\documentclass[12pt,a4paper,reqno]{amsart}
\usepackage{graphicx}
\usepackage{stmaryrd} 
\usepackage{geometry}
\usepackage{amssymb}
\usepackage{etex}
\usepackage{pifont}
\usepackage{amssymb,amsmath,mathrsfs,graphicx,float,mathdots}
\usepackage[T1]{fontenc}
\usepackage{extarrows}
\usepackage[all]{xy}
\usepackage[english,francais]{babel}
\usepackage{rotating}
\usepackage{hyperref}
\usepackage{pdfpages}
\usepackage{cellspace}
\usepackage{txfonts}
\usepackage{tikz}
\usetikzlibrary{calc}
\tikzstyle{marron}=[fill=brown!65!black,color=brown!65!black]
\tikzstyle{rouge}=[fill=red,color=red]
\tikzstyle{bleu}=[fill=blue!50!black,color=blue!50!black]
\tikzstyle{vert}=[fill=green!50!black,color=green!50!black]

\theoremstyle{plain}
\newtheorem{thm}{Theorem}[section]
\newtheorem{pro}[thm]{Proposition}
\newtheorem{lem}[thm]{Lemma}
\newtheorem{cor}[thm]{Corollary}

\newtheorem{theoalph}{Theorem}

\theoremstyle{definition}
\newtheorem{defi}[thm]{Definition}

\newtheorem{rem}[thm]{Remark}
\newtheorem{rems}[thm]{Remarks}

\def\og{\leavevmode\raise.3ex\hbox{$\scriptscriptstyle\langle\!\langle$~}}
\def\fg{\leavevmode\raise.3ex\hbox{~$\!\scriptscriptstyle\,\rangle\!\rangle$}}

\addtocounter{section}{0}             
\numberwithin{equation}{section}       

\begin{document}
\selectlanguage{english}

\title{Infinitesimal deformations of rational surface automorphisms}
\thanks{Author supported by ANR Grant ``BirPol" ANR-11-JS01-004-01, ANR Grant ``MicroLocal" ANR-15-CE40-0007 and ANR Grant ``HodgeFun" ANR-16-CE40-0011.}

\author{Julien Grivaux}

\address{
CNRS, I2M (Marseille) \& IH\'{E}S
}

\email{jgrivaux@math.cnrs.fr}

\begin{abstract}
If $X$ is a rational surface without nonzero holomorphic vector field and $f$ is an automorphism of $X$, we study in several examples the Zariski tangent space of the local deformation space of the pair $(X, f)$.

\end{abstract}

\maketitle

\tableofcontents

\section{Introduction}
Biregular automorphisms of rational surfaces with positive topological entropy present a major interest in complex dynamics (\emph{see} the recent survey \cite{Cantat}) but their construction remains still a difficult problem of algebraic geometry. For an overview of this problem, we refer to \cite{DesertiGrivaux} and to the references therein. It is a paradoxal fact that these automorphisms, although hard to construct, can occur in holomorphic families of arbitrary large dimension, as shown recently in \cite{BK1}. Besides, the automorphism group of a given rational surface can carry many automorphisms of positive entropy, \textit{see} \cite{Blanc2} and \cite{BlancDeserti} for recent results on this topic. In this paper we study deformations of families of rational surface automorphisms using deformation theory (this was initiated in \cite{DesertiGrivaux}), and investigate a great number of examples.
\par \medskip
Let us first describe the general setup. For the basic definitions concerning deformation theory, we refer the reader to \S \ref{pfff}. If $X$ is a complex compact manifold, Kuranishi's theorem shows the existence of a semi-universal deformation $(\mathscr{K}, \mathcal{B}_X)$ of the manifold $X$, which means that any local deformation of $X$ can be obtained by pulling back $\mathscr{K}$ by a germ of holomorphic map whose differential at the origin is unique. The Zariski tangent space at the marked point of $\mathcal{B}_X$ identifies canonically with $\mathrm{H}^ 1(X, \mathrm{TX})$.
The space $\mathcal{B}_X$ is singular if and only if there are obstructed first order deformations, that is elements in $\mathrm{H}^ 1(X, \mathrm{T}X)$, or equivalently deformations over the double point, that cannot be lifted to deformations over a smooth base.
\par \medskip
After the initial works of Kodaira, Spencer, Kuranishi, Horikawa, and others, deformation theory has been developed in an abstract categorical formalism mainly by Grothendieck, Artin and Schlessinger in order to cover a wide range of situations (deformations of manifolds or schemes with extra additional structure such as marked points or level structures, deformations of submanifolds, deformations of morphisms, deformations of representations, and so on) in a unified way. 
\par \medskip
In the present paper, we develop this theory for pairs $(X, f)$ where $X$ is a complex compact manifold and $f$ is a biholomorphism of $X$. Assuming that $X$ carries no nonzero holomorphic vector field (this guarantees that the Kuranishi family is universal), $f$ acts naturally on the Kuranishi space $\mathcal{B}_X$. Then the restriction of the Kuranishi family $\mathscr{K}$ to the fixed locus $Z_f$ of this action is universal for the deformation functor of pairs $(X, f)$. Besides the Zariski tangent space of $Z_f$ at the origin identifies with fixed vectors in $\mathrm{H}^1(X, \mathrm{T}X)$ under the action of $f_*$. The number of moduli of $(X, f)$, that  can be thought intuitively as the maximum number of parameters of nontrivial deformations of $(X, f)$, is the dimension of $Z_f$. Knowing the action to $f_*$ on $\mathrm{H}^1(X, \mathrm{T}X)$ we deduce a bound
\begin{equation} \label{sharp}
\mathrm{m}(X, f) \leq \mathrm{dim} \,\mathrm{ker}\, (f_*-\mathrm{id})
\end{equation}
for the number of moduli of $(X, f)$, with equality if and only if $Z_f$ is smooth. To get a finer geometric picture, we deal separately with the different possibilities:
\begin{enumerate}
\item[--] If $f_*$ has no nonzero fixed vector, then $Z_f$ is a point, which means that there exists no nontrivial deformations of $(X, f)$ over any base (reduced or not). In particular $(X, f)$ is rigid.
\item[--] If $\mathrm{ker}\, (f_*-\mathrm{id})$ is nonzero, assume that we can produce a deformation $(\mathfrak{X}, f)$ of $(X, f)$ over a smooth base $(B,b)$ whose Kodaira-Spencer map from $\mathrm{T}_b B$ to $\mathrm{ker}\, (f_*-\mathrm{id})$ is surjective. Then $Z_f$ is smooth, and $(\mathfrak{X}, \mathfrak{f})$ is complete (which means that pullbacks of this deformation encode all local deformations of $(X, f)$ over any base).
\item[--] If $\mathrm{ker}\, (f_*-\mathrm{id})$ is nonzero but we don't know any specific deformation of the pair $(X, f)$ over a smooth base, then everything can a priori happen concerning the dimension of $Z_f$, the bound \eqref{sharp} is the better estimate that can be obtained. In particular, nothing prevents the reduced complex space $Z_f^{\mathrm{red}}$ to be a point; in this case, $(X, f)$ is also rigid.
\end{enumerate}
The main difficulty in order to apply \eqref{sharp} in practical examples is to compute the action of $\mathrm{Aut}\,(X)$ on $\mathrm{H}^1\big(X, \mathrm{T}X\big)$, and this is far more delicate than the action of~$\mathrm{Aut}(X)$ on the Neron-Severi group of $X$. The reason for this is that the first action is not defined for birational morphisms, whereas the second is. Our purpose in this article is to compute this action in various examples for rational surface automorphisms.
The first case we deal with is the case of rational surfaces carrying a reduced effective anticanonical divisor. The most significant result we obtain (\textit{cf.} Theorem \ref{cusp}), is:

\begin{theoalph}  \label{anticanonique}
Let $X$ be a basic rational surface with $\mathrm{K}_X^2<0$ endowed with an automorphism $f$, and assume that $|-\mathrm{K}_X|=\{C\}$ for an irreducible curve $C$ (such a curve is automatically $f$-invariant). Let $P_f$ and $\theta_f$ denote the characteristic polynomials of $f^*$ acting on $\mathrm{NS}_{\mathbb{Q}}(X)$ and $\mathrm{H}^0\big(C, \mathrm{N}^*_{C/X}\big)$ respectively, and let $Q_f$ be the characteristic polynomial of $f_*$ acting on $\mathrm{H}^1\big(X, \mathrm{T}X\big)$. Lastly, let $a_f$ be the multiplier\footnote{The multiplier of $f$ is the action of $f$ on the complex line $\mathrm{H}^0 (C, \omega_C)$.} of $f$. Then:
\par \smallskip
\begin{itemize}
\item[(i)] If $C$ is cuspidal, then 
\[
Q_f(x)=\frac{P_f(x)\, \theta_f(x)\, (x-a_f^{-5}) (x-a_f^{-7})}{(x-1)(x-a_f^{-1})} \cdot
\]
\par \smallskip
\item[(ii)] If $C$ is smooth,
\[
Q_f(x)=\dfrac{P_f(x) \, \theta_f(x) \, (x-a_f)}{x-1} \cdot
\] 
\end{itemize}
\end{theoalph}
We also prove an analogous statement if the effective anticanonical divisor is a reduced cycle of rational curves. It is worth saying that Theorem \ref{anticanonique} doesn't seem to generalise easily to the case of plurianticanonical divisors. Even for Coble surfaces (that is $|-\mathrm{K}_X|$ is empty but $|-2\mathrm{K}_X|$ is not), computing the action of the automorphisms group on the vector space of infinitesimal deformations seems a non-trivial problem. 
\par \medskip
We give three applications of Theorem \ref{anticanonique}. The first one deals with quadratic birational transformations of the projective plane leaving a cuspidal cubic curve globally invariant. These examples were introduced independently by McMullen \cite{MC} and Bedford-Kim \cite{BK} (\emph{see} \cite{Di} for a unified approach). The result we get is the following (for the definitions of the orbit data and of the polynomial  $P_{\tau}$, we refer the reader to \S \ref{McCormick}):
\begin{theoalph} \label{hebinoui}
Let  $(\tau, n_1, n_2, n_3)$ be an admissible orbit data, and let $\mu$ be a root of $P_{\tau}$ that is not a root of unity. Let $f$ be a birational quadratic map realizing the orbit data, fixing the cuspidal cubic $\mathscr{C}$, and having multiplier $\mu$ when restricted to the cubic $\mathscr{C}$. If $X$ is the corresponding rational surface and $g$ is the lift of $f$ as an automorphism of $X$, then $\mathrm{m}(X, g)\leq 3-|\tau|$. In particular, if $\tau$ has order three, $g$ is rigid.
\end{theoalph} 
The result appears in the paper as Theorem \ref{canada}. In some sense this result is remarkable because even if the number of the blowups $n_1+n_2+n_3$ in these examples can be arbitrarily large, the number of moduli remains uniformly bounded (and is sometimes zero). Up to our knowledge, this yields the first known examples of rigid rational surface automorphisms with positive topological entropy. Next, we discuss examples of automorphisms fixing a smooth elliptic curve which are produced in \cite{Blanc} by a classical construction that appears in \cite{Gizz}. The result we get, corresponding to Theorem \ref{moose} in the paper, is:
\begin{theoalph}
Let $p$, $q$, $r$ be three pairwise distinct points on smooth cubic curve $\mathscr{C}$ in $\mathbb{P}^2$, let $\sigma_p$, $\sigma_q$, $\sigma_r$ be the three birational involutions fixing pointwise $\mathscr{C}$ given by the Blanc-Gizatullin construction, and let $X$ be the corresponding blowup of $\mathbb{P}^2$ along $15$ points on which $\sigma_p$, $\sigma_q$ and $\sigma_r$ lift to automorphisms. If  $\psi$ is the lift of $\sigma_p \circ  \sigma_q \circ  \sigma_r$ to $X$, then every deformation of the pair $(X, \psi)$ is obtained by deforming the cubic $\mathscr{C}$ and the points $p, q$ and $r$ on it.
\end{theoalph}

Lastly, we discuss deformations of automorphisms of unnodal Halphen surfaces. The results in this section are certainly  well-known to experts (\emph{see e.g.} \cite[\S 2]{CD}).
\par \medskip
Using different techniques, we study a particular class of examples which admits plurianticanonical divisors: rational Kummer surfaces. Let $\mathcal{E}$ be the elliptic curve obtained by taking the quotient of the complex line $\mathbb{C}$ by the hexagonal lattice $\Lambda=\mathbb{Z}[\mathbf{j}]$, where $\mathbf{j}^3=1$. The group $\mathrm{GL}(2;\Lambda)$ acts linearly on the complex plane and preserves the lattice $\Lambda\times\Lambda;$ therefore any element $M$ of $\mathrm{GL}(2;\Lambda)$ induces an automorphism $f_M$ on~$\mathcal{A}=\mathcal{E}\times \mathcal{E}$ that commutes with the automorphism $\phi$ defined by $\phi(x,y)=(\mathbf{j}x,\mathbf{j}y)$. The automorphism~$f_M$ induces an automorphism $\varphi_M$ on the desingularization $X$ of~$\mathcal{A}/\langle\phi\rangle$. The surface $X$ is called a rational Kummer surface, it can be explicitly obtained by blowing up a very special configuration of $12$ points in $\mathbb{P}^2$. By means of two different approaches, one using the Atiyah-Bott fixed point theorem and the other one using classical techniques of sheaf theory, we compute the action of $\mathrm{GL}(2;\Lambda)$ on $\mathrm{H}^1(X, \mathrm{T}X)$. This gives again new examples of rigid rational surface automorphisms with positive topological entropy.
\begin{theoalph}
For any matrix $M$ of infinite order in $\mathrm{GL}(2; \mathbb{Z}[\mathbf{j}])$, the automorphism $\varphi_M$ is rigid.
\end{theoalph}
We refer to Theorem \ref{yahou} for a more detailed statement. Note that the link between $\mathbb{P}^2$ and $\mathcal{A}$ via the surface $X$ has already been very fruitful in foliation theory (\emph{see}  \cite{CantatFavre} and \cite{LinsNeto, Puchuri}). To contribute to this dictionary, we provide an explicit description of the  birational map of $\mathbb{P}^2$ induced by $\varphi_M$ after blowing down twelve exceptional curves in $X$. The methods works in a similar manner in the case of the square lattice $\Lambda=\mathbb{Z}[\mathbf{i}]$, $(\mathbf{i}^2=-1)$ with slightly different results.
\par \medskip
In the last part of the paper, we address the following problem of effective algebraic geometry: if $g$ is an explicit Cremona transformation (that is given by three homogeneous polynomials of the same degree without common factor) that lifts to an automorphism $f$ of a rational surface $X$ after a finite number of blowups, how to compute the action of $f_*$ on $\mathrm{H}^1(X, \mathrm{T}X)$? Although some rational surface automorphisms have a purely geometric construction (like Coble's automorphisms for instance), others don't. A typical example is given by the automorphisms constructed in \cite{BK1} and then in \cite{DesertiGrivaux}, which are given by their analytic form. Many examples introduced by physicists being also of this type, it seems necessary to develop specific methods to deal with this problem in order to give a complete picture of the subject. The strategy for solving this problem is long and not particularly easy to grasp because we must develop some specific machinery in order to construct explicit bases of $\mathrm{H}^1(X, \mathrm{T}X)$. To help the reader understand how the algorithm works, we carry out completely the computation in one specific example, which is the one constructed in \cite[Thm 3.5]{DesertiGrivaux}; it is an automorphism obtained by blowing up $15$ successive points of $\mathbb{P}^2$, which correspond to three infinitely near points of length five. This is one the most simple possible example we know with iterated blowups.
\par \medskip
\textbf{Acknowledgments} The author would like to thank Julie D\'{e}serti for many discussions, Igor Dolgachev and Laurent Meerseman for useful comments, Philippe Goutet for the nice pictures and the LaTex editing of the Maple files; and lastly the anonymous referee for his very careful reading and his numerous remarks and comments that led to a considerable improvement of the paper.
\section{Preliminary results in deformation theory} \label{pfff}
For general background on the theory of deformations of complex compact manifolds, we refer to the references \cite{Kodaira},\cite{Meersseman}, \cite{balaji_introduction_2010}, and also \cite{Sernesi} for the algebraic setting.
\subsection{Background}
\subsubsection{Deformations and Kuranishi space} \label{rond} $ $ \par \medskip
-- If $X$ is a smooth complex compact manifold, a (local) deformation\footnote{For some authors, what we call \textit{deformation} is called \textit{marked deformation}, because we specify the isomorphism between the central fiber and $X$.} of $X$ is an equivalence class of cartesian diagram
\[
\xymatrix{X \ar[r]  \ar[d] & \mathfrak{X} \ar[d]^-{\pi} \\
\{b\} \ar[r] & B
}
\]
where $(B, b)$ is a germ of pointed complex space, $\mathfrak{X}$ is a complex space, and $\pi$ is is a flat and proper holomorphic morphism. In other words, the deformation is given by the morphism $(\mathfrak{X}, \pi)$ over the marked base $(B, b)$ together with a specific identification between the central fiber $\mathfrak{X}_b$ and $X$. \par \medskip
Every complex manifold defines a contravariant deformation functor 
\[
\mathrm{Def}_X \colon \{ \textrm{germs of marked complex spaces} \} \longrightarrow \mathbf{Set}
\]
given by
\[
\mathrm{Def}_X(B, b)=\left\{\,\textrm{deformations of}\,\, X\,\, \textrm{over} \,\,(B, b)\, \right\} / \,\textrm{isomorphism}.
\]
\par \medskip
-- If $(\mathfrak{X}_1, \pi_1)$ and $(\mathfrak{X}_2, \pi_2)$ are deformations of a complex manifold $X$ over the same marked base $B$, a morphism between these two deformations if a holomorphic map $f \colon \mathfrak{X}_1 \rightarrow \mathfrak{X}_2$ such that $\pi_2 \circ f = \pi_1$, and such that the diagram
\[
\xymatrix{(\mathfrak{X}_1)_b \ar[d]_-{\sim} \ar[r]^-{f_b}  \ar[d] & (\mathfrak{X}_2)_b \ar[d]^-{\sim} \\
X \ar[r]^-{\mathrm{id}} & X
}
\]
commutes.
\par \medskip
-- Recall that there is a one to one correspondance between isomorphic classes of complex spaces whose associated reduced space is a point and local artinian $\mathbb{C}$-algebras, that is local $\mathbb{C}$ -algebras that are also finite-dimensional $\mathbb{C}$-vector spaces. This correspondence is obtained by attaching to such a complex space $B$ the algebra $\mathcal{O}(B)$ of global holomorphic functions on $B$. The inverse mapping is the map $R \rightarrow (\mathrm{spec}\,R)^{\mathrm{an}}$, where $\mathrm{spec}\,R$ is the complex algebraic scheme attached to $R$ and $Z \rightarrow Z^{\mathrm{an}}$ is the functor that associates to a complex algebraic scheme its analytification, which is a complex space.
\par \medskip
-- An infinitesimal deformation of $X$ is a deformation of $X$ over a complex space $B$ whose associated reduced complex space is the marked point $b$. Concretely, if $B=(\mathrm{spec}\, R)^{\mathrm{an}}$, a deformation of $X$ over $B$ is the data of a complex space $(X, \mathcal{O}_{\mathfrak{X}})$ on $X$ of flat $R$-algebras together with a specific isomorphism $\mathcal{O}_{\mathfrak{X}}/ \mathfrak{m} \mathcal{O}_{\mathfrak{X}} \simeq \mathcal{O}_{X}$ of sheaves of $\mathbb{C}$-algebras, where $\mathfrak{m}$ denotes the maximal ideal of the local artinian algebra $R$.
\par \medskip
-- A first-order deformation of $X$ is a deformation of $X$ over the double point $(\mathrm{spec}\,\mathbb{C}[t]/t^2)^{\mathrm{an}}$. There is a natural isomorphism between isomorphism classes of first-order deformations of $X$ and $\mathrm{H}^1(X, \mathrm{TX})$ \cite[Proposition 6.2.10]{Huybrechts}.
\par \medskip
-- Let $(\mathfrak{X}, \pi, B)$ de a deformation of $X$. The Kodaira-Spencer map of $\mathfrak{X}$ is a linear map $\mathrm{KS}(\mathfrak{X})$ from $\mathrm{T}_b B$ to $\mathrm{H}^1({X}, \mathrm{T}{X})$. It admits the following intrinsic description: the set of morphisms of complex spaces from $(\mathrm{spec}\,\mathbb{C}[t]/t^2)^{\mathrm{an}}$ to $B$ are exactly the points of the tangent bundle $\mathrm{T} B$. For any vector $v$ in $\mathrm{T}_b B$, the element $\mathrm{KS}_b(\mathfrak{X})(v)$ is exactly the class of the first-order deformation $
\mathfrak{X} \times _{B} (\mathrm{spec}\,\mathbb{C}[t]/t^2)^{\mathrm{an}}$ of $X$, the map from $(\mathrm{spec}\,\mathbb{C}[t]/t^2)^{\mathrm{an}}$ to $B$ being given by the tangent vector $v$.
\par \medskip
-- A deformation of the manifold $X$ is called universal (resp. semi-universal\footnote{Some authors use the terminology \textit{versal}.}, resp. complete) if any deformation of $X$ is the pullback of this deformation under a germ of holomorphic map that is unique (resp. whose differential at the marked point is unique, resp. without any further conditions).
\par \medskip
-- A deformation over a smooth base is complete if and only if its Kodaira-Spencer map is surjective (Kodaira's completeness theorem \cite{KodairaOriginal}, \textit{see} \cite[Theorem 6.1]{Kodaira}). Besides, the Kodaira-Spencer map of a semi-universal deformation is (almost by definition) an isomorphism.
\par \medskip
-- A complex manifold is rigid if all fibers of a deformation of $X$ over a smooth base are biholomorphic to $X$. Thanks to a theorem of Grauert and Fischer \cite{GF}, this is equivalent of saying that every deformation of $X$ over a smooth base is locally trivial.
\par \medskip
-- Any complex manifold admits a semi-universal deformation (Kuranishi's theorem); its base $\mathcal{B}_X$ is called the Kuranishi space and the deformation $\mathscr{K}$ is called the Kuranishi family. The Zariski tangent space at the origin of $\mathcal{B}_X$ is $\mathrm{H}^ 1(X, \mathrm{TX})$.
\par \medskip-- If $X$ has no nonzero holomorphic vector fields on $X$ and if $\mathfrak{X}$ is any deformation of $X$, then the group of automorphisms of $\mathfrak{X}$ is trivial (we provide some details about the history of this result as well as a proof in Appendix \ref{thermidor}). In particular, the Kuranishi family is universal.
\par \medskip
-- Another way of stating Kuranishi's theorem is that the deformation functor $\mathrm{Def}_X$ attached to $X$ is quasi-representable (\textit{see} \cite[Def. 1.2]{Wavrik}) by the Kuranishi family $(\mathcal{B}_X,\mathscr{K})$. If $X$ carries no nonzero holomorphic vector field, then $\mathrm{Def}_X$ is represented by $\mathcal{B}_X$ (\textit{i.e.} the Kuranishi family is universal).
\par \medskip
-- The space $\mathcal{B}_X$ is in general neither reduced nor irreducible (for pathologies, \textit{see} \cite{Rollenske}). If $\mathcal{B}_X$ is smooth, then we say that $X$ is unobstructed. This is in particular the case when the cohomology group $\mathrm{H}^2(X, \mathrm{T}X)$ vanishes, thanks to Kodaira's existence theorem \cite[Thm. 5.6 pp. 270]{Kodaira}.
\par \medskip
 -- The number of moduli of $X$, denoted by $\mathrm{m}(X)$, is the maximum of the dimensions of the irreducible components of $\mathcal{B}_X$. The number $\mathrm{m}(X)$ is in $\llbracket 0, \mathrm{h}^1(X, \mathrm{T}X) \rrbracket$. Besides:
\begin{center}
$\mathrm{m}(X)=0 \Leftrightarrow \mathcal{B}_X^{\mathrm{red}}=\{0\} \Leftrightarrow X$ is rigid. \\
$\mathrm{m}(X)=\mathrm{h}^1(X, \mathrm{T}X) \Leftrightarrow \mathcal{B}_X$ is smooth $\Leftrightarrow X$ is unobstructed.
\end{center}

\subsubsection{Families of complex manifolds}
In this section, we recall briefly some of the material formerely introduced in \cite[\S 5]{DesertiGrivaux} concerning the generic number of parameters of a family of complex manifolds, and relate it to the number of moduli.
\par \medskip
-- A family of complex manifolds is a triplet $(\mathfrak{X}, \pi, B)$ where $\pi \colon \mathfrak{X} \rightarrow B$ is a proper submersion between smooth complex manifolds.
\par \medskip
-- For any point $b$ in $B$, the family $\mathfrak{X}$ induces a deformation of $\mathfrak{X}_b$. We denote by $\mathrm{KS}_b(\mathfrak{X})$  the corresponding Kodaira-Spencer map, which is a linear map from $\mathrm{T}_b B$ to $\mathrm{H}^1(\mathfrak{X}_b, \mathrm{T} \mathfrak{X}_b)$.
\par \medskip
-- The function $b \rightarrow \mathrm{rank} \, \{\mathrm{KS}_b(\mathfrak{X})\}$ is generically constant on $B$. If $\mathfrak{X}$ is an algebraic family, this is proved in \cite[Proposition 5.5]{DesertiGrivaux}. For arbitrary deformations, it follows from \cite[Satz 7.7(1)]{Flenner} that the function $b \rightarrow \mathrm{h}^1(\mathfrak{X}_b, \mathrm{T} \mathfrak{X}_b)$ is constructible, in particular it is generically constant. Then the argument of \textit{loc. cit.} applies.
\par \medskip
-- The generic number of parameters of a family $(\mathfrak{X}, \pi, B)$, denoted by $\mathfrak{m}(\mathfrak{X})$, is the generic rank of the function $b \rightarrow \mathrm{rank} \, \{\mathrm{KS}_b(\mathfrak{X})\}$.
\par \medskip
-- Let $X$ be a complex compact manifold without nonzero holomorphic vector field. By the semicontinuity theorem \cite[Thm. 7.8]{Kodaira}, the fibers of the Kuranishi family have no nonzero vector fields either. Then it follows from \cite[Corollary 1]{Meersseman} that the Kuranishi family is semi-universal (and even universal) at any point of the base $\mathcal{B}_X$.  As a corollary we get the following important result: 
\begin{pro} \label{dontmove}
Let $X$ be a complex manifold without nonzero holomorphic vector field. Then the number of moduli $\mathrm{m}(X)$ of $X$ is the supremum of $\mathfrak{m}(\mathfrak{X})$ where $\mathfrak{X}$ runs through deformations over $X$ over a smooth base.
\end{pro}

\begin{proof}
Let $\mathfrak{X}$ be a deformation of $X$ over a smooth base $B$. 
We can write $\mathfrak{X}$ as $\varphi^*\mathscr{K}$ where $\varphi \colon B \rightarrow \mathcal{B}_X$ is a germ of holomorphic map. Since $\mathscr{K}$ is semi-universal at all points of the base, this implies that $\mathfrak{m}(\mathfrak{X})$ is the generic rank of $\varphi$. Since the base of $\mathfrak{X}$ is smooth, we get the inequality $\mathfrak{m}(\mathfrak{X}) \leq \mathrm{m}(X)$. To prove the equality, let $Z$ be an irreducible component of maximal dimension of $\mathcal{B}_X$ and take for $\varphi$ the composition $B \rightarrow Z^{\mathrm{red}} \rightarrow Z$ where $Z^{\mathrm{red}}$ is the reduction of $Z$, and the first morphism is a resolution of singularities. Then $\mathfrak{m}(\varphi^* \mathscr{K})=\mathrm{dim}\,Z^{\mathrm{red}}=\mathrm{m}(X)$.
 
\end{proof}

\subsection{Deformations of automorphisms}

\subsubsection{Setting}

Let $X$ be a complex compact manifold without nonzero holomorphic vector field, and let $f$ be a biholomorphism of $X$. A deformation of the pair $(X, f)$ is an equivalence class of cartesian diagram
\[
\xymatrix{(X, f) \ar[r] \ar[d] & (\mathfrak{X}, \mathfrak{f}) \ar[d]^-{\pi} \\
\{b\} \ar[r] & B
}
\]
where $(B, b)$ is a germ of marked complex space, $\pi$ is flat and proper, $\mathfrak{f}$ is a biholormorphism of $\mathfrak{X}$ commuting with $\pi$, and the top horizontal arrow commutes with the automorphisms. There is also a deformation functor $\mathrm{Def}_{(X, f)}$ from germs of marked complex spaces to sets encoding the deformations of $(X, f)$ modulo isomorphisms.
\begin{lem} \label{hop}
For any marked base $(B, b)$, the natural map
$\mathrm{Def}_{(X, f)}(B, b) \rightarrow \mathrm{Def}_X (B, b)$ is injective.
\end{lem}
\begin{proof}
We must prove that if $(\mathfrak{X}, \mathfrak{f})$ is in $\mathrm{Def}_X (B, b)$, then $\mathfrak{f}$ is uniquely determined by the deformation $\mathfrak{X}$. Taking two possible biholomorphisms $\mathfrak{f}$ and $\mathfrak{f}'$, $\mathfrak{f}' \circ \mathfrak{f}^{-1}$ is an automorphism of $\mathfrak{X}$, so it is the identity morphism since we have assumed that $X$ carries no nonzero holomorphic vector field (\textit{see} \S \ref{rond} and Appendix \ref{thermidor}).
\end{proof}

If $(\mathfrak{X}, \mathfrak{f})$ is a deformation of a pair $(X, f)$, then we have $(\mathfrak{f}_b)_* \circ \mathrm{KS}_b=\mathrm{KS}_b$, so that the image of $\mathrm{KS}_b$ is contained in $\ker (f_*- \mathrm{id})$. It is therefore natural to define the Kodaira-Spencer map of the pair $(X, f)$ as the unique map 
\[
\mathrm{KS}_b (\mathfrak{X}, \mathfrak{f}) \colon \mathrm{T}_b B \longrightarrow \mathrm{ker} \, (f_*-\mathrm{id})
\]
such that the composition $\mathrm{T}_b B \xrightarrow{\mathrm{KS_b(\mathfrak{X}, \mathfrak{f})}} \mathrm{ker} \, (f_*-\mathrm{id}) \hookrightarrow \mathrm{H}^1(X, \mathrm{T}X)$ is $\mathrm{KS}_b(\mathfrak{X})$.

\begin{defi}
We say that $(X, f)$ is \textit{rigid} if for any deformation $(\mathfrak{X}, \mathfrak{f})$ of $(X, f)$ over a smooth base $B$, for any $b'$ in $B$, $(\mathfrak{X}_{b'}, \mathfrak{f}_{b'})$ is biholomorphic to $(X, f)$.
\end{defi}
Remark that $(X, f)$ is rigid if and only if any deformation of $(X, f)$ over a smooth base is locally trivial. Indeed, if $(\mathfrak{X}, \mathfrak{f})$ is such a deformation, all fibers of $\mathfrak{X}$ are biholomorphic so using the Fischer-Grauert theorem \cite{GF}, $\mathfrak{X}$ is locally trivial. Lemma \ref{hop} implies that $(\mathfrak{X}, \mathfrak{f})$ is also locally trivial. 

\subsubsection{The invariant locus}
In this section, we fix a complex compact manifold $X$ without nonzero holomorphic vector field. For any deformation $\mathfrak{X}$ of $X$, we have a new deformation $\mathfrak{X}^{f}$ of $X$ obtained from $\mathfrak{X}$ by pre-composing the deformation $\mathfrak{X}$ with $f$, that is
by considering the diagram
\[
\xymatrix{
X \ar[r]^{f} \ar[d] & X \ar[r] \ar[d] & \mathfrak{X} \ar[d] & \\
\{b\} \ar@{=}[r] & \{ b \} \ar[r] & B
}
\]
Remark that if $(\mathfrak{X}, \mathfrak{f})$ is in $\mathrm{Def}_{(X, f)}$, then $\mathfrak{f} \colon \mathfrak{X} \rightarrow \mathfrak{X}^f$ is an isomorphism of deformations. Since $\mathscr{K}$ is universal, there exists a unique germ of biholomorphism $\varphi_f$ fixing the marked point $b$ such that $\mathscr{K}^{f} \simeq \varphi_f^*\mathscr{K}$ as deformations of $X$\footnote{For semi-universal deformations $\varphi_f$ is not unique anymore, and this can be the source of many problems. \emph{See} \cite{Rim} and \cite{Siebert} for further details.}. This means that there exists a biholomorphism $\widehat{f}$ of $\mathscr{K}$ such that the diagrams
\[
\xymatrix{
\mathscr{K} \ar[r]^-{\widehat{f}} \ar[d] &\mathscr{K} \ar[d] \\
B \ar[r]^-{\varphi_f} & B
} \qquad \qquad
\xymatrix{\mathscr{K} \ar[r]^-{\widehat{f}}&\mathscr{K}\\
X \ar[u] \ar[r]^-{f} & X \ar[u]}
\]
commute.
\begin{defi}
Given a pair $(X, f)$, if $\mathcal{B}_X$ is the Kuranishi space of $X$, we define the $f$-invariant locus of $X$ as the subscheme $Z_f$ of $\mathcal{B}_X$ defined as the pullback (as a complex space) of the diagonal of $\mathcal{B}_X$ under the map $(\mathrm{id}, \varphi_f) \colon \mathcal{B}_X \rightarrow \mathcal{B}_X \times \mathcal{B}_X$. We also define the number of moduli $\mathrm{m}(X, f)$ of $(X, f)$ as the maximum of the dimensions of the irreducible components of $Z_f$.

\end{defi}
If we identify $\mathrm{T}_b \mathcal{B}_X$ with $\mathrm{H}^1(X, \mathrm{T}X)$, then $
\mathrm{T}_b Z_f$ identifies with $\mathrm{ker}\, (f_*-\mathrm{id}).
$
Taking the pullback of the above left diagram under the inclusion $Z_f \hookrightarrow \mathcal{B}_X$, we get a diagram
\[
\xymatrix{
\mathscr{K}_{|Z_f} \ar[rr]^-{} \ar[rd] &&\mathscr{K}_{|Z_f} \ar[ld] \\
& Z_f &
}
\]
Hence $(\mathscr{K}_{|Z_f}, \widehat{f})$ is an element of $\mathrm{Def}_{(X,f)}$.
\begin{pro} \label{bienjoue}
Let $X$ be a complex compact manifold with no holomorphic vector field, let $f$ be a biholomorphism of $X$, let $\mathcal{B}_X$ the Kuranishi space of $X$, and let $Z_f$ be the $f$-invariant locus. Then the element $(\mathscr{K}_{|Z_f}, \widehat{f})$ is universal for the functor $\mathrm{Def}_{(X, f)}$.
\end{pro}

\begin{proof}
Let $(\mathfrak{X}, \mathfrak{f})$ be a deformation of $(X, f)$ over a germ of marked complex space $(M, m)$. Since the deformation $\mathscr{K}$ is universal, there exists a unique germ $\psi \colon (M, m) \rightarrow (\mathcal{B}_X, b)$ of holomorphic map such that $\mathfrak{X} \simeq \psi^*\mathscr{K}$. Now $\mathfrak{f} \colon \mathfrak{X} \rightarrow \mathfrak{X}^f$ is an isomorphism of deformations. This means that $ \psi^*\mathscr{K}$ and $(\varphi_f \circ \psi)^*\mathscr{K}$ are isomorphic. Hence that $\varphi_f \circ \psi=\psi$ (as $\mathscr{K}$ is universal), so that the map $\psi$ factors through $Z_f$. We conclude that the deformations $\mathfrak{X}$ and $\psi^*\mathscr{K}_{|Z_f}$ are isomorphic. Thanks to Lemma \ref{hop}, $(\mathfrak{X},  \mathfrak{f})$ is isomorphic to $(\psi^* \mathscr{K}_{| Z_f} , \widehat{f})$.
\end{proof}

\begin{rems} \label{rem:hihihi}$ $ \par
\begin{enumerate}
\item[(i)] As usual in deformation theory, the singularities of the parameter space (here $Z_f$) correspond to obstructed deformations.  Let us sketch a way to construct explicitly examples with $Z_f$ singular. Assume that we are given a pair $(X, f)$ such that
\[
\begin{cases}
\mathrm{H}^0(X, \mathrm{T}_X)=\mathrm{H}^1(X, \mathrm{T}_X)=\{0\}. \\
\textrm{The scheme} \,\, \mathrm{Fix}\,(f):=\{ x \,\, \textrm{such that} \,\, f(x)=x\}\,\,\textrm{is not smooth.}
\end{cases}
\]
Let $p$ denote a point where $\mathrm{Fix}\,(f)$, and consider the manifold $Y$ obtained by blowing up the manifold $X$ at $p$. It can be shown (\textit{see} \cite[Proposition 5.3]{DesertiGrivaux}) that the Kuranishi space $\mathcal{B}_Y$ can be identified with an open neighborhood of $p$, the Kuranishi family being given by $\mathcal{K}_q=\mathrm{Bl}_q(X)$. The automorphism $f$ lifts to an automorphism of $Y$, and the action of $f$ on $\mathcal{B}_Y$ is simply the action of $f$ in a neighborhood of $p$. Hence the fixed locus $Z_f$ identifies with the scheme $\mathrm{Fix}\,(f)$.
\par \medskip
\item[(ii)] If $f$ is linearizable (\emph{e.g.} $f$ is of finite order), then $Z_f$ is smooth. 
\end{enumerate}
\end{rems}
The following theorem is a generalisation of Kodaira's completeness theorem \cite[Thm. 6.1]{Kodaira} for deformations of automorphisms. Its proof is a direct consequence of a general result on complete families for semi-universal deformation functors, due to Wavrik \cite[Thm 1.8]{Wavrik}. 
\begin{thm}[Completeness theorem] \label{thm:kodaira}
Let $X$ be a complex compact manifold without nonzero holomorphic vector field, and let $f$ be an automorphism of $X$. Consider a deformation $(\mathfrak{X}, \mathfrak{f})$ of the pair $(X, f)$ over a smooth base $(B, b)$, and assume that $\mathrm{KS}_{b}(\mathfrak{X}, \mathfrak{f})$ is surjective. Then $Z_f$ is smooth, and $(\mathfrak{X}, \mathfrak{f})$ is complete at $b$.
\end{thm}

\begin{proof}
By Proposition \ref{bienjoue}, we can write $\mathfrak{X}=\psi^* \mathscr{K}$ where $\psi \colon B \rightarrow Z_f$ is holomorphic. Since the differential $\mathrm{d} \psi_b \colon \mathrm{T}_b B \rightarrow \mathrm{T}_{\psi(b)} Z_f $ is surjective, $Z_f$ is smooth and $\psi$ is a submersion. Hence $\psi$ admits locally a right inverse $\chi \colon Z_f \rightarrow B$. This gives
$\chi^* \mathfrak{X}=\mathscr{K}_{| Z_f}$, and we conclude using Lemma \ref{hop}.
\end{proof}
Lastly, we provide a concrete interpretation of the number of moduli of a pair $(X, f)$, whose proof is entirely similar to the proof of Proposition \ref{dontmove}.
\begin{pro}
Let $X$ be a complex compact manifold without nonzero holomorphic vector field, and let $f$ be a biholomorphism of $X$. Then $\mathrm{m}(X, f)$ is the supremum of $\mathfrak{m}(\mathfrak{X})$ where $\mathfrak{X}$ runs through deformations of $X$ over a smooth base such that $f$ extends to an automorphism of the deformation $\mathfrak{X}$. Besides, $\mathrm{m}(X, f) \leq \mathrm{dim} \,\mathrm{ker}\, (f_*-\mathrm{id})$, and
\begin{center}
$\mathrm{m}(X, f)=0 \Leftrightarrow Z_f^{\mathrm{red}}=\{ 0 \} \Leftrightarrow (X, f) $ is rigid. \\
$\mathrm{m}(X, f)= \mathrm{dim} \,\mathrm{ker}\, (f_*-\mathrm{id})\Leftrightarrow Z_f$ is smooth $\Leftrightarrow (X, f)$ is unobstructed.
\end{center}
\end{pro}

\section{Automorphisms of anticanonical basic rational surfaces}\label{sec:anticanonical}

\subsection{General results on rational surfaces}

\subsubsection{Basic rational surfaces}\label{allright} $ $
\par \medskip
-- Let $X$ be a smooth projective surface, let $\mathrm{Aut}\,(X)$ be its group of biholomorphisms, let $\mathrm{NS}_{\mathbb{Q}}(X)$ be its rational Neron-Severi group, and let $
\varepsilon_X \colon \mathrm{Aut}\,(X) \rightarrow \mathrm{GL}\{\mathrm{NS}_{\mathbb{Q}}(X)\}
$
be the natural representation given by $f \rightarrow {f}^*$. 
The first dynamical degree $\lambda_1(f)$ of an element $f$ of $\mathrm{Aut}\, (X)$ is the spectral radius of $\varepsilon_X(f)$.
\par \medskip
-- Thanks to \cite{Grom} and \cite{Yo}, the topological entropy of $f$ is $\log\, \lambda_1 (f)$. Although we won't use it in this paper, the notion of first dynamical degree can be defined for birational maps between projective complex surfaces and is invariant by birational conjugacy.
\par \medskip
-- Let $X$ be a rational surface. It is well-known that $X$ is isomorphic to a finite blowup of $\mathbb{P}^2$ or a Hirzebruch surface $\mathbb{F}_n$. By definition, a basic rational surface is a finite blowup of $\mathbb{P}^2$.
\par \medskip
-- For any projective surface $X$, the Hirzebruch-Riemann-Roch theorem \cite[Thm. 5.1.1]{Huybrechts} and the Gauss-Bonnet theorem \cite[p. 416]{GH} give 
\[
\mathrm{h}^0(X, \mathrm{T}X)-\mathrm{h}^1(X, \mathrm{T}X)+\mathrm{h}^2(X, \mathrm{T}X)=\frac{7c_1(X)^2-5c_2(X)}{6}
=\frac{7\,\mathrm{K}_X^2-5 \chi(X)}{6}\cdot
\]
If $X$ is a basic rational surface obtained by successively blowing up the projective plane $N$ times, we have $\mathrm{h}^2(X, \mathrm{T}X)=0$ (since the quantity $\mathrm{h}^2(X, \mathrm{T}X)$ is a birational invariant thanks to Serre duality, and vanishes for $X=\mathbb{P}^2$). Besides, $\mathrm{K}_X^2=9-N$, and $\chi(X)=3+N$ so that assuming that $X$ has no nonzero holomorphic vector field, $\mathrm{h}^1(X, \mathrm{T}X)=2N-8$.
\par \medskip
-- Thanks to a result of Nagata \cite[Th. 5]{Nagata}, if $X$ is a rational surface and if $\mathrm{Im}\,(\varepsilon_X)$ is infinite, then $X$ is basic. Besides, by result of Harbourne \cite[Cor. 4.1]{Harb1}, for any rational surface $X$, $\mathrm{Ker}\,(\varepsilon_X)$ and $\mathrm{Im}\,(\varepsilon_X)$ cannot be both infinite. Combining these two results, if $X$ is a rational surface endowed with an automorphism whose action on the Picard group is of infinite order (\emph{e.g.} with positive topological entropy), then $X$ is basic and has no nonzero holomorphic vector field. 
\par\medskip
-- Let us now give a brief description of the Kuranishi space attached to a basic rational surface without nonzero holomorphic vector field. Note that if $X$ is such a surface, $\mathrm{H}^2(X, \mathrm{T}X)$ vanishes so $X$ is unobstructed. For any integer $N$, let $S_N$ be the Fulton-MacPherson configuration space of ordered (possibly infinitely near) $N$-uplets of points in the projective plane (\textit{see} \cite{FultonMcPherson}), and let $\mathfrak{X}_N$ be the corresponding family of rational surfaces given by $(\mathfrak{X}_N)_{\widehat{\xi}}=\mathrm{Bl}_{\widehat{\xi}}\, \mathbb{P}^2$. The group $\mathrm{PGL}(3; \mathbb{C})$ acts naturally on $(\mathfrak{X}_N, \pi, S_N)$ via its standard action on $\mathbb{P}^2$. We denote by $S_N^{\dag}$ the Zariski open subset of $S_N$ where the action is free. It parametrizes rational surfaces without nonzero holomorphic vector field.  For any $\widehat{\xi}$ in $S_N$, let $O_{\widehat{\xi}}$ be the $\mathrm{PGL}(3; \mathbb{C})$-orbit passing through $\widehat{\xi}\,$. We have the following result \cite[Thm. 5.1]{DesertiGrivaux}:

\begin{pro}\label{pro:pgl3} 
The family $\mathfrak{X}_N$ induces a complete deformation at every point of $S_N^{\dag}$, and for any point $\widehat{\xi}$ of $S_N$ we have an exact sequence 
\[
0 \longrightarrow \mathrm{T}_{\widehat{\xi}} O_{\widehat{\xi}} \longrightarrow \mathrm{T}_{\widehat{\xi}} S_N \xrightarrow{\mathrm{KS}_{\widehat{\xi}}(\mathfrak{X}_N)} \mathrm{H}^1(X, \mathrm{T}X) \longrightarrow 0.
\]
\end{pro}

-- As a corollary, if $\widehat{\xi}$ is in $\mathrm{S}_N^{\dag}$, the restriction of $\mathrm{X}_N$ to any smooth submanifold of codimension $8$ passing through $\widehat{\xi}$ and transverse to $O_{\widehat{\xi}}$ yields a deformation whose Kodaira-Spencer map is an isomorphism at every point, it is the Kuranishi space of $X$.
\par\medskip
-- Proposition \ref{pro:pgl3} allows to compute the generic number of parameters of families of basic rational surfaces (\emph{see} \cite[Thm.  5.8]{DesertiGrivaux} for a precise statement). As a particular case, we have the following result:
\begin{pro}\label{hopp}
Let $N\geq 4$ be an integer, and $\mathcal{V}$ be a connected submanifold of $S_N^{\dag}$ which is stable under the action of $\mathrm{PGL}(3; \mathbb{C})$. Then the Kodaira-Spencer of the deformation ${\mathfrak{X}_N}_{| \mathcal{V}}$ has everywhere rank $\mathrm{dim}\,\mathcal{V}-8$. 
\end{pro}

\subsubsection{Anticanonical surfaces}
{Anticanonical surfaces} are surfaces whose anticanonical class is effective. These surfaces play a crucial role in the theory of rational surfaces (\emph{see e.g.} \cite{Looijenga} and \cite{Harb2}).
Let us give at first a classification of possible reduced\footnote{If $D$ is reducible but non reduced, the situation gets more complicated, even if we know that the irreducible components are still smooth rational curves.} anticanonical divisors on a projective surface. We start by a simple lemma:

\begin{lem} \label{lem:visse}
Let $Y$ is a smooth surface and $\pi \colon X \rightarrow Y$ is a point blowup and $D$ is an effective anticanonical divisor on $X$. Then $\pi_*(D)$ is an effective anticanonical divisor on $Y$, and $D=\pi^* \pi_*(D)-E$. Besides, $D$ is connected if and only if $\pi(D)$ is connected.
\end{lem}

\begin{proof}
Let $E$ be the exceptional divisor, $U=X \setminus E$ and $p=\pi(E)$. Then $D_{| U}$ is an anticanonical divisor on $U$. It follows that $K_Y$ and $\pi_* D$ are isomorphic on $Y \setminus \{p\}$, and therefore on $Y$ thanks to Hartog's theorem. Hence $\pi_* D$ is anticanonical.
The divisor $\pi^*\pi_* D-D$ is equal to $mE$ for some $m$ in $\mathbb{Z}$. Besides, since $\pi^* \mathrm{K}_Y \sim\mathrm{K}_X-E$, we have
\[
\pi^* \pi_*D -D  \sim \pi^*(-\mathrm{K}_Y) + K_X \sim E
\]
so
$\pi^*\pi_* D-D$ is linearly equivalent to $E$. Hence $(m-1) E$ is linearly equivalent to zero, which forces $m=1$ since $E^2 \neq 0$.
\par \medskip
For the last point, let us write $D=\overline{Z}+mE$ where $Z$ is an effective divisor in $Y$, $\overline{Z}$ is the strict transform of $Z$ in $X$ and $m$ is in $\mathbb{N}$. Note that $p$ belongs to $Z$, otherwise $\pi^* \pi_* D-E=\pi^* Z-E=\overline{Z}-E$ is not equal to $D$. Hence, if $m \geq 1$, $\overline{Z}$ and $E$ meet. This implies the required result.
\end{proof}
\begin{cor}
If $X$ is a basic rational surface, any effective anticanonical divisor is connected.
\end{cor}
It is possible to classify connected reduced anticanonical divisors (\emph{see} \cite[Th. 4.2]{DJS} in a slightly different context):
\begin{pro} \label{pro:classification}
Let $X$ be a smooth projective surface, and let $D$ be a reduced and connected effective divisor representing the class $- \textrm{K}_X$. Then $D$ is either an irreducible reduced curve of arithmetic genus $1$, or a cycle of smooth rational curves, or one of the three exceptional configurations shown in the picture below\emph{:}
\par \medskip
\begin{center}
\includegraphics{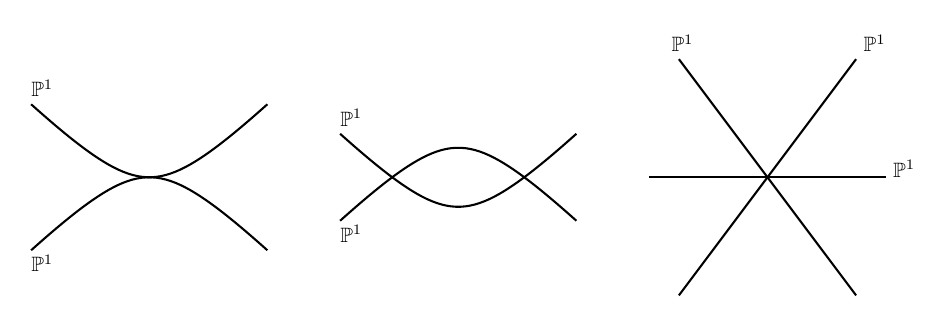}
\end{center}
\end{pro}

\begin{proof}
Let us write $D=\sum_{i=1}^k D_i$. For any $i$ the arithmetic genus of $D_i$ is 
\[
g(D_i)=1+\frac{1}{2}D_i \,(D_i+K_X)=1-\frac{1}{2} \sum_{j \neq i} D_i . D_j
\]
so that $g(D_i) \leq 1$. 
\begin{enumerate}
\item[(i)] Assume that for some $i$, $g(D_i)=1$. Then $D_i$ is a connected component of $D$, so that $D=D_i$
and $D_i$ is a reduced curve of arithmetic genus $1$.
\item[(ii)] Assume that for all $i$, $g(D_i)=0$. Then all the $D_i$'s are smooth rational curves and we have $D_i . \sum_{j \neq i} D_j=2$. If for some indices $i$ and $j$ we have $D_i . D_j=2$ then $D_i+D_j$ is a connected component of $D$ and we are in one of the first two exceptional configurations. Otherwise, all $D_i$'s intersect transversally. If for some indices $i$, $j$ and $k$ the intersection $D_i \cap D_j \cap D_k$ is nonempty then $D_i+D_j+D_k$ is again a connected component of $D$ and we are in the third exceptional configuration. The last remaining possibility is that for any $i$, there exists exactly two components $D_j$ and $D_k$ intersecting $D_i$, and satisfying $D_j \cap D_k=\varnothing$. Hence $D$ is a cycle of smooth rational curves.
\end{enumerate}

\end{proof}

\begin{rem}
All possible configurations of reduced effective anticanonical divisors can occur on rational surfaces carrying an automorphism of positive entropy, except possibly the case of an irreducible nodal curve of arithmetic genus one (even in the case of asymptotically stable birational maps, this case is not ruled out in \cite[Th. 4.2]{DJS}). To see this, we list all the possible cases. Let $D$ denote a reduced anticanonical divisor. The construction of Blanc and Gizatullin \cite{Blanc} provides examples of rational surfaces carrying a smooth anticanonical elliptic curve (\emph{see also} \cite[Ex. 3.3]{Di} for an example producing a quadratic transformation fixing the square torus and lifting to an automorphism of positive entropy). This construction will be given in details in \S \ref{noirhaha}. All other configurations except the irreducible nodal curve can occur using quadratic transformations in $\mathbb{P}^2$. The case of a cuspidal elliptic curve goes back to McMullen \cite{MC} (\emph{see}  \cite[Thm 3.5 \& Thm 3.6]{Di}), we will also recall it in \S \ref{McCormick}. For the first exceptional configuration, this is said to be possible in the last paragraph before \cite[\S 4.1]{Di}, although the precise result is not stated. For the second configuration, \emph{see} \cite[Thm. 4.5]{Di}. For the third configuration, \emph{see} \cite[Thm. 4.4]{Di}. 
\end{rem}
We end this section by discussing the existence of holomorphic vector fields on anticanonical surfaces.
\begin{lem} 
Let $X$ be a basic rational surface with $\mathrm{K}_X^2<0$ admitting an irreducible and reduced anticanonical curve. Then $X$ admits no nonzero holomorphic vector field.
\end{lem} 

\begin{proof}
Let $G$ be the connected component of the identity in the automorphism group of $X$, and let $C$ be an irreducible and reduced curve in $|-\mathrm{K}_X|$. Since $\mathrm{K}_X^2<0$, $C$ is fixed by $G$, as well as any $(-1)$-curve on $X$. If $\pi \colon X \rightarrow \mathbb{P}^2$ is a presentation of $X$ as an iterated blowup of the projective plane, then any element in $G$ fixes all the exceptional curves occuring is the successive blowups, so it descends to an automorphism of $\mathbb{P}^2$. This automorphism must fix the curve $\pi(C)$, which as well as the center of the blowups. Since $C$ is irreducible and reduced, so is $\pi(C)$. It is easy to see using Lemma \ref{lem:visse} that if a singular point of $\pi(C)$ is blown up, then $C$ won't be irreducible. Therefore, the blown up locus of $\pi$ lies in the smooth locus of $\pi(C)$. But the subgroup of $\mathrm{PGL}(3; \mathbb{C})$ fixing an irreducible reduced cubic in $\mathbb{P}^2$ and a smooth point on it is finite. This proves that $G=\{ \textrm{id} \}$, so that $X$ has no nonzero holomorphic vector field.
\end{proof}

\subsection{Action of the automorphisms group on first order deformations}

\subsubsection{Equivariant bundles} \label{serre}
In this section we recall some basic facts on equivariant vector bundles, mainly to fix the conventions. If $X$ is a complex manifold and $G$ is a group acting biholomorphically on $X$, a left $G$-action on $E$ is a collection of vector bundle isomorphisms $\mu_g  \colon E \xrightarrow{\sim} g^* E$ compatible with the group action, \textit{i.e.} such that $\mu_{\textbf{1}}$ is the identity morphism, and for any $g, h$ in $G$, the composition 
\[
E \xrightarrow{\mu_h} h^* E   \xrightarrow{h^* \mu_g} h^*(g^*E) \simeq (gh)^* E
\]
is $\mu_{gh}$. A vector bundle endowed with a left $G$-action is called a left equivariant $G$-bundle. If $E$ is a $G$-equivariant bundle, then $G$ acts on the cohomology spaces $\mathrm{H}^i(X, E)$.
\par \medskip
Right $G$-actions are defined in a reversed way: $\mu_g$ goes from $g^*E$ to $E$. Any right $G$-action can be transformed in a left $G$-action (and vice-versa) by mapping $\mu_g$ to $g^* \mu_{g^{-1}}$. If $E$ is a left $G$-equivariant bundle, then $E^*$ is a right equivariant bundle, but we will always consider it with the associated left $G$-equivariant structure\footnote{In the case of ordinary representations (\textit{i.e.} vector bundles over a point), this corresponds to define the dual representation as the contragredient representation.}.
\par \medskip
The tangent bundle $TX$ is naturally a left $G$-bundle, and the cotangent bundle $\Omega^1_X$ is naturally a right $G$-bundle. In particular the canonical bundle $\mathrm{K}_X$ is naturally a right $G$-bundle. Another interesting example is the following: if $D$ is a divisor on $X$ that is globally invariant by the action of $G$, then $\mathcal{O}_X(D)$ is naturally a right $G$-equivariant bundle, the map $\mu_g$ being given by $\mu_g(f)=f \circ g$. Its dual (as a right $G$-equivariant bundle) is $\mathcal{O}_X(-D)$.
\par \medskip
Assume that $X$ is compact of dimention $n$. For any left $G$-equivariant bundle $E$,  Serre's duality isomorphism
\[
\mathrm{H}^i(X, E)^* \simeq \mathrm{H}^{n-i} (X, E^* \otimes \mathrm{K}_X)
\]
can be lifted to an isomorphism of left $G$-modules as follows: $E^*$ and $K_X$ can be considered as left $G$-equivariant bundles using the aforementioned construction. Then for the left-hand side, we pick the contragredient representation of $\mathrm{H}^i(X, E)$. The story carries on in the same way if $E$ is a right $G$-module .

\subsubsection{Elliptic curves}

-- We start with some preliminary results on equivariant bundles on elliptic curves, which are probably classical although we couldn't find them in the literature. For any divisor $D$ on a curve $C$, we denote by $[D]$ the holomorphic line bundle $\mathcal{O}_C(D)$.
\par \medskip
-- If $C$ is an irreducible reduced curve of arithmetic genus $1$, let $C^{\mathrm{reg}}$ be the smooth locus of $C$. Given any base point $P_0$ in $C^{\textrm{reg}}$, the map from $C^{\mathrm{reg}}$ to $\mathrm{Pic}^0(C)$ given by $P \rightarrow [P]-[P_0]$ is an isomorphism. Hence $C^{\mathrm{reg}}$ is naturally endowed with the structure of an algebraic group having $P_0$ as origin. This group is of the form $\mathbb{C}/\Gamma$ where $\Gamma$ is a discrete subgroup of $\mathbb{C}$ of rank $0$, $1$ or $2$ depending on the three cases $C$ cuspidal, $C$ nodal or $C$ smooth. 
\par \medskip
-- If we endow $C^{\mathrm{reg}}$ with its algebraic group structure, the map from $C^{\mathrm{reg}} \oplus \mathbb{Z}$ to  $\mathrm{Pic}\,(C)$ given by 
\[
(P, n) \rightarrow [P]+ (n-1)[P_0]
\] 
is a group isomorphism.
\par \medskip
-- Any biholomorphism $\varphi$ of $C$ preserves $C^{\mathrm{reg}}$ and lifts to an affine map $z \rightarrow az+b$ of $\mathbb{C}$, where the multiplier $a$ satisfies $a \Gamma=\Gamma$ and is independent of the lift. A more intrinsic way of defining the multiplier is as follows: if $\omega_C$ is the dualizing sheaf of $C$, $\mathrm{H}^0(C, \omega_C)$ is a complex line. The action $\varphi^*$ on $\mathrm{H}^0(C, \omega_C)$ is the multiplication by $a$. The number $b$ is called the translation factor; it is well defined modulo $\Gamma$ but it depends of the choice of the origin. 
\begin{pro} \label{marre}
Let $C$ be an irreducible reduced curve of arithmetic genus $1$, let $\mathcal{L}$ be a holomorphic line bundle on $C$ of positive degree $n$, let $\mathcal{G}_{\mathcal{L}}$ be the stabilizer of $\mathcal{L}$ in $\mathrm{Aut}\,(C)$, and let $\chi$ be the character of $\mathcal{G}_{\mathcal{L}}$ given by the multiplier. 
\begin{enumerate}
\item[(i)] Assume that $C$ is smooth. 
\begin{enumerate}
\item[--] \textbf{Similitudes}. For any morphism $\varphi$ in $\mathcal{G}_{\mathcal{L}}$ that is not a translation, let $P$ be any fixed point of $\varphi$. Then the line bundle $\mathcal{L}$ can be endowed with a right action of the cyclic  group $\langle \varphi \rangle $ generated by $\varphi$ such that the action on the fiber $\mathcal{L}_P$ is trivial, and that the corresponding representation on $\mathrm{H}^0(C, \mathcal{L})$ is \emph{:}
\[
\left\{
\begin{alignedat}{2}
\displaystyle &\bigoplus_{0 \leq k \leq n ,\,\,\, k \neq n-1}  \chi^{k} & \qquad &\textrm{if} \quad \mathcal{L} \sim n[P] \\
\displaystyle &\bigoplus_{0 \leq k \leq n-1}^{\vphantom{a}} \chi^{k}&\qquad &\textrm{otherwise.} \\
\end{alignedat} \right.
\]
\item[--] \textbf{Translations}. The intersection of $\mathcal{G}_{\mathcal{L}}$ with the group of translations by elements of $C$  is isomorphic to the group $T_n$ of $n$-torsion points of $C$. If $b$ is a primitive $n$-torsion point and $\langle b \rangle$ is the cyclic group generated by $b$ in $H_n$, then $\mathcal{L}$ can be endowed with a right $\langle b \rangle$-action whose representation of $\langle b \rangle$ on $\mathrm{H}^0(C, \mathcal{L})$ is $\bigoplus_{0 \leq k \leq n-1}^{\vphantom{a}} \nu^k$ where $\nu$ is any primitive character of $\langle b \rangle$.
\end{enumerate}
\item[(ii)] If $C$ is cuspidal, the morphism $\chi \colon \mathcal{G}_{\mathcal{L}} \rightarrow \mathbb{C}^{\times}$ is an isomorphism, and there is a unique global fixed point $P$ on $C^{\mathrm{reg}}$ under the action of $\mathbb{C}^{\times}$. Besides, $\mathcal{L}$ can be endowed with an action of $\mathbb{C}^{\times}$ which is trivial on $\mathcal{L}_P$, and such that 
the corresponding representation of $\mathbb{C}^{\times}$ on $\mathrm{H}^0(C, \mathcal{L})$ is $\bigoplus_{0 \leq k \leq n, \,\, k \neq n-1} \,\, \chi^{k}$.
\end{enumerate}
\end{pro}

\begin{proof}
Let $\varphi$ be an automorphism of $C$ that lifts to the map $z \rightarrow az+b$. If  $\mathcal{L} \sim (n-1)[0]+[w]$, we have $\varphi_*{\mathcal{L}} \sim (n-1)[b]+[aw+b] \sim (n-1)[0]+[aw+nb]$ so that the isomorphism class of 
$\mathcal{L}$ is fixed by $\varphi$ if and only if $aw+nb=w \,\,\mathrm{mod}\,\, \Gamma$. 
\par \smallskip
If $\mathrm{rank}\,(\Gamma)=2$, pick such a $\varphi$ with $a \neq 1$ (so that $\varphi$ is not a translation), and choose for the origin of $C$ a fixed point $P$ of $\varphi$. Then the translation factor of $\varphi$ vanishes so that $aw=w \,\,\mathrm{mod}\,\, \Gamma$. Remark now that the divisor $D=(n-1)\, 0 + w$ is invariant by $\varphi$ and $\mathcal{L}$ is isomorphic to $[D]$. There is therefore a natural right $\langle \varphi \rangle $-action on the line bundle $\mathcal{L}$, and the corresponding action on $\mathrm{H}^0(C, \mathcal{L})$ is given by pre-composing with $\varphi$. For any meromorphic function $u$ in $\mathrm{H}^0(C, \mathcal{L})$, 
let $\alpha_0(u), \ldots, \alpha_n(u)$ be the Laurent coefficients of $u$ near $0$: 
\[
u(z)=\sum_{i=0}^{n} {\alpha_i(u)}{z^{-i}} + \mathrm{O}(z).
\]
There is a natural isomorphism between $\mathrm{H}^0(C, \mathcal{L})$ and $\mathbb{C}^n$ obtained as follows:
\[
\begin{cases}
u \rightarrow \{\alpha_0(u), \alpha_1(u), \ldots \alpha_{n-1}(u)\} &\textrm{if} \quad w \neq 0 \\
u \rightarrow \{\alpha_0(u), \alpha_2(u), \ldots \alpha_{n}(u)\} &\textrm{if} \quad w=0.
\end{cases}
\]
In both cases, the linear form $\alpha_i$ satisfies $\alpha_i (u \circ \varphi)=a^{-i} \,\alpha_i(u)$. Besides, we can identify $\mathcal{L}_P$ with $\mathbb{C}$ via the map given by 
\[
\begin{cases}
u \rightarrow \alpha_{n-1}(u)  &\textrm{if} \quad w \neq 0 \\
u \rightarrow \alpha_{n}(u) &\textrm{if} \quad w=0.
\end{cases}
\] 
Hence the action of $\langle \varphi \rangle$ on $\mathcal{L}_P$ is the character $\chi^{-(n-1)}$ if $w \neq 0$, and $\chi^{-n}$ if $w=0$. To get a trivial action on $\mathcal{L}_P$, we multiply it by $\chi^{n-1}$ if $w \neq 0$, and by $\chi^n$ if $w=0$. This yields the first part of (i). 
\par \medskip
For the second part, let us put $z=\dfrac{w}{n}-\dfrac{(n-1)b}{2} \cdot$ The divisor 
\[
D=z+(z+b)+ \ldots +(z+(n-1)b)
\] 
is $\langle b \rangle$-invariant, and linearly equivalent to $(n-1) \,0 + w$ so that $\mathcal{L}$ is isomorphic to $[D]$. We have an isomorphism between $\mathrm{H}^0(C, \mathcal{L})$ and $\mathbb{C}^n$ given by the list of the residues at the points of $D$. Hence the representation of $\langle b \rangle$ on $\mathrm{H}^0(C, \mathcal{L})$ is isomorphic to the representation of $\langle b \rangle$ on $\mathbb{C}^n$ that associates to $b$ the matrix of the permutation $(1, 2, \ldots ,n-1)$. This yields the second part of (i).
\par \medskip 
If $\Gamma=\{ 0 \}$, the group of solutions $G_w$ is the set of elements $(a, b)$ in the affine group of the form $(1, w/n) (a, 0) (1, -w/n)$ for $a$ in $\mathbb{C}^{\times}$, it is therefore conjugate of the standard torus $\mathbb{C}^{\times}$ of scalar multiplications in the affine group. Let $\pi \colon \mathbb{P}^1 \rightarrow C$ be the normalization map, it is a set-theoretic bijection. Then $\pi^{-1} \mathcal{L}$ is the subsheaf of the holomorphic line bundle $(n-1)[0]+[w]$ on $\mathbb{P}^1$ consisting of sections $s$ such that $s-s(\infty)$ vanishes at order $2$ at $\infty$. Let $\mu$ denote the action of $G_w$ on $\mathbb{P}^1$. We define the action of $G_w$ on $\pi^{-1} \mathcal{L}$ in the natural way as follows: for any $g=(a,b)$ in $G_w$, the isomorphism $\psi_g \colon \mu(g)^{-1} (\pi^{-1} \mathcal{L}) \rightarrow \pi^{-1} \mathcal{L}$ that gives the $G_w$-structure is given by 
\[
\psi_g (f) (t)=f(at+b) \times \frac {(at+b)^{n-1} (at+b-w)}{t^{n-1} (t-w)}.
\]
Remark that 
\[
\dfrac {(t+b/a)^{n-1} (t-(w-b)/a)}{t^{n-1} (t-w)} \underset{t \rightarrow \infty}{\sim} 1+ \dfrac{(nb-w)/a+w}{t}+ \mathrm{O}(t^{-2})=1+\mathrm{O}(t^{-2})
\]
so the $G_w$ action is well-defined on the sheaf $\pi^{-1} \mathcal{L}$. A basis  for $\mathrm{H}^0(\mathbb{P}^1, \pi^{-1} \mathcal{L})$ is given by the sections $s_i = \dfrac{(t-w/n)^i}{t^{n-1}(t-w)}$ for $0 \leq i \leq n, \,i \neq n-1$. Then for $g=(a, b)$ in $G_w$, we have $g^* s_i=a^{i} s_i$.
\end{proof}

\begin{cor} \label{enfin}
Let $C$ be an irreducible reduced curve of arithmetic genus $1$, and let $\mathcal{L}$ be a holomorphic line bundle on $C$ such that $n=\deg \mathcal{L} >0$. Let $\varphi$ be in $\mathcal{G}_{\mathcal{L}}$, let $a$ be the multiplier of $\varphi$, and assume that $\varphi$ acts on $\mathcal{L}$. Let $\mu({\varphi})$ denote the action of $\varphi$ on $\mathrm{H}^0(C, \mathcal{L})$.
\begin{enumerate}
\item[(i)] Assume that $C$ is smooth. 
\begin{enumerate}
\item[--] If $a \neq 1$, let $P$ be a fixed point of $\varphi$, and let $\beta$ be the action of $\varphi$ on the fiber $\mathcal{L}_P$. Then $\mu(\varphi)$ is diagonalizable with eigenvalues 
\[
\begin{cases}
\beta, \beta a, \ldots, \beta a^{n-2}, \beta a^{n}& \qquad  \textrm{if} \quad \mathcal{L} \sim n[p] \\
\beta , \beta a, \ldots, \beta a^{n-1}  & \qquad \textrm{otherwise.}
\end{cases}
\]
\item[--] If $a=1$ and if the translation factor of $\varphi$ is a primitive $\ell$-torsion point \emph{(}where $\ell | n$\emph{)}, then $\mu(\varphi)$ is diagonalizable, and there exists a complex number $\beta$ such that the eigenvalues of $\mu(\varphi)$ are $\beta, \beta \omega, \ldots, \beta \omega^{\ell-1}$ where $\omega$ is any primitive $\ell$-root of unity.
\item[(ii)] If $C$ is cuspidal, let $a$ be the multiplier of $\varphi$. If $\beta$ is the action of $\varphi$ on the fiber $\mathcal{L}_{P}$ where $P$ is the fixed point of $C^{\mathrm{reg}}$ under $\mathcal{G}_{\mathcal{L}}$, then $\mu(\varphi)$ is diagonalizable with eigenvalues $\beta, \beta a, \beta a^2, \ldots, \beta a^{n-2}, \beta a^n$.
\end{enumerate}
\end{enumerate}
\end{cor}

\begin{proof}
Since $C$ is compact, $\mathrm{Aut} \,(\mathcal{\mathcal{L}})=\mathbb{C}^{\times}$. Therefore two different actions of $\varphi$ on $\mathcal{L}$ differ by a scalar. The result follows from Proposition \ref{marre}. 
\end{proof}
Before stating the main result, we prove a technical lemma:
\begin{lem} \label{mexico}
Let $C$ be a cuspidal curve of arithmetic genus $1$. Then $\mathrm{H}^0(C, \Omega^1_C)$ and $\mathrm{H}^1(C, \Omega^1_C)$ are two-dimensional vector spaces. For any automorphism $\varphi$ of $C$, if $a$ denotes the multiplier of $\varphi$, then the eigenvalues of the action of $\varphi$ by pullback on $\mathrm{H}^0(C, \Omega^1_C)$ and $\mathrm{H}^1(C, \Omega^1_C)$ are $\{a^{-5}, a^{-7} \}$ and $\{1, a^{-1}\}$ respectively.
\end{lem}
\begin{proof}
The curve $C$ is isomorphic to the standard cuspidal cubic in the projective plane $\mathbb{P}^2$ given in affine coordinates by the equation $y^2=x^3$. We denote by $Q$ the cusp. The normalization of $C$ is locally given by the map $s \rightarrow (s^2, s^3)$, which is a set-theoretic bijection. Hence we can identify the complex space $C$ in a neighborhood of $Q$ with a neibourhood of the origin in $\mathbb{C}$ endowed with the subsheaf of $\mathcal{O}_{\mathbb{C}}$ consisting ot holomorphic functions $\varphi$ such that $\varphi'(0)=0$.
\par \medskip
\underline{\textit{The sheaf $\mathcal{T}$ of torsion differentials}}
\par \medskip
Let us first investigate the maximal torsion subsheaf $\mathcal{T}$ of $\Omega^1_C$. 
We claim the following:
\begin{enumerate} 
\item[--] The sheaf $\mathcal{T}$ is generated by the section $\tau=2xdy-3ydx$. 
\item[--] The annihilator of $\tau$ is $(x^2, y)$, and $\mathrm{H}^0(C, \mathcal{T})=\mathbb{C} \tau \oplus \mathbb{C} x\tau$.
\end{enumerate}
To prove the claim, let us introduce some notation: let $A=\mathcal{O}_{(0,0)}/(y^2-x^3)$ be the local ring of the curve $C$ at $Q$. Then $(\Omega^1_C)_{Q}$ is
the quotient of $A^2$ by the submodule generated by $(3x^2, -2y)$: any differential $\alpha dx_{| C} + \beta dy_{|C}$ corresponds to the couple $(\alpha, \beta)$. Now a differential on $C$ is a torsion element if and only if its pullback on the normalization vanishes. Hence torsion differentials are characterized by the equation $2s \alpha + 3s^2 \beta=0$. Identifying the local ring $A$ with the subring of functions $\varphi$ in $(\mathcal{O}_{\mathbb{C}})_{0}$ such that $\varphi'(0)=0$, 
\[
(\mathcal{T})_{Q}=\left\{ (\alpha, \beta) \,\,\mathrm{in}\,\, A^2/(3s^4, -2s^3) \,\, \textrm{such that}\,\, 2 \alpha + 3s \beta=0 \right\}.
\] 
For $(\alpha, \beta)$ in $(\mathcal{T})_{Q}$, we can write $3\beta = a + bs^2 + c s^4 + s^3 \gamma (s)$ where $\gamma' (0)=0$, \textit{i.e.} $\gamma \in A$. Since $2\alpha + 3s \beta=0$, we get $a=0$, and $2\alpha=-bs^3-cs^5-s^4 \gamma(s)$. Hence  
\begin{align*}
(\alpha, \beta) &= \left( -\frac{b}{2}s^3-\frac{c}{2}s^5-\frac{1}{2} s^4 \gamma(s), \frac{a}{3} + \frac{b}{3}s^2 + \frac{c}{3} s^4 + \frac{1}{3} s^3 \gamma (s) \right) \\
&= \frac{b}{6} (-3s^3, 2s^2) +\frac{c}{6}(-3s^5, 2s^4) - \frac{\gamma(s)}{6} (3s^4, -2s^3)
\end{align*}
Hence $(\mathcal{T})_{Q}$ is generated by the two elements $(-3s^3, 2s^2)$ and $(-3s^5, 2s^4)$. It is easy to see that these elements are nonzero and linearly independant in $ A^2/(3s^4, -2s^3)$. Now $-3s^3 dx_{|C} +2 s^2 dy_{|C}=\tau$ and $-3s^5 dx_{|C} +2 s^4 dy_{|C}= x\tau$. This gives the first point. The elements $x^2$ and $y$ annihilate $\tau$. This give the second point, since $A/(x^2, y)$ is the two dimensional vector space generated by $1$ and $x$.
\par \medskip
\underline{\textit{Locally free subsheaves of ${\Omega}^1_C / {\mathcal{T}}$}}
\par \medskip
We can describe explicitly the sheaf ${\Omega}^1_C / {\mathcal{T}}$: it is the image of the pullback map from differentials on the curve $C$ to differentials on the normalization. It consists of all $1$-forms (that we identify with holomorphic functions) that vanish at zero. This sheaf is torsion free but not locally $\mathcal{O}_C$-free. However, it contains natural locally free subsheaves: for any nonzero complex number $\zeta$, let us consider the subsheaf $\mathcal{F}_{\zeta}$ of ${\Omega}^1_C / {\mathcal{T}}$ consisting of forms $\alpha dx_{| C} + \beta dy_{|C}$ such that $\alpha(Q) - \zeta \beta(Q)=0$. We claim that this sheaf is locally free of rank one, generated by the form $ \zeta dx - dy$. Indeed, using the model on the normalization,  
\[
(\mathcal{F}_{\zeta})_{Q}=\{ f \in (\mathcal{O}_{\mathbb{C}})_{0} \, \, \textrm{such that}\,\, f(0)=0 \, \, \textrm{and}\,\, 3 f'(0)+ \zeta f''(0)=0 \}.
\]
All we have to check is that for $f$ in $(\mathcal{F}_{\zeta})_{Q}$,  the germ of holomorphic function $\displaystyle \frac{f(s)}{2 \zeta s-3 s^2}$ near the origin lies in $(\mathcal{O}_C)_Q$. For this we compute the derivative at the origin:
\[
\left(\frac{f(s)}{2 \zeta s- 3 s^2} \right)'(s)=\frac{f'(s)(2\zeta s- 3 s^2)-f(s)(2 \zeta -6 s)}{(2 \zeta s-3 s^2)^2} \cdot
\]
The numerator is
\[
(f' (0)+f''(0)s) (2 \zeta s-3 s^2)- \left(f'(0)s+\frac{f''(0)}{2} s^2\right)(2 \zeta - 6 s) + o(s^2)
\]
which is 
$
-(\zeta f''(0) +  3 f'(0))s^2+o(s^2).
$
Hence the derivative at the origin vanishes, which proves the claim.
\par \medskip
\underline{\textit{Automorphisms of $C$: the fundamental exact sequence}}
\par \medskip
On $C$, we introduce the coordinate $t$ such that $x=t^{-2}$ and $y=t^{-3}$. We denote by  $P_0$ be the point at infinity (for the coordinate $s$) on $C$, it is $[0 :1 :0]$ which corresponds to $t=0$. Any automorphism $\varphi$ of $\mathrm{Aut}\,(C)$ is given in the coordinate $t$ by $t \rightarrow at+b$. For any nonzero complex number $\zeta$, let $G_{\zeta}$ denote the set of $\varphi$ in $\mathrm{Aut}\, (C)$ fixing $t=(2\zeta)^{-1}$, that is such that $a + 2b \zeta=1$ where $\varphi$ is given away from the cusp by the affine map $t \rightarrow at+b$ . Then $G_\zeta$ is isomorphic to $\mathbb{C}^{\times}$, and $\underset{\zeta \in \mathbb{C}^{\times}}{\cup} G_{\zeta}$ is dense in the affine group $\mathrm{Aut}\, (C)$. Hence we can assume without loss of generality that the automorphism $\varphi$ lies in $G_{\zeta}$ for some $\zeta$ in $\mathbb{C}^{\times}$. Let us consider the morphism of sheaves
$
\Delta \colon \Omega^1_C \rightarrow \mathbb{C}_{Q}
$
given by $\alpha dx + \beta dy \rightarrow \alpha(Q)-\zeta \beta(Q)$. We claim that the sequence
\begin{equation} \label{sol}
0 \rightarrow \mathcal{O}_C \left(-4P_0+\frac{3}{2\zeta} \right) \rightarrow \Omega^1_{C}/\mathcal{T} \xrightarrow{\Delta} \mathbb{C}_{Q} \rightarrow 0
\end{equation}
where the first arrow is the multiplication with $\zeta dx-dy$, is exact. We have already seen the exactness of this sequence near the cusp $Q$ (the kernel of the second arrow is the sheaf $\mathcal{F}_{\zeta}$). Using the coordinate $t$, we have
\[
\zeta dx-dy=\displaystyle \frac{3-2\zeta t}{t^4} dt
\] 
so on the regular part of $C$, $\zeta dx-dy$ has a single zero at $t=\displaystyle \frac{3}{2\zeta}$ and a pole of order four at $t=0$.
\par \medskip
\underline{\textit{The action of $G_{\zeta}$}}
\par \medskip
Let us define a right action of $G_{\zeta}$ on the exact sequence \eqref{sol}. On the middle sheaf $\Omega^1_{C}/\mathcal{T}$, the action is simply the pullback. We define the action of a function $\varphi$ on the skyscraper sheaf $\mathbb{C}_Q$ by the multiplication multiplication with $a^{-3}$. Let us check that $\Delta$ becomes equivariant. We have the formula
\[
\Delta(f(s)ds)=\frac{f'(0)}{2} + \frac{\zeta f''(0)}{6} \cdot
\]
In the coordinate $s$, $\varphi(s)=\displaystyle \frac{s}{a+bs} \cdot $ Hence
\[
\varphi^* (f(s) ds) = -f \left ( \frac{s}{a+bs} \right) \frac{a ds}{(a+bs)^2} =g(s) ds
\]
and
\[
\begin{cases}
g'(0)=\displaystyle \frac{f'(0)}{a^2} \vspace{0.3cm}\\
g''(0)=\displaystyle \frac{f''(0)+6bf'(0)}{a^3}
\end{cases}
\]
Hence we get
\begin{align*}
\Delta \left(\varphi^* (f(s) ds) \right) &= \Delta (g(s)ds) \\
&=\displaystyle \frac{g'(0)}{2} +\displaystyle \frac{\zeta g''(0)}{6} \\
&=\displaystyle \frac{1}{a^3} \left( (a+2 \zeta b) \frac{f'(0)}{2}+ \frac{\zeta f''(0)}{6} \right) \\
&=a^{-3} \Delta(f(s)ds).
\end{align*}
Another way to see this point is as follows: the automorphism $\varphi$ of $\mathrm{Aut}\,(C)$ given by $t \rightarrow at+b$ is induced by the map
\begin{equation} \label{cirm}
(x, y) \rightarrow \left( \frac{a^2x-2aby+b^2x^2}{(a^2-b^2x)^2}, \frac{a^3y-3a^2bx^2+3ab^2xy-b ^3y^2}{(a^2-b^2x)^3}\right)
\end{equation}
in the affine coordinates $(x, y)$ that is regular on $\mathbb{P}^2$ near the cusp. The action of the automorhism $\varphi$ on the two dimensional vector space $(\Omega^1_C)_{|Q}$ in the basis $(dx_{|Q}, dy_{|Q})$ is given by the Jacobian matrix of this map at the origin, which is 
\[
a^{-3} \times \begin{pmatrix}
a & -2b \\ 
0 & 1
\end{pmatrix}. 
\]
Hence
\begin{align*}
(\varphi^* (\alpha dx + \beta dy))_{|Q}&= a^{-3} \left( \alpha(Q) (a dx_{|Q}-2b dy_{|Q})+\beta(Q) dy_{|Q} \right) \\
&=a^{-3} \left( a \alpha(Q) dx_{|Q} + (\beta(Q)-2b\alpha(Q)) dy_{|Q} \right) 
\end{align*}
so that
\begin{align*}
\Delta(\varphi^* (\alpha dx + \beta dy))&=a^{-3} \times \left( a \alpha(Q)-\zeta(\beta(Q)-2b\alpha(Q)) \right ) \\
&=a^{-3} \left( (a+2b\zeta) \alpha(Q)-\zeta \beta(Q) \right)\\
&=a^{-3} \Delta(\alpha dx + \beta dy).
\end{align*}
The sequence \eqref{sol} is naturally (right) $G_{\zeta}$-equivariant: the action of an element $\varphi$ on a section $f$ of $\mathcal{O}_C \left(-4P_0+ \displaystyle \frac{3}{2\zeta} \right)$ is given by 
\[
\varphi . f = \displaystyle \frac{\varphi^* \left (f (\zeta dx - dy)\right)}{\zeta dx - dy} \cdot
\]
Explicitly, 
\[
\varphi.f(t)= f(at+b) \times \displaystyle \frac {at^4 (3-2 \zeta b-2 \zeta at)}{(at+b)^4(3-2\zeta t)}=a f(at+b) \times  \frac {t^4 (at+b-3/2\zeta)}{(at+b)^4(t-3/2\zeta)} \cdot
\]
With this action, the sequence \eqref{sol} is $G_{\zeta}$-equivariant. Lastly, let us consider the exact sequence
\begin{equation} \label{marco}
0 \rightarrow \mathcal{T} \rightarrow \Omega^1_C \rightarrow \Omega^1_C/\mathcal{T} \rightarrow 0
\end{equation}
which is also $G_{\zeta}$-equivariant. A direct computer-assisted calculation using equation \eqref{cirm} shows that the action of any element $\varphi$ of $\mathrm{Aut}(C)$ on $\mathcal{T}$ is diagonal with eigenvalues $a^{-5}$ and $a^{-7}$ in the basis $(\tau, x\tau)$, where $a$ is the multiplier of $\varphi$. 
\par \medskip
\underline{\textit{End of the proof}}
\par \medskip
We use the sequences \eqref{cirm} and \eqref{marco}, as well as the multiplicativity of the characteristic polynomial for equivariant exact sequences. First let us remark that since $\Omega^1_C/\mathcal{T}$ identifies with the sheaf of differentials on $\mathbb{P}^1$ that vanish at zero, so it has no global sections. Hence $\mathrm{H}^0(C, \Omega^1_C)$ is isomorphic as a right $\mathrm{Aut}(C)$-module to $\mathrm{H}^0(C, \mathcal{T})$. This gives the first point of the lemma. We now consider the right $G_{\zeta}$-equivariant exact sequence
\[
0 \rightarrow \mathrm{H}^0(C, \mathbb{C}_Q) \rightarrow \mathrm{H}^1 \left(C, \mathcal{O}_C \left(-4P_0+\frac{3}{2 \zeta}\right) \right) \rightarrow \mathrm{H}^1(C, \Omega^1_C) \rightarrow 0.
\]
Thanks to Serre duality (\textit{see} \S \ref{serre}), 
\[
\mathrm{H}^1 \left(C, \mathcal{O}_C \left(-4P_0+ \frac{3}{2\zeta} \right) \right) \simeq \mathrm{H}^0 \left(C, \mathcal{O}_C \left(4P_0-\frac{3}{2\zeta} \right) \otimes \omega_C \right)^*. 
\]
Note that $\omega_C$ is trivial as a usual holomorphic line bundle, but not as an equivariant holomorphic line bundle. Here we consider it as a right $\mathrm{Aut}(C)$-bundle, the action of the automorphism group $\mathrm{Aut}(C)$ being given by multiplication with the multiplier. On the other hand, the right $G_{\zeta}$-equivariant structure on $\mathcal{O}_C \left(4P_0-\displaystyle\frac{3}{2\zeta} \right)$ is given by
\[
f \rightarrow f(at+b) \times  \displaystyle \frac {(at+b)^4(t-3/2\zeta)}{at^4 (at+b-3/2\zeta)} \cdot
\]
Hence the right $G_{\zeta}$-equivariant structure on $\mathcal{O}_C \left(4P_0-\displaystyle\frac{3}{2\zeta} \right) \otimes \omega_C$ is given by
\[
f \rightarrow f(at+b) \times  \displaystyle \frac {(at+b)^4(t-3/2\zeta)}{t^4 (at+b-3/2\zeta)} \cdot
\]
The action of $\varphi$ on the equivariant line bundle $\mathcal{O}_C(4P_0-3/2\zeta) \otimes \omega_C$ of degree three at the fixed point $1/2\zeta$ is the identity, so that
thanks to Proposition \ref{marre} (ii), the eigenvalues of the action of $\varphi$ on $\mathrm{H}^0 \left(C, \mathcal{O}_C \left(4P_0- \frac{3}{2\zeta} \right) \otimes \omega_C\right)$ are $1, a$ and $a^3$. Hence the eigenvalues of the action of $\varphi$ on the contragredient representation $\mathrm{H}^0 \left(C, \mathcal{O}_C \left(4P_0- \frac{3}{2\zeta} \right) \otimes \omega_C\right)^*$ are $1, a^{-1}$ and $ a^{-3}$. This finishes the proof.
\end{proof}
%
\begin{thm}  \label{cusp}
Let $X$ be a basic rational surface such that $\mathrm{K}_X^2<0$ endowed with an automorphism $f$, and assume that there exists an irreducible curve $C$ in $|-\mathrm{K}_X|$ (such a curve is automatically fixed by $f$). Let $P_f$ and $\theta_{f}$ denote the characteristic polynomials of $f^*$ acting on $\mathrm{NS}_{\mathbb{Q}}(X)$ and $\mathrm{H}^0\big(C, \mathrm{N}^*_{C/X}\big)$ respectively, and let $Q_f$ denote the characteristic polynomial of $f_*$ acting on $\mathrm{H}^1\big(X, \mathrm{T}X\big)$. Lastly, let $a_f$ be the multiplier of $f$.
\par \smallskip
\begin{enumerate}
\item[(i)] If $C$ is cuspidal, $Q_f(x)=\displaystyle \frac{P_f(x)\, \theta_f(x) \, (x-a_f^{-5})(x-a_f^{-7})}{(x-1)(x-a_f^{-1})} \cdot$
\item[(ii)] If $C$ is smooth, 
$
Q_f(x)=\dfrac{P_f(x)\,\theta_f(x)\, (x-a_f)}{x-1} \cdot
$
\end{enumerate}
\end{thm}

\begin{rem}
Since the canonical class is fixed by $f$, $(x-1)$ always divides $P_f$. In the cuspidal case, we don't know if $(x-a_f^{-1})$ always divides $P_f$, but this will be the case in the forthcoming examples of McMullen.
\end{rem}

\begin{proof} 
For classical tools in algebraic geometry concerning curves and curves on surfaces used in this proof (\emph{e.g.} arithmetic and geometric genus, normal and conormal sheaves, dualizing sheaf, Riemann-Roch theorem), we refer the reader to \cite[Chap. II]{BPVDV}.
\par \medskip
Let $\langle f \rangle$ be the group generated by $f$. We consider the long exact sequence of cohomology associated with the short exact sequence
\[
0 \longrightarrow \Omega_X^1(-C) \longrightarrow \Omega_X^1 \longrightarrow \Omega^1_{X | C} \longrightarrow 0
\]
of right $\langle f \rangle$-equivariant sheaves on $X$. According to Serre duality \cite[III Cor. 7.7]{Hartshorne}, for $0 \leq j \leq 2$, \[\mathrm{H}^j\big(X, \mathrm{T}X\big)^* \simeq \textrm{H}^{2-j}\big(X, \Omega^1_X \otimes \textrm{K}_X\big) \simeq  \textrm{H}^{2-j}\big(X, \Omega^1_X (-C)\big).\] 
We have $\mathrm{h}^2 \bigl(X, \mathrm{T}X \bigr)=0$ and 
\[
\mathrm{h}^2(X, \Omega^1_X) \simeq \mathrm{h}^0 \bigl(X, \mathrm{T}X \otimes \mathrm{K}_X\bigr)=\mathrm{h}^0 \bigl(X, \mathrm{T}X(-C)\bigr)\leq \mathrm{h}^0 \bigl(X, \mathrm{T}X\bigr)=0.
\] 
Hence we have an exact sequence of right $\langle f \rangle$-modules
\begin{equation} \label{mainsequence}
0 \longrightarrow \mathrm{H}^0\bigl(C, \Omega^1_{X | C}\bigr) \longrightarrow \textrm{H}^1\bigl(X,\mathrm{T}X\bigr)^* \longrightarrow \mathrm{H}^{1, 1}(X) \stackrel{\pi}{\longrightarrow} \mathrm{H}^1\bigl(C, \Omega^1_{X | C}\bigr) \longrightarrow 0. 
\end{equation} 
Remark that the right $\langle f \rangle$-module structure on $\textrm{H}^1\bigl(X,\mathrm{T}X\bigr)^*$ is simply given by the transpose of the action $f_*$ of $f$ on $\mathrm{T}X$. We can now write down the conormal exact sequence \cite[II Prop. 8.12]{Hartshorne}:
\begin{equation}
0 \longrightarrow \mathrm{N}^*_{C/X} \longrightarrow \Omega^1_{X| C} \longrightarrow \Omega_C^1\longrightarrow 0
\end{equation}
\noindent where the injectivity of the first arrow holds because its kernel is a torsion subsheaf of $\mathrm{N^*_{C/X}}$, the latter being locally free. Since the arithmetic genus of $C$ is $1$, the dualizing sheaf $\omega_C$ is trivial. Besides, the conormal bundle $\mathrm{N}^*_{C/X}$ has degree $- \mathrm{K}_X^2$ which is positive by assumption, so that combining Riemann-Roch and Serre duality \cite[p. 82--83]{HarrisM}, we get 
$\mathrm{h}^0\big(C, \mathrm{N}^*_{C/X}\big)=-\mathrm{K}_X^2$ and $\mathrm{h}^1\big(C, \mathrm{N}^*_{C/X}\big)$ vanishes. Therefore $\mathrm{H}^1 \bigl(C, \Omega^1_{X | C}\bigr) \simeq \mathrm{H}^1 \bigl(C, \Omega^1_{C}\bigr)$. It follows that $\pi$ can be identified with the pullback morphism from $\mathrm{H}^1 \bigl(X, \Omega_X^1)$ to $\mathrm{H}^1 \bigl(C, \Omega_C^1)$ induced by the injection of the curve $C$ in $X$.
We have two exact sequences of right $\langle f \rangle$-modules
\[
\begin{cases}
\, 0 \longrightarrow \mathrm{H}^0\big(C, \Omega^1_{X| C}\big) \longrightarrow \mathrm{H}^1\big(X, \mathrm{T}X\big)^* \longrightarrow \mathrm{H}^{1,1}(X) \longrightarrow \mathrm{H}^1(C, \Omega^1_C) \longrightarrow 0 \\
\, 0 \longrightarrow \mathrm{H}^0\bigl(C, \mathrm{N}^*_{C/X}\bigr) \longrightarrow \mathrm{H}^0\bigl(C, \Omega^1_{X| C}\bigr) \longrightarrow \mathrm{H}^0 \bigl(C, \Omega_C^1 \bigr) \longrightarrow 0
\end{cases}
\]
If $C$ is smooth, this yields $
Q_f(x) \,(x-1)=P_f(x) \, \theta_C(x)\, (x-a_f).$
If $C$ is cuspidal, Lemma \ref{mexico} gives
\[
Q_f(x) \, (x-1) (x-a_f^{-1})=P_f(x)\, \theta_C(x)\, (x-a_f^{-5}) (x-a_f^{-7}).
\]
\end{proof}

\subsubsection{Cycle of rational curves}
We now investigate the case of effective anticanonical divisors given by a cycle of smooth rational curves and leave the three other exceptional configurations to the reader. 
\par \medskip
Remark that if $f$ is an automorphism of a rational surface with positive entropy, then $|-\mathrm{K}_X|$ is either empty or consists of a single divisor (otherwise $f$ would preserve a rational fibration). In this last case, by replacing $f$ by an iterate, we can assume that all irreducible components of the divisor are globally invariant by $f$.

\begin{thm} \label{cycle}
Let $X$ be a basic rational surface without nonzero holomorphic vector field. Assume that there exists a reducible and reduced effective anticanonical cycle $D=\sum_{i=1}^r D_i$ on $X$ such that $D_i^2<0$ for all $i$. For any automorphism $f$ of $X$ leaving each~$D_i$ globally invariant, let $P_f$ and $\theta_{i, f}$ denote the characteristic polynomials of $f^*$ acting on ${\mathrm{NS}_{\mathbb{Q}}(X)}$, and $\mathrm{H}^0\big(D_i, \mathrm{N}^*_{D_i/X}(-S_i)\big)$ respectively, where $S_i=\mathrm{sing}(D) \cap D_i$. Lastly, let $Q_f$ be the characteristic polynomial of $f_*$ acting on $\mathrm{H}^1\big(X, \mathrm{T}X\big)$. Then
\[
Q_f(x)=\frac{P_f(x)}{(x-1)^r} {\prod_{i=1}^r\theta_{i, f}(x)}.
\]
\end{thm}

\begin{proof}
Let $S$ be the singular locus of $D$. Then we have an exact sequence
\[
0 \longrightarrow \Omega^1_{X| D} \longrightarrow \oplus_{i=1}^r \, \Omega^1_{X| D_i} \longrightarrow \Omega^1_{X| S}  \longrightarrow 0.
\]
Since $\mathrm{N}^*_{D_i/X}$ has positive degree, $\mathrm{H}^1\bigl(D_i ,\Omega^1_{X| D_i} \bigr) \simeq \mathrm{H}^1\bigl(D_i ,\Omega^1_{D_i} \bigr) \simeq \mathbb{C}$. Besides, as $\mathrm{H}^0 \bigl(D_i ,\Omega^1_{D_i} \bigr)=0$, we obtain the isomorphism $\mathrm{H}^0\bigl(D_i ,\Omega^1_{X| D_i} \bigr) \simeq \mathrm{H}^0\bigl(D_i ,\mathrm{N}^*_{D_i/X} \bigr)$. It follows that 
\[
\mathrm{H}^0 \bigl(D, \Omega^1_{X| D} \bigr) \simeq \bigoplus_{i=1}^r \, \mathrm{H}^0 \bigl(D,  \mathrm{N}^*_{D_i/X}(-S_i)\bigr)
\]
and that the map from $\bigoplus_{i=1}^r \, \mathrm{H}^0 \bigl(D_i, \Omega^1_{X| D_i} \bigr)$ to $\mathrm{H}^0 \bigl(S, \Omega^1_{X | S} \bigr)$ is onto. Therefore
\[
\mathrm{H}^1 \bigl(D,  \Omega^1_{X| D}\bigr) \simeq \bigoplus_{i=1}^r \mathrm{H}^1 \, \bigl(D_i ,\Omega^1_{X| D_i} \bigr) \simeq \bigoplus_{i=1}^r \, \mathrm{H}^1 \bigl(D_i ,\Omega^1_{D_i} \bigr) \simeq \mathbb{C}^r.
\]
We now write down the exact sequence (\ref{mainsequence}) used in the proof of Theorem \ref{cusp}:
\[
0 \longrightarrow \bigoplus_{i=1}^r \, \mathrm{H}^0 \bigl(D,  \mathrm{N}^*_{D_i/X}(-S_i)\bigr)\longrightarrow \mathrm{H}^1\big(X, \mathrm{T}X\big)^*\longrightarrow \mathrm{H}^{1, 1}(X)\longrightarrow \mathbb{C}^r \longrightarrow \mathrm{H}^0\bigl(X, \mathrm{T}X \bigr)^* \longrightarrow 0.
\]
Since we assumed that there were no nonzero holomorphic vector fields on $X$, we get 
\[
Q_f(x) \,(x-1)^r=P_f(x) \, \prod_{i=1}^r \theta_{i, f}(x).
\] 
\end{proof}

\subsection{Examples}
We have chosen three main examples of application of Theorem \ref{cusp}, but this doesn't exhaust the list of possibilities. For instance, other cases of quadratic Cremona transformations leaving a cubic invariant can be dealt with, \emph{e.g.} \cite[Ex. 3.3 and Th. 4.4]{Di}. 

\subsubsection{Quadratic transformations fixing a cuspidal cubic} \label{McCormick}

We start by recalling the construction of automorphisms of rational surfaces of positive entropy obtained from quadratic transformations of the projective plane fixing a cuspidal cubic. This construction goes back to McMullen \cite{MC}, but we will follow \cite{Di}.
\par \medskip
Let $\mathscr{C}$ be the cuspidal cubic in $\mathbb{P}^2$ given in homogeneous coordinates $(x : y : z)$ by the equation $y^2z=x^3$, and let $P_0=(0: 1: 0)$. The stabilizer in $\mathrm{PGL}(3; \mathbb{C})$ of the cubic $\mathscr{C}$ is isomorphic to $\mathbb{C}^{\times}$, each parameter $\delta$ in $\mathbb{C}^{\times}$ corresponding to the linear transformation $(x : y : z) \rightarrow (\delta x : y : \delta^3 z)$. There is a natural algebraic group structure on $\mathscr{C}^{\mathrm{reg}}$ with $P_0$ as origin compatible with the group structure on $\mathrm{Pic}^0(\mathscr{C})$, and the isomorphism $t \rightarrow (t: 1: t^3)$ between $\mathbb{C}$ and $\mathscr{C}^{\mathrm{reg}}$ is an isomorphism of algebraic groups (\emph{see} \cite[Ex. 6.11.4]{Hartshorne}).
\par \medskip
Let $\tau$ be a permutation of the set $\{1, 2, 3\}$, and let $n_1$, $n_2$, $n_3$ be three integers. The set $\{\tau, n_1, n_2, n_3\}$ is called an \textit{orbit data} (\emph{see} \cite{BK}). If $|\tau|$ is the order of $\tau$ in the symmetric group $\mathfrak{S}_3$, we define two polynomials ${p}_{\tau}$ and $P_{\tau}$ as follows:
\[
\begin{cases}
p_{\tau}(x)=1-2x+\sum_{j=\tau(j)} x^{1+n_j}+\sum_{j\neq\tau(j)} x^{n_j}(1-x) \\
P_{\tau}(x)=x^{1+n_1+n_2+n_3}\,p_{\tau}(x^{-1})+(-1)^{|\tau|}p_{\tau}(x).
\end{cases}
\]
The degree of $P_{\tau}$ is $n_1+n_2+n_3+1$, and $P^*_{\tau}=P_{\tau}$. Let $f$ be a quadratic birational transformation of $\mathbb{P}^2$, and assume that the base locus of $f$ consists of three distinct points $p_1^+$, $p_2^+$ and $p_3^+$. If $\Delta_{ij}$ denotes the line between $p_i^+$ and $p_j^+$, then the base locus of $f^{-1}$ consists of the three points $p_1^-=f(\Delta_{23})$, $p_2^-=f(\Delta_{13})$ and $p_3^-=f(\Delta_{12})$. We say that $f$ \textit{realizes the orbit data} if for any $j$ with $1 \leq j \leq 3$, the orbit $\{f^k(p_j^-)\}_{0 \leq k < n_j-1}$ consists of $n_j$ pairwise distinct points outside the base loci of $f$ and $f^{-1}$ and if $f^{n_j-1}(p_j^-)=p_{\tau(j)}^+$. In this case, the birational map $f$ lifts to an automorphism of the rational surface obtained by blowing up the $n_1+n_2+n_3$ points corresponding to the orbits of the points $p_j^-$. The corresponding characteristic polynomial for the action on the Picard group of this surface is $P_{\tau}$ (\emph{see} \cite[\S 2.1]{Di}). We say that an orbit data is \textit{admissible} if it satisfies the following conditions:
\[
\begin{cases}
\textrm{If } n_1=n_2=n_3 \,\,\,\textrm{then}\,\,\, \tau=\mathrm{id}. \\
\textrm{If } n_i=n_j \,\,\,\textrm{for}\,\,\, i \neq j \,\,\,\textrm{either $\tau(i) \neq j$ or $\tau(j) \neq i$}.\\
\textrm{All}\,\,\, n_i \,\,\, \textrm{are at least} \,\,\, 3\,\,\, \textrm{and one is at least}\,\,\, 4.
\end{cases}
\]

\begin{pro}[{\cite[Thm. 1 \& 3]{Di}}] \label{Diller}$ $
\begin{itemize}
\item[(i)] Let $p_1^+$, $p_2^+$ and $p_3^+$ be three points in $\mathscr{C}^{\mathrm{reg}}$ such that $p_1^++p_2^++p_3^+\neq0$, and let $\mu$ be a given number in $\mathbb{C}^{\times}$. Then there exists a unique quadratic transformation $f$ of $\mathbb{P}^2$ fixing $\mathscr{C}$, having $p_1^+$, $p_2^+$ and $p_3^+$ as base points, and such that $f_{| \mathscr{C}}$ has multiplier $\mu$. Besides, the translation factor of $f_{| \mathscr{C}}$ is $\epsilon=\frac{1}{3}(p_1^++p_2^++p_3^+)\, \mu$ and the points $p_j^-$ are given by $p_j^-=\mu p_j^+-2\epsilon$ for $1 \leq j \leq 3$.
\par \smallskip
\item[(ii)] Let $(\tau, n_1, n_2, n_3)$ be an admissible orbit data, and assume furthermore that $\mu$ is a root of $P$ that is not a root of unity. Then there exist $p_1^+$, $p_2^+$ and $p_3^+$ with $p_1^++p_2^++p_3^+\neq0$ such that the quadratic map $f$ given in $(i)$ realizes the orbit data $(\tau, n_1, n_2, n_3)$ and has multiplier $\mu$. Besides, such a quadratic map is unique modulo conjugation by the centralizer of $\mathscr{C}$ in $\mathrm{PGL}(3; \mathbb{C})$.
\end{itemize}
\end{pro}

\begin{pro} \label{miracle}
Let $p_1^+$, $p_2^+$ and $p_3^+$ be three points in $\mathscr{C}^{\mathrm{reg}}$ such that $p_1^++p_2^++p_3^+\neq0$, let $(\tau, n_1, n_2, n_3)$ be an admissible orbit data, let $\mu$ be a root of $P_{\tau}$ that is not a root of unity, and let $f$ be the corresponding birational quadratic map given by Proposition \ref{Diller}. Then the eigenvalues of $\mathrm{d}\!f$ at the unique fixed point on $\mathscr{C}^{\textrm{reg}}$ are $\mu$ and $\mu^{3-n_1-n_2-n_3}$. 
\end{pro}

\begin{proof}
Due to a lack of conceptual arguments, we provide a direct proof by calculation, which is computer assisted. We sketch the argument and refer to the Maple file for computational details.
\par \medskip
Since the stabilizer of $\mathscr{C}$ in $\mathrm{PGL}(3; \mathbb{C})$ acts transitively on $\mathscr{C}^{\textrm{reg}}$, we can assume without loss of generality that $p_1^+=(1: 1: 1)$. We put $p_2^+=(\alpha: 1: \alpha^3)$ and $p_3^+=(\beta: 1: \beta^3)$ and assume that $\alpha+\beta \neq -1$. If $\sigma \colon \mathbb{P}^2 \dashrightarrow \mathbb{P}^2$ is the Cremona involution given by $\sigma\, (x: y: z)=(yz: xz: xy)$ and $N$ is the linear transformation sending the three points $(1: 0: 0)$, $(0: 1: 0)$ and $(0: 0: 1)$ to $p_1^+$, $p_2^+$ and $p_3^+$ respectively, then the base locus of the quadratic transformation $\sigma \circ N$ consists of the three points $\{p_i^+\}_{1 \leq i \leq 3}$. It is then possible to find a matrix $M$
such that $f=M \circ \sigma \circ N$ fixes the cuspidal cubic $\mathscr{C}$. Up to multiplying by an element of the centralizer of $\mathscr{C}$ in $\mathrm{PGL}(3; \mathbb{C})$, we can arrange that the multiplier of $f_{\vert \mathscr{C}}$ is $\mu$. 
\par \medskip
The translation factor is $\epsilon=\frac{1}{3}(\alpha+\beta+1)\,\mu$, so that the unique point on $\mathscr{C}^{\textrm{reg}}$ fixed by $f$ is $p=\frac{(\alpha+\beta+1)\, \mu}{3(1-\mu)} \cdot$ We have $p_1^-=\frac{1}{3}{(1-2\alpha-2\beta)\,\mu}$, $p_2^-=\frac{1}{3}{(-2+\alpha-2\beta)\,\mu}$ and $p_3^-=\frac{1}{3}{(-2-2\alpha+\beta)\,\mu}$. We now write down the conditions in order that $f$ realizes the orbit data. We have $f^n(t)=\mu^n(t-p)+p$, so that we must solve the equations \[
\mu^{n_i-1} (p_i^--p)+p=p_i^+ \quad \textrm{for} \textrm 1 \leq i \leq 3.
\] 
If we take two of these three equations, we have two affine equations in the variables $\alpha$ and $\beta$ that give uniquely $\alpha$ and $\beta$ as rational fractions in $\mu$. The third equation is automatically satisfied if $\mu$ is a root of $P_{\tau}$. Then we can compute the eigenvalues of $\mathrm{d}\!f$ at $p$: the multiplier $\mu$ is of course an eigenvalue, and the other one $\zeta$ is a rational fraction in $\mu$. Now an explicit calculation yields the formula $\zeta \,\mu^{n_1+n_2+n_3-3}=1$.
\end{proof}
\begin{thm} \label{canada}
Let  $(\tau, n_1, n_2, n_3)$ be an admissible orbit data, let $\mu$ be a root of $P_{\tau}$ that is not a root of unity, and let $f$ be a birational quadratic map realizing the orbit data $(\tau, n_1, n_2, n_3)$, fixing the cuspidal cubic $\mathscr{C}$, and having multiplier $\mu$ when restricted to the cubic $\mathscr{C}$. If $X$ is the corresponding rational surface and if $g$ is the lift of $f$ as an automorphism of $X$, then $\mathrm{m}(X, g)\leq 3-|\tau|$. In particular, if $\tau$ has order three, then $g$ is rigid.
\end{thm}

\begin{proof}
Recall that $X$ is obtained by blowing up $n_1+n_2+n_3$ points, and that the conormal bundle of the strict transform of $\mathscr{C}$ in $X$ has degree $n_1+n_2+n_3-9$. Thanks to Theorem \ref{cusp} (a), Corollary \ref{enfin} (ii) and Proposition \ref{miracle}, the characteristic polynomial $Q_g$ of $g$ acting on $\mathrm{H}^1(X, \mathrm{T}X)$ is given by
\[
Q_g(x)=\frac{P_{\tau}(x) \, \left(  \prod_{\substack{j=6 \\ j\neq 7}}^{n_1+n_2+n_3-3}(x-\mu^{-j}) \right) (x-\mu^{-5})(x-\mu^{-7})} {(x-1)(x-\mu^{-1})} =\frac{P_{\tau}(x) \, \prod_{\substack{j=5}}^{n_1+n_2+n_3-3}(x-\mu^{-j})}{(x-1)(x-\mu^{-1})} \cdot
\]
We look at the multiplicity of the root $1$ in $Q_g$. Since $\mu$ is not a root of unity, we can reduce the problem to the polynomial $P_{\tau}$. We claim that $1$ occurs with multiplicity $4-|\tau|$. We deal with the three different cases.
\begin{enumerate}
\item[--]
If $\tau=(1)(2)(3)$, $P_{\tau}(1)=P'_{\tau}(1)=P^{(2)}_{\tau}(1)=0$, and $P^{(3)}_{\tau}(1)=6n_1n_2n_3 \left(\frac{1}{n_1}+\frac{1}{n_2}+\frac{1}{n_3}-1 \right)<0.$
\item[--]
If $\tau=(12)(3)$, $P_{\tau}(1)=P'_{\tau}(1)=0$, and $P^{(2)}_{\tau}(1)=2[4n_3-(n_1+n_2)(n_3-1)]\neq 0$. Now $n_3-1$ cannot divide $4n_3$ except if $n_3=3$ or $n_3=5$. If $n_3=3$, we would have $n_1+n_2=6$, which is not possible. If $n_3=5$, we get $n_1+n_2=5$ which is also excluded.
\item[--]
If $\tau=(123)$, $P_{\tau}(1)=0$, and $P'_{\tau}(1)=9-n_1-n_2-n_3 < 0$.
\end{enumerate}
\end{proof}

\subsubsection{The construction of Blanc and Gizatullin} \label{noirhaha}

We start by a brief recollection of the construction of automorphisms fixing a smooth elliptic curve in \cite{Blanc2}.
\par \medskip
Let $\mathscr{C}$ be a smooth cubic in $\mathbb{P}^2$. For any point $p$ in $\mathscr{C}$, let $\sigma_p$ be the birational involution of $\mathbb{P}^2$ defined as follows: if $\ell$ is a generic line of $\mathbb{P}^2$ passing through $p$, $(\sigma_p)_{| \ell}$ is the involution of $\ell$ fixing the two other intersection points of $\ell$ with $\mathscr{C}$. The involution $\sigma_p$ has five distinct\footnote{If we choose $p$ as the origin of $\mathscr{C}$, then the four other base points are the $2$-torsion points of $\mathscr{C}$.} base points (including $p$), and becomes an automorphism on the rational surface obtained by blowing up these five points. 
We denote the set of five base points of $\sigma_p$ by $S_p$. If $S_p=\{p, p_1, p_2, p_3, p_4\}$, then the $p_i$'s are the four points of $\mathscr{C}$ such that $(pp_i)$ is tangent to $\mathscr{C}$ at $p_i$. Let $E_{p}, E_{p_1}, E_{p_2}, E_{p_3}, E_{p_4}$ be the exceptional divisors of the blowup. The lift of $\sigma_p$ maps $E_p$ to the unique conic of $\mathbb{P}^2$ passing through all points of $S_p$ and each $E_{p_i}$ to the line $(p p_i)$. Thus the action of $\sigma_p$ on the Picard group of the blown-up surface is given by

\begin{equation}
\left \{ 
\begin{split} \label{formule}
&E_p \rightarrow 2H-E_p-E_{p_1}-E_{p_2}-E_{p_3}-E_{p_4} \\
&E_{p_i} \rightarrow H-E_p-E_{p_i} \quad 1 \leq i \leq 4 \\
&H \rightarrow 3H-2E_p-E_{p_1}-E_{p_2}-E_{p_3}-E_{p_4}
\end{split}
\right.
\end{equation}
where $H$ is the pull-back of the hyperplane class (\emph{see} \cite[Lemma 17]{Blanc2}).
\par \medskip
Let us now fix three points $p$, $q$, $r$ on $\mathscr{C}$ such that the set $S_p$, $S_q$ and $S_r$ do not overlap (the construction works with arbitrary many points and without the genericity condition, \emph{see} \cite{Blanc2} for further details). Then $\sigma_p$, $\sigma_q$ and $\sigma_r$ lift to automorphisms of the rational surface $X$ obtained by blowing up $\mathbb{P}^2$ in the $15$ points $S_p$, $S_q$, $S_r$ (we will still denote the lifts by $\sigma_p$, $\sigma_q$ and $\sigma_r$) and provides an embedding of the free product $\mathbb{Z}/2 \mathbb{Z} \star \mathbb{Z}/2 \mathbb{Z} \star \mathbb{Z}/2 \mathbb{Z}$ in $\mathrm{Aut}\,(X)$. 
\par \medskip
This construction can be made in families: the family of smooth cubics in $\mathbb{P}^2$ with three (ordered) distinct marked points is a smooth quasi-projective variety $\mathcal{V}$ of dimension $9+3=12$. We have a natural deformation $\mathfrak{T}$ of rational surfaces over $\mathcal{V}$: for any point $(C, p, q, r)$ in $\mathcal{V}$, the corresponding rational surface $X$ is the blowup of the projective plane $\mathbb{P}^2$ at $S_p$, $S_q$ and $S_r$. Besides, $\mathfrak{T}$ is endowed with three involutions $\mathfrak{s}_p$, $\mathfrak{s}_q$ and $\mathfrak{s}_r$.
\par \medskip
Let $\Psi=\mathfrak{s}_p \circ \mathfrak{s}_q \circ \mathfrak{s}_r$; it is an automorphism of $\mathfrak{T}$. For any $v$ in $\mathcal{V}$, the characteristic polynomial of $\Psi^*_v$ acting on the Picard group of $\mathfrak{T}_v$ is $P_{\Psi_{v}}=(t^2-18t+1)(t-1)^4(t+1)^{10}$ so that the first dynamical degree of $\Psi_v$ is $9+4 \sqrt{5}$. In particular $\Psi_v$ has positive entropy.

\begin{thm} \label{moose}
The family $(\mathfrak{T}, \Psi)$ is complete at all of its fibers.
\end{thm}

\begin{proof}
The three involutions $\sigma_p$, $\sigma_q$ and $\sigma_r$ fix pointwise the strict transform $C$ of the curve $\mathscr{C}$, and act by multiplication by $-1$ on any fiber of the conormal bundle $\mathrm{N}^*_{C/X}$. Thus for any point $v$ in the base $\mathcal{V}$, $\Theta_{\Psi_v}=(t+1)^{6}$ and $a_{\Psi_v}=1$. Hence we get by Theorem \ref{cusp} the formula 
\[
Q_{\Psi_v}=(t^2-18t+1)(t-1)^4(t+1)^{16}
\]
so that $\mathrm{dim} \, \mathrm{ker}\,(\Psi_{v}^*-\mathrm{id})=4$. The base $\mathcal{V}$ of $\mathfrak{T}$ is a submanifold of $S_{15}^{\dag}$ of dimension $12$ which is $\mathrm{PGL}(3; \mathbb{C})$-invariant. By Proposition \ref{pro:pgl3}, the Kodaira-Spencer map of $(\mathfrak{T}, \Psi)$ has rank $4$ at every point. It is therefore surjective, and the result follows from Theorem \ref{thm:kodaira}.  
\end{proof}

\subsubsection{Unnodal Halphen surfaces}

We start with a short reminder about Halphen surfaces. We refer to \cite[Prop. 2.1]{CD} and \cite[\S 7]{Grr} for more details. 

\begin{defi}
A Halphen surface of index $m$ is a rational surface $X$ such that $|-m\mathrm{K}_X|$ has no fixed part and defines a base point free pencil.
\end{defi}

If $X$ is a Halphen surface of index $m$, $\mathrm{K}_X^2=0$ and the generic fiber of the pencil $-m\mathrm{K}_X$ is a smooth elliptic curve. In fact, Halphen surfaces are exactly the minimal rational surfaces. They are obtained by blowing up a pencil of curves of degree $3m$ in $\mathbb{P}^2$ with $9$ base points\footnote{Such a pencil is called a Halphen pencil of index $m$. If $m=1$, this is an ordinary pencil of cubics in the plane.} along the base locus of the pencil, and $|-m\mathrm{K}_X|$ is the strict transform of the pencil. If $m=1$, all members of the pencil are anticanonical divisors. If $m \geq 2$, the Riemann-Roch theorem yields the existence of an anticanonical divisor $\mathfrak{D}$, and the only multiple fiber of the elliptic pencil is $m \mathfrak{D}$. Besides, $\mathrm{N}^*_{\mathfrak{D}/X}$ is a torsion point of index $m$ in $\mathrm{Pic}(\mathfrak{D})$.
\par \medskip
Let us explain a concrete way (given at the end of \cite[\S 2.2]{CD}, \emph{see also} \cite[\S 10.5.1]{DS}) to construct Halphen surfaces directly without Halphen pencils. For simplicity, we will only consider the case where the anticanonical divisor is smooth and reduced. Let $\mathscr{C}$ be a smooth cubic in $\mathbb{P}^2$, and consider nine pairwise distinct points\footnote{The construction works also for infinitely near points, but all points $p_i$ must be based on the various strict transforms of $\mathscr{C}$.} $p_1, \ldots, p_9$ on $\mathscr{C}$. Let $\mathfrak{h}$ be the class $\mathcal{O}_{\mathbb{P}^2}(1)_{| \mathscr{C}}$ in $\mathrm{Pic}(\mathscr{C})$. We define $X$ as the blowup of $\mathbb{P}^2$ at the nine points $p_i$. Let $\mathfrak{d}=3 \mathfrak{h} -\sum_{i=1}^9 [p_i]$, and let $C$ be the strict transform of $\mathscr{C}$ in $X$. Considering $\mathfrak{d}$ as a divisor on $C$, $\mathfrak{d}=\mathrm{N}^*_{C/X}$ in $\mathrm{Pic}^0(C)$. If $m \mathfrak{d}=0$, then $X$ is a Halphen surface of index $m$. Let $o$ be an inflexion point of $\mathscr{C}$, so that $\mathfrak{h} \sim 3o$. Choosing $o$ as an origin in $\mathscr{C}$ and denoting by $\oplus$ the group law in $\mathscr{C}$, the condition $m\mathfrak{d}=0$ in $\mathrm{Pic}^0(C)$ means that $m \, (p_1 \oplus \ldots \oplus p_9)=o$ in $\mathscr{C}$.
\par \medskip
We can put this construction in families. We must distinguish the cases $m=1$ and $m \geq 2$, which behave slightly differently.
\par \medskip
If $m \geq 2$, 
let $\mathcal{U}_m$ be the set of pairs $\{\mathscr{C}, p_1, \ldots, p_9\}$ where $\mathscr{C}$ is a smooth cubic, the $p_i$'s lie on $\mathscr{C}$ and are pairwise distinct, and $m \mathfrak{d}=0$ in $\mathrm{Pic}^0(\mathscr{C})$; this is a smooth quasi-projective variety of dimension $9+9-1=17$ with a natural action of $\mathrm{PGL}(3; \mathbb{C})$. Besides, we have a universal family $\mathfrak{H}_m$ of Halphen surfaces of index $m$ over $\mathcal{U}_m$. Since Halphen surfaces of index $m \geq 2$ have a unique anticanonical divisor, the cubic $\mathscr{C}$ is entirely determined by the points $p_i$, so that we can see $\mathcal{U}_m$ as a locally closed smooth algebraic variety in the configuration space $(\mathbb{P}^2)^9 / \mathfrak{S}_9$ stable by the action of $\mathrm{PGL}(3; \mathbb{C})$. 
\par \medskip
If $m=1$, the nine points $p_i$ don't determine the cubic curve $\mathscr{C}$. Let $\mathcal{U}_1$ denote the set of pencils of cubics on $\mathbb{P}^2$ with smooth generic fiber and no infinitely near points in the base locus, which is a Zariski open subset of $\mathbb{P}^{16}$ that can be seen as a smooth locally closed subset of $(\mathbb{P}^2)^9 / \mathfrak{S}_9$ via the map associating to each pencil its base locus. The variety $\mathcal{U}_1$ carries a natural family $\mathfrak{H}_1$ of Halphen surfaces of index $1$ obtained by blowing up the base locus of the pencils. Besides, $\mathcal{U}_1$ is again stable by the action of $\mathrm{PGL}(3; \mathbb{C})$.
\par \medskip
Lastly, let us recall some results about \textit{unnodal} Halphen surfaces. We refer the reader to \cite[\S 2.3]{CD} for more details. A Halphen surface is called unnodal if all members of the pencil $|-m\mathrm{K}_X|$ are irreducible. This condition is generic among Halphen surfaces of index $m$. In particular, every genericity condition we have assumed on the Halphen set is satisfied for unnodal Halphen surfaces. Besides, their automorphisms group admits a particularly nice description: if $X$ is unnodal, the lattice $\mathfrak{L}(X)=(\mathrm{K}^{\perp}_X \cap \mathrm{H}^2(X, \mathbb{Z}))/\mathbb{Z} {\mathrm{K}_X}$ embeds naturally as a finite-index subgroup of $\mathrm{Aut}(X)$; and this group acts by translation on the fibers of the elliptic fibration. The action of any element $\alpha$ of $\mathfrak{L}(X)$ on the Picard group of $X$ is given by the explicit formula\footnote{There is a sign mistake in \cite{GIZ}, also pointed out in \cite{Cantat}.}
\begin{equation} \label{omg}
f_{\alpha}^*(D)=D-m\,(D.\mathrm{K}_X)\, \alpha+\left\{m\,(D. \alpha)-\frac{m^2}{2} (D.\mathrm{K}_X)\, \alpha^2 \right\} \mathrm{K}_X.
\end{equation}
In particular, $f_{\alpha}^*$ acts unipotently on $\mathrm{Pic}(X)$, but $f_{\alpha}^*$ is not of finite order (it is a true parabolic element). Thus $X$ has no nonzero holomorphic vector field.
\begin{pro} \label{yazuka}
Let $X$ be an unnodal Halphen surface of some index $m$ carrying a smooth anticanonical curve, and let $\alpha$ be an element in the lattice $\mathfrak{L}(X)$. If $U$ is a small neighborhood of a point defining $X$ in $\mathcal{U}_m$, let $\mathfrak{f}_{\alpha}$ be a lift of $f_\alpha$\footnote{The lift $\mathfrak{f}_{\alpha}$ exists because the lattices $\mathfrak{L}(X)$ form a local system of abelian groups over $\mathcal{U}_m$.} on the family $(\mathfrak{H}_m)_{|U}$. Then $\{(\mathfrak{H}_m)_{|U}, \mathfrak{f}_{\alpha}\}$ is a complete deformation of $(X, f_{\alpha})$.
\end{pro}

\begin{proof}
Let $\{X_t, f_t\}_t$ be a local deformation of the pair $(X, f_{\alpha})$.  Then for any $t$, the automorphism $f_t$ remains a parabolic isometry of $\mathrm{H}^2(X_t, \mathbb{Z})$. Therefore, thanks to the main result of \cite{GIZ} (\emph{see} \cite{Grr}), $X_t$ is a Halphen surface. Now thanks to Proposition \ref{pro:pgl3}, we can write $X_t$ as the blowup of nine points $p_i(t), 1 \leq i \leq 9$ varying holomorphically with $t$. For $t$ small enough, $X_t$ is anticanonical so the points $p_i(t)$ lie on a plane cubic curve $\mathscr{C}_t$. Since $X_t$ is Halphen, the point $\mathfrak{d}_t$ is a torsion point in $\mathrm{Pic}^0(\mathscr{C}_t)$. We can see the map $t \rightarrow \mathfrak{d}_t$ as a local holomorphic section of the Jacobian variety of $X_t$ whose values are torsion elements. Since the order of $\mathfrak{d}_0$ is $m$, it follows that all $\mathfrak{d}_t$ have order $m$. Thus the points $p_i(t)$ define an element in $\mathcal{U}_m$, which proves that $\{X_t\}_t$ is obtained by pullback from $\mathfrak{H}_m$ and then $\{X_t, f_t\}_t$ is obtained by pullback from $\{(\mathfrak{H}_m)_{|U}, \mathfrak{f}_{\alpha}\}$.
\end{proof}

The proof of Proposition \ref{yazuka} relies heavily on Gizatullin's result. Let us explain how it is possible to obtain this result (at least for $m \geq 2$) using our method.

\begin{pro} \label{hapft}
Let $X$ be an unnodal Halphen surface of index $m$ carrying a smooth anticanonical curve. Then there is a natural exact sequence
\[
0 \longrightarrow \mathbb{C} \longrightarrow \mathrm{H}^1(X, \mathrm{TX})^* \longrightarrow \mathrm{K}_X^{\perp} \longrightarrow 0
\]
of right $\mathfrak{L}(X)$-modules, where $\mathrm{K}_X^{\perp}$ denotes the orthogonal of the canonical class in $\mathrm{Pic}\,(X)$, and $\mathfrak{L}(X)$ acts trivially on $\mathbb{C}$.
In particular the action of $\mathfrak{L}(X)$ on $\mathrm{H}^1(X, \mathrm{T}X)$ is unipotent and for any $\alpha$ in $\mathfrak{L}(X)$, 
\[
8 \leq \mathrm{dim}\,((f_{\alpha})_*-\mathrm{id}) \leq 9
\] 
where $(f_{\alpha})_*$ is the action of $f_{\alpha}$ on $\mathrm{H}^1(X, \mathrm{T}X)$.
\end{pro}

\begin{rem}
Proposition \ref{yazuka} is stronger because it gives
$ \mathrm{dim}\, \mathrm{ker}\, ((f_{\alpha})_*-\mathrm{id}) =\left\{
\begin{alignedat}{1}
&8 \,\,\textrm{if} \,\, m=1 \\
&9 \,\,\textrm{if} \,\, m \geq 2.
\end{alignedat}
\right.
$
\end{rem} 

\begin{proof}
We adapt the proof of Theorem \ref{cusp} in this situation. To do so, we study the conormal exact sequence \eqref{mainsequence} of ${C}$. 
\par \medskip
\textbf{Assume that} $\mathbf{m\geq 2}$. Since $\mathrm{N}^*_{C/X}$ is a torsion point of order $m$ in $\mathrm{Pic}^0(C)$, $\mathrm{H}^0(C, \mathrm{N}^*_{C/X})=\{0\}$, so that by Riemann-Roch $\mathrm{H}^1(C, \mathrm{N}^*_{C/X})=\{0\}$. Hence we get isomorphisms
\[
\mathrm{H}^i ({C}, \Omega^1_{X \vert {C}})\simeq \mathrm{H}^i ({C}, \Omega^1_{C})  \qquad i \in \{0,1\},
\]
so that the kernel of the restriction map $\mathrm{H}^1(X, \Omega_X^1) \rightarrow \mathrm{H}^1(C, \Omega_{X|C}^1)$ identifies with $\mathrm{K}_X^{\perp}$. Since $\mathfrak{L}(X)$ acts trivially on $\mathrm{H}^0(C, \Omega^1_C)$, the result follows using \eqref{mainsequence}.
\par \medskip
\textbf{Assume that} $\mathbf{m=1}$. For any smooth fiber ${{C}}_t$ of the pencil, $\mathrm{N}_{{C}_t/X} $ is canonically isomorphic to $\mathcal{O}_{{C}_t}\otimes_{\mathbb{C}} \mathrm{T}_t \mathbb{P}^1$ and the extension class of the normal exact sequence (which is dual to \eqref{mainsequence}) identifies with the Kodaira-Spencer map of the family $\{{C}_s\}_{\vert s-t\vert < \epsilon}$ via the isomorphism 
\[
\mathrm{Ext}^1_{\mathcal{O}_{C_t}}({C_t}, \mathcal{O}_{{C}_t} \otimes_{\mathbb{C}} \mathrm{T}_t \mathbb{P}^1, \mathrm{T} {C}_t) \simeq \mathrm{H}^1({C}_t, \mathrm{T} {C}_t) \otimes_{\mathbb{C}} \mathrm{T}^*_t \mathbb{P}^1.
\] 
There are Halphen surfaces of index $1$ such that the complex structure of the smooth fibers of the elliptic pencil remains constant (they are described in \cite[Prop. B]{GIZ}) but they are not unnodal. Thus, if $C$ is a generic fiber of the elliptic fibration, the Kodaira-Spencer map $\kappa \colon \mathrm{T} \mathbb{P}^1 \rightarrow  \mathrm{H}^1({C}, \mathrm{T} {C})$ is nonzero. It follows that \eqref{mainsequence} is not holomorphically split over $\mathcal{O}_{C}$, so that the maps 
\[
\mathrm{H}^0 ({C}, \mathrm{N}^*_{{C}/X}) \rightarrow \mathrm{H}^0 ({C}, \Omega^1_{X \vert {C}}) \qquad \textrm{and} \qquad
\mathrm{H}^1 ({C}, \Omega^1_{X \vert {C}}) \rightarrow \mathrm{H}^1 ({C}, \Omega^1_{{C}})
\]
are isomorphisms. We conclude using \eqref{mainsequence} again.
 \end{proof}

As a corollary, we get a new proof of Proposition \ref{yazuka} for $m \geq 2$: indeed, the family $\{(\mathfrak{H}_m)_{|U}, \mathfrak{f}_{\alpha}\}$ is parameterized by a smooth base of dimension $17$ which is stable under the action of $\mathrm{PGL}(3; \mathbb{C})$. Thanks to Proposition \ref{hop}, its Kodaira-Spencer map has rank $9$ at any point of $U$. Then the conclusion follows from Theorem \ref{thm:kodaira}.

\section{Kummer surfaces}\label{Sec:kummer}

By definition, a Kummer surface $X$ is a desingularization of a quotient $\mathcal{A}/G$ where $\mathcal{A}$ is an abelian surface and $G$ is a finite group of automorphisms of $A$. These subgroups have been classified (\emph{see} \cite{Fujiki}), and the geometry of the corresponding Kummer surfaces have been studied in \cite{Yoshi}. Many situations can occur. The most famous case is $G=\{ \pm \textrm{id} \}$, and in this case $X$ is a $\textrm{K}3$ surface. In this part, we will deal with two special pairs $(\mathcal{A}, G)$ such that $X$ is a rational surface. There are many other cases apart these two ones where this happens (\emph{see} \cite[Thm 2.1]{Yoshi}).

\subsection{Rational Kummer surface associated with the hexagonal lattice}\label{Subsec:const}

\subsubsection{Basic properties}

Let $\mathcal{E}$ be the elliptic curve obtained by taking the quotient of the complex line $\mathbb{C}$ by the lattice $\Lambda=\mathbb{Z}[\mathbf{j}]$ of Eisenstein integers, and let $\mathcal{A}$ be the abelian surface $\mathcal{E} \times \mathcal{E}$, and let $\phi$ be the automorphism of order $3$ defined by $\phi(x,y)=(\mathbf{j}x,\mathbf{j}y).$ Since $\phi^2=\phi^{-1}$, the automorphisms $\phi$ and $\phi^2$ have the same $9$ fixed points which are
\[
\left\{\!\begin{alignedat}{5}
p_1&=(0,0)&\quad p_2&=\left(0,\frac{2}{3}+\frac{\mathbf{j}}{3}\right)&\quad p_3&=\left(0,\frac{1}{3}+\frac{2\mathbf{j}}{3}\right)
\\
p_4&=\left(\frac{2}{3}+\frac{\mathbf{j}}{3},0\right)&\quad p_5&=\left(\frac{2}{3}+\frac{\mathbf{j}}{3},\frac{2}{3}+\frac{\mathbf{j}}{3}\right) &\quad p_6&=\left(\frac{2}{3}+\frac{\mathbf{j}}{3},\frac{1}{3}+\frac{2\mathbf{j}}{3}\right) \\
p_7&=\left(\frac{1}{3}+\frac{2\mathbf{j}}{3},0\right)&\quad p_8&=\left(\frac{1}{3}+\frac{2\mathbf{j}}{3},\frac{2}{3}+\frac{\mathbf{j}}{3}\right)&\quad p_9&=\left(\frac{1}{3}+\frac{2\mathbf{j}}{3},\frac{1}{3}+\frac{2\mathbf{j}}{3}\right) 
\end{alignedat}\right.
\]
We denote this set by $S$, it is a subgroup of the $3$-torsion points in $\mathcal{A}$. Let $G$ be the group of order $3$ generated by $\phi$ in $\mathrm{Aut}(\mathcal{A})$. Then $\mathcal{A}/G$ is a singular surface, and the nine singularities corresponding to the points of $S$ are of type $A_{3}$. Their blowup produces a smooth projective surface $X$ called a rational Kummer surface, and the nine exceptional divisors are of self-intersection $-3$.
\par \medskip
To avoid using singular surfaces, we use a slightly different construction yielding the same surface $X$: first we blow up the set $S$ in $\mathcal{A}$ and denote by $\widetilde{\mathcal{A}}$ the resulting surface and by $\widetilde{E}_i$ the exceptional divisors corresponding to the points $p_i$. The group $G$ acts on $\widetilde{\mathcal{A}}$, and the quotient $\widetilde{\mathcal{A}}/G$ is $X$. For $1 \leq i \leq 9$, let $E_i$ be the image of $\widetilde{E}_i$ in $X$, it is a rational curve of self-intersection $-3$. We have the following diagram, where $\delta \colon \widetilde{\mathcal{A}} \rightarrow \mathcal{A}$ is the blowup map and 
$\pi\colon\widetilde{\mathcal{A}}\to X$ is the projection
\[
\xymatrix{& \widetilde{\mathcal{A}}\ar[ld]_{\delta}\ar[rd]^{\pi}&\\
A&&X
}
\]
We can describe more precisely the map $\pi$: $(\widetilde{\mathcal{A}}, \pi)$ is the cyclic covering of $X$ of order $3$ branched along the rational curves $E_i$. In particular, for $1 \leq i \leq 9$, we have $\pi^*E_i=3 \widetilde{E}_i$.
Let us recall the following well-known fact (\emph{see} \cite[Ex.4 p. 103]{Brunella}):

\begin{lem} \label{douze}
The surface $X$ is a basic rational surface that can be obtained by blowing $12$ distinct points in $\mathbb{P}^2$.
\end{lem}

\begin{proof}
Let us consider the four following curves $(C_i)_{1 \leq i \leq 4}$ in ${\mathcal{A}}$
\begin{equation}
C_1=E \times \{ 0 \}, \quad C_2=\{ 0 \} \times E, \quad C_3=\Delta_{\mathcal{A}}, \quad C_4=\Gamma_{-\phi}
\end{equation}
where $\Delta_{\mathcal{A}}$ is the diagonal of $\mathcal{A}$ and $\Gamma_{-\phi}$ is the graph of $-\phi$. Define $8$ other curves $(C_i)_{5 \leq i \leq 12}$ as follows:
\[
\left\{\!\begin{alignedat}{7}
C_5&=C_1+p_2 &\quad C_6&=C_2+p_4 &\quad C_7&=C_3+p_4 &\quad C_8&=C_4+p_4 \\
C_9&=C_1+p_3 &\quad C_{10}&=C_2+p_7 &\quad C_{11}&=C_3+p_7 &\quad C_{12}&=C_4+p_7.\end{alignedat} \right.
\]

\shorthandoff{:!}
\begin{center}
\begin{tikzpicture}
    \begin{scope}[label distance=0.05cm]
        \coordinate[label={below right:\footnotesize$p_1$}] (P1) at (-2,-2);
        \coordinate[label={below right:\footnotesize$p_2$}] (P2) at (-2,0);
        \coordinate[label={below right:\footnotesize$p_3$}] (P3) at (-2,2);
        \coordinate[label={below right:\footnotesize$p_4$}] (P4) at (0,-2);
        \coordinate[label={below right:\footnotesize$p_5$}] (P5) at (0,0);
        \coordinate[label={below right:\footnotesize$p_6$}] (P6) at (0,2);
        \coordinate[label={below right:\footnotesize$p_7$}] (P7) at (2,-2);
        \coordinate[label={below right:\footnotesize$p_8$}] (P8) at (2,0);
        \coordinate[label={below right:\footnotesize$p_9$}] (P9) at (2,2);
    \end{scope}
    \def\sep{0.04}
    \def\adj{{-\sep*sqrt(3)/2},{\sep/2}}
    \draw[marron] ($ (P1)!-0.25!(P6) $) -- ($ (P1)!0.475!(P6) $);
        \draw[marron] ($ (P1)!0.525!(P6) $) -- ($ (P1)!1.25!(P6) $) node[above right] {$C_{4}$};
    \draw[marron] ($ (P8) + 0.75*(1,2)$) -- ($ (P8) - 0.975*(1,2)$);
        \draw[marron] ($ (P8) - 1.025*(1,2)$) -- ($ (P8) - 1.5*(1,2)$);
    \draw[marron] ($ (P4)!-0.25!(P9) + (\adj)$) -- ($ (P4)!0.4625!(P9) + (\adj)$);
        \draw[marron] ($ (P4)!0.525!(P9) + (\adj)$) -- ($ (P4)!1.25!(P9) + (\adj)$) node[above right] {$C_{8}$};
    \draw[marron] ($ (P4)!-0.25!(P9) - (\adj)$) -- ($ (P4)!0.475!(P9) - (\adj)$);
        \draw[marron] ($ (P4)!0.5375!(P9) - (\adj)$) -- ($ (P4)!1.25!(P9) - (\adj)$);
    \draw[marron] ($ (P2) + 1.5*(1,2) + (\adj)$) -- ($ (P2) + 1.050*(1,2) + (\adj)$);
        \draw[marron] ($ (P2) + 0.925*(1,2) + (\adj)$) -- ($ (P2) - 0.975*(1,2) + (\adj)$);
        \draw[marron] ($ (P2) - 1.050*(1,2) + (\adj)$) -- ($ (P2) - 1.5*(1,2) + (\adj)$);
    \draw[marron] ($ (P2) + 1.5*(1,2) - (\adj)$) -- ($ (P2) + 1.075*(1,2) - (\adj)$);
        \draw[marron] ($ (P2) + 0.950*(1,2) - (\adj)$) -- ($ (P2) - 0.950*(1,2) - (\adj)$);
        \draw[marron] ($ (P2) - 1.025*(1,2) - (\adj)$) -- ($ (P2) - 1.5*(1,2) - (\adj)$);
    \draw[marron] ($ (P7) - 0.5*(1,2) + 2*(\adj)$) -- ($ (P7) + 0.75*(1,2) + 2*(\adj)$);
    \draw[marron] ($ (P7) - 0.5*(1,2)$) -- ($ (P7) + 0.75*(1,2)$) node[right] {$C_{12}$};
    \draw[marron] ($ (P7) - 0.5*(1,2) - 2*(\adj)$) -- ($ (P7) + 0.75*(1,2) - 2*(\adj)$);
    \draw[marron] ($ (P3) - 1.5*(1,2) + 2*(\adj)$) -- ($ (P3) - 1.0500*(1,2) + 2*(\adj)$);
        \draw[marron] ($ (P3) - 0.9750*(1,2) + 2*(\adj)$) -- ($ (P3) + 0.5*(1,2) + 2*(\adj)$);
    \draw[marron] ($ (P3) - 1.5*(1,2)$) -- ($ (P3) - 1.0375*(1,2)$);
        \draw[marron] ($ (P3) - 0.9625*(1,2)$) -- ($ (P3) + 0.5*(1,2)$);
    \draw[marron] ($ (P3) - 1.5*(1,2) - 2*(\adj)$) -- ($ (P3) - 1.025*(1,2) - 2*(\adj)$);
        \draw[marron] ($ (P3) - 0.95*(1,2) - 2*(\adj)$) -- ($ (P3) + 0.5*(1,2) - 2*(\adj)$);
    \draw[marron] ($ (P5) - 1.5*(1,2) + 2*(\adj)$) -- ($ (P5) - 1.075*(1,2) + 2*(\adj)$);
        \draw[marron] ($ (P5) - 0.975*(1,2) + 2*(\adj)$) -- ($ (P5) + 0.9*(1,2) + 2*(\adj)$);
        \draw[marron] ($ (P5) + 1.050*(1,2) + 2*(\adj)$) -- ($ (P5) + 1.5*(1,2) + 2*(\adj)$);
    \draw[marron] ($ (P5) - 1.5*(1,2)$) -- ($ (P5) - 1.05*(1,2)$);
        \draw[marron] ($ (P5) - 0.95*(1,2)$) -- ($ (P5) + 0.925*(1,2)$);
        \draw[marron] ($ (P5) + 1.075*(1,2)$) -- ($ (P5) + 1.5*(1,2)$);
    \draw[marron] ($ (P5) - 1.5*(1,2) - 2*(\adj)$) -- ($ (P5) - 1.025*(1,2) - 2*(\adj)$);
        \draw[marron] ($ (P5) - 0.925*(1,2) - 2*(\adj)$) -- ($ (P5) + 0.950*(1,2) - 2*(\adj)$);
        \draw[marron] ($ (P5) + 1.1*(1,2) - 2*(\adj)$) -- ($ (P5) + 1.5*(1,2) - 2*(\adj)$);
    \def\adj{\sep,0}
    \draw[rouge] ($ (P1)!-0.25!(P3) $) -- ($ (P1)!1.25!(P3) $) node[above] {$C_{2}$};
    \draw[rouge] ($ (P4)!-0.25!(P6) + (\adj)$) -- ($ (P4)!1.25!(P6) + (\adj)$);
    \draw[rouge] ($ (P4)!-0.25!(P6) - (\adj)$) -- ($ (P4)!1.25!(P6) - (\adj)$) node[above] {$C_{6}$};
    \draw[rouge] ($ (P7)!-0.25!(P9) + 2*(\adj)$) -- ($ (P7)!1.25!(P9) + 2*(\adj)$);
    \draw[rouge] ($ (P7)!-0.25!(P9) $) -- ($ (P7)!1.25!(P9) $) node[above] {$C_{10}$};
    \draw[rouge] ($ (P7)!-0.25!(P9) - 2*(\adj)$) -- ($ (P7)!1.25!(P9) - 2*(\adj)$);
    \def\adj{0,\sep}
    \draw[vert] ($ (P1)!1.25!(P7) $) -- ($ (P1)!-0.5!(P7) $) node[left] {$C_{1}$};
    \draw[vert] ($ (P2)!1.25!(P8) + (\adj)$) -- ($ (P2)!-0.5!(P8) + (\adj)$) node[left] {$C_{5}$};
    \draw[vert] ($ (P2)!1.25!(P8) - (\adj)$) -- ($ (P2)!-0.5!(P8) - (\adj)$);
    \draw[vert] ($ (P3)!1.25!(P9) + 2*(\adj)$) -- ($ (P3)!-0.5!(P9) + 2*(\adj)$);
    \draw[vert] ($ (P3)!1.25!(P9) $) -- ($ (P3)!-0.5!(P9) $) node[left] {$C_{9}$};
    \draw[vert] ($ (P3)!1.25!(P9) - 2*(\adj)$) -- ($ (P3)!-0.5!(P9) - 2*(\adj)$);
    \def\adj{{-\sep/sqrt(2)},{\sep/sqrt(2)}}
    \draw[bleu] ($ (P1)!-0.125!(P9) $) -- ($ (P1)!1.125!(P9) $) node[right] {$C_{3}$};
    \draw[bleu] ($ (P3) + 0.75*(1,1) + (\adj)$) -- ($ (P3) - 0.75*(1,1) + (\adj)$);
    \draw[bleu] ($ (P3) + 0.75*(1,1) - (\adj)$) -- ($ (P3) - 0.75*(1,1) - (\adj)$);
    \draw[bleu] ($ (P4)!-0.25!(P8) + (\adj)$) -- ($ (P4)!1.25!(P8) + (\adj)$) node[right] {$C_{7}$};
    \draw[bleu] ($ (P4)!-0.25!(P8) - (\adj)$) -- ($ (P4)!1.25!(P8) - (\adj)$);
    \draw[bleu] ($ (P7) + 0.75*(1,1) + 2*(\adj)$) -- ($ (P7) - 0.75*(1,1) + 2*(\adj)$);
    \draw[bleu] ($ (P7) - 0.75*(1,1)$) -- ($ (P7) + 0.75*(1,1) $) node[right] {$C_{11}$};
    \draw[bleu] ($ (P7) - 0.75*(1,1) - 2*(\adj)$) -- ($ (P7) + 0.75*(1,1) - 2*(\adj)$);
    \draw[bleu] ($ (P2)!-0.25!(P6) + 2*(\adj)$) -- ($ (P2)!1.25!(P6) + 2*(\adj)$);
    \draw[bleu] ($ (P2)!-0.25!(P6) $) -- ($ (P2)!1.25!(P6) $);
    \draw[bleu] ($ (P2)!-0.25!(P6) - 2*(\adj)$) -- ($ (P2)!1.25!(P6) - 2*(\adj)$);
\end{tikzpicture}
\end{center}
\shorthandon{:!}

The strict transforms of the $12$ curves $C_i\, (1 \leq i \leq 12)$ in $\widetilde{\mathcal{A}}$
are $\phi$-invariant elliptic curves of self-intersection $-3$, since each of them pass through exactly three points of $S$ with multiplicity one). Their images by $\pi$ give $12$ smooth rational curves $(\mathscr{E}_i)_{1 \leq i \leq 12}$ of self-intersection $-1$.  Blowing down these $12$ curves, we get a smooth surface $Y$. Since $\mathrm{H}^1(\mathcal{A}, \mathbb{Z})^G \simeq (\Lambda^* \times  \Lambda^*)^G=\{0\}\footnote{For any $G$-module $M$, we put $M^G=\{m \in M \, \, \textrm{s. t.}\, \ \forall g \in G, g.m=m\}.$}$ where $\Lambda^*$ denotes the dual lattice of $\Lambda$,  $\mathrm{b}_1(Y)$ vanishes. Now we compute the Euler characteristic of $Y$:
\[
\chi(Y)=\chi(X)-12=\chi\big(X\smallsetminus\{E_i\}_i\big)+6=\frac{\chi\big(\widetilde{A}\smallsetminus\{\widetilde{E}_i\}_i\big)}{3}+6=\frac{\chi(\widetilde{A})}{3}=\chi(A)+3=3=\chi(\mathbb{P}^2)
\]
so that $b_{2}(Y)=1$. To conclude that the surface $Y$ is isomorphic to $\mathbb{P}^2$ and not to a fake projective plane, it suffices to prove that $Y$ is not a surface of general type (\emph{see} \cite[p. 487]{GH}). Denoting by $\kappa$ the Kodaira dimension, we have $\kappa(Y)=\kappa(X) \leq \kappa(\widetilde{\mathcal{A}})=\kappa(\mathcal{A})=0$. This finishes the proof.
\end{proof}

\begin{rem} \label{degre}
After blowing down the $12$ exceptional curves, the exceptional divisors $E_i$ are lines in $\mathbb{P}^2$, since they are of self-intersection one. Each point belongs to three lines and each line passes through four points. 
\end{rem}

The Picard group of $X$ can be described explicitly in the following way: since $X$ is rational, $\mathrm{Pic}(X)$ is isomorphic to $\mathrm{H}^2(X, \mathbb{Z}_X)$. First we compute: 
\[
\mathrm{H}^2\big(\widetilde{\mathcal{A}}, \mathbb{Z}_{\widetilde{\mathcal{A}}}\big) \simeq  \mathrm{H}^2({\mathcal{A}}, \mathbb{Z}_{{\mathcal{A}}}) \oplus  \left(\bigoplus_{i=1}^9 \mathbb{Z} [\widetilde{E}_i] \right) \simeq \wedge^2_{\mathbb{Z}} \big(\Lambda^* \times \Lambda^*\big) \oplus  \left(\bigoplus_{i=1}^9 \mathbb{Z} [\widetilde{E}_i] \right).
\]
Hence we get an isomorphism of $\mathrm{GL}(2; \Lambda)$-modules:
\begin{equation} \label{Picard}
\mathrm{Pic}(X) \simeq \mathrm{H}^2\big(\widetilde{\mathcal{A}}, \mathbb{Z}_{\widetilde{\mathcal{A}}}\big)^G \simeq \left(  \wedge^2_{\mathbb{Z}} (\Lambda^* \times \Lambda^*)\right)^G \, \oplus \,\bigoplus_{i=1}^9 \,\mathbb{Z} [\widetilde{E}_i].
\end{equation}
It is easy to see that $\left( \wedge^2_{\mathbb{Z}}(\Lambda^* \times \Lambda^*)\right)^G$ is a free $\mathbb{Z}$-module of rank $4$, a basis (over $\mathbb{Q}$) of this module being given by the curves $(C_i)_{1 \leq i \leq 4}$. Thus $\mathrm{Pic}(X)$ is a free $\mathbb{Z}$-module of rank $13$ (which is $1$ plus the the number of points blown up, as expected).
We end this section by a description of the canonical class of $X$.

\begin{lem} \label{canonique}
We have $\vert -\mathrm{K}_X\vert=\vert -2\mathrm{K}_X\vert=\varnothing$ and $\vert -3\mathrm{K}_X\vert=\sum_{i=1}^9 E_i$.
\end{lem}

\begin{proof}
Since $\pi$ is a cyclic covering, the generalized Riemann-Hurwitz formula for branched cyclic covers gives $\pi^* \mathrm{K}_X=\mathrm{K}_{\widetilde{\mathcal{A}}}(-2 \sum_{i=1}^9 \widetilde{E}_i)$. On the other hand, $\mathrm{K}_{\widetilde{\mathcal{A}}}=\delta^*\mathrm{K}_{\mathcal{A}}+\sum_{i=1}^9 \widetilde{E}_i$ so that $\pi^* \mathrm{K}_X=-\sum_{i=1}^9 \widetilde{E}_i$. It follows that $-3\mathrm{K}_X \sim \sum_{i=1}^9 E_i$. Now if $D$ belongs to $\vert -3\mathrm{K}_X\vert$, we have $D.E_i=-3$, so that $D$ must contain $E_i$ as a component. Thus $D=\sum_{i=1}^9 E_i + D'$ where $D'$ is effective. But $D' \sim 0$ so that $D'=0$ and $\vert -3\mathrm{K}_X\vert=\sum_{i=1}^9 E_i$.
\end{proof}

\subsubsection{Linear automorphisms of the Kummer surface}

The group $\mathrm{GL}(2;\Lambda)$ acts linearly on~$\mathbb{C}^2$ and preserves the lattice $\Lambda\times\Lambda.$ Therefore any element $M$ of $\mathrm{GL}(2;\Lambda)$ induces an automorphism $f_M$ on $\mathcal{A}$ that commutes with the automorphism $\phi$ of $A$, hence leaves the set $S$ globally invariant. Each $f_M$ lifts to an automorphism $\tilde{f}_M$ of the blown up abelian surface $\widetilde{\mathcal{A}}$ that still commutes to the action of the group $G$ generated by $\phi$. Thus $\tilde{f}_M$ descends to an automorphism $\varphi_M$ of $X$. The map $M \rightarrow \varphi_M$ embeds $\mathrm{GL}(2; \Lambda)/G$ as a subgroup of $\mathrm{Aut}\,(X)$.
\par \medskip
 Let $H$ be the group of matrices in $\mathrm{SL}(2; \Lambda)$ that are congruent to the identity matrix modulo the ideal $(1-\mathbf{j}) \mathbb{Z}[\mathbf{j}]$.

\begin{lem}
The natural morphism $H \rightarrow \mathrm{GL}(2; \Lambda)/G$ is injective and its image is exactly the stabilizer of $S$.
\end{lem}

\begin{proof}
The injectivity is obvious since $G \cap \mathrm{SL}(2; \Lambda)=\mathrm{id}$. For the surjectivity, note that the automorphism $f_M$ fixes the nine points $p_i$ if an only if $3$ divides $(2+\mathbf{j})\,(a-1)$, $(2+\mathbf{j})\,b$, $(2+\mathbf{j})\,c$ and $(2+\mathbf{j})\,(d-1)$ where $M=\Big(\begin{array}{cc} a & b\\ c& d\end{array}\Big)$. Since $3=(2+\mathbf{j})\,(1-\mathbf{j})$, $M$ belongs to the stabilizer of $S$ if and only if its reduction modulo $(1-\mathbf{j}) \mathbb{Z}[\mathbf{j}]$ is the identity matrix. Among the units of $\mathbb{Z}[\mathbf{j}]$, only $1$, $\mathbf{j}$ and $\mathbf{j}^2$ are congruent to $1$ modulo $1-\mathbf{j}$. Thus $\det M \in \{ 1, \mathbf{j}, \mathbf{j}^2 \}$ and we are done.
\end{proof}

Let $M$ be an element of $\mathrm{GL}(2; \Lambda)$ with eigenvalues $\alpha$ and $\beta$, and $V=(\Lambda^* \times \Lambda^*) \otimes_{\mathbb{Z}} \mathbb{C}$. Then the complex eigenvalues of the endomorphism $\wedge^2\,M$ of the real vector space $\wedge^2_{\mathbb{R}} V$ are $\vert\alpha\vert^2$, $\vert\beta\vert^2$, $\alpha \overline{\beta}$, $\overline{\alpha} \beta$, $\alpha \beta$ and $\overline{\alpha} \overline{\beta}$.
Note that for dimension reasons we have a $G$-equivariant exact sequence
\[
0 \longrightarrow \big(\wedge^2_{\mathbb{R}} V\big)^G \longrightarrow \wedge^2_{\mathbb{R}} V \longrightarrow \wedge^2_{\mathbb{C}} V \longrightarrow 0
\]
since $\phi$ acts by multiplication by $\mathbf{j}^2$ on $\wedge^2_{\mathbb{C}} V$. Besides $f_M$ acts by $\det\,(M)=\alpha\beta$ on $\wedge^2_{\mathbb{C}} V$, so the eigenvalues of the corresponding $\mathbb{R}$-linear endormorphism are $\alpha\beta$ and $\overline{\alpha} \overline{\beta}$. Thus the eigenvalues of $\wedge^2\,M$ on the subspace $\big(\wedge^2_{\mathbb{R}} V\big)^G$ are exactly $\vert\alpha\vert^2$, $\vert\beta\vert^2$, $\alpha \overline{\beta}$ and $\overline{\alpha} \beta$. We conclude that for any element $M$ in $\mathrm{GL}(2; \Lambda)$ with spectral radius $r_M$,  the spectral radius of $M$ acting on~$\mathrm{NS}_{\mathbb{Q}}(X)$ is $r_M^{\,2}$. This means that the first dynamical degree of $\tilde{f}_M$ is given by the formula $\lambda_1(\tilde{f}_M)=r_M^{\, 2}$. More precisely, the characteristic polynomial of $\tilde{f}_M^{\,*}$ acting on $\mathrm{NS}_{\mathbb{Q}}(X)$ is 
\[
\big(x-1\big)^9 \big(x-\vert\alpha\vert^2\big)\big(x-\vert\beta\vert^2\big)\big(x-\overline{\alpha}\beta\big)\big(x-\alpha \overline{\beta}\big).
\]

\subsubsection{Explicit realisation in the Cremona group}
According to Lemma \ref{douze}, the Kummer surface $X$ is obtained by blowing up $\mathbb{P}^2$ along $12$ points, and an explicit morphism $p \colon X \rightarrow \mathbb{P}^2$ is obtained by blowing down the $-1$ curves $(\mathscr{E}_i)_{1 \leq i \leq 12}$. We can therefore define for any matrix $M$ in $\mathrm{GL}(2; \mathbb{Z}[\mathbf{j}])$ a Cremona transformation $\psi_M \colon \mathbb{P}^2 \dasharrow \mathbb{P}^2$ by the formula
\[
\psi_M=p^{-1} \circ \varphi_M \circ p.
\]
It is an interesting question to find explicit formulas for the map $\psi_M$. The first step to understand the maps $\psi_M$ is to describe explicitly the configuration of lines $(p(E_i))_{1 \leq i \leq 9}$ in $\mathbb{P}^2$ since the intersection points of these lines give the indeterminacy locus of $p^{-1}$. Let us put \[
\left\{ \begin{alignedat}{3}
q_i&=p(\mathscr{E}_i) \quad &\textrm{for} \quad &1 \leq i \leq 12 \\
\Delta_j&=p(E_j) \quad &\textrm{for} \quad &1 \leq j \leq 9.
\end{alignedat} \right.
\]
The two configurations $\{C_i, p_j\}$ and $\{\Delta_j, q_i \}$ are projectively dual. 
\shorthandoff{:!}
\begin{center}
\begin{tikzpicture}
    \begin{scope}[label distance=0.25cm]
        \coordinate[label={right:\footnotesize$q_6$}] (q5)  at (0,3);
        \coordinate[label={below:\footnotesize$q_{2}$}] (q1)  at (3,0);
    \end{scope}
    \begin{scope}[label distance=0.05cm]
        \coordinate[label={above left:\footnotesize$q_{11}$}] (q11)  at (1,1);
        \coordinate[label={above left:\footnotesize$q_8$}] (q8)  at (-1,1);
        \coordinate[label={above left:\footnotesize$q_5$}] (q6)  at (0,1);
        \coordinate[label={above left:\footnotesize$q_7$}] (q7)  at (-1,0);
        \coordinate[label={above left:\footnotesize$q_{12}$}] (q12) at (0,0);
        \coordinate[label={above left:\footnotesize$q_{9}$}] (q10) at (1,0);
        \coordinate[label={above left:\footnotesize$q_{3}$}] (q3)  at (0,-1);
        \coordinate[label={above left:\footnotesize$q_{4}$}] (q4)  at (1,-1);
        \coordinate[label={above left:\footnotesize$q_{1}$}] (q2)  at (-1,-1);
        \coordinate[label={above left:\footnotesize$q_{10}$}] (q9)  at ({-2*sqrt(2)},{-2*sqrt(2)});
    \end{scope}
    \foreach \lbl in {q8,q12,q4} {
       \draw[marron] (\lbl) circle (5pt);
    }
    \foreach \lbl in {q3,q7,q11} {
       \draw[bleu] (\lbl) circle (5pt);
    }
    \foreach \lbl in {q2,q6,q10} {
       \draw[vert] (\lbl) circle (5pt);
    }
    \foreach \lbl in {q1,q5,q9} {
       \draw[rouge] (\lbl) circle (5pt);
    }
    \draw ($ (q3)!5!(q12) $) -- ($ (q3)!-2.5!(q12) $) node[below] {$\Delta_5$};
    \draw[rounded corners=0.25cm] ($ (q2)!3.5!(q7)!2!(q5) $) -- (q5) -- ($ (q2)!3.5!(q7) $) -- ($ (q2)!-2.5!(q7) $) node[below] {$\Delta_4$};
    \draw[rounded corners=0.25cm] ($ (q4)!3.5!(q10)!2!(q5) $) -- (q5) -- ($ (q4)!3.5!(q10) $) -- ($ (q4)!-2.5!(q10) $) node[below] {$\Delta_6$};
    \draw[rounded corners=0.25cm] ($ (q2)!3.5!(q3)!2!(q1)$) -- (q1) -- ($ (q2)!3.5!(q3) $) -- ($ (q2)!-2.5!(q3) $) node[left] {$\Delta_1$};
    \draw ($ (q7)!5!(q12) $) -- ($ (q7)!-2.5!(q12) $) node[left] {$\Delta_3$};
    \draw[rounded corners=0.25cm] ($ (q8)!3.5!(q6)!2!(q1)$) -- (q1) -- ($ (q8)!3.5!(q6) $) -- ($ (q8)!-2.5!(q6) $) node[left] {$\Delta_2$};
    \draw[>>-] ($ (q8) + 0.67*(-1,-1)$) -- ($ (q8) + 0.67*(1,1) $);
    \draw[>-] ($ (q4) + 0.67*(-1,-1)$) -- ($ (q4) + 0.67*(1,1) $);
    \draw[<-] ($ (q6)!-0.75!(q7) $) -- ($ (q6)!1.95!(q7) $);
        \draw[rounded corners=0.25cm] ($ (q6)!2.05!(q7) $) -- ($ (q6)!3!(q7) $) -- (q9) -- ($ (q6)!3!(q7)!2!(q9) $) node[below] {$\Delta_8$};
    \draw ($ (q2)!4!(q12) $) -- ($ (q2)!-2.5!(q12) $) node[below left] {$\Delta_7$};
    \draw[<<-] ($ (q10)!-0.75!(q3) $) -- ($ (q10)!1.95!(q3) $);
        \draw[rounded corners=0.25cm] ($ (q10)!2.05!(q3) $) -- ($ (q10)!3!(q3) $) -- (q9) -- ($ (q10)!3!(q3)!2!(q9) $) node[left] {$\Delta_9$};
\end{tikzpicture}
\end{center}
\shorthandon{:!}
We have the following result (which is almost \cite[Proposition 1]{LinsNeto} with a small additional ingredient).

\begin{pro} 
There exists linear coordinates $x$, $y$, $z$ on $\mathbb{P}^2$ such that 
\[
\left\{\!\begin{alignedat}{8}
\Delta_1&=\{y=z \} & \quad\Delta_2&=\{y=\mathbf{j} z \} & \quad\Delta_3&=\{y=\mathbf{j}^2\,z \} & \quad\Delta_4&=\{x=z\} &\quad\Delta_5&=\{x=\mathbf{j}^2 z\}  \\
\Delta_6&=\{x=\mathbf{j} z \} &\quad\Delta_7&=\{y=x\} & \quad\Delta_8&=\{y=\mathbf{j}^2 x\} & \quad\Delta_9&=\{y=\mathbf{j} x \}
\end{alignedat}\right.
\]
and
\[
\left\{\!\begin{alignedat}{7}
q_1&=(1:1:1) & \quad q_2&=(1: 0 : 0) & \quad q_3&=(\mathbf{j}^2: 1: 1) & \quad q_4&=(\mathbf{j}: 1: 1)\\
q_5&=(\mathbf{j}^2: \mathbf{j}: 1) & \quad q_6&=(0:1:0) & \quad q_7&=(1:\mathbf{j}^2:1) & \quad q_8&=(1:\mathbf{j}:1)\\
q_9&=(\mathbf{j}: \mathbf{j}^2: 1) & \quad q_{10}&=(0:0:1) & \quad q_{11}&=(\mathbf{j}:\mathbf{j}:1) & \quad q_{12}&=(\mathbf{j}^2:\mathbf{j}^2:1)
\end{alignedat} \right.
\]
\end{pro}

\begin{proof}
Since no line passes through three points in the set $\{q_1, q_2, q_6, q_{10}\}$, there exist unique linear coordinates on $\mathbb{P}^2$ such that $q_1=(1:1:1)$, $q_2=(1: 0 : 0)$, $q_6=(0:1:0)$ and $q_{10}=(0:0:1)$. Then $\Delta_1=\{y=z\}$, $\Delta_4=\{x=z\}$ and $\Delta_7=\{x=y\}$. Now the line $\{z=0\}$ (resp. $\{x=0\}$) cannot be equal to $\Delta_5$ or $\Delta_6$ since it passes through $q_2$ (resp. $q_{10}$). Hence there exist nonzero complex numbers $\alpha$ and $\beta$ such that $\Delta_5=\{x=\alpha z\}$ and $\Delta_6=\{x=\beta z\}$. Then $\Delta_3=\{y=\alpha z \}$ and $\Delta_2=\{y=\beta z\}$. Hence in the affine chart $z=1$, $q_3=(\alpha, 1)$, $q_4=(\beta, 1)$, $q_7=(1, \alpha)$, $q_{12}=(\alpha, \alpha)$, $q_9=(\beta, \alpha)$, $q_8=(1, \beta)$, $q_5=(\alpha, \beta)$ and $q_{11}=(\beta, \beta)$. Using that $\{q_{10}, q_3, q_9, q_8 \}$ are aligned, we get $\alpha^2=\beta$ and $\alpha=\beta^2$ so that $\alpha \in \{\mathbf{j}, \mathbf{j}^2 \}$. This gives two distinct configurations, which are not projectively isomorphic, although they are complex conjugate\footnote{This point is not clearly explained in \cite[Proposition 1]{LinsNeto} where the terminology ``essentially unique'' can be slightly misleading.}.
Using the affine coordinate $x$ as a coordinate on $\Delta_1$, we see that the cross-ratio $[q_1, q_2, q_3, q_4]$ is 
\[
[q_1, q_2, q_3, q_4]=[1, \infty, \alpha, \alpha^2]=\frac{1-\alpha}{1-\alpha^2}=\frac{1}{1+\alpha}=-\alpha.
\]
But this cross ratio is easy to compute, since the points $q_i, 1 \leq i \leq 4$ can be identified with the intersections of the strict transforms of the curves $C_i, 1 \leq i \leq 4$ with the exceptional divisor $\widetilde{E}_1$ via the isomorphisms $\widetilde{E}_1 \simeq E_1 \simeq \Delta_1$. Hence we get
\[
[q_1, q_2, q_3, q_4]=[0, \infty, 1, -\mathbf{j}]=-\mathbf{j}^2
\]
so that $\alpha=\mathbf{j}^2$ and we are done.
\end{proof}

The next step to understand the  birational maps $\psi_M$ is to compute their degrees. Although this result is not strictly necessary for us, we include it because it can be useful for explicit computations.
\par\medskip

\begin{pro} \label{dur} 
For any matrix $M=\left(\begin{array}{cc} a & b \\ c & d \end{array}\right)$ in $\mathrm{GL}(2; \mathbb{Z}[\mathbf{j}])$,
\[
\deg \psi_M=|a+d | ^2+| c-\mathbf{j}b |^2+\left(1+\frac{\sqrt{3}}{2} \right)\,|  \mathbf{i}a-\mathbf{j}^2 b-\mathbf{j}c-\mathbf{i}d|^2+\left(1-\frac{\sqrt{3}}{2} \right)\,|  \mathbf{i}a+\mathbf{j}^2 b+\mathbf{j}c-\mathbf{i}d|^2-3.
\]
\end{pro}
\begin{proof}
Let us recall the three following commutative diagrams:
\[
\xymatrix{\widetilde{\mathcal{A}} \ar[r]^-{\tilde{f}_M} \ar[d]_-{\delta}&\widetilde{\mathcal{A}} \ar[d]^-{\delta}\\
\mathcal{A} \ar[r]^-{{f}_M}& \mathcal{A} \\
}
\qquad \qquad
\xymatrix{\widetilde{\mathcal{A}} \ar[r]^-{\tilde{f}_M} \ar[d]_-{\pi}&\widetilde{\mathcal{A}} \ar[d]^-{\pi}\\
X \ar[r]^-{\varphi_M}& X \\
}
\qquad \qquad
\xymatrix{X \ar[r]^-{\varphi_M} \ar[d]_-{p}&X \ar[d]^-{p}\\
\mathbb{P}^2 \ar@{-->}[r]^-{\psi_M}& \mathbb{P}^2 \\
}
\]
\par \medskip
We have $
p^{*} \Delta_1=E_1+ \sum_{i=1}^4 \mathscr{E}_i.
$
Hence, if $\zeta$ and $\epsilon$ denote the divisors $\sum_{i=1}^4 C_i$ and $\sum_{i=1}^9 \widetilde{E}_i$ on $\mathcal{A}$ and $\widetilde{\mathcal{A}}$ respectively, 

\begin{alignat*}{3}
(p \circ \pi)^*\Delta_1&=3 \widetilde{E}_1&&+(\delta^* C_1-\widetilde{E}_1-\widetilde{E}_4-\widetilde{E}_7)+(\delta^* C_2-\widetilde{E}_1-\widetilde{E}_2-\widetilde{E}_3)\\
&&&+(\delta^* C_3-\widetilde{E}_1-\widetilde{E}_5-\widetilde{E}_9)+(\delta^* C_4-\widetilde{E}_1-\widetilde{E}_4-\widetilde{E}_8)\\
&=\delta^*\zeta &&-\epsilon
\end{alignat*}

so that
\[
\deg\, \psi_M= p^* \Delta_1\,. \, (p \circ \varphi_M)^* \Delta_1=\frac{(p \circ \varphi_M \circ \pi)^* \Delta_1 \, . \, (p \circ \pi)^* \Delta_1}{3}=\frac{f_M^* \zeta \, . \,\zeta+ \tilde{f}_M^* \epsilon\, . \,\epsilon}{3}=\frac{f_M^* \zeta \, . \,\zeta}{3}-3.
\]
A routine computation yields the value of the cohomology classes of the curves $C_i$ in $\mathrm{H}^2(\mathcal{A}, \mathbb{C})$:
\begin{align*}
[C_1]&=\frac{\mathbf{i}\sqrt{3}}{3}dz \wedge d \bar{z} \\
[C_2]&=\frac{\mathbf{i}\sqrt{3}}{3}dw \wedge d \bar{w} \\
[C_3]&=\frac{\mathbf{i}\sqrt{3}}{3}(d\bar{z} \wedge dw - dz\wedge d \bar{w}+dz \wedge d \bar{z}+ dw \wedge d\bar{w})\\
[C_4]&=\frac{\mathbf{i}\sqrt{3}}{3}\left(\mathbf{j}dz \wedge d \bar{w}-\mathbf{j}^2 d\bar{z} \wedge dw+dz \wedge d \bar{z}+dw \wedge d \bar{w} \right).
\end{align*}
Thus, if $\mu=(\mathbf{j}-1)$,
\[
[\zeta]=\frac{\mathbf{i} \sqrt{3}}{3} \left(3dz \wedge d \bar{z}+3dw \wedge d \bar{w}+\mu\, dz\wedge d \bar{w} -\bar{\mu}\, d\bar{z}\wedge dw \right)
\]
so that 
\[
\mathrm{deg}\, \psi_M=3(\vert a\vert^2+\vert b\vert^2+\vert c\vert^2+\vert d\vert^2)+2 \,\Re\left\{(\mathbf{j}-1) (a \bar{c}+b\bar{d}-\bar{a}b-\bar{c}d)+\mathbf{j}b\bar{c}-a\bar{d}\right\}-3.
\]
The result follows by a direct calculation.
\end{proof}
\begin{rem} $ $
\begin{enumerate}
\item[(i)] The degree of $\psi_M$ depends only on the class of $M$ in $\mathrm{GL}(2; \Lambda)/G$, and this can be verified directly on the explicit expression of $\deg\psi_M$ given by Proposition \ref{dur}.
\item[(ii)] We have $\mathrm{deg}\, \psi_M=Q(a,b,c,d)-3$ where $Q$ is a positive-definite hermitian form. Hence for any positive integer $N$, the set
\[
\left\{M \in \mathrm{GL}(2; \mathbb{Z}[\mathbf{j}]) \,\, \textrm{s.t.} \,\, \textrm{deg}\, \Psi_M \leq N \right\}
\]
is finite.
\end{enumerate}
\end{rem}

\begin{thm} \label{quenelle} 
The matrices $M_1=\begin{pmatrix} 0 & 1 \\ -1 & 1 \end{pmatrix}$, $M_2=\begin{pmatrix} 0 & 1 \\ -\mathbf{j} & 0 \end{pmatrix}$  and $M_3=\begin{pmatrix} 0 & 1 \\ 1 & 0 \end{pmatrix}$  generate $\mathrm{GL}(2; \mathbb{Z}[\mathbf{j}])$ as a semigroup. Besides, the corresponding Cremona transformations have the $\vphantom{A^A}$following explicit description\emph{:}
\[
\begin{cases}
\psi_{M_1} \colon (x:y:z) \rightarrow (x+y+z: \mathbf{j}x + y + \mathbf{j}^2 z:\mathbf{j}x+\mathbf{j}^2y+z)\\
\psi_{M_2} \colon (x:y:z) \rightarrow (x+y+z: x + \mathbf{j}^2 y + \mathbf{j} z:x+\mathbf{j}y+\mathbf{j}^2 z)\\
\psi_{M_3} \colon (x:y:z) \rightarrow  (a(x:y:z), b(x:y:z), c(x:y:z))\\
\end{cases}
\]
where 
\[
\begin{cases}
a(x:y:z)=x^2+y^2+z^2-\mathbf{j}^2(xy+xz+yz) \\
b(x:y:z)=x^2+\mathbf{j}y^2+\mathbf{j}^2z^2-\mathbf{j}xy-xz-\mathbf{j}^2yz \\
c(x:y:z)=x^2+\mathbf{j}^2y^2+\mathbf{j}z^2-xy-\mathbf{j}xz-\mathbf{j}^2yz \\
\end{cases}
\]
\end{thm}

\begin{proof}
The ring $\mathbb{Z}[\mathbf{j}]$ being euclidean, any matrix in $\mathrm{GL}(2;\mathbb{Z}[\mathbf{j}])$ can be put in diagonal form after performing elementary row and column operations, corresponding to multiplications on the right and on the left by matrices of the type $\begin{pmatrix} 1 & \alpha \\ 0 & 1 \end{pmatrix}$ and $\begin{pmatrix} 1 & 0 \\ \beta & 1 \end{pmatrix}$. First note that the diagonal matrices can be expressed using $M_1$, $M_2$ and $M_3$: we have
\[
\begin{pmatrix}1 & 0 \\ 0 & -\mathbf{j} \end{pmatrix}=M_2M_3 \qquad \textrm{and} \qquad M_3M_2M_3^2=\begin{pmatrix} -\mathbf{j} & 0 \\ 0 & 1 \end{pmatrix}
\]
and these two matrices generate all invertible diagonal matrices, since $\mathbb{Z}[\mathbf{j}]^{\times}$ is the cyclic subgroup of $\mathbb{C}^{\times}$ generated by $-\mathbf{j}$. We now deal with the transvection matrices:

\[
\begin{alignedat}{1}
\begin{pmatrix} 1 & 1 \\ 0 & 1 \end{pmatrix}&=\begin{pmatrix} -1 & 0 \\ 0 & -1 \end{pmatrix} \times M_3 M_1(M_2M_3)^3\\
\begin{pmatrix} 1 & 0 \\ 1 & 1 \end{pmatrix}&=\begin{pmatrix} -1 & 0 \\ 0 & -1 \end{pmatrix} \times M_1 (M_2M_3)^3 M_3\\
\begin{pmatrix} 1 & \mathbf{j} \\ 0 & 1 \end{pmatrix}&=\begin{pmatrix} \mathbf{j}^2 & 0 \\ 0 & \mathbf{j} \end{pmatrix} \times \begin{pmatrix} 1 & 1 \\ 0 & 1 \end{pmatrix} \times \begin{pmatrix} \mathbf{j} & 0 \\ 0 & \mathbf{j}^2 \end{pmatrix}\\
\begin{pmatrix} 1 &0 \\  \mathbf{j} & 1 \end{pmatrix}&=\begin{pmatrix} \mathbf{j} & 0 \\ 0 & \mathbf{j}^2 \end{pmatrix} \times \begin{pmatrix} 1 & 0 \\ 1 & 1 \end{pmatrix} \times \begin{pmatrix} \mathbf{j}^2 & 0 \\ 0 & \mathbf{j} \end{pmatrix}\\
\end{alignedat}
\]
\par \medskip
Hence $M_1$, $M_2$ and $M_3$ generate $\mathrm{GL}(2; \mathbb{Z}[\mathbf{j}])$ as a group. Since these matrices are of finite order, they also generate $\mathrm{GL}(2; \mathbb{Z}[\mathbf{j}])$ as a semigroup.
\par\medskip
It remains to compute the explicit expressions of the $\psi_{M_i}, 1\leq i \leq 3$. First it is easy to see that $\psi_{M_1}$ and $\psi_{M_2}$ have no indeterminacy point, because the maps $\varphi_{M_1}$ and $\varphi_{M_2}$ preserve globally the twelve exceptional curves $\mathscr{E}_j, 1 \leq j \leq 12$.  Hence $\psi_{M_1}$ and  $\psi_{M_2}$ are linear (this could also be checked using Proposition \ref{dur}), and therefore entirely determined by its action on the points $q_k$. In order to compute $\psi_{M_3}$, we first note that it is a quadratic involution. Indeed, $\varphi_{M_3}$ leaves globally invariant the set $\{\mathscr{E}_{1}, \mathscr{E}_{5}, \mathscr{E}_{9}, \mathscr{E}_{2}, \mathscr{E}_{6}, \mathscr{E}_{10}, \mathscr{E}_{3}, \mathscr{E}_{7}, \mathscr{E}_{11}\}$, so that the indeterminacy locus of $\psi_{M_3}$ consists of simple points in the list $\{q_4, q_8, q_{12}\}$. Since $\psi_{M_3}$ is not linear (otherwise all $\psi_M$ would also be linear), it must be a quadratic involution whose indeterminacy locus is the set $\{q_4, q_8, q_{12}\}$. It is completely determined by its action on the remaining points points $q_k$ for $k \notin \{4,8,12\}$.
\end{proof}
\subsection{Action of the automorphism group on infinitesimal deformations}\label{Subsec:action}
In this section, we will present two different approaches to prove the following theorem:

\begin{thm} \label{yahou} 
Let $X$ be the rational Kummer surface associated with the lattice $\mathbb{Z}[\mathbf{j}]$ of Eisenstein integers. For any element $M$ in the group $\mathrm{GL}(2; \mathbb{Z}[\mathbf{j}])$, let $\sigma_M$ be the permutation of $S$ given by the action of $f_M$, and let $\mathcal{P}_M$ be the set of the $9$ eigenvalues of the permutation matrix associated with $\sigma_M$.
Then the characteristic polynomial $Q_M$ of the endomorphism $(\varphi_M^{-1})_{*}$ of $\mathrm{H}^1\big(X, \mathrm{T}X\big)$ is given by the formula 
\[
Q_M(x)=\prod_{\lambda \in \mathcal{P}_M \setminus \{ 1 \}} \left(x-\frac{\lambda}{\alpha \beta^2}\right) \left(x-\frac{\lambda}{\alpha^2 \beta}\right).
\]
\end{thm}

\begin{rems}$ $ \par
\begin{enumerate}
\item[(i)] We have $\# \mathcal{P}_M=9$, so that $\deg Q_M=16=2 \times 12-8$.
\item[(ii)] If $M$ is in the group $G$ generated by $\phi$, then $Q_M(x)=(x-1)^{16}$ as expected. Indeed, if $M=\left(\begin{array}{cc} \mathbf{j}& 0 \\ 0 & \mathbf{j} \end{array}\right)$ then $\alpha=\beta=\mathbf{j}$ and $\mathcal{P}_M=\{1, 1, 1, 1, 1, 1\}$.
\end{enumerate}
\end{rems}

\begin{cor}\label{bon} $ $
\begin{enumerate}
\item[(i)] If $M$ is a matrix in $H$, then $Q_M(x)=(x-\alpha)^8 (x-\beta)^8$.
\item[(ii)]If $M$ is not of finite order in $\mathrm{GL}(2; \mathbb{Z}[\mathbf{j}])/\langle \,\mathbf{j}\, \mathrm{id} \,\rangle$, then $\varphi_M$ is rigid.
\end{enumerate}
\end{cor}

The first approach of the proof of Theorem \ref{yahou} will use the Atiyah-Bott fixed point theorem to compute the trace of the action of $\varphi_M$ on $\mathrm{H}^1(X, \mathrm{T}X)$ for $M$ in $H$. Knowing this for all iterates of $M$, we obtain all eigenvalues of $\tilde{f}_M$. We will limit ourselves to the case where $M$ is in $H$ (that is to the statement $(i)$ of Corollary \ref{bon}), but all other cases can be dealt with using the same method. Since these calculations don't bring anything new conceptually (and since the second proof is valid for any $M$ in $\mathrm{GL}(2;\Lambda)$) we omit them. 
The other approach is to the machinery of sheaves: the tangent bundles $\mathrm{T} \mathcal{A}$, $\mathrm{T} \widetilde{\mathcal{A}}$ and $\mathrm{T} X$ are related by some exact sequences, allowing to compare their respective cohomology groups.

\subsubsection{First proof by the Atiyah-Bott formula} 
Let us divide the set of fixed points of $\varphi_M$ into two parts: the first part $\Theta_1$ consists of fixed points outside the exceptional divisors $E_i\,$; and the second part $\Theta_2$ consists of the remaining fixed points.

\begin{pro} \label{fix} 
Let $M$ be a matrix in $H$ such that $M^3 \neq \mathrm{id}$. Then the automorphism $\varphi_M$ has $\vert\mathrm{Tr}(M)\vert^2+11$ fixed points on the Kummer surface $X$.
\end{pro}

\begin{proof}
The first step consists in proving that the fixed points of $\varphi_M$ are non-degenerate, \emph{i.e.} \!that $1$ is never an eigenvalue of the differential of $\varphi_M$ at a fixed point. We deal with fixed points in $\Theta_1$ and $\Theta_2$ separately.
\par \smallskip
The points of $\Theta_1$ correspond to points $p$ in $A \smallsetminus S$ such that $f_M(p)$ lies in the orbit $G\cdot p$, modulo the action of $G$. For any such point, the differential of $\varphi_M$ identifies with $M$, $\mathbf{j}M$ or $\mathbf{j}^2 M$, so that 1 is never an eigenvalue.
\par \smallskip
The set $\Theta_2$ can be described as follows: the automorphism $\tilde{f}_M$ acts by the projective transformation $\mathbb{P}(M)$ on the curve $E_i$ via the identification 
\[
E_i \simeq \widetilde{E}_i \simeq \mathbb{P} (\mathrm{T}_{p_i}V) \simeq \mathbb{P}(V).
\]
Therefore we get two fixed points $q_{\alpha}$ and $q_{\beta}$ on each~$E_i$ corresponding to the eigenspaces of $M$. This gives a concrete description of $\Theta_2$, which consists of $18$ points (two distinct points on each $E_i$).
To compute the differential of $\tilde{f}_M$ at such a fixed point, we can assume without loss of generality that this point lies in $E_{1}$. Let $(e,f)$ be an eigenvector of $M$ for the eigenvalue $\alpha$ and assume for simplicity that $e \neq 0$.  If $(x,y)$ are the standard coordinates on $V$, we define coordinates $(u,v)$ in $\widetilde{\mathcal{A}}$ near $\widetilde{E}_1$ by putting $x=u$ and $y=uv$.
Then if we set $Z=u^3$ and $T=v$, $(Z, T)$ are holomorphic coordinates on $X$ near $q_{\alpha}$. In these coordinates, $q_{\alpha}=\left(0, \displaystyle\frac{f}{e}\right)$ and $\tilde{f}_M$ is given by
\[
(Z, T) \mapsto \left( Z (a+bT)^3, \frac{c+dT}{a+bT} \right).
\]
so that the eigenvalues of $\mathrm{d} \tilde{f}_M$ at $q_{\alpha}$ are $\alpha^3$ and $\alpha^{-2}$, which are different from $1$.
\par \bigskip
We can now apply the Lefschetz fixed point formula. As $\mathrm{H}^1\big(X, \mathbb{R}\big)=\mathrm{H}^3\big(X, \mathbb{R}\big)=0$, we get
that the number of fixed points of $\varphi_M$ is equal to $\mathrm{Tr} \,  \, {\widetilde{f}^{\, \,* \vphantom{\bigl(}}}_{M\, | \mathrm{H}^2(X, \mathbb{R})}+2$. Then we can use formula (\ref{Picard}): the trace on the part $\left(\wedge^2_{\mathbb{Z}}(\Lambda^* \times \Lambda^*)\right)^G \otimes_{\mathbb{Z}} \mathbb{R}$ is the sum 
\[
\vert\alpha\vert^2 + \vert\beta\vert^2 + \alpha \overline{\beta} + \overline{\alpha} \beta=\vert\alpha + \beta\vert^2;
\]
and since $M$ belongs to $H$, $\varphi_M$ acts trivially on the factor $\bigoplus_{i=1}^9 \,\mathbb{R} [\widetilde{E}_i]$ so that the trace on this factor is $9$. This gives the required result.
\end{proof}

We will now study in greater details the set $\Theta_1$. Let us define three subsets  $S_1$, $S_2$, $S_3$ of $A$ as follows:
\[
S_1=\{ p \in A\smallsetminus S \, \, \, \vert \,\, f_M\,(p)=p \},  S_2=\{ p \in A\smallsetminus S \, \, \, \vert \,\, f_M\,(p)=\phi(p) \},  S_3=\{ p \in A\smallsetminus S \, \, \, \vert \,\, f_M\,(p)=\phi^2(p) \}.
\]
To compute the cardinality of the $S_i$'s, we use the following statement: 

\begin{lem}\label{malin} 
For any matrix $P$ in $\mathrm{GL(2; \Lambda)}$ such that $1$ is not an eigenvalue of $P$, the automorphism $f_P$ of the complex torus $A$ has $\big| 1-\mathrm{Tr} \, (P) + \det P \,\big|^2$ fixed points. 
\end{lem}

\begin{proof}
Since the fixed points of $f_P$ are non-degenerate, we can apply the holomorphic Lefschetz fixed point formula (\cite[p. 426]{GH}, \cite{McMullen2}, \cite{Cantat}); this gives $$1-\overline{\mathrm{Tr}\, (P)}+\overline{\det P}=\#\, \mathrm{Fix}\, (f_P) \, \times \, \displaystyle\frac{1}{1-{\mathrm{Tr}\, (P)}+{\det P}}.$$
\end{proof}

\begin{cor} 
Let $M$ be a matrix in $H$ such that $M^3 \neq \mathrm{id}$.  One has the following equalities: 
\[
\#\, S_1= \big\vert\mathrm{Tr} \, (M)-2\big\vert^2-9\quad\quad\quad\text{and}\quad\quad\quad\#\, S_2= \#\, S_3= \big\vert \mathrm{Tr} \, (M)+1\big\vert^2-9.
\]
\end{cor}

\begin{proof}
This follows directly from Lemma \ref{malin}, since 
\[
S_1=\mathrm{Fix}\,(f_M) \smallsetminus S, \quad\quad\quad S_2=\mathrm{Fix}\, (f_{\mathbf{j}^2 M})\smallsetminus S\quad\quad\quad \text{and} \quad\quad\quad S_3=\mathrm{Fix}\, (f_{\mathbf{j} M})\smallsetminus S.
\]
\end{proof}

The group $G$ acts transitively on each set $S_i$. Then $\Theta_1$ can be written as the disjoint union of the quotients
$S_i/G$. Remark that we can obtain in this way the result of Proposition \ref{fix}: indeed, if $t=\mathrm{Tr} \, (M)$,
\[
\#\, \Theta_1=\frac{\big\vert t-2\big\vert^2}{3}-3+ 2 \left(\frac{\big\vert 1+t\big\vert^2}{3}-3\right)=\big\vert t\big\vert^2-7
\]
Since $\#\, \Theta_2=18$, this gives $\#\,\mathrm{Fix} \, (\varphi_M)=\big\vert t\big\vert^2+11$.
 \par \bigskip
As $\mathrm{H}^0 \big(X, \mathrm{T}X \big)=\mathrm{H}^2 \big(X, \mathrm{T}X \big)= \{0\}$, the holomorphic Atiyah-Bott fixed point formula \cite[Thm. 4.12]{AtiyahBott} yields the following result :
 
\begin{pro} \label{somme}
The trace of $(\varphi_M^{-1})_*$ acting on $\mathrm{H}^1 \big(X, \mathrm{T}X \big)$ is equal to the sum $-\displaystyle\sum_{x^{\vphantom{A}}} \dfrac{\mathrm{Tr} \, \,\, \mathrm{d}(\varphi_M^{-1})_{x}}{\det \, (\mathrm{id}- \mathrm{d}(\varphi_M)_x)}$ where $x$ runs through the fixed points of $\varphi_M$.
 \end{pro}
 
\begin{cor}\label{cor:ouf}
Let $M$ be a matrix in $H$ such that $M^3 \neq\mathrm{id}$. Then the trace of $(\varphi_M^{-1})_{*}$ acting on $\mathrm{H}^1 \big(X, \mathrm{T}X \big)$ is $8 \, \mathrm{Tr}  \, (M)$.
\end{cor} 

\begin{proof}
We divide the fixed point set of $\varphi_M$ in four parts: $S_1/G$, $S_2/G$, $S_3/G$ and $\Theta_2$. The first three parts correspond to $\Theta_1$. We start by computing the contribution of $\Theta_2$ in the sum of Proposition \ref{somme}. Let us put $t=\mathrm{Tr} \, (M)=\alpha+\alpha^{-1}$. Any pair of fixed points in a divisor $E_i$ yields the term
 \[
 \frac{\alpha^{-3}+\alpha^2}{(1-\alpha^3)(1-\alpha^{-2})}+ \frac{\alpha^{3}+\alpha^{-2}}{(1-\alpha^{-3})(1-\alpha^{2})}
 \]
 \par \medskip
\noindent so that the contribution of $\Theta_2$ is $-9t -\displaystyle\frac{6}{t-2}-\displaystyle\frac{3}{1+t} \cdot$ Now the contributions of $S_1/G$, $S_2/G$ and $S_3/G$ are respectively
\[
\left( \frac{\big\vert t-2\big\vert^2}{3}-3\, \right)\, \frac{t}{2-t}, \quad\quad\quad \,\left( \frac{\big\vert 1+t\big\vert^2}{3}-3\, \right)\, \frac{-\mathbf{j}^2\, t}{1+t} \quad\quad\quad \, \text{and} \quad\quad\quad \left( \frac{\big\vert 1+t\big\vert^2}{3}-3\, \right)\, \frac{-\mathbf{j}\, t}{1+t}
\]
so that the contribution of $\Theta_1$ is $t+\displaystyle\frac{3t}{t-2}-\displaystyle\frac{3t}{1+t} \cdot$
Adding the contributions of $\Theta_1$ and $\Theta_2$ we get $-8t$.
\end{proof}

\subsubsection{Second proof by sheaf theory} \label{sheaf}

We start by fixing some conventions and notations that are specific to this section. For any group $G$ a $G$-module will mean
a right $G$-module, that is a \textit{contravariant} representation of $G$ on some vector space. We denote the vector space $\mathbb{C}^2$ by $V$. The same convention goes for equivariant vector bundles: they will all be right equivariant (\textit{see} \S \ref{serre}). If nothing else is specified, we consider $V$ as a natural right $\mathrm{GL}(2; \Lambda)$-module, where a matrix $M$ acts by $M^{-1}$. We start by the two following exact sequences of sheaves:
\par \medskip
\underline{\textit{First exact sequence of sheaves on $\widetilde{\mathcal{A}}$ relating $\mathrm{T} {A}$ and $\mathrm{T}\widetilde{\mathcal{A}}$ via the blowup map $\delta$}}
\begin{equation}\label{exact1}
0\longrightarrow\mathrm{T}\widetilde{\mathcal{A}}\longrightarrow \delta^*\mathrm{T}\mathcal{A}\longrightarrow \bigoplus_{i=1}^9 \iota_{E_i *^{\vphantom{A}}}\mathrm{T}_{\widetilde{E}_i}(-1) \longrightarrow 0.
\end{equation}
This sequence is $\mathrm{GL}(2; \Lambda)$-equivariant and $G$-equivariant. Let us explain its construction: the first map is the differential of the map $\delta \colon \widetilde{\mathcal{A}} \rightarrow \mathcal{A}$ corresponding to the blowup of the nine points in $S$; as a sheaf morphism it is injective. For any $i$ with $1 \leq i \leq 9$ and any point $[\ell]$ in $E_i$ (corresponding to a line $\ell$ in $\mathrm{T}_{p_i} \mathcal{A}$), the image of $(\delta_{*})_{[\ell]} \colon \mathrm{T}_{[\ell]} \widetilde{\mathcal{A}} \rightarrow  \mathrm{T}_{p_i} {\mathcal{A}}$ is precisely the line $\ell$. Thanks to the Euler exact sequence \cite[Prop. 2.4.4]{Huybrechts}
\begin{equation} \label{euler}
0 \longrightarrow  \mathcal{O}_{E_i}(-1) \longrightarrow \mathrm{T}_{p_i} \mathcal{A} \otimes \mathcal{O}_{E_i} \longrightarrow  \mathrm{T}{E}_i (-1) \longrightarrow 0
\end{equation}
the complex line $\displaystyle \frac{\mathrm{T}_{p_i} {\mathcal{A}}}{\mathrm{Im}\, (\delta_{*})_{[\ell]}}$ identifies canonically with the fiber of $\mathrm{T}_{\widetilde{E}_i}(-1)$ at $[\ell]$.
\par \medskip
Remark that $\mathrm{T}_{\widetilde{E}_i}(-1)$ is isomorphic to $\mathcal{O}_{E_i}(1)$, but this isomorphism is not canonical and cannot be made compatible in any way with the action of $\mathrm{GL}(2;\Lambda)$.
\par \medskip

\underline{\textit{Second exact sequence of sheaves on $X$ relating $\mathrm{T}\widetilde{\mathcal{A}}$ and $\mathrm{T}X$ via the cyclic cover $\pi$}}
\begin{equation}\label{exact2}
0\longrightarrow \pi_*(\mathrm{T}\widetilde{\mathcal{A}})^G\longrightarrow\mathrm{T}X \longrightarrow  \bigoplus_{i=1}^9  \iota_{E_i *^{\vphantom{A}}} \mathrm{N}_{E_i/X}\longrightarrow 0.
\end{equation}
This sequence is $\mathrm{GL}(2; \Lambda)$-equivariant. Let us again explain its construction: for any open subset $U$ of $X$, the sections of the sheaf $\pi_*(\mathrm{T}\widetilde{\mathcal{A}})^G$ on $U$ are exactly the $G$-invariant holomorphic vector fields on $\pi^{-1}(U)$. We can take holomorphic coordinates $(z, w)$ and $(x, y)$ near a point of $\widetilde{E}_i$ and its image in $X$ such that $\pi (z, w)=(z^3, w)$, $\phi(z, w)=(\mathbf{j}z, w)$ and $\widetilde{E}_i=\{z=0\}$. Therefore a $G$-invariant holomorphic vector field is of the form $z \alpha(z^3, w) {\partial_z}+ \beta(z^3, w) \partial_w$, which is (outside $E_i$) the pull-back of the holomorphic vector field $3x \,\alpha(x, y) \partial_x+\beta(x, y) \partial_y$.
The latter holomorphic vector field extends uniquely across $E_i$, and the extension at $E_i$ is tangent to $E_i$. Conversely, the same calculation shows that every such holomorphic vector field yields a $G$ invariant holomorphic vector field upstairs. To conclude, it suffices to note that a holomorphic vector field on an open subset of $X$ is tangent to $E_i$ if and only if its restriction to $E_i$ maps to zero via the morphism $\mathrm{T}X_{E_i} \rightarrow \mathrm{N}_{E_i/X}$.
\par \bigskip
Let us introduce some notation. First we consider the natural representation $Z$ of the symmetric group $\mathfrak{S}_{9}$ in $\mathbb{C}^9$. There is a natural group morphism $\mathrm{GL}(2; \Lambda) \rightarrow \mathfrak{S}_{9}$ given by the action on the set $S=\{p_1, \ldots, p_9\}$, so that we will consider $Z$ as a $\mathrm{GL}(2; \Lambda)$-module. As an $H$-module, $Z$ is the sum of nine copies of the trivial representation.

Let us denote by $\mathcal{F}$ the sheaf $\displaystyle\bigoplus_{i=1}^9 \iota_{E_i *^{\vphantom{A}}}\mathrm{T}_{\widetilde{E}_i}(-1)$. 

\begin{lem} \label{technique} 
The following assertions hold \emph{:}
\begin{itemize}
\item[(a)] For any integer $i$ with $0 \leq i \leq 2$, the natural morphisms 
\[
\begin{cases}
\mathrm{H}^i \big(\mathcal{A}, {\mathrm{T}}\mathcal{A}\big) \longrightarrow \mathrm{H}^i \big(\widetilde{\mathcal{A}}, \delta^* {\mathrm{T}}\mathcal{A}\big)\\
\mathrm{H}^i \big(X, \pi_* {\mathrm{T}}\widetilde{\mathcal{A}}\big) \longrightarrow  \mathrm{H}^i \big(\widetilde{\mathcal{A}}, \pi^* \pi_* {\mathrm{T}} \widetilde{\mathcal{A}}\big) \longrightarrow \mathrm{H}^i \big(\widetilde{\mathcal{A}}, {\mathrm{T}} \widetilde{\mathcal{A}}\big)
\end{cases}
\]
are $\mathrm{GL(2; \Lambda)}$-equivariant isomorphisms.
\item[(b)] The cohomology groups $\mathrm{H}^0 \big(\widetilde{\mathcal{A}}, \mathcal{F}\big)^{G}$ and $\mathrm{H}^1 \big(\widetilde{\mathcal{A}}, \mathcal{F}\big)^{G}$ vanish.
\item[(c)] The cohomology group $\mathrm{H}^1 \big(X, \pi_* ({\mathrm{T}}\widetilde{\mathcal{A}})^G\big)$ vanishes. Besides, $\mathrm{H}^2 \big(X, \pi_* ({\mathrm{T}}\widetilde{\mathcal{A}})^G\big)$ is isomorphic to $\det V \otimes V$ as a $\mathrm{GL}(2; \Lambda)$-module.
\item[(d)] For $1 \leq i \leq 9$, $\oplus_{i=1}^9 \mathrm{H}^1 \big(E_i, \mathrm{N}_{E_i/X}\big)$ is isomorphic to $(\mathrm{det}\, V \otimes V) \otimes Z $ as a $\mathrm{GL}(2; \Lambda)$-module.
\end{itemize}
\end{lem}

\begin{proof} 
{\it ($a$)} We write the Leray spectral sequence for the pair $(\delta^* {\mathrm{T}}\mathcal{A}, \mathcal{A})$: we have 
\[
\mathrm{E}_2^{p,q}=\mathrm{H}^p\big(\mathcal{A}, \mathrm{R}^q \delta_* \mathcal{O}_{\widetilde{\mathcal{A}}} \otimes {\mathrm{T}}\mathcal{A}\big)\quad\quad\quad\text{and}\quad\quad\quad\mathrm{E}_{\infty}^{p,q}=\mathrm{Gr}^p \mathrm{H}^{p+q} \big(\widetilde{\mathcal{A}}, \delta^* {\mathrm{T}}\mathcal{A}\big).
\] 
Since $\delta$ is the projection of a point blowup, it is known that $\mathrm{R}^q \delta_* \mathcal{O}_{\widetilde{\mathcal{A}}}$ vanishes for $q>0$ so that the spectral sequence degenerates and we get $\mathrm{H}^p(\mathcal{A}, {\mathrm{T}}\mathcal{A})=\mathrm{E}^{p,0}_2 \simeq \mathrm{E}_{\infty}^{p,0} = \mathrm{H}^{p} \big(\widetilde{\mathcal{A}}, \delta^* {\mathrm{T}}\mathcal{A}\big)$. The argument is the same for the second morphism: $\pi$ being finite, $\mathrm{R}^q \pi_* {\mathrm{T}} \widetilde{\mathcal{A}}=0$ for $q>0$.
\par \smallskip

{\it ($b$)} The vanishing of $\mathrm{H}^1 \big(\widetilde{\mathcal{A}}, \mathcal{F}\big)^{G}$ is straightforward as $\mathrm{H}^1(\widetilde{\mathcal{A}}, \mathcal{F})=0$. We see that $\mathrm{H}^0 (\widetilde{\mathcal{A}}, \mathcal{F})$ is isomorphic as a $G$-module to the direct sum $\bigoplus_{i=1}^9 \mathrm{T}_{p_i} A$, where $\phi$ acts by the inverse of its differential at each point $p_i$. Therefore, $\mathrm{H}^0 (\widetilde{\mathcal{A}}, \mathcal{F}) \simeq V^9$ and the result follows since $V^G=\{0\}$.  
\par \smallskip

{\it ($c$)} According to {\it ($a$)} and since $G$ is finite, we have isomorphisms \[\mathrm{H}^i\big(X, \pi_* (\mathrm{T} \widetilde{\mathcal{A}})^G\big) \simeq \mathrm{H}^i\big(X, \pi_* \mathrm{T} \widetilde{\mathcal{A}}\big)^G \simeq \mathrm{H}^i\big(\widetilde{\mathcal{A}}, \mathrm{T} \widetilde{\mathcal{A}}\big)^G.\] Using (\ref{exact1}) combined with {\it ($b$)} and {\it ($a$)}, we obtain $\mathrm{GL}(2; \Lambda)$-equivariant isomorphisms $\mathrm{H}^i\big(\widetilde{\mathcal{A}}, \mathrm{T} \widetilde{\mathcal{A}}\big)^G \simeq \mathrm{H}^i \big(\mathcal{A}, \mathrm{T}\mathcal{A}\big)^G$ for $i=1$, $2$. Now $\mathrm{H}^i \big(\mathcal{A}, \mathrm{T}\mathcal{A}\big)$ is isomorphic to $V^* \otimes V$ (resp. $\wedge^2 V^* \otimes V$) for $i=1$ (resp. $i=2$), where a matrix $M$ in $\mathrm{GL}(2; \Lambda)$ acts by $\vphantom{A}^{t} \overline{M} \otimes~M^{-1}$ $\big($resp. $\wedge^2\,(\,\vphantom{A}^{t} \overline{M}) \otimes M^{-1}\big)$. This proves that $\mathrm{H}^1 \big(\mathcal{A}, \mathrm{T}\mathcal{A}\big)^G$ vanishes and that 
\[
\mathrm{H}^2 \big(\mathcal{A}, \mathrm{T}\mathcal{A}\big)^G=\mathrm{H}^2\big(\mathcal{A}, \mathrm{T}\mathcal{A}\big)\simeq \det V \otimes V
\]
since~$\overline{\,\mathbf{j}\,}^2 \,\mathbf{j}^{-1}=1$. This yields the result.
\par \smallskip

{\it ($d$)} Taking the determinant of the Euler exact sequence \eqref{euler}, we obtain that the sheaves $\oplus_{i=1}^9 \mathrm{K}_{\widetilde{E}_i}$ and $\oplus_{i=1}^9 \det \,\mathrm{T}^*_{p_i} \mathcal{A} \otimes_{\mathbb{C}} \mathcal{O}_{\widetilde{E}_i}(-2)$ are naturally isomorphic, and this isomorphism is compatible with the action of $\mathrm{GL}(2; \Lambda)$. Since $\pi^* \mathrm{N}_{E_i /X}$ is canonically isomorphic to $\mathcal{O}_{\widetilde{E}_i}(-3)$, we get by Serre duality a chain of $\mathrm{GL}(2; \Lambda)$-equivariant isomorphisms
\begin{small}
\[
\bigoplus_{i=1}^9 \,\mathrm{H}^1 \big(E_i, \mathrm{N}_{E_i/X}\big) \simeq \bigoplus_{i=1}^9 \,\mathrm{H}^0 \big(\widetilde{E}_i, \mathcal{O}_{\widetilde{E}_i}(3) \otimes \mathrm{K}_{\widetilde{E}_i}\big)^* \simeq \bigoplus_{i=1}^9 \,\det \,\mathrm{T}_{p_i} \mathcal{A} \otimes_{}\mathrm{H}^0 \big(\widetilde{E}_i, \mathcal{O}_{\widetilde{E}_i}(1)\big)^*  \simeq (\det V \otimes V) \otimes Z.
\]
\end{small}
\end{proof}

We can now prove Corollary \ref{bon}. Using the exact sequence (\ref{exact2}), Lemma \ref{technique} {\it ($c$)}, {\it ($d$)} and the fact that $\mathrm{H}^2\big(X, \mathrm{T}X\big)$ vanishes, we get an exact sequence of $H$-modules
\[
0 \longrightarrow \mathrm{H}^1\big(X, \mathrm{T}X\big) \longrightarrow V^9 \longrightarrow V \longrightarrow 0.
\]
The result follows.

\subsection{The case of the square lattice} \label{cubicc}

For the sake of completeness and also because it is an interesting case, we also provide briefly the corresponding results for square lattices (that is $\Lambda$ is the ring of Gau\ss{} integers $\mathbb{Z}[\mathbf{i}]$, $E$ is the elliptic curve $\mathbb{C}/\Lambda$, $\mathcal{A}=E \times E$, $\phi(z, w)=(\mathbf{i}z, \mathbf{i}w)$ and $G$ is the group of order $4$ generated by $\phi$). The same strategy works but the results are slightly different. In the square case, the group $G$ is of order $4$, the map $\phi$ (which is this case the multiplication by $\mathbf{i}$) has $4$ fixed points and $12$ new other points are fixed by ${\phi}^2$. More precisely, if we put
\[
\left\{\!\begin{alignedat}{7}
p_1&=(0,0)&\quad p_2&=\left(0,\frac{1+\mathbf{i}}{2}\right)&\quad p_3&=\left(\frac{1+\mathbf{i}}{2},0\right)&\quad p_4&=\left(\frac{1+\mathbf{i}}{2},\frac{1+\mathbf{i}}{2}\right)\\
p_5&=\left(0, \frac{1}{2}\right)&\quad p_6&=\left(\frac{1}{2}, 0\right)&\quad p_7&=\left(\frac{1}{2}, \frac{1+\mathbf{i}}{2}\right)&\quad p_{8}&=\left(\frac{1+\mathbf{i}}{2}, \frac{1}{2}\right) \\
p_{9}&=\left(\frac{1}{2}, \frac{1}{2}\right)&\quad p_{10}&=\left(\frac{1}{2}, \frac{\mathbf{i}}{2}\right) &\quad p'_5&=\left(0, \frac{\mathbf{i}}{2}\right)&\quad p'_6&=\left(\frac{\mathbf{i}}{2}, 0\right) \\
p'_7&=\left(\frac{\mathbf{i}}{2}, \frac{1+\mathbf{i}}{2}\right)&\quad p'_{8}&=\left(\frac{1+\mathbf{i}}{2}, \frac{\mathbf{i}}{2}\right)&\quad p'_{9}&=\left(\frac{\mathbf{i}}{2}, \frac{\mathbf{i}}{2}\right)&\quad p'_{10}&=\left(\frac{\mathbf{i}}{2}, \frac{1}{2}\right)
\end{alignedat}
\right.
\]
then $\mathrm{Fix}\,(\phi)=\{p_1, p_2, p_3, p_4 \}$, $\mathrm{Fix}\, (\phi^2)\setminus \mathrm{Fix}\, (\phi)=\{p_5, \ldots, p_{10}, p'_5, \ldots, p'_{10}\}$, and $p'_i=\phi(p_i)$ for $5 \leq i \leq 10$. We put $S'=\mathrm{Fix}\,(\phi)$, $S''=\mathrm{Fix}\, (\phi^2) \setminus \mathrm{Fix}\, (\phi)$ and $S=S' \cup S''$. We also denote by $Z'$ (resp. $Z''$, resp. $Z$) the natural representation of the symmetric group $\mathfrak{S}_4$ (resp. $\mathfrak{S}_{12}$, resp. $\mathfrak{S}_6$) in $\mathbb{C}^4$ (resp. $\mathbb{C}^{12}$, resp. $\mathbb{C}^6$). There is a natural group morphism $\mathrm{GL}(2; \Lambda) \rightarrow \mathfrak{S}_4$ (resp. $\mathrm{GL}(2; \Lambda) \rightarrow \mathfrak{S}_{12}$) given by the action on the set $S'$ (resp. $S''$), so that we will
consider $Z'$ and $Z''$ as $\mathrm{GL}(2; \Lambda)$-modules. Note that $\mathrm{GL}(2; \Lambda)$ acts on the set of pairs $\{p_i, p'_i\}$, so that the morphism $\mathrm{GL}(2; \Lambda) \rightarrow \mathfrak{S}_{12}$ factors as
\[
\mathrm{GL}(2; \Lambda) \xrightarrow{(\xi, \eta)}  \mathfrak{S}_6 \times (\mathbb{Z}/2\mathbb{Z})^6 \rightarrow \mathfrak{S}_{12}
\]
where the factor $\mathfrak{S}_6$ corresponds to the action on the set of pairs, and each factor in $\mathfrak{S}_2$ corresponds to the action on the corresponding pair. We consider $Z$ as a $\mathrm{GL}(2; \Lambda)$-module via the representation $\rho$ given as follows: for $1 \leq i \leq 6$,
\begin{equation} \label{baleze}
\rho(M). \textbf{e}_i=\left(\eta(M)\right)_i \, \textbf{e}_{\, \xi(M)(i)}.
\end{equation}
As before, let $\widetilde{\mathcal{A}}$ be the blowup of $\mathcal{A}$ along the $16$ points of $S$, and denote by $\widetilde{E}_1$, $\ldots$, $\widetilde{E}_4$, $\widetilde{E}_5$, $\ldots$, $\widetilde{E}_{10}$, $\widetilde{E}'_5$, $\ldots$, $\widetilde{E}'_{10}$ the corresponding exceptional divisors.  
The quotient $X/G$ is a basic rational surface that can be obtained by blowing up $\mathbb{P}^2$ in $13$ points. We consider again the diagram
\[
\xymatrix{& \widetilde{\mathcal{A}}\ar[ld]_{\delta}\ar[rd]^{\pi}&\\
A&&X
}
\]
If we put $E_i=\pi(\widetilde{E}_i)$, then $\pi$ is the cyclic covering of order $4$ along the divisor $\sum_{i=1}^4 4E_i+ \sum_{i=5}^{10} 2E_i$. In particular,
\[
\pi^*(E_i)= \begin{cases} 4\widetilde{E}_i &\textrm{if} \quad 1 \leq i \leq 4 \\
2\widetilde{E}_i+2\widetilde{E}'_i &\textrm{if} \quad 5 \leq i \leq 10 \end{cases}
\]

\begin{thm} \label{yahou2} 
Let $X$ be the rational Kummer surface associated with the lattice $\mathbb{Z}[\mathbf{i}]$ of Gau\ss{} integers. For any element $M$ in $\mathrm{GL}(2; \mathbb{Z}[\mathbf{i}])$, let $\mathcal{P}_{M}$ be the set of eigenvalues of $\rho(M)$, let $\sigma'_M$ be the permutation of $S'$ given by the action of $f_M$, and let $\mathcal{P}'_M$ be the set of $4$ eigenvalues of the permutation matrix associated with $\sigma'_M$. 
Then the characteristic polynomial $Q_M$ of the endomorphism $({\varphi_M}^{-1})_*$ of $\mathrm{H}^1\big(X, \mathrm{T}X\big)$ is given by the formula 
\[
Q_M(x)=\prod_{\lambda \in \mathcal{P}_M} \left(x-\frac{\lambda}{\alpha \beta}\right) \prod_{\mu \in \mathcal{P}'_M} \left\{\left(x-\frac{\mu}{\alpha^3 \beta}\right) \left(x-\frac{\mu}{\alpha \beta^3}\right) \left(x-\frac{\mu}{\alpha^2 \beta^2}\right)\right\}.
\]
\end{thm}

\begin{rems} $ $ \par
\begin{enumerate}
\item[(i)] We have $\# \mathcal{P}_M=6$ and $\# \mathcal{P}'_M=4$ so that $\deg Q_M=18=2 \times 13-8$.
\item[(ii)] If $M$ is in the group $G$ generated by the automorphism $\phi$, then $Q_M(x)=(x-1)^{18}$ as expected. Indeed, if $M=\left(\begin{array}{cc} \mathbf{i}& 0 \\ 0 & \mathbf{i} \end{array}\right)$ then $\alpha=\beta=\mathbf{i}$, $\mathcal{P}_M=\{-1, -1, -1, -1, -1, -1\}$ and $\mathcal{P}'_M=\{1, 1, 1, 1\}$.
\end{enumerate}
\end{rems}

\begin{proof}
We follow the strategy developed for hexagonal lattices in \S \ref{sheaf}. Coming back to Lemma \ref{technique}, the main difference is that $\mathrm{H}^2(\mathcal{A}, \mathrm{T}\mathcal{A})^G=0$ since $\overline{\mathbf{i}}^2\, \mathbf{i}^{-1}=\mathbf{i} \neq 1$. Hence we get an isomorphism of $\mathrm{GL}(2; \Lambda)$-modules
\[
\mathrm{H}^1\big(X, \mathrm{T}X\big)\simeq \bigoplus_{i=1}^{10} \mathrm{H}^1\big(E_i, \mathrm{N}_{E_i/X}\big).
\]
The right-hand side of the previous isomorphism splits as the direct sum of two $\mathrm{GL}(2; \Lambda)$-modules: $\bigoplus_{i=1}^{5} \mathrm{H}^1\big(E_i, \mathrm{N}_{E_i/X}\big)$ and $\bigoplus_{i=5}^{10} \mathrm{H}^1\big(E_i, \mathrm{N}_{E_i/X}\big)$. Besides, we have
\[
\begin{cases}
\pi^*\mathrm{N}_{E_i/X} \simeq \mathcal{O}_{\widetilde{E}_i} (-4) &\textrm{if} \quad 1 \leq i \leq 4\\
\pi^*\mathrm{N}_{E_i/X} \simeq \mathcal{O}_{\widetilde{E}_i} (-2)\oplus \mathcal{O}_{\widetilde{E}'_i}(-2) &\textrm{if} \quad  5\leq i \leq 10.\\
\end{cases}
\]
Using Serre duality, we have $\mathrm{GL}(2;\Lambda)$-equivariant isomorphisms
\begin{align*}
\bigoplus_{i=1}^4 \,\mathrm{H}^1 \big(E_i, \mathrm{N}_{E_i/X}\big) &\simeq \bigoplus_{i=1}^4 \,\mathrm{H}^0 \big(\widetilde{E}_i, \mathcal{O}_{\widetilde{E}_i}(4) \otimes \mathrm{K}_{\widetilde{E}_i}\big)^* \simeq \bigoplus_{i=1}^4 \,\det \,\mathrm{T}_{p_i} \mathcal{A} \otimes_{}\mathrm{H}^0 \big(\widetilde{E}_i, \mathcal{O}_{\widetilde{E}_i}(2)\big)^*  \\
&\simeq (\det V \otimes \mathrm{Sym}^2\, V) \otimes Z'
\end{align*}

and
\begin{align*}
\bigoplus_{i=5}^{10} \,\mathrm{H}^1 \big(E_i, \mathrm{N}_{E_i/X}\big) &\simeq \left( \bigoplus_{i=5}^{10} \, \left\{\mathrm{H}^0 \big(\widetilde{E}_i, \mathcal{O}_{\widetilde{E}_i}(2) \otimes \mathrm{K}_{\widetilde{E}_i}\big)^* \oplus \,\mathrm{H}^0 \big(\widetilde{E}'_i, \mathcal{O}_{\widetilde{E}'_i}(2) \otimes \mathrm{K}_{\widetilde{E}'_i}\big)^* \right\} \right)^G\\
&\simeq \left(\bigoplus_{i=5}^{10} \, \left\{ \det \,\mathrm{T}_{p_i} \mathcal{A} \otimes_{}\mathrm{H}^0 \big(\widetilde{E}_i, \mathcal{O}_{\widetilde{E}_i}\big)^* \oplus\, \det \,\mathrm{T}_{p'_i} \mathcal{A} \otimes_{}\mathrm{H}^0 \big(\widetilde{E}'_i, \mathcal{O}_{\widetilde{E}'_i}\big)^* \right\} \right)^G\\
&\simeq \left(\bigoplus_{i=5}^{10} \, \left\{ \det \,\mathrm{T}_{p_i} \mathcal{A} \oplus\, \det \,\mathrm{T}_{p'_i} \mathcal{A}  \right\} \right)^G\\
&\simeq (\det V \otimes Z'')^G \simeq \det V \otimes Z
\end{align*}
where $Z$ is defined via \eqref{baleze}. Therefore we get an isomorphism of $\mathrm{GL}(2; \Lambda)$-modules
\[
\mathrm{H}^1(X, \mathrm{T}X) \simeq (\det V \otimes \mathrm{Sym}^2\, V) \otimes Z' \oplus\, (\det V \otimes Z).
\]
\end{proof}

\begin{cor} $ $ \par
\begin{enumerate}
\item[(i)] Let $M$ be a matrix in $\mathrm{GL}(2; \mathbb{Z}[\mathbf{i}])$ such that $M \equiv \left(\begin{array}{cc} 1& 0 \\ 0 & \mathbf{i} \end{array}\right)$ mod $2 \mathbb{Z}[\mathbf{i}]$. Then \[
Q_M(x)=\left(x+{\det M}\right)^4 \, \left(x+\frac{1}{\alpha^2} \right)^4 \left(x+\frac{1}{\beta^2}\right)^4 (x+1)^4.
\]
\item[(ii)] There exist infinitely many $M$ such that $\varphi_M$ is rigid.
\end{enumerate}
\end{cor}

\begin{proof} $ $
\begin{enumerate}
\item[(i)] The action of $\tilde{f}_M$ on the finite set $S$ depends only of the class of $M$ modulo the ideal $2 \mathbb{Z}[\mathbf{i}]$. Therefore, if $M  \equiv \left(\begin{array}{cc} 1& 0 \\ 0 & \mathbf{i} \end{array}\right)$ mod $2 \mathbb{Z}[\mathbf{i}]$, the action on $S''$ is the permutation \[
(p_5, p'_5)\,(p_8, p'_8)\, (p_9, p_{10})(p'_9, p'_{10})
\] 
and the corresponding element in the group $\mathfrak{S}_6 \times (\mathbb{Z}/2 \mathbb{Z})^6$ is $(9,10), (-1, 1, 1, -1, 1, 1)$. Thus 
\[
\rho(M)=
\left(\begin{array}{cccccc} -1&0&0&0&0&0 \\ 0&1&0&0&0&0 \\ 0&0&1&0&0&0 \\ 0&0&0&-1&0&0 \\0&0&0&0&0&1 \\ 0&0&0&0&1&0
\end{array}\right)
\]
so that $\mathcal{P}_M=\{1, 1, 1, -1, -1, -1\}$. On the other hand, the action on $S'$ is trivial. Lastly, $\det M \in \{\mathbf{i}, -\mathbf{i} \}$. Hence we get the result of $(ii)$.
\item[(ii)] It suffices to prove that there are infinitely many matrices $M$ in $\mathrm{GL}(2; \mathbb{Z}[\mathbf{i}])$ congruent to $\left(\begin{array}{cc} 1& 0 \\ 0 & \mathbf{i} \end{array}\right)$ mod $2 \mathbb{Z}[\mathbf{i}]$ such that $\pm \mathbf{i}$ are not eigenvalues of $M$, which is straightforward: for instance the matrices $\left(\begin{array}{cc} 2n+1& 2\mathbf{i}n \\ -2n & -\mathbf{i}(2n-1) \end{array}\right)$, $n \in \mathbb{Z}$ work.
\end{enumerate}
\end{proof}

To conclude this section, let us mention that it can be proved as in Lemma \ref{canonique} that $-2 \mathrm{K}_X$ is linearly equivalent to $\sum_{i=1}^4 E_i$, so that $|-\mathrm{K}_X|=\varnothing$ and $|-2\mathrm{K}_X|=\sum_{i=1}^4 E_i$. Hence $X$ is a rational Coble surface\footnote{This fact was pointed out to us by I. Dolgachev.} (that is $-2 \mathrm{K}_X$ is effective but $-\mathrm{K}_X$ is not). Therefore we see that such surfaces can carry rigid automorphisms.

\section{Realisation of infinitesimal deformations using divisors}\label{Sec:realization}
Let $X$ be a rational surface, and let $f$ be an automorphism of $X$. We will present a practical method to compute the actions $f^*$ and $f_*$ of $f$ on $\mathrm{H}^1\big(X,\mathrm{T}X\big)$ using divisors. This is much more delicate than for the action on the Picard group. In practical examples, the method is effective (\emph{see} \S \ref{Sec:explicit}).

\subsection{1-exceptional divisors}\label{Subsec:1exc}
\par \medskip
Let $X$ be a complex surface, and let $D$ be a divisor on $X$. Attached to $D$ is the holomorphic line bundle $\mathcal{O}_X(D)$ whose sections are the meromorphic fonctions $f$ such that $D+\mathrm{div}(f)$ is effective. For any holomorphic vector bundle $\mathscr{E}$ on $X$, there is an exact sequence
\begin{equation} \label{cute}
0 \longrightarrow \mathscr{E} \longrightarrow \mathscr{E}(D) \longrightarrow \mathscr{E}(D)_{|D} \longrightarrow 0.
\end{equation}

\begin{defi}
Let $X$ be a complex surface and let $D$ be an effective divisor. We say that $D$ is \textit{$1$-exceptional} if the natural morphism $\mathrm{H}^1\big(X, \mathrm{T}X\big) \longrightarrow \mathrm{H}^1\big(X, \mathrm{T}X(D)\big)$ induced by the first arrow of \eqref{cute} for $\mathscr{E}=\mathrm{T}X$ vanishes.
\end{defi}

Remark that since $\mathrm{H}^1(\mathbb{P}^2, \mathrm{T} \mathbb{P}^2)$ vanishes, any effective divisor on $\mathbb{P}^2$ is $1$-exceptional.
\par \medskip
The importance of these divisors comes from the following immediate observation: if $D$ is $1$-exceptional, then we have an exact sequence
\begin{equation}\label{div}
0 \longrightarrow \mathrm{H}^0\big(X, \mathrm{T}X\big) \longrightarrow \mathrm{H}^0\big(X, \mathrm{T}X(D)\big) \longrightarrow \mathrm{H}^0 \bigl(D, \mathrm{T}X(D)_{| D}\bigr) \longrightarrow \mathrm{H}^1\big(X, \mathrm{T}X\big) \longrightarrow 0
\end{equation}
which is the long cohomology sequence associated with \eqref{cute}. Thus, if $X$ has no nonzero holomorphic vector field, 
\[
\mathrm{H}^1\big(X, \mathrm{T}X\big) \simeq \frac{\mathrm{H}^0 \bigl(D, \mathrm{T}X(D)_{| D}\bigr)}{\mathrm{H}^0\big(X, \mathrm{T}X(D)\big)} \cdot
\]
\par \medskip
Let us now explain how to construct these divisors on rational surfaces. We fix two complex surfaces $X$ and $Y$ such that $Y$ is the blowup of $X$ at a point $p$. Let $\pi$ be the blowup map, and let $E$ be the exceptional divisor. There is a natural morphism of sheaves $\mathrm{T}Y \longrightarrow \mathrm{\pi^*} \mathrm{T}X$ given by the differential of $\pi$, which is injective.
By the same argument as the proof of the exactness of the sequence \eqref{exact1}, the quotient sheaf $\mathrm{\pi^*} \mathrm{T}X / \mathrm{T}Y$ is canonically isomorphic to $\iota_{E*} \mathrm{T}_E(-1)$ so that we have an exact sequence
\begin{equation} \label{crucial}
0 \longrightarrow \mathrm{T}Y \longrightarrow \mathrm{\pi^*} \mathrm{T}X \longrightarrow \iota_{E*} \mathrm{T}_E(-1) \longrightarrow 0.
\end{equation}

For any locally free sheaf $\mathcal{F}$ on $X$, the natural pullback morphism $\mathrm{H}^i\big(X, \mathcal{F}\big) \longrightarrow \mathrm{H}^i\big(Y, \pi^* \mathcal{F}\big)$ is an isomorphism; this follows by writing down the Leray spectral sequence for the sheaf $\pi^* \mathcal{F}$ and using that $\mathrm{R}^j \pi_*\mathcal{O}_Y=0$ for $j>0$ (\emph{see} Lemma~\ref{technique} {\it (i)} where we already used this argument). 
Therefore, for every divisor $D$ on $X$, we have a long exact sequence
\begin{small}
\[
\xymatrix@C=15pt{
0 \ar[r] & \mathrm{H}^0\big(Y, \mathrm{T}Y(\pi^* D)\big) \ar[r] & \mathrm{H}^0\big(X, \mathrm{T}X (D)\big) \ar[r]^-{\mathrm{ev.}}_-{\mathrm{at}\, p} & \mathrm{T_p X} \otimes L_p \ar[r]^-{\delta_D} &\mathrm{H}^1\big(Y, \mathrm{T}Y(\pi^* D)\big) \ar[r] & \mathrm{H}^1\big(X, \mathrm{T}X (D)\big)\ar[r] & 0}
\]
\end{small}
where $\delta_D$ is the connection morphism, and $L_p$ is the fiber of $\mathcal{O}_X(D)$ at $p$.
\begin{pro}\label{ex} 
Let $D$ be an effective divisor on $X$.
\begin{enumerate}
\item[(i)] There is a natural isomorphism $\mathrm{H}^0(X, \mathrm{T}X(D)) \xrightarrow{\sim} \mathrm{H}^0(Y, \mathrm{T}Y(\pi^*D+E)).$
\item[(ii)] If $D$ is $1$-exceptional on $X$, then $\pi^* D+E$ is $1$-exceptional on $Y$.
\end{enumerate}
\end{pro}

\begin{proof} 
We have a morphism of short exact sequences
\[
\xymatrix@R=15pt
@C=15pt{0 \ar[r]& \mathrm{T}Y \big(\pi^*D\big)\ar[r] \ar[d] &\mathrm{\pi^*} \mathrm{T}X\big(\pi^*D\big) \ar[r] \ar[d]& \iota_{E*} \mathrm{T}_E(-1) \otimes \mathcal{O}_Y\big(\pi^*D\big) \ar[r] \ar[d] &0 \\
0 \ar[r]& \mathrm{T}Y\big(\pi^*D+E\big) \ar[r] &\mathrm{\pi^*} \mathrm{T}X\big(\pi^*D+E\big)\ar[r] & \iota_{E*} \mathrm{T}_E(-1) \otimes \mathcal{O}_Y\big(\pi^*D+E\big)  \ar[r] &0\\
}
\]
where the right vertical arrow vanishes. It follows immediately that the morphism of sheaves $\pi^* \mathrm{T}X(\pi^*D) \rightarrow \pi^* \mathrm{T}X(\pi^*D+E)$ factors through $\mathrm{T}Y(\pi^*D+E)$. This gives (i). 
\par\medskip
For (ii), let us start by proving that the composition of the connection morphism $\delta_D$ with the natural morphism $\mathrm{H}^1\big(Y, \mathrm{T}Y (\pi^*D)\big) \longrightarrow \mathrm{H}^1\big(Y, \mathrm{T}Y(\pi^* D+E)\big)$ vanishes. We have a commutative diagram
\[
\xymatrix@C=15pt @R=15pt{ \mathrm{H}^0\big(\iota_{E*} \mathrm{T}_E(-1) \otimes \mathcal{O}_Y(\pi^*D) \big) \ar[r]^-{\delta_D} \ar[d] &\mathrm{H}^1\big(Y, \mathrm{T}Y(\pi^* D)\big) \ar[d] \\
\mathrm{H}^0\big(\iota_{E*} \mathrm{T}_E(-1) \otimes \mathcal{O}_Y(\pi^*D+E) \big) \ar[r]& \mathrm{H}^1\big(Y, \mathrm{T}Y(\pi^*D +E)\big)\\
}
\]
\noindent and we get our claim since the left vertical arrow vanishes. We now consider the diagram
\[
\xymatrix@R=15pt
@C=15pt{&\mathrm{H}^1\big(Y, \mathrm{T}Y\big) \ar[r] \ar[d] ^{u}& \mathrm{H}^1\big(X, \mathrm{T}X\big) \ar[d] ^{v}\\
\mathrm{T}_p X \otimes L_p \ar[r]^-{\delta_D} &\mathrm{H}^1\big(Y, \mathrm{T}Y(\pi^* D)\big) \ar[r] & \mathrm{H}^1\big(X, \mathrm{T}X(D)\big)}
\]
where the bottom line is exact. Since $D$ is $1$-exceptional, $v=0$. This implies that the image of $u$ lies in the image of~$\delta_D$.
\end{proof}

\begin{rem}
The first point of the proposition can be rephrased as follows:  let $Z$ be a section of $\mathrm{T}X(D)$. Then $Z_{
| X \setminus \{p\}}$ can be considered as a section of $\mathrm{T}Y(\pi^* D)$ on $Y \setminus E$. This section extends uniquely to a section of $\mathrm{T}Y(\pi^*D+E)$.
\end{rem}

Let us introduce some extra notations. For any effective divisor $D$ on a surface $X$, we put:
\[
\begin{cases}
V(D)= {\mathrm{H}^0\big(X, \mathrm{T}X(D)\big)} \\
W(D)= \mathrm{H}^0\big(X, \mathrm{T}X(D)_{| D}\big).
\end{cases}
\]
We denote by $\mathfrak{h}(\mathbb{P}^2)$ the set of holomorphic vector fields on $\mathbb{P}^2$. A (possibly infinitely near) point $\widehat{P}$ in $\mathbb{P}^2$ is a point $P$ in $\mathbb{P}^2$ together with a sequence 
\[
X_N \xrightarrow{\pi_N} X_{N-1} \xrightarrow{\pi_{N-1}} \ldots \xrightarrow{\pi_2} X_1 \xrightarrow{\pi_1} X_0=\mathbb{P}^2
\]
where the $\pi_i$'s are point blowups and $\pi_1$ is the blowup of $P$. The surface $X_N$ is called the blowup of $\mathbb{P}^2$ along $\widehat{P}$ and denoted by $\mathrm{Bl}_{\widehat{P}} \mathbb{P}^2$. If $N=1$ the point is simple, if $N \geq 2$ it is infinitely near. The integer $N$ is called the length of $\widehat{P}$.

\begin{defi} \label{penible}
For any (possibly infinitely near) point $\widehat{P}$ in $\mathbb{P}^2$, and any effective divisor $D_{\mathrm{base}}$ on $\mathbb{P}^2$, we define a divisor~$D_{\widehat{P}, \,D_{\mathrm{base}}}$ on $\mathrm{Bl}_{\widehat{P}}\, \mathbb{P}^2$ as follows: 
\begin{enumerate}
\item[--] If $(\pi_1$, $\ldots$, $\pi_N)$ is the sequence of blowups defining $\widehat{P}$, let $E_1$, $\ldots$, $E_N$ be the corresponding exceptional divisors on $X_1$, $\ldots$, $X_N$. 
\item[--] Let $D_0, \ldots, D_N$ be $N+1$ divisors on $X_0, \ldots, X_N$ defined inductively by $D_{0}^{\vphantom{A^a}}=D_{\mathrm{base}}$ and for $1 \leq i \leq N$ $D_{i}=\pi_{i-1}^* D_{i-1}+E_i$.
\item[--]The divisor $D_{\widehat{P}, \,D_{\mathrm{base}}}$ is defined by $D_{\widehat{P}, \,D_{\mathrm{base}}}=D_N$. 
\item[--] If $D_{\mathrm{base}}$ is empty we put  $D_{\widehat{P},\varnothing}=D_{\widehat{P}}$.
\end{enumerate}
\end{defi}

Note that this definition extends readily to a finite number of (possibly infinitely near) points in $\mathbb{P}^2$.
\par \medskip
According to Proposition \ref{ex} the divisor $D_{\widehat{P}, \,D_{\mathrm{base}}}$ is always $1$-exceptional on $\mathrm{Bl}_{\widehat{P}}\, \mathbb{P}^2$ since $D_0$ is $1$-exceptional and \{$D_i$ is $1$-exceptional\} $\Rightarrow$ \{$D_{i+1}$ is $1$-exceptional\} . 

\begin{lem}\label{W} 
Let $\widehat{P}$ be an infinitely near point in $\mathbb{P}^2$ of length $N$. Then
\begin{enumerate}
\item[(i)] There is a natural isomorphism $\mathfrak{h}(\mathbb{P}^2) \xrightarrow{\sim} V(D_{\widehat{P}})$.
\item[(ii)] $\dim W(D_{\widehat{P}})=2N$.
\end{enumerate}
\end{lem} 

\begin{proof} 
Using the notation of Definition \ref{penible}, Proposition \ref{ex} (i) yields isomorphisms 
\[
\mathfrak{h}(\mathbb{P}^2) \simeq V(D_1) \simeq V(D_N)=V(D_{\widehat{P}})
\]
so that using the exact sequence \ref{div} we have 
\[
\mathrm{h}^1\big(X, \mathrm{T}X\big)=\dim W(D_{\widehat{P}})-\Big(8-\mathrm{h}^0\big(X, \mathrm{T}X\big)\Big).
\]
Since $\mathrm{h}^1\big(X, \mathrm{T}X\big)-\mathrm{h}^0\big(X, \mathrm{T}X\big)=2N-8$, we get the result. 
\end{proof}

\begin{pro} \label{latotale}
Let $X$ be a rational surface without nonzero holomorphic vector field obtained by blowing $\mathbb{P}^2$ in $k$ (possibly infinitely near) points ${\widehat{P}_1} , \ldots, {\widehat{P}_k}$. Let $D_{\mathrm{base}}$ be an effective divisor on $\mathbb{P}^2$, and let $V^{\dag}(D_{\mathrm{base}})$ be a direct factor of~$\mathfrak{h}(\mathbb{P}^2)$ in $V(D_{\mathrm{base}})$. We denote by $\widehat{D}$ the $1$-exceptional divisor $D_{\widehat{P}_1 \cup \ldots \cup \widehat{P}_k,\, D_{\mathrm{base}}}$. Then there is a natural isomorphism between~$V(D_{\mathrm{base}})$ and $V(\widehat{D})$, and the associated morphism
\[
\bigoplus_{i=1}^k W(D_{\widehat{P}_i}) \oplus V^{\dag}(D_{\mathrm{base}})  \longrightarrow W(\widehat{D})
\]
is an isomorphism. 
\end{pro}

\begin{proof} For the first isomorphism, we argue exactly as in the proof of Lemma \ref{W} (i) using Proposition \ref{ex} (i). For the second isomorphism, set~$K=D_{{\widehat{P}}_1}+ \ldots + D_{{\widehat{P}}_k}$.
Let us write down the exact sequence (\ref{div}) for the divisors $K$ and $\widehat{D}$. Since $V(K)^{\vphantom{A^A}} \simeq\mathfrak{h}(\mathbb{P}^2)$, we get a commutative diagram
\[
\xymatrix{0 \ar[r]&\mathfrak{h}(\mathbb{P}^2) \ar[r] \ar[d]& W(K) \ar[r] \ar[d] & \mathrm{H}^1\big(X, \mathrm{T}X\big) \ar[r] \ar@{=}[d] & 0 \\
0 \ar[r] & V(D_{\mathrm{base}}) \ar[r]& W(\widehat{D}) \ar[r] & \mathrm{H}^1\big(X, \mathrm{T}X\big) \ar[r] & 0
}
\]
As $V(D_{\mathrm{base}})=\mathfrak{h}(\mathbb{P}^2) \oplus V^{\dag}(D_{\mathrm{base}})$, we obtain that $W(\widehat{D})$ is isomorphic to $W(K) \oplus V^{\dag}(D_{\mathrm{base}})$.
\end{proof}

We end the section with a statement which has no theoretical interest but which is particularly useful in practical computations:

\begin{lem}\label{reductiondespoles}
Let $D$, $D'$ be two divisors on a rational surface $X$ such that:
\[
\begin{cases}
D' \, \, \textrm{is}\, \,  1-\textrm{exceptional}. \\
\textrm{The natural map from} \, \, W(D) \, \, \mathrm{to} \, \, W(D') \, \, \textrm{is surjective}.
\end{cases}
\]
\par \smallskip
\noindent Then $D$ is $1$-exceptional.
\end{lem}

\begin{proof}
We have a commutative diagram
\[
\xymatrix{W(D) \ar[r] \ar[d] & \mathrm{H}^1\big(X, \mathrm{T}X\big) \ar@{=}[d] \\
W({D'}) \ar[r] & \mathrm{H}^1\big(X, \mathrm{T}X\big) 
}
\]
where the bottom horizontal arrow is onto. Hence the top horizontal arrow is also onto.
\end{proof}

\subsection{Geometric bases}\label{Subsec:geobasis}

In this section, we construct a basis for the vector space $W(\widehat{D})$ using coverings (the divisor~$\widehat{D}$ has been introduced in Proposition \ref{latotale}), we will call such a basis a \textit{geometric basis}. Since $\mathrm{T}X(D)_{| D}$ is the sheaf quotient~$\frac{\mathrm{T}X(D)}{\mathrm{T}X}$, a global section can be represented as a family of sections $(Z_i)_{i \in I}$ of $\mathrm{T}X(D)$ on open sets $U_i$ covering $D$ such that $\vphantom{\Bigl(}Z_i-Z_j$ is holomorphic on $U_{ij}$ for all $i$, $j$ in $I$.
\par \medskip

We start with the simplest case, namely $D_{\mathrm{base}}=\varnothing$ and $k=1$, so that $\widehat{D}=D_{\widehat{P}}$. Let us construct this basis by induction on the length $N$ of $\widehat{P}$. 
\par \medskip

Let $\widehat{P}$ be an infinitely near point of $\mathbb{P}^2$ of length $N$, let $q$ be a point in one of the exceptional divisors, let $\widehat{P}\,' = \widehat{P} \cup \{ q \}$, and put $X=\mathrm{Bl}_{\widehat{P}}\, \mathbb{P}^2$ and $Y=\mathrm{Bl}_{\widehat{P}\,'}\, \mathbb{P}^2$. Assume that we are given a covering of $D_{\widehat{P}}$ by open sets $(U_i)_{i \in I}$ of $X$ such that :
\smallskip
\begin{itemize}
\item[\ding{172}] there exists a unique $i_0$ in $I$ such that $q \in U_{i_0}$;
\par \smallskip
\item[\ding{173}] the evaluation map $\mathrm{H}^0\big(U_{i_0}, \mathrm{T}U_{i_0}\big) \longrightarrow \mathrm{T}_{q} X$ is surjective;
\smallskip
\item[\ding{174}] there exists a basis $Z_1, \ldots, Z_{2N}$ of $W(D_{\widehat{P}})$, where for each $\alpha$, $Z_{\alpha}$ is a collection of holomorphic sections  $(Z_{i \alpha})_{i \in I}$ of~$\mathrm{T}X(D_{\widehat{P}})$ on the $U_i$'s such that for all $i$ and $j$, the section $Z_{i \alpha}-Z_{j \alpha}$ is holomorphic on each $U_{i \alpha} \cap U_{j \alpha}$.
\end{itemize}
\bigskip

Let $\pi\colon Y \rightarrow  X$ be the blowup map, and let $E$ be the exceptional divisor. Set $U'_i=\pi^{-1} (U_i)$. For any basis $(v_1, v_2)$ of $\mathrm{T}_q \, (\mathrm{Bl}_{\widehat{P}}\, \mathbb{P}^2)$, thanks to \ding{173}, we can choose two holomorphic vector fields $T_1, T_2$ on $U_{i_0}$ which extend $(v_1, v_2)$. We will consider the vector fields $T_1$ and $T_2$ as sections of $\mathrm{T}Y(E)$ on $U'_{i_0}$. Let us now consider the sections $(Z'_{1}, \ldots, Z'_{2N+2})$ of~$W\big(D_{\widehat{P}'}\big)$ defined using the covering $(U'_i)_{i \in I}$ as follows :
\medskip
\begin{center}
\begin{tabular}{c | c c c | c |c}
&$\vphantom{a}^{\vphantom{\bigl(}}_{\vphantom{\bigl(}} Z'_1$& $\ldots\ldots$&$Z'_{2N}$&$Z'_{2N+1}$&$Z'_{2N+2}$\\
\hline
$\vphantom{a}^{\vphantom{\bigl(}}_{\vphantom{\bigl(}} U'_{i_0}$ &$Z_{1 i_0}$ & $\ldots\ldots$ &$Z_{2N i_0}$& $T_1$ & $T_2$\\
\hline
$\vphantom{a}^{\vphantom{\bigl(}}_{\vphantom{\bigl(}} U'_j \quad j \neq i_0$ & $Z_{1j}$ & $\ldots\ldots$ & $Z_{2Nj}$&$0$&$0$\\
\end{tabular}
\end{center}
\medskip
\par
Then we have the following result:

\begin{pro} \label{champs}
The family $(Z'_{1}, \ldots, Z'_{2N+2})$ is a basis of $W\big(D_{\widehat{P}\, '}\big)$.
\end{pro}

\begin{proof} 
Thanks to Lemma \ref{W}, it suffices to show that this family is free. Let $\lambda_1, \ldots, \lambda_{2N+2}$  be complex numbers such that $\lambda_1 Z'_1 + \lambda_2 Z'_2 + \ldots + \lambda_{2N+2} Z'_{2N+2}=0$. Then for all $j \neq i_0$, $$\lambda_1 Z'_{1j} + \lambda_2 Z'_{2j} + \ldots + \lambda_{2N} Z'_{2N j}$$ is holomorphic on~$U'_j$, and~$\lambda_1 Z'_{1i_0} + \lambda_2 Z'_{2i_0} + \ldots + \lambda_{2N} Z'_{2N i_0}+ \lambda_{2N+1} T_1 + \lambda_{2N+2} T_2$ is holomorphic on $U'_{i_0}$. 
\par \medskip
Therefore, $\lambda_1 Z'_{1j} + \lambda_2 Z'_{2j} + \ldots + \lambda_{2N} Z'_{2N j}$ is holomorphic on $U'_j$ also for $j=i_0$ and $\lambda_{2N+1} T_1 + \lambda_{2N+2} T_2$ (considered as a vector field on $U_{i_0}$) vanishes at $q$. Since $(Z_1, \ldots, Z_{2N})$ is a basis of $W(D_{\widehat{P}})$ we get $\lambda_1= \ldots = \lambda_{2N}=0$; and as $(v_1, v_2)$ is a basis of $\mathrm{T}_q X$, $\lambda_{2N+1}=\lambda_{2N+2}=0$. 
\end{proof}

\begin{rems} $ $
\begin{itemize}
\item[(i)] In practical computations, $U_{i_0}$ is a domain for some holomorphic coordinates $(z, w)$, $T_1=\partial_z$ and $T_2=\partial_w$. 
\item[(ii)] If $q'$ is a point on $\widehat{P}'$, the covering $(U'_i)_{i \in I}$ does not satisfy in general conditions \ding{173} and \ding{174}. Therefore it is necessary to refine the covering at each blowup.
\end{itemize}
\end{rems}

\par
Using Proposition \ref{latotale}, we can now easily construct a basis of $W\big(\widehat{D}\big)$. Indeed, we have already explained how to construct a basis of each $W\big(D_{\widehat{P}_k}\big)$. Then it suffices to take a basis of $V(D_{\mathrm{base}})^{\dag}$ and to pull back these meromorphic vector fields to the surface $X$. 

\subsection{Algebraic bases}\label{Subsec:algbasis}

In practical computations, if an element of $W(\widehat{D})$ is given by local meromorphic vector fields on open sets of a covering, it is not a priori obvious to decompose this element in a geometric basis (as constructed in~\S\ref{Subsec:geobasis}). In this section, we will construct another basis (we call it an algebraic basis) which solves this problem.
\par \medskip
We start by defining the residue morphisms. Let $\widehat{P}=\{p_1, \ldots, p_N\}$, and let $\overline{E}_1, \ldots, \overline{E}_N$ be the corresponding strict transforms of the exceptional divisors. We fix holomorphic coordinates $(x_i, y_i)$ near $p_i$ such that $p_i=(0,0)$ in these coordinates, and we introduce new holomorphic coordinates $(u_i, v_i)$ by putting $x_i=u_i$ and $y_i=u_i v_i$. We can  consider $(u_i, v_i)$ as holomorphic coordinates on an open subset of $X$ which contains $\overline{E}_i$ except a finite number of points. 
\par \medskip
Let  $(0, \lambda_i)$ be a generic point of $\overline{E}_i$, let $Z$ be an element of $W\big(D_{\widehat{P}}\big)$, and let $Z_i$ be a section of $\mathrm{T}X(D)$ which lifts $Z$ near this point. Then $Z_i$ admits a Laurent expansion
\[
Z_i(u_i, v_i)=\sum_{n=1}^{n_0} \sum_{m=0}^{\infty} \frac{ (v_i-\lambda_i)^m (a_{nm}\, \partial_u+b_{nm}\, \partial_v)}{u^n}\, + \{\mathrm{holomorphic \,\, terms}\}.
\]

\begin{defi}
For any generic complex number $\lambda_i$, we define the \textit{$i^{\mathrm{\,th}}$ residue morphism} 
\[
\mathrm{res}_{E_i} \colon W(D_{\widehat{P}}) \longrightarrow \mathbb{C}^2
\] 
by the formula $\mathrm{res}_{E_i} (Z)=(b_{10}, b_{11})$. We also define $\mathrm{res}_{\widehat{P}} \colon W(D_{\widehat{P}}) \longrightarrow \mathbb{C}^{2N}$ by $\mathrm{res}_{\widehat{P}}=\bigoplus_{i=1}^N \mathrm{res}_{E_i}$.
\end{defi}
\begin{rem}
The definition of these residues depends on the coordinates and on the generic parameters on the exceptional divisors.
\end{rem}
\begin{lem}
The morphism $\mathrm{res}_{\widehat{P}}$ is an isomorphism.
\end{lem}

\begin{proof}
Let us take a geometric basis associated to the coordinates $(x_i, y_i)_{1 \leq i \leq N}$ as constructed in \S\ref{Subsec:geobasis}. Then the matrix of $\mathrm{res}_{\widehat{P}}$ is lower triangular by blocks, where the diagonal blocks are the $2 \times 2$ matrices whose columns are the vectors $\mathrm{res}_{E_i}(\partial_{x_i})$ and $\mathrm{res}_{E_i}(\partial_{y_i})$. 
We have $\partial_{x_i}=\partial_{u_i}-\frac{v_i}{u_i}\partial_{v_i}$ and $\partial_{y_i}=\frac{1}{u_i} \partial_{v_i}$. Therefore we obtain $\mathrm{res}_{E_i}(\partial_{x_i})=(-1, -\lambda_i)$ and $\mathrm{res}_{E_i}(\partial_{y_i})=(1,0)$, so that the diagonal blocks are all invertible.
\end{proof}

\par \bigskip
We can now define algebraic basis in the general case. We take the notation of Proposition \ref{latotale}. In order to avoid cumbersome notation, we will assume for simplicity that $D_{\mathrm{base}}$ is irreducible, and leave the general case to the reader. 
\par \bigskip
Let $\chi_1$, $\ldots$, $\chi_r$ be a basis of $V(D_{\mathrm{base}})^{\dag}$.

\begin{lem}\label{base}
Assume that $D_{\mathrm{base}}$ is irreducible, and let $Z$ be an element of $W(\widehat{D})$. Then there exist unique complex numbers $\alpha_1(Z)$, $\ldots$, $\alpha_r(Z)$ such that for any generic point $\xi$ of ${D}_{\mathrm{base}}$ and any lift of $\widetilde{Z}$ of $Z$ near $\xi$ \emph{(}with respect to the sheaf morphism $\mathrm{T}X(D) \rightarrow \mathrm{T}X(D)_{|D}$\emph{)}, 
\[
\widetilde{Z}=\alpha_1 (Z)\,\, \chi_1+ \ldots + \alpha_m(Z) \,\, \chi_r + \, \mathrm{holomorphic \, \,terms}
\]
in a neighborhood of $\xi$ .
\end{lem}

\begin{proof}
The existence of such a decomposition is straightforward: the $\alpha_i(Z)$'s are the coefficients of $Z$ in $V(D_{\mathrm{base}})^{\dag}$ when decomposed in a geometric basis. For the unicity, we remark that $\alpha_1 \,\, \chi_1+ \ldots + \alpha_r \,\, \chi_r$ is holomorphic near a point of $D_{\mathrm{base}}$, it must be holomorphic on $\mathbb{P}^2$ since $D_{\mathrm{base}}$ is irreducible. Since $V(D_{\mathrm{base}})^{\dag}$ is a direct factor of $\mathfrak{h}(\mathbb{P}^2)$, all the coefficients $\alpha_i$ must vanish.
\end{proof}

\begin{cor}\label{basealg} 
With the notations of Proposition \ref{latotale}, the map
\[
\big(\mathrm{res}_{\widehat{P}_1}, \ldots, \mathrm{res}_{\widehat{P}_k}, \alpha_1, \ldots, \alpha_r\big) \colon W(\widehat{D}) \longrightarrow \mathbb{C}^{2N_1+ \ldots + 2N_k + r}
\]
is an isomorphism.
\end{cor}

By definition, the associated algebraic basis of $W(\widehat{D})$ is the image of the canonical basis of $\mathbb{C}^{2N_1+ \ldots + 2N_k + r}$ by the inverse isomorphism.

\section{An explicit example on $\mathbb{P}^2$ blown up in 15 points}\label{Sec:explicit}

\subsection{The strategy}\label{Subsec:strategy}

Let $\phi \colon \mathbb{P}^2 \dashrightarrow \mathbb{P}^2$ be the birational map given by 
\[
(x:y:z) \rightarrow (xz^2+y^3:yz^2:z^3).
\] 
The map $\phi$ induces an isomorphism (which we still denote by $\phi$) between $\mathrm{Bl}_{\widehat{P}_1}\,\mathbb{P}^2$ and $\mathrm{Bl}_{\widehat{P}_2}\,\mathbb{P}^2$, where $\widehat{P}_1$ and $\widehat{P}_2$ are infinitely near points of~$\mathbb{P}^2$ of length $5$ centered at~$(1: 0: 0)$ described in \cite[\S 2.1]{DesertiGrivaux}. For any complex number $\alpha$, we put \begin{align*}
&A=\left(\begin{array}{ccc}\alpha & 2(1-\alpha) & 2+\alpha-\alpha^2\\
-1 & 0 & \alpha+1 \\
1 & -2 & 1-\alpha
\end{array} \right)
\end{align*}
and we consider it as an element of $\mathrm{PGL}(3; \mathbb{C})$. Then the map $A \phi$ lifts to an automorphism $\psi$ of the rational surface $X$ obtained by blowing up the projective plane $15$ times at $\widehat{P}_1$, $A \widehat{P}_2$ and $A \phi A \widehat{P}_2$ (\emph{see} \cite[\S 3.3]{DesertiGrivaux}). Recall that the parameter $\alpha$ is not really interesting because two different values of $\alpha$ correspond to linearly conjugate automorphisms, however we keep it in order to check some calculations in the sequel. Our aim is to compute the map $\psi_{*}$ acting on $\mathrm{H}^1(X, \mathrm{T}X)$. 
\par \medskip
Let us recall the construction of $\widehat{P}_1$ $\big($resp. $\widehat{P}_2$$\big)$. In affine coordinates $(y,z)$, we blow up the point $(0,0)_{y,z}$ in $\mathbb{P}^2$ and we denote by $E$ the exceptional divisor. We put $y=u_1$, $z=u_1v_1$, we blow up $(0,0)_{u_1, v_1}$ and $F$ is the exceptional divisor. Then we put $u_1=r_2s_2$, $v_1=s_2$, we blow up $(0,0)_{r_2, s_2}$ and we denote by $G$ the exceptional divisor. Next, we put $r_2=r_3s_3$, $s_2=s_3$, we blow up $(-1,0)_{r_3,s_3}$ $\big($resp. $(1,0)_{r_3, s_3}$$\big)$ and $H$ (resp. $K$) is the exceptional divisor. Then, we put $r_3=r_4s_4-1$, $s_3=s_4$ $\big($resp. $r_3=c_4d_4+1$, $s_3=d_4$$\big)$, we blow up $(0,0)_{r_4, s_4}$ $\big($resp. $(0,0)_{c_4, d_4}$$\big)$ and we denote by $L$ (resp. $M$) the last exceptional divisor. Lastly, we put $r_4=r_5s_5$, $s_4=s_5$ $\big($resp. $c_4=c_5d_5$, $d_4=d_5$$\big)$.  
\par \medskip
Set  $X_1=\mathrm{Bl}_{\widehat{P}_1}\,\mathbb{P}^2$ and $X_2=\mathrm{Bl}_{\widehat{P}_2}\,\mathbb{P}^2$. Let $\overline{\Delta}$ be the strict transform of the line $\Delta=\{z=0\}$ in $X_2$. Then
\begin{equation}\label{image}
{\phi}(E)=E, \quad {\phi}(F)=K, \quad {\phi}(G)=G, \quad {\phi}(H)=F, \quad {\phi}(L)=\overline{\Delta}\,.
\end{equation}
We have $D_{\widehat{P}_1}=E+2F+4G+5H+6L$ and $D_{\widehat{P}_2}=E+2F+4G+5K+6M$. Let $D_1$ be the divisor on $X_1$ given by 
\[
D_1=E+2F+3G+4H+5L.
\]
By explicit calculation, the natural morphism from $W(D_1)$ to $W\big(D_{\widehat{P}_1}\big)$ is surjective so that $D_1$ is $1$-exceptional on $X_1$ by Lemma \ref{reductiondespoles}. We now define a divisor $D_2$ on $X_2$ as follows (\emph{see} Definition \ref{penible}):
\[
D_2=D_{\widehat{P}_2,\,5\overline{\Delta}}.
\]
By (\ref{image}), we have ${\phi}_* D_1 \leq D_2$. Let $\mathfrak{D}_1$ and $\mathfrak{D}_2$ be the two $1$-exceptional divisors on $X$ given by:
\[
\mathfrak{D}_1=D_1+AD_{\widehat{P}_2}+A \phi A \,D_{\widehat{P}_2}, \qquad \qquad
\mathfrak{D}_2=D_{\widehat{P}_1}+AD_2+A \phi A \,{D_{\widehat{P}_2}}.
\]
Then $\mathfrak{D}_1 \leq \mathfrak{D}_2$ and $f_* \mathfrak{D}_1 \leq \mathfrak{D}_2$. Therefore the morphism $\psi_*$ acting on $\mathrm{H}^1(X, \mathrm{T}X)$ 
can be obtained as the composition
\begin{equation} \label{comp}
\xymatrix{
\displaystyle\frac{W(\mathfrak{D}_1)}{V(\mathfrak{D}_1)} \ar[r]^-{\psi_*}& \displaystyle\frac{W(\mathfrak{D}_2)}{V(\mathfrak{D}_2)} \ar[r]^-{\sim}& \displaystyle\frac{W(\mathfrak{D}_1)}{V(\mathfrak{D}_1)} 
}
\end{equation}
where the inverse of last arrow is induced by the natural morphism from $W(\mathfrak{D}_1)$ to $W(\mathfrak{D}_2)$. 
To compute $\psi_*$, the strategy runs as follows:
\par \medskip
\noindent \underline{\textit{Calculations for the pair $(D_1, D_2)$}}.
\par \smallskip
\begin{itemize}
\item[\textit{Step 1} --] Express the vectors of a geometric basis of $W(D_1)$ in an algebraic basis.
\item[\textit{Step 2} --] Compute $\phi_* \colon W(D_1) \longrightarrow W(D_2)$ where $W(D_1)$ is endowed with a geometric basis and $W(D_2)$ is endowed with an algebraic basis.
\item[\textit{Step 3} --] Find bases of $V(D_1)$ and $V(D_2)$ whose vectors are expressed in algebraic bases of $W(D_1)$ and $W(D_2)$ respectively.
\end{itemize}
\par \medskip
\noindent \underline{\textit{Calculations for the pair $({\mathfrak{D}}_1, {\mathfrak{D}}_2)$}}.
\par \smallskip
\begin{itemize}
\item[\textit{Step 1} --] Compute $\psi_* \colon W({\mathfrak{D}}_1) \longrightarrow W({\mathfrak{D}}_2)$ where $W({\mathfrak{D}}_1)$ and $W({\mathfrak{D}}_2)$ are endowed with algebraic bases.
\item[\textit{Step 2} --] Find bases of $V({\mathfrak{D}}_1) $ and $V({\mathfrak{D}}_2)$ whose vectors are expressed in algebraic bases of $W({\mathfrak{D}}_1)$ and $W({\mathfrak{D}}_2) $ respectively.
\end{itemize}

\subsection{Calculations for the pair $(D_1, D_2)$}\label{Subsec:comp1} 

We give results of the calculations for the first three steps listed above.
\par \smallskip
\noindent {\textit{Step 1} --}  We use the coordinates $(y,z)$, $(u_1, v_1)$, $(r_2, s_2)$, $(r_3, s_3)$, $(r_4, s_4)$ and $(r_5, s_5)$ to compute the algebraic and geometric bases on $X_1$. Let $(U_i)_{1 \leq i\leq 5}$ be the covering of $D_1$ given by the following picture:
\par \medskip
\begin{center}
\includegraphics{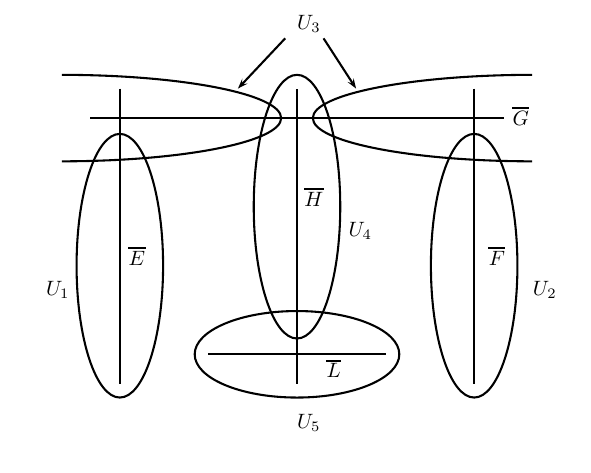}
\end{center}
\par \medskip
Then a geometric basis of $W(D_1)$ is given by the ten following vectors:
\smallskip
\begin{center}
\begin{tabular}{c|c|c|c|c|c}
&$U_1$&$U_2$&$U_3$&$U_4$&$U_5$\\
\hline
$\mathfrak{e}_1$&$\partial{y}$&$\partial{y}$&$\partial{y}$&$\partial{y}$&$\partial{y}$ \\
\hline
$\mathfrak{e}_2$&$\partial{z}$&$\partial{z}$&$\partial{z}$&$\partial{z}$&$\partial{z}$ \\
\hline
$\mathfrak{e}_3$&$0$&$\partial{u_1}$&$\partial{u_1}$&$\partial{u_1}$&$\partial{u_1}$ \\
\hline
$\mathfrak{e}_4$&$0$&$\partial{v_1}$&$\partial{v_1}$&$\partial{v_1}$&$\partial{v_1}$ \\
\hline
$\mathfrak{e}_5$&$0$&$0$&$\partial{r_2}$&$\partial{r_2}$&$\partial{r_2}$ \\
\hline
$\mathfrak{e}_6$&$0$&$0$&$\partial{s_2}$&$\partial{s_2}$&$\partial{s_2}$ \\
\hline
$\mathfrak{e}_7$&$0$&$0$&$0$&$\partial{r_3}$&$\partial{r_3}$ \\
\hline
$\mathfrak{e}_8$&$0$&$0$&$0$&$\partial{s_3}$&$\partial{s_3}$ \\
\hline
$\mathfrak{e}_9$&$0$&$0$&$0$&$0$&$\partial{r_4}$ \\
\hline
$\mathfrak{e}_{10}$&$0$&$0$&$0$&$0$&$\partial{s_4}$ \\
\end{tabular}
\par\medskip -- \textit{Matrix $\mathscr{P}$} --
\end{center}
\par \medskip
If $\lambda_1$, $\lambda_2$, $\lambda_3$, $\lambda_4$ and $\lambda_5$ are generic parameters on $E$, $F$, $G$, $H$ and $L$ respectively and if $(\mathfrak{l}_i)_{1 \leq i \leq 10}$ is the associated algebraic basis, a direct calculation yields:
\par \medskip
\begin{center}
\begin{tabular}{c|c|c|c|c|c|c|c|c|c|c}
&$\mathfrak{e}_1$&$\mathfrak{e}_2$&$\mathfrak{e}_3$&$\mathfrak{e}_4$&$\mathfrak{e}_5$&$\mathfrak{e}_6$&$\mathfrak{e}_7$&$\mathfrak{e}_8$&$\mathfrak{e}_9$&$\mathfrak{e}_{10}$\\
\hline
$\mathfrak{l}_1$&$-\lambda_1$&$1$&$0$&$0$&$0$&$0$&$0$&$0$&$0$&$0$\\
\hline
$\mathfrak{l}_2$&$-1$&$0$&$0$&$0$&$0$&$0$&$0$&$0$&$0$&$0$\\
\hline
$\mathfrak{l}_3$&$2$&$0$&$1$&$-\lambda_2$&$0$&$0$&$0$&$0$&$0$&$0$\\
\hline
$\mathfrak{l}_4$&$0$&$0$&$0$&$-1$&$0$&$0$&$0$&$0$&$0$&$0$\\
\hline
$\mathfrak{l}_5$&$0$&$0$&$0$&$-2\lambda_3$&$1$&$-\lambda_3$&$0$&$0$&$0$&$0$\\
\hline
$\mathfrak{l}_6$&$0$&$0$&$0$&$-2$&$0$&$-1$&$0$&$0$&$0$&$0$\\
\hline
$\mathfrak{l}_7$&$-\lambda_4^2$&$\lambda_4^3$&$0$&$-3\lambda_4$&$0$&$-2\lambda_4$&$1$&$-\lambda_4$&$0$&$0$\\
\hline
$\mathfrak{l}_8$&$-2\lambda_4$&$3\lambda_4^2$&$0$&$-3$&$0$&$-2$&$0$&$-1$&$0$&$0$\\
\hline
$\mathfrak{l}_9$&$0$&$2\lambda_5^2$&$0$&$-4\lambda_5$&$0$&$-3\lambda_5$&$0$&$-2\lambda_5$&$1$&$-\lambda_5$\\
\hline
$\mathfrak{l}_{10}$&$0$&$4\lambda_5$&$0$&$-4$&$0$&$-3$&$0$&$-2$&$0$&$-1$\\
\end{tabular}
\par\medskip -- \textit{Matrix $\mathscr{K}$} --
\end{center}
\par \medskip

\noindent {\textit{Step 2} --}  We start by fixing a basis of $\mathrm{H}^0\big({\mathbb{P}^2, \mathrm{T} \mathbb{P}^2 (m\Delta)}\big)$. We divide 
$\mathrm{H}^0\big({\mathbb{P}^2, \mathrm{T} \mathbb{P}^2 (m\Delta)}\big)$ in four subspaces, the first one corresponding to holomorphic vector fields.
\par \bigskip
\noindent --Vector fields of {\textrm{type A}} that span a subspace of dimension $8$:
\[ a_1=\partial_y, \quad a_2=y\,\partial_y, \quad a_3=z\,\partial_y, \quad a_4=\partial_z, \quad a_5=y\,\partial_z, \quad a_6=z\,\partial_z, \quad a_7=yz\,\partial_z, \quad a_8=z^2\,\partial_z.
\]
\noindent --Vector fields of {\textrm{type B}} that span a subspace of dimension $\frac{m(m+5)}{2}$:
\[
b_{p,q}=\frac{y^p}{z^q}\,\partial_y \qquad 1 \leq q \leq m, \quad 0 \leq p \leq q+1.
\]
\noindent --Vector fields of {\textrm{type C}} that span a subspace of dimension $\frac{m(m+5)}{2}$:
\[
c_{p,q}=\frac{y^p}{z^q}\,\partial_z \qquad 1 \leq q \leq m, \quad 0 \leq p \leq q+1.
\]
\noindent --Vector fields of {\textrm{type D}} that span a subspace of dimension $m$:
\[
d_{p}=\frac{y^{p+2}}{z^p}\,\partial_y+\frac{y^{p+1}}{z^{p-1}}\,\partial_z \qquad 1 \leq p \leq m.
\]
\par \medskip
We take for $\mathrm{H}^0\big({\mathbb{P}^2, \mathrm{T} \mathbb{P}^2(m\Delta)}\big)^{\dag}$ the subspace spanned by meromorphic vector fields of type \textrm{B}, \textrm{C}, and \textrm{D} (in our example, $m=5$). Then we use the coordinates $(y,z)$, $(u_1, v_1)$, $(r_2, s_2)$, $(r_3, s_3)$, $(c_4, d_4)$ and $(c_5, d_5)$ on $X_2$. If $\mu_1$, $\mu_2$, $\mu_3$, $\mu_4$ and $\mu_5$ are the generic parameters on $E$, $F$, $G$, $K$ and $M$ respectively, we consider an algebraic basis of $W(D_2)$ associated with these parameters consisting of ten vectors $(\mathfrak{m}_i)_{1 \leq i \leq 10}$ corresponding to the exceptional divisors, and $55$ vectors ($25$ of type \textrm{B}, $25$ of type \textrm{C} and $5$ of type \textrm{D}) corresponding to a basis of $\mathrm{H}^0\big({\mathbb{P}^2, \mathrm{T} \mathbb{P}^2 (5\Delta)}\big)^{\dag}$.
\par \medskip
By direct computation, we have
\begin{small}
\begin{align*}
&\phi_* \partial_y=\left(1-\frac{4y^3}{z^2}+\frac{3y^6}{z^4}\right)\, \partial_y+\left(-\frac{3y^2}{z}+\frac{3y^5}{z^3} \right)\,\partial_z,&&\phi_* \partial_z=2\left(\frac{y^4}{z^3} -\frac{y^7}{z^5}\right)\, \partial_y+\left( 1+\frac{y^3}{z^2}-\frac{2y^6}{z^4}\right)\,\partial_z,\\
&\phi_* \partial_{u_1}=\left(1-2\frac{y^3}{z^2}+\frac{y^6}{z^4} \right)\, \partial_y+\left( \frac{z}{y}-2\frac{y^2}{z}+\frac{y^5}{z^3}\right)\,\partial_z,&&\phi_* \partial_{v_1}=2\frac{y^5}{z^3} \, \partial_y+\left( y+2\frac{y^4}{z^2}\right)\,\partial_z,\\
&\phi_* \partial_{r_2}=\left(\frac{z}{y}-2\frac{y^2}{z}+\frac{y^5}{z^3} \right)\, \partial_y+\left( \frac{z^2}{y^2}-2y+\frac{y^4}{z^2}\right)\,\partial_z,&&\phi_* \partial_{s_2}=\left(\frac{y^2}{z}+\frac{y^5}{z^3} \right)\, \partial_y+\left(2y+\frac{y^4}{z^2}\right)\,\partial_z,\\
&\phi_* \partial_{r_3}=\left(\frac{z^2}{y^2}-2y+\frac{y^4}{z^2} \right)\, \partial_y+\left(\frac{z^3}{y^3}-2z+\frac{y^3}{z}\right)\,\partial_z,&&\phi_* \partial_{s_3}=\frac{2y^2}{z}\, \partial_y+3y\,\partial_z,\\
&\phi_* \partial_{r_4}=\left( \frac{z^3}{y^3}-2z+\frac{y^3}{z}\right)\, \partial_y+\left(\frac{z^4}{y^4}-\frac{2z^2}{y}+y^2 \right)\,\partial_z,&&\phi_* \partial_{s_4}=\left(\frac{z}{y}+\frac{y^2}{z} \right)\, \partial_y+\left(\frac{z^2}{y^2}+2y\right)\,\partial_z.
\end{align*}
\end{small}
\par \medskip

Therefore, the matrix of $\phi_* \colon W(D_1) \longrightarrow W(D_2)$ is:
\medskip
\begin{small}
\begin{center}
\begin{tabular}{c|c|c|c|c|c|c|c|c|c|c}
&$\phi_*\mathfrak{e}_1$&$\phi_*\mathfrak{e}_2 $&$\phi_*\mathfrak{e}_3 $&$\phi_*\mathfrak{e}_4 $&$\phi_*\mathfrak{e}_5 $&$\phi_*\mathfrak{e}_6 $&$\phi_*\mathfrak{e}_7 $&$\phi_*\mathfrak{e}_8 $&$\phi_*\mathfrak{e}_9 $&$\phi_*\mathfrak{e}_{10} $\\
\hline
$ \mathfrak{m}_1$&$-\mu_1 $&$1 $&$0 $&$0 $&$0 $&$0 $&$0 $&$0 $&$0$&$0$\\
\hline
$\mathfrak{m}_2 $&$ -1$&$0 $&$0 $&$0 $&$0 $&$0 $&$ $&$0 $&$0 $&$0$\\
\hline
$\mathfrak{m}_3 $&$2 $&$0 $&$1 $&$-\mu_2 $&$-2\mu_2 $&$0 $&$\mu_2^2 $&$\mu_2 $&$0$&$0$\\
\hline
$\mathfrak{m}_4 $&$ 0$&$0 $&$0 $&$-1 $&$-2 $&$0 $&$2\mu_2$&$1 $&$0$&$0$\\
\hline
$\mathfrak{m}_5 $&$0 $&$0 $&$0 $&$2\mu_3(\mu_3-1) $&$(\mu_3-1)^2 $&$\mu_3(\mu_3-1) $&$0 $&$0 $&$0$&$0$\\
\hline
$\mathfrak{m}_6 $&$0 $&$0 $&$0 $&$4\mu_3-2 $&$2(\mu_3-1) $&$2\mu_3-1 $&$0 $&$0 $&$0$&$0$\\
\hline
$\mathfrak{m}_7 $&$2\mu_4^2 $&$-\mu_4^3 $&$\mu_4^2 $&$\mu_4 $&$0 $&$0 $&$0 $&$0 $&$0$&$0$\\
\hline
$\mathfrak{m}_8 $&$4\mu_4 $&$-3\mu_4^2 $&$2\mu_4 $&$1 $&$0 $&$0 $&$0 $&$0 $&$0$&$0$\\
\hline
$\mathfrak{m}_9 $&$0 $&$0 $&$0 $&$0 $&$0 $&$0 $&$0$&$0$&$0$&$0$\\
\hline
$\mathfrak{m}_{10} $&$0 $&$0 $&$0 $&$0 $&$0 $&$0 $&$0$&$0$&$0$&$0$\\
\hline
&$-4b_{3,2}$&$2b_{4,3}$&$-2b_{3,2}$&&$-2b_{2,1}$&$b_{2,1}$&&$2b_{2,1} $&&$b_{2,1} $\\
$\scriptscriptstyle{\mathrm{H}^0\big({\mathbb{P}^2, \mathrm{T} \mathbb{P}^2(5\Delta)}\big)^{\dag}}$&$-3c_{2,1}$&$c_{3,2}$&$-2c_{2,1}$&&&&&&&\\
&$3d_{4} $&$-2d_5 $&$d_{4} $&$2d_3 $&$d_3 $&$d_3 $&$d_2 $&&$d_1 $&\\
\end{tabular}
\par\medskip -- \textit{Matrix $\mathscr{L}$} --

\end{center}
\end{small}
\par \medskip
\noindent {\textit{Step 3} --} For typographical reasons, we will write the vectors of $V(D_1)$ and $V(D_2)$ as lines (and not as columns).
\par \medskip
\begin{center}
\begin{tabular}{c|c|c|c|c|c|c|c|c|c|c}
&$\mathfrak{l}_{1}$&$\mathfrak{l}_{2}$&$\mathfrak{l}_{3}$&$\mathfrak{l}_{4}$&$\mathfrak{l}_{5}$&$\mathfrak{l}_{6}$&$\mathfrak{l}_{7}$&$\mathfrak{l}_{8}$&$\mathfrak{l}_{9}$&$\mathfrak{l}_{10}$ \\
\hline
$a_1$&$-\lambda_1 $&$-1 $&$2 $&$0 $&$0 $&$0 $&$ -\lambda_4^2$&$ -2\lambda_4$&$0 $&$0$\\
\hline
$a_2$&$0 $&$0 $&$0 $&$0 $&$0 $&$0 $&$-3 $&$0 $&$0 $&$0$\\
\hline
$a_3$&$0 $&$0 $&$0 $&$0 $&$0 $&$0 $&$0 $&$0 $&$-3 $&$0 $\\
\hline
$a_4$&$1 $&$0 $&$0 $&$0 $&$0 $&$0 $&$\lambda_4^3 $&$3\lambda_4^2 $&$2\lambda_5^2 $&$4\lambda_5 $\\
\hline
$a_5$&$ 0$&$0 $&$-\lambda_2 $&$-1 $&$-2 \lambda_3 $&$-2 $&$-3\lambda_4 $&$-3 $&$-4\lambda_5 $&$-4 $\\
\hline
$a_6$&$ 0$&$0 $&$0 $&$0 $&$0 $&$ 0$&$2 $&$0 $&$0 $&$0 $\\
\hline
$a_7$&$0 $&$0 $&$0 $&$0 $&$0 $&$0 $&$0 $&$0 $&$0 $&$0 $\\
\hline
$a_8$&$0 $&$0 $&$0 $&$0 $&$0 $&$0 $&$0 $&$0 $&$0 $&$0 $\\
\end{tabular}
\par\medskip -- \textit{Matrix $^{\mathit t}\mathscr{M}$} --
\end{center}
\par \medskip
\begin{center}
\begin{tabular}{c|c|c|c|c|c|c|c|c|c|c}
&$\mathfrak{m}_{1}$&$\mathfrak{m}_{2}$&$\mathfrak{m}_{3}$&$\mathfrak{m}_{4}$&$\mathfrak{m}_{5}$&$\mathfrak{m}_{6}$&$\mathfrak{m}_{7}$&$\mathfrak{m}_{8}$&$\mathfrak{m}_{9}$&$\mathfrak{m}_{10}$ \\
\hline
$a_1$&$-\mu_1 $&$-1 $&$2 $&$0 $&$0 $&$0 $&$ -\mu_4^2$&$ -2\mu_4$&$0 $&$0$\\
\hline
$a_2$&$0 $&$0 $&$0 $&$0 $&$0 $&$0 $&$3 $&$0 $&$0 $&$0$\\
\hline
$a_3$&$0 $&$0 $&$0 $&$0 $&$0 $&$0 $&$0 $&$0 $&$3 $&$0 $\\
\hline
$a_4$&$1 $&$0 $&$0 $&$0 $&$0 $&$0 $&$-\mu_4^3 $&$-3\mu_4^2 $&$2\mu_5^2 $&$4\mu_5 $\\
\hline
$a_5$&$ 0$&$0 $&$-\mu_2 $&$-1 $&$-2 \mu_3 $&$-2 $&$-3\mu_4 $&$-3 $&$-4\mu_5 $&$-4 $\\
\hline
$a_6$&$ 0$&$0 $&$0 $&$0 $&$0 $&$ 0$&$-2 $&$0 $&$0 $&$0 $\\
\hline
$a_7$&$0 $&$0 $&$0 $&$0 $&$0 $&$0 $&$0 $&$0 $&$0 $&$0 $\\
\hline
$a_8$&$0 $&$0 $&$0 $&$0 $&$0 $&$0 $&$0 $&$0 $&$0 $&$0 $\\
\end{tabular}
\par\medskip -- \textit{Matrix $^{\mathit t}\mathscr{N}_{\,\,\,\mathrm{a}}$} --
\end{center}
\par \medskip
\begin{center}
\begin{tabular}{c|c|c|c|c|c|c|c|c|c|c}
&$\mathfrak{m}_{1}$&$\mathfrak{m}_{2}$&$\mathfrak{m}_{3}$&$\mathfrak{m}_{4}$&$\mathfrak{m}_{5}$&$\mathfrak{m}_{6}$&$\mathfrak{m}_{7}$&$\mathfrak{m}_{8}$&$\mathfrak{m}_{9}$&$\mathfrak{m}_{10}$ \\
\hline
$b_{0,1}$&$0 $&$0 $&$0 $&$0 $&$0 $&$0 $&$2\mu_4^5 $&$10\mu_4^4 $&$3\mu_5^3 $&$9\mu_5^2 $\\
\hline
$b_{1,1}$&$-1 $&$ 0$&$0 $&$0 $&$0 $&$0 $&$\mu_4^3 $&$3\mu_4^2 $&$-2\mu_5^2 $&$-4\mu_5 $\\
\hline
$b_{2,1}$&$0 $&$0 $&$2\mu_2 $&$2 $&$3\mu_3 $&$3 $&$4\mu_ 4$&$4 $&$5\mu_5 $&$5 $\\
\hline
$b_{0,2}$&$0 $&$0 $&$0 $&$0 $&$0 $&$0 $&$-9\mu_4^8 $&$-72\mu_4^7 $&$0 $&$0 $\\
\hline
$b_{1,2}$&$0 $&$0 $&$0 $&$0 $&$0 $&$0 $&$-3\mu_4^6 $&$-18\mu_4^5 $&$0 $&$0 $\\
\hline
$b_{2,2}$&$-{\mu_1}^{-1} $&$\mu_1^{-2} $&$0 $&$0 $&$0 $&$0 $&$-\mu_4^4 $&$-4\mu_4^3 $&$0 $&$0 $\\
\hline
$b_{3,2}$&$0 $&$0 $&$0 $&$0 $&$0 $&$0 $&$0 $&$0 $&$0 $&$0 $\\
\hline
$b_{0,3}$&$0 $&$0 $&$0 $&$0 $&$0 $&$0 $&$52\mu_4^{11}$&$572\mu_4^{10} $&$-28\mu_5^6 $&$-168\mu_5^5 $\\
\hline
$b_{1,3}$&$0 $&$0 $&$0 $&$0 $&$0 $&$0 $&$15\mu_4^9 $&$135\mu_4^8 $&$12 \mu_5^5 $&$60\mu_5^4 $\\
\hline
$b_{2,3}$&$0 $&$0 $&$0 $&$0 $&$0 $&$0 $&$4\mu_4^7 $&$28\mu_4^6 $&$-5\mu_5^4 $&$-20\mu_5^3 $\\
\hline
$b_{3,3}$&$-\mu_1^{-2} $&$2\mu_1^{-3} $&$0 $&$0 $&$0 $&$0 $&$\mu_4^5 $&$5\mu_4^4 $&$2\mu_5^3 $&$6\mu_5^2 $\\
\hline
$b_{4,3}$&$0 $&$0 $&$0 $&$0 $&$0 $&$0 $&$0 $&$0 $&$0 $&$0 $\\
\hline
$b_{0,4}$&$0 $&$0 $&$0 $&$0 $&$0 $&$0 $&$-340\mu_4^{14} $&$-4760\mu_4^{13} $&$0 $&$0 $\\
\hline
$b_{1,4}$&$0 $&$0 $&$0 $&$0 $&$0 $&$0 $&$-91\mu_4^{12} $&$-1092\mu_4^{11} $&$0 $&$0 $\\
\hline
$b_{2,4}$&$0 $&$0 $&$0 $&$0 $&$0 $&$0 $&$-22\mu_4^{10} $&$-220\mu_4^9 $&$0 $&$0 $\\
\hline
$b_{3,4}$&$0 $&$0 $&$0 $&$0 $&$0 $&$0 $&$-5\mu_4^8 $&$-40\mu_4^7 $&$0 $&$0 $\\
\hline
$b_{4,4}$&$-\mu_1^{-3} $&$3\mu_1^{-4} $&$0 $&$0 $&$0 $&$0 $&$-\mu_4^6 $&$-6\mu_4^5 $&$0 $&$0 $\\
\hline
$b_{5,4}$&$0 $&$0 $&$0 $&$0 $&$0 $&$0 $&$0 $&$0 $&$0 $&$0 $\\
\hline
$b_{0,5}$&$0 $&$0 $&$0 $&$0 $&$0 $&$0 $&$2394\mu_4^{17} $&$40698\mu_4^{16} $&$429\mu_5^9 $&$3861\mu_5^8 $\\
\hline
$b_{1,5}$&$0 $&$0 $&$0 $&$0 $&$0 $&$0 $&$612\mu_4^{15} $&$9180\mu_4^{14} $&$-165\mu_5^8 $&$-1320\mu_5^7 $\\
\hline
$b_{2,5}$&$0 $&$0 $&$0 $&$0 $&$0 $&$0 $&$140\mu_4^{13} $&$1820\mu_4^{12} $&$60\mu_5^7 $&$420\mu_5^6 $\\
\hline
$b_{3,5}$&$0 $&$0 $&$0 $&$0 $&$0 $&$0 $&$30\mu_4^{11} $&$330\mu_4^{10} $&$-21\mu_5^6 $&$-126\mu_5^5 $\\
\hline
$b_{4,5}$&$0 $&$0 $&$0 $&$0 $&$0 $&$0 $&$6\mu_4^9 $&$54\mu_4^8 $&$7\mu_5^5 $&$35\mu_5^4 $\\
\hline
$b_{5,5}$&$-\mu_1^{-4} $&$4\mu_1^{-5} $&$0 $&$0 $&$0 $&$0 $&$\mu_4^7 $&$7\mu_4^6 $&$-2\mu_5^4 $&$-8\mu_5^3 $\\
\hline
$b_{6,5}$&$0 $&$0 $&$0 $&$0 $&$0 $&$0 $&$0 $&$0 $&$0 $&$0 $\\
\end{tabular}
\par\medskip -- \textit{Matrix $^{\mathit t}\mathscr{N}_{\,\,\,\mathrm{b}}$} --
\end{center}
\par\medskip
\begin{center}
\begin{tabular}{c|c|c|c|c|c|c|c|c|c|c}
&$\mathfrak{m}_{1}$&$\mathfrak{m}_{2}$&$\mathfrak{m}_{3}$&$\mathfrak{m}_{4}$&$\mathfrak{m}_{5}$&$\mathfrak{m}_{6}$&$\mathfrak{m}_{7}$&$\mathfrak{m}_{8}$&$\mathfrak{m}_{9}$&$\mathfrak{m}_{10}$ \\
\hline
$c_{0,1}$&$0 $&$0 $&$0 $&$0 $&$0 $&$0 $&$4\mu_4^6 $&$24\mu_4^5 $&$0 $&$0 $\\
\hline
$c_{1,1}$&$\mu_1^{-1} $&$-\mu_1^{-2} $&$0 $&$0 $&$0 $&$0 $&$\mu_4^4 $&$4\mu_4^3 $&$0 $&$0 $\\
\hline
$c_{2,1}$&$0 $&$0 $&$0 $&$0 $&$0 $&$0 $&$0 $&$0 $&$0 $&$0 $\\
\hline
$c_{0,2}$&$0 $&$0 $&$0 $&$0 $&$0 $&$0 $&$-25\mu_4^9 $&$-225\mu_4^8 $&$-18\mu_5^5 $&$-90\mu_5^4 $\\
\hline
$c_{1,2}$&$0 $&$0 $&$0 $&$0 $&$0 $&$0 $&$-5\mu_4^7 $&$-35\mu_4^6 $&$6\mu_5^4 $&$24\mu_5^3 $\\
\hline
$c_{2,2}$&$\mu_1^{-2} $&$-2\mu_1^{-3} $&$0 $&$0 $&$0 $&$0 $&$-\mu_4^5 $&$-5\mu_4^4 $&$-2\mu_5^3 $&$-6\mu_5^2 $\\
\hline
$c_{3,2}$&$0 $&$0 $&$0 $&$0 $&$0 $&$0 $&$0 $&$0 $&$0 $&$0 $\\
\hline
$c_{0,3}$&$0 $&$0 $&$0 $&$0 $&$0 $&$0 $&$182\mu_4^{12} $&$2184\mu_4^{11} $&$0 $&$0 $\\
\hline
$c_{1,3}$&$0 $&$0 $&$0 $&$0 $&$0 $&$0 $&$33\mu_4^{10} $&$330\mu_4^9 $&$0 $&$0 $\\
\hline
$c_{2,3}$&$0 $&$0 $&$0 $&$0 $&$0 $&$0 $&$6\mu_4^8 $&$48\mu_4^7 $&$0 $&$0 $\\
\hline
$c_{3,3}$&$\mu_1^{-3} $&$-3\mu_1^{-4} $&$0 $&$0 $&$0 $&$0 $&$\mu_4^6 $&$6\mu_4^5 $&$0 $&$0 $\\
\hline
$c_{4,3}$&$0 $&$0 $&$0 $&$0 $&$0 $&$0 $&$0 $&$0 $&$0 $&$0 $\\
\hline
$c_{0,4}$&$0 $&$0 $&$0 $&$0 $&$0 $&$0 $&$-1428\mu_4^{15} $&$-21420\mu_4^{14} $&$330\mu_5^8 $&$2640\mu_5^7 $\\
\hline
$c_{1,4}$&$0 $&$0 $&$0 $&$0 $&$0 $&$0 $&$-245\mu_4^{13} $&$-3185\mu_4^{12} $&$-96\mu_5^7 $&$-672\mu_5^6 $\\
\hline
$c_{2,4}$&$0 $&$0 $&$0 $&$0 $&$0 $&$0 $&$-42\mu_4^{11} $&$-462\mu_4^{10} $&$28\mu_5^6 $&$168\mu_5^5 $\\
\hline
$c_{3,4}$&$0$&$0$&$0$&$0$&$0$&$0$&$-7\mu_4^9 $&$-63\mu_4^8 $&$-8\mu_5^5 $&$-40\mu_5^4 $\\
\hline
$c_{4,4}$&$\mu_1^{-4} $&$-4\mu_1^{-5} $&$0 $&$0 $&$0 $&$0 $&$-\mu_4^7 $&$-7\mu_4^6 $&$2\mu_5^4 $&$8\mu_5^3 $\\
\hline
$c_{5,4}$&$0 $&$0 $&$0 $&$0 $&$0 $&$0 $&$0 $&$0 $&$0 $&$0 $\\
\hline
$c_{0,5}$&$0 $&$0 $&$0 $&$0 $&$0 $&$0 $&$11704\mu_4^{18} $&$210672\mu_4^{17} $&$0 $&$0 $\\
\hline
$c_{1,5}$&$0 $&$0 $&$0 $&$0 $&$0 $&$0 $&$1938\mu_4^{16} $&$31008\mu_4^{15} $&$0 $&$0 $\\
\hline
$c_{2,5}$&$0 $&$0 $&$0 $&$0 $&$0 $&$0 $&$320\mu_4^{14} $&$4480\mu_4^{13} $&$0 $&$0 $\\
\hline
$c_{3,5}$&$0 $&$0 $&$0 $&$0 $&$0 $&$0 $&$52\mu_4^{12} $&$624\mu_4^{11} $&$0 $&$0 $\\
\hline
$c_{4,5}$&$0 $&$0 $&$0 $&$0 $&$0 $&$0 $&$8\mu_4^{10} $&$80\mu_4^9 $&$0 $&$0 $\\
\hline
$c_{5,5}$&$\mu_1^{-5} $&$-5\mu_1^{-6} $&$0 $&$0 $&$0 $&$0 $&$\mu_4^8 $&$8\mu_4^7 $&$0 $&$0 $\\
\hline
$c_{6,5}$&$0 $&$0 $&$0 $&$0 $&$0 $&$0 $&$0 $&$0 $&$0 $&$0 $\\
\end{tabular}
\par\medskip -- \textit{Matrix $^{\mathit t}\mathscr{N}_{\,\,\,\mathrm{c}}$} --
\end{center}
\par \medskip
\begin{center}
\begin{tabular}{c|c|c|c|c|c|c|c|c|c|c}
&$\mathfrak{m}_{1}$&$\mathfrak{m}_{2}$&$\mathfrak{m}_{3}$&$\mathfrak{m}_{4}$&$\mathfrak{m}_{5}$&$\mathfrak{m}_{6}$&$\mathfrak{m}_{7}$&$\mathfrak{m}_{8}$&$\mathfrak{m}_{9}$&$\mathfrak{m}_{10}$ \\
\hline
$d_1$&$0 $&$0 $&$0 $&$0 $&$0 $&$0 $&$0 $&$0 $&$1 $&$0 $\\
\hline
$d_2$&$0 $&$0 $&$\mu_2^2 $&$2\mu_2 $&$0 $&$0 $&$1 $&$0 $&$0 $&$0 $\\
\hline
$d_3$&$0 $&$0 $&$0 $&$0 $&$\mu_3^2 $&$2\mu_3 $&$2\mu_4 $&$2$&$2\mu_5 $&$2 $\\
\hline
$d_4$&$0 $&$0 $&$0 $&$0 $&$0 $&$0 $&$\mu_4^2 $&$2\mu_4 $&$0 $&$0 $\\
\hline
$d_5$&$0 $&$0 $&$0 $&$0 $&$0 $&$0 $&$0 $&$0 $&$\mu_5^2 $&$2\mu_5 $\\
\end{tabular}
\par\medskip -- \textit{Matrix $^{\mathit t}\mathscr{N}_{\,\,\,\mathrm{d}}$} --
\end{center}
\par \medskip
Although we won't need it, it is easy to verify that \addtocounter{equation}{1}
\[
\left\{
\begin{array}{lll}
\phi_*(a_1)=a_1-4b_{3,2}-3c_{2,1}+3d_4 &\hspace{1cm}\phi_*(a_2)=a_2-3d_2 &\hspace{1cm}\phi_*(a_3)=a_3-3d_1\\
\phi_*(a_4)=a_4-2b_{4,3}+c_{3,2}-2d_5 &\hspace{1cm}\phi_*(a_5)=a_5+2d_3 &\hspace{1cm}\phi_*(a_6)=a_6+2d_2 \tag{\theequation} \label{inutile}\\
\phi_*(a_7)=a_7 &\hspace{1cm}\phi_*(a_8)=a_8. &
\end{array}
\right.
\]

\subsection{Calculations for the pair $({\mathfrak{D}}_1, {\mathfrak{D}}_2)$}\label{Subsec:comp2}

 In this part we provide the last two steps of the calculations. \par \medskip
\noindent {\textit{Step 1} --} We have isomorphisms 
\begin{align*}
&W(\mathfrak{D}_1) \simeq W(D_1) \oplus W\big(A\cdot D_{\widehat{P}_2}\big) \oplus  W\big(A f A\cdot D_{\widehat{P}_2}\big), &&
W(\mathfrak{D}_2) \simeq W\big(D_{\widehat{P}_1}\big) \oplus W(A D_2) \oplus W\big(AfA\cdot D_{\widehat{P}_2}\big).
\end{align*}
We transport the bases of $W(D_2)$ and $W(D_{\widehat{P}_2})$ by $A$ and $A\phi A$. Therefore, if we take algebraic bases of $W\big(D_{\widehat{P}_1}\big)$, $W\big(D_{\widehat{P}_2}\big)$ and $W(D_2)$, the matrix of $\psi_{*} \,\, \colon \, W(\mathfrak{D}_1) \longrightarrow W(\mathfrak{D}_2)$ has the form 
\par \medskip
\begin{center} 
\begin{tabular}{l|c|c|c}
&$A \phi \cdot (\mathfrak{l}_i)_{1 \leq i \leq 10} $&$ A \phi \cdot (A \mathfrak{m}_i)_{1 \leq i \leq 10} $&$ A \phi \cdot (A \phi A \mathfrak{m}_i)_{1 \leq i \leq 10}$ \\
\hline
$(\mathfrak{l}_i)_{1 \leq i \leq 10}$&$0_{10 \times 10}$ &$0_{10 \times 10}$ &$\mathscr{Q}$ \\
\hline
${\vphantom{\bigl(}}^{\displaystyle{(A \mathfrak{m}_i)_{1 \leq i \leq 10}}}_{\displaystyle{\mathrm{H}^0\big({\mathbb{P}^2, \mathrm{T} \mathbb{P}^2(5 A\Delta)}\big)^{\dag}}}$&$\mathscr{L} \mathscr{K}^{-1}$ &$0_{55 \times 10}$ & $0_{55 \times 10}$\\
\hline
$(A \phi A \mathfrak{m}_i)_{1 \leq i \leq 10}$&$0_{10 \times 10}$ &$\mathrm{id}_{10 \times 10}$ &$0_{10 \times 10}$ \\
\end{tabular}
\par \medskip -- \textit{Matrix $\mathscr{Y}$} --

\end{center}
\par \bigskip
\noindent The matrix $\mathscr{Q}$ is the $10 \times 10$ matrix given explicitly by
\begin{small}
\[
\left(\begin{array}{cccccccccc}
-4\lambda_1-1 & q_{1,2} & 0 & 0 & 0 & 0 & 0 & 0 & 0 & 0\\
-4& 4\mu_1-1& 0 & 0 & 0 & 0 & 0 & 0 & 0 & 0\\
8-25\lambda_2 & \mu_1(25\lambda_2-8) & 6\lambda_2-1 & q_{3,4} & 0 & 0 & 0 & 0 & 0 & 0\\
-25 & 25\mu_1 & 6 & 1-6\mu_2 & 0 & 0 & 0 & 0 & 0 & 0\\
-50\lambda_3 & 8+50\lambda_3\mu_1 & 4(1+3\lambda_3) & -4\mu_2(1+3\lambda_3) & -1 & \mu_3+\lambda_3& 0 & 0 & 0 & 0\\
-50 & 50\mu_1 & 12 & -12\mu_2 & 0 & 1 & 0 & 0 & 0 & 0\\
q_{7,1}& q_{7,2}& 18\lambda_4-42& \mu_2(42-18\lambda_4)& -10 & 10\mu_3-10 & -1 & \mu_4+\lambda_4 & 0 & 0\\
q_{8,1}& q_{8,2} & 18 & -18\mu_2 & 0 & 0 & 0 & 1 & 0 & 0\\
q_{9,1}& q_{9,2}& q_{9,3}& q_{9,4}& -110 & 110\mu_3-62 & -12 & 12\mu_4 & -1 &q_{9,10}\\
q_{10,1}& q_{10,2}& 24 & -24\mu_2 & 0 & 0 & 0 & 0 & 0 & 1
\end{array}
\right)
\]
\end{small}
where
\begin{small}
\begin{align*}
q_{1,2}&=\mu_1(1+4\lambda_1)-\lambda_1\\
q_{3,4}&=\mu_2(1-6\lambda_2)+\lambda_2\\
q_{7,1}&=2\alpha+200+\lambda_4(3\mu_4^2-\lambda_4^2-75-4\lambda_4)+2\mu_4^3\\
q_{7,2}&=40+\mu_1\big(75\lambda_4-2\alpha-200+4\lambda_4^2-2\mu_4^3-3\mu_4^2\lambda_4+\lambda_4^3\big)-(\mu_4+\lambda_4)^2\\
q_{8,1}&=3(\mu_4^2-\lambda_4^2)-75-8\lambda_4 \\
q_{8,2}&=\mu_1\big(75+8\lambda_4+3(\lambda_4^2-\mu_4^2)\big)-2(\mu_4+\lambda_4)\\
q_{9,1}&=672-4\mu_5(\lambda_5+1)+24(\mu_4^3+\alpha)-2(\lambda_5^2+\mu_5^2)-100\lambda_5\\
q_{9,2}&=\mu_1\big(4\mu_5(1+\lambda_5)-24(\mu_4^3+\alpha)+2(\mu_5^2+\lambda_5^2)+100\lambda_5-672\big)+320-2\alpha-12\mu_4^2\\ 
q_{9,3}&=24\lambda_5-92-4\alpha\\
q_{9,4}&=\mu_2(92+4\alpha-24\lambda_5)\\ 
q_{9,10}&=\mu_5+1+\lambda_5, \quad q_{10,1}=-4(25+\lambda_5+\mu_5)\\
q_{10,2}&=4\mu_1(25+\lambda_5+\mu_5).
\end{align*}
\par \smallskip
\end{small}
\noindent {\textit{Step 2} --} We have $V(\mathfrak{D}_1)=\mathfrak{h}(\mathbb{P}^2)$ and $V(\mathfrak{D}_2)=\mathrm{H}^0\big(\mathbb{P}^2, \mathrm{T} \mathbb{P}^2 \bigl(5A \Delta\bigr)\big)$. Let us decompose a basis of $V(\mathfrak{D}_1)$ in an algebraic basis of $W(\mathfrak{D}_1)$:
\par \medskip
\begin{center}
\begin{tabular}{l|c}
&$(a_i)_{1 \leq i \leq 8}$\\
\hline
$(\mathfrak{l}_i)_{1 \leq i \leq 10}$& $\mathscr{M}_{\,\,\,\mathrm{a}}$\\
\hline
$(A\mathfrak{m}_i)_{1 \leq i \leq 10}$& $\mathscr{M}_{\,\,\,\mathrm{a}}'$\\
\hline
$(A \phi A\mathfrak{m}_i)_{1 \leq i \leq 10}$& $\mathscr{M}_{\,\,\,\mathrm{a}}''$\\
\end{tabular}
\par \medskip
-- \textit{Matrix $\mathscr{V}_1$} --
\end{center}
\par \medskip
To compute quickly the matrices $\mathscr{M}_{\,\,\,\mathrm{a}}'$ and $\mathscr{M}_{\,\,\,\mathrm{a}}''$, we make the following remark (which can be proved by an easy computation): for any holomorphic vector field defined in a neighborhood $(0,0)$, if we lift it as a section of $W(D_{\widehat{P}_1})$ or~$W\big(D_{\widehat{P}_2}\big)$, this section depends only on the seven Taylor components $\partial_y$, $y\, \partial_y$, $z\,\partial_y$ ,$\partial_z$, $y\,\partial_z$, $z\,\partial_z$, $y^2\, \partial_z$ of the vector field at $(0,0)$. The corresponding sections in algebraic bases are given by the tables
\par \medskip
\begin{center}
\begin{tabular}{c|c|c|c|c|c|c|c} 
& $\partial_y$& $y\, \partial_y$& $z\,\partial_y$ &$\partial_z$& $y\,\partial_z$& $z\,\partial_z$& $y^2\, \partial_z$\\
\hline
$\mathfrak{l}_1$&$-\lambda_1$&$0 $&$0 $&$1 $&$0 $&$0 $&$0$\\
\hline
$\mathfrak{l}_2$&$-1 $&$0 $&$0 $&$0 $&$0 $&$0 $&$0$\\
\hline
$\mathfrak{l}_3$&$2 $&$0 $&$0 $&$0 $&$-\lambda_2$&$0 $&$0$\\
\hline
$\mathfrak{l}_4$&$0 $ &$0 $&$0 $&$0 $&$-1 $&$0 $&$0$\\
\hline 
$\mathfrak{l}_5$&$0 $ &$0 $&$0 $&$0 $&$-2 \lambda_3 $&$0 $&$0$\\
 \hline
$\mathfrak{l}_6$&$0 $&$0 $&$0 $&$0 $&$-2 $&$0 $&$0$\\
\hline
$\mathfrak{l}_7$&$ -\lambda_4^2$&$-3$&$0 $&$\lambda_4^3 $&$-3\lambda_4 $&$2 $&$0$\\
\hline
$\mathfrak{l}_8$&$ -2\lambda_4$&$0 $&$0 $&$3\lambda_4^2 $&$-3 $&$0 $&$0$\\
\hline
$\mathfrak{l}_9$&$0 $&$0 $&$-3 $&$2\lambda_5^2 $&$-4\lambda_5 $&$0 $&$-2$\\
\hline$\mathfrak{l}_{10}$&$0$&$0 $&$0 $&$4\lambda_5 $&$-4 $&$0 $&$0$\\
\end{tabular}
\par\medskip -- \textit{Matrix $\mathscr{Z}_1$} --
\end{center}
\par \bigskip
\begin{center}
\begin{tabular}{c|c|c|c|c|c|c|c} 
& $\partial_y$& $y\, \partial_y$& $z\,\partial_y$ &$\partial_z$& $y\,\partial_z$& $z\,\partial_z$& $y^2\, \partial_z$\\
\hline
$\mathfrak{m}_1$&$-\mu_1$&$0 $&$0 $&$1 $&$0 $&$0 $&$0$\\
\hline
$\mathfrak{m}_2$&$-1 $&$0 $&$0 $&$0 $&$0 $&$0 $&$0$\\
\hline
$\mathfrak{m}_3$&$2 $&$0 $&$0 $&$0 $&$-\mu_2$&$0 $&$0$\\
\hline
$\mathfrak{m}_4$&$0 $ &$0 $&$0 $&$0 $&$-1 $&$0 $&$0$\\
\hline 
$\mathfrak{m}_5$&$0 $ &$0 $&$0 $&$0 $&$-2 \mu_3 $&$0 $&$0$\\
 \hline
$\mathfrak{m}_6$&$0 $&$0 $&$0 $&$0 $&$-2 $&$0 $&$0$\\
\hline
$\mathfrak{m}_7$&$ -\mu_4^2$&$3$&$0 $&$-\mu_4^3 $&$-3\mu_4 $&$-2 $&$0$\\
\hline
$\mathfrak{m}_8$&$ -2\mu_4$&$0 $&$0 $&$-3\mu_4^2 $&$-3 $&$0 $&$0$\\
\hline
$\mathfrak{m}_9$&$0 $&$0 $&$3 $&$2\mu_5^2 $&$-4\mu_5 $&$0 $&$-2$\\
\hline$\mathfrak{m}_{10}$&$0$&$0 $&$0 $&$4\mu_5 $&$-4 $&$0 $&$0$\\
\end{tabular}
\par\medskip -- \textit{Matrix $\mathscr{Z}_2$} --
\end{center} 
\par \medskip
Therefore, in order to compute $\mathscr{M}_{\,\,\,\mathrm{a}}'$ and $\mathscr{M}_{\,\,\,\mathrm{a}}''$, it suffices to extract the seven aforementioned Taylor components of the vector field $A^{-1}_* a_i$ $\big($resp. $(A \phi A)^{-1}_* a_i$$\big)$ at $(0,0)$. Then we multiply the resulting vector by the matrix $\mathscr{Z}_1$ $\big($resp. $\mathscr{Z}_2$$\big)$ and we obtain $\mathscr{M}_{\,\,\,\mathrm{a}}'$ $\big($resp. $\mathscr{M}_{\,\,\,\mathrm{a}}''$$\big)$. We won't give the exact expressions of $\mathscr{M}_{\,\,\,\mathrm{a}}'$ and $\mathscr{M}_{\,\,\,\mathrm{a}}''$ because of lack of space.
\par \medskip
We now deal with $V(\mathfrak{D}_2)$. We decompose a basis of vectors of $V(\mathfrak{D}_2)$ in an algebraic basis of $W(\mathfrak{D}_2)$. We get the following matrix
\par \medskip
\begin{center}
\begin{tabular}{l|c|c|c|c}
&$(A a_i)_{1 \leq i \leq 8}$&$(A b_{p,q})_{0 \leq p \leq q+1\leq m+1}$&$(A c_{p,q})_{0 \leq p \leq q+1\leq m+1}$&$(A d_{p})_{1 \leq p \leq m}$\\
\hline
$(\mathfrak{l}_i)_{1 \leq i \leq 10}$&$\mathscr{N}_{\,\,\,\mathrm{a}}'$&$\mathscr{N}_{\,\,\,\mathrm{b}}'$&$\mathscr{N}_{\,\,\,\mathrm{c}}'$&$\mathscr{N}_{\,\,\,\mathrm{d}}' $\\
\hline
$(A \mathfrak{m}_i)_{1 \leq i \leq 10}$&$\mathscr{N}_{\,\,\,\mathrm{a}}$&$\mathscr{N}_{\,\,\,\mathrm{b}} $&$\mathscr{N}_{\,\,\,\mathrm{c}} $&$\mathscr{N}_{\,\,\,\mathrm{d}}  $\\
$(A b_{p,q})_{0 \leq p \leq q+1\leq m+1}$&$0_{25\times 8} $&$ \mathrm{id}_{25 \times 25}$&$0_{25\times 25}  $&$0_{25\times 25}  $\\
$(A c_{p,q})_{0 \leq p \leq q+1\leq m+1}$&$0_{25\times 8} $&$0_{25\times 25} $&$\mathrm{id}_{25 \times 25} $&$0_{25\times 25}  $\\
$(A d_{p})_{1 \leq p \leq m}$&$0_{5\times 8} $&$0_{5 \times 25} $&$0_{5 \times 25}$&$\mathrm{id}_{5 \times 5} $\\
\hline
$(A \phi A \mathfrak{m}_i)_{1 \leq i \leq 10}$&$\mathscr{N}_{\,\,\,\mathrm{a}}''$&$\mathscr{N}_{\,\,\,\mathrm{b}}'' $&$\mathscr{N}_{\,\,\,\mathrm{c}}'' $&$\mathscr{N}_{\,\,\,\mathrm{d}}'' $\\
\end{tabular}
\par \medskip
-- \textit{Matrix $\mathscr{V}_2$} --
\end{center}
\par \medskip
The matrices $\mathscr{N}_{\,\,\,\mathrm{a}}'$ $\big($resp. $\mathscr{N}_{\,\,\,\mathrm{a}}''\big)$, $\mathscr{N}_{\,\,\,\mathrm{b}}'$ $\big($resp. $\mathscr{N}_{\,\,\,\mathrm{b}}''\big)$, $\mathscr{N}_{\,\,\,\mathrm{c}}'$ $\big($resp. $\mathscr{N}_{\,\,\,\mathrm{c}}'' \big)$ and $\mathscr{N}_{\,\,\,\mathrm{d}}'$ $\big($resp. $\mathscr{N}_{\,\,\,\mathrm{d}}''\big)$ appearing in $\mathscr{V}_2$ are computed in the same way as $\mathscr{N}_{\,\,\,\mathrm{a}}'$ $\big($resp. $\mathscr{N}_{\,\,\,\mathrm{a}}''$$\big)$.

\subsection{The result}
We can now state and prove the following result.

\begin{thm} Let $X$ be the rational surface obtained by blowing up the projective plane $15$ times at the infinitely near points $\widehat{P}_1$, $A \widehat{P}_2$ and $A \phi A \widehat{P}_2$, and let $\psi$ be the lift of $A \phi$ as an automorphism of $X$. Then the characteristic polynomial~$Q_{\psi}$ of the map $\psi_*$ acting on the space $\mathrm{H}^1 \bigl(X, \mathrm{T}X \bigr)$ of infinitesimal deformations of~$X$ is
\[
Q_{\psi}(x)=\big(x^2+3x+1\big) \, \big(x^2+18x+1\big) \, \big(x^2-7x+1\big) \, \big(x^2+x+1\big) \, \big(x-1\big)^2 \, \big(x+1\big)^4 \, \big(x^2-x+1\big)^4.
\]
Besides, there is only one nontrivial Jordan block, which is a $2 \times 2$ Jordan block attached with the eigenvalue~$-1$.
\end{thm}

\begin{proof} Let $\mathscr{E}_1$ be the subspace of $W(\mathfrak{D}_1)$ of dimension $22$ defined by\par  \[
\mathscr{E}_1= \mathrm{Span} \, 
\left\{\begin{array}{c} 
\mathfrak{l}_1,\, \mathfrak{l}_2,\, \mathfrak{l}_3,\, \mathfrak{l}_4,\, \mathfrak{l}_5,\, \mathfrak{l}_6,\, \mathfrak{l}_7, \,\mathfrak{l}_8,\, \mathfrak{l}_9,\, \mathfrak{l}_{10}, \\
 A\mathfrak{m}_1,\, A\mathfrak{m}_2,\, A\mathfrak{m}_3,\, A\mathfrak{m}_4,\, A\mathfrak{m}_5,\, A\mathfrak{m}_6,\, A\mathfrak{m}_7, \\
A\phi A \mathfrak{m}_1,\, A\phi A \mathfrak{m}_2,\, A\phi A \mathfrak{m}_3,\, A\phi A \mathfrak{m}_4,\, A\phi A \mathfrak{m}_5\bigr.
\end{array}\right\}.
\]
\par \medskip
\noindent A direct calculation shows that $\mathscr{E}_1$ is a direct factor of $V(\mathfrak{D}_1)$ in $W(\mathfrak{D}_1)$.
If $\mathscr{E}_2$ is the image of $\mathscr{E}_1$ in $W(\mathfrak{D}_2)$ by the natural injection, then $\mathscr{E}_2$ is a direct factor of $V(\mathfrak{D}_2)$ in $W(\mathfrak{D}_2)$.
Therefore, the composition of the morphisms (\ref{comp}) can be expressed as 
\[
\xymatrix{
&\mathscr{E}_1 \ar[r]^-{j} & \mathscr{E}_1 \oplus V(\mathfrak{D}_1) \simeq W(\mathfrak{D}_1) \ar[r]^-{\psi_*} & W(\mathfrak{D}_2) \simeq \mathscr{E}_2 \oplus V(\mathfrak{D}_2) \ar[r] ^-{p}&\mathscr{E}_2 \simeq \mathscr{E}_1}.
\]
The matrix of $j$ (resp. $p$) can be computed using $\mathscr{V}_1$ $\big($resp. $\mathscr{V}_2 \big)$ and the matrix of $\psi_*$ is $\mathscr{Y}$.
\end{proof}

\begin{cor}
$m(X, f) \leq 2$.
\end{cor}

\appendix

\section{Automorphisms of deformations} \label{thermidor}

In this appendix, we provide a detailed proof of the following result:

\begin{pro}
Let $X$ be a complex compact manifold without nonzero holomorphic vector field. Then deformations of $X$ admit no nontrivial automorphisms.
\end{pro}

This result is well-known in the folklore (although it seems to follow from the material developed in \cite{kuranishi_locally_1962}), but we have been unable to find a precise reference in the analytic case. It is often quoted as the following ad hoc form of Kuranishi's theorem: \og{}\textit{If $X$ has no nonzero holomorphic vector field, $X$ admits a universal deformation}\fg{} (\textit{see} \cite[Theorem 22.3]{gross_calabi-yau_2003}, \cite[Theorem 3.6.3.1 (3)]{balaji_introduction_2010}). In older papers, this fact is also considered as known among experts: in the first lines of \cite{wavrik_obstructions_1969}, this statement is quoted and the reader is refered to the expository paper \cite{MR1608786}.
In \cite{MR1608786}, this fact indeed appears as a remark at the end of \S 5. 
\par \medskip
In the algebraic context, the statement can be found (for infinitesimal deformations) in \cite[Corollary 2.6.3]{Sernesi}. We write in details one proof that is outlined in \textit{loc. cit.}, that can be adapted to the analytic context.

\begin{proof}[Proof of Proposition A.1]
Let $(\mathfrak{X}, \pi, B)$ be a deformation of $X$ endowed with an automorphism $u$. Let $(R, \mathfrak{m})$ be the local ring $\mathcal{O}_{B, b}$. The automorphism $u$ acts as an $R$-linear algebra automorphism $u^*$ of $\mathcal{O}_{\mathfrak{X}}$, and the goal is to prove that this automorphism is the identity. First we reduce the problem to infinitesimal deformations.
\par \medskip
We claim that it suffices to prove that $u^*$ induces the identity on all sheaves $\mathcal{O}_{\mathfrak{X}}/{\mathfrak{m}^n \mathcal{O}_{\mathfrak{X}}}$ for $n \geq 1$. Indeed, if it is the case, for any $x$ in $X$ and any element $f$ in $\mathcal{O}_{\mathfrak{X}, x}$,
$u_x^*(f)-f$ maps to zero in all $\mathcal{O}_{\mathfrak{X}, x}/{\mathfrak{m}_R^n \mathcal{O}_{\mathfrak{X}, x}}$, 
hence in all $\mathcal{O}_{\mathfrak{X}, x}/{\mathfrak{m}_x^n \mathcal{O}_{\mathfrak{X}, x}}$. Since $\bigcap_{n \geq 1} \mathfrak{m}_x^n =\{0\}$, this element is zero. This being done, we can replace $\mathfrak{X}$ by the infinitesimal deformation $\mathfrak{X} \times _B (\mathrm{spec}\, R/{\mathfrak{m}_R^n R})^{\mathrm{an}}$. Hence we can assume without loss of generality that $B^{\mathrm{red}}$ is a single point. 
\par \medskip
Next we use the fact that every local artinian algebra can be obtained from the ground field (here it is $\mathbb{C}$) by a finite sequence of \textit{small} extensions. Recall that a small extension of local artinian algebras is an exact sequence
\[
0 \rightarrow (t) \rightarrow R_1 \xrightarrow{\varphi} R_2 \rightarrow 0
\]
where $R_1$, $R_2$ are artinian local algebras, $\varphi$ is an algebra morphism, and $t$ is an element of $R_1$ that is annihilated by the maximal ideal of $R_1$ (so that the ideal generated by $t$ is just the ground field $\mathbb{C}$). Therefore, we are led to the following problem: assume to be given a small extension as above, and assume that
\begin{enumerate}
\item[--] $\mathfrak{X}$ is a deformation over $(\mathrm{spec}\, R_1)^{\mathrm{an}}$ endowed with an automorphism $u$.
\item[--] $u$ acts trivially on the deformation $(X, \mathcal{O}_{\mathfrak{X}}/t \mathcal{O}_{\mathfrak{X}})$ of $X$ over $(\mathrm{spec}\, R_2)^{\mathrm{an}}$.
\end{enumerate}
Then we must show that $u$ fixes $\mathcal{O}_{\mathfrak{X}}$. Since $\mathfrak{X}$ is flat over $B_1$, we have an exact sequence
\[
0 \rightarrow \mathcal{O}_{\mathfrak{X}} / {\mathfrak{m}_{R_1} \mathcal{O}_{\mathfrak{X}}} \xrightarrow{t} \mathcal{O}_{\mathfrak{X}} \rightarrow \mathcal{O}_{\mathfrak{X}} \otimes_{R_1} R_2 \rightarrow 0, 
\]
which is
\[
0 \rightarrow \mathcal{O}_X \xrightarrow{t} \mathcal{O}_{\mathfrak{X}} \rightarrow \mathcal{O}_{\mathfrak{X}}/ t \mathcal{O}_{\mathfrak{X}} \rightarrow 0.
\]
The moprhism $u^*$ acts trivially on $\mathcal{O}_{\mathfrak{X}}/ t \mathcal{O}_{\mathfrak{X}}$ (by hypothesis), and also on  $\mathcal{O}_X$.
The map
\[
\delta \colon \mathcal{O}_{\mathfrak{X}} \rightarrow \mathcal{O}_{\mathfrak{X}}
\]
defined by $\delta(g)=u^*(g)-g$ factors through a map from $\mathcal{O}_{\mathfrak{X}}/t \mathcal{O}_{\mathfrak{X}}$ to $\mathcal{O}_X$. Now we use that $u^*$ fixes the ring $R_1$, this implies that for any $s$ in $R_1$ and any section $g$ of $\mathcal{O}_{\mathfrak{X}}$, 
\[
u^*(sg)-sg=u^*(s) u^*(g) - sg =s \times \underbrace{(u^*(g)-g)}_{\in\, (t)} =0.
\]
Hence $\delta$ factors through a map from $\mathcal{O}_X$ to itself. We claim that this map is a derivation. Indeed, 
if $g_1$ and $g_2$ are sections of $\mathcal{O}_X$ and $\tilde{g}_1$, $\tilde{g}_2$ are lifts of $g_1$ and $g_2$ respectively on  $\mathfrak{X}$, then
\begin{align*}
\delta(\tilde{g}_1 \tilde{g}_2)&=u^*(\tilde{g}_1 \tilde{g}_2)-\tilde{g}_1 \tilde{g}_2 \\
&=u^*(\tilde{g}_1) (u^*(\tilde{g}_2)-\tilde{g}_2) +(u^*(\tilde{g}_1)-\tilde{g}_1) \tilde{g}_2 \\
&= \tilde{g}_1 (u^*(\tilde{g}_2)-\tilde{g}_2) +(u^*(\tilde{g}_1)-\tilde{g}_1) \tilde{g}_2 +  \underbrace{(u^*(\tilde{g}_1)-\tilde{g}_1) (u^*(\tilde{g}_2)-\tilde{g}_2) }_{0}\\
&=g_1 \delta(g_2)+ g_2 \delta(g_1).
\end{align*}
We are now done: $\delta$ is a globally defined derivation of $\mathcal{O}_X$, \textit{i.e.} a global holomorphic vector field. By hypothesis, $\delta$ must be zero, so $u^*$ acts by the identity on $\mathcal{O}_{\mathfrak{X}}$.
\end{proof}

\bibliographystyle{plain}
\bibliography{biblioexemple}
\nocite{*}

\clearpage
\phantomsection
\addcontentsline{toc}{section}{Maple computations}


\section*{Supplementary material}

# Maple computations

# Contents

# 1 The example of $\mathbb{P}^2$ blown up in 15 points

```
> restart:
> with(LinearAlgebra):
```

## 1.1 First definitions

```
> restart:
> with(LinearAlgebra):
```

The function `f` is the function $\phi$ of the text given in the affine coordinates $(y, z)$.

```
> f1:=(y,z)->(y*z^2)/(z^2+y^3):
> f2:=(y,z)->(z^3)/(z^2+y^3):
> f:=(y,z)->(f1(y,z),f2(y,z)):
```

The inverse of `f` is denoted by `g`.

```
> g1:=(y,z)->y*z^2/(z^2-y^3):
> g2:=(y,z)->z^3/(z^2-y^3):
> g:=(y,z)->(g1(y,z),g2(y,z)):
```

We define the action of the matrix `A` in the coordinates $(y, z)$, its inverse is `B`.

```
> A:=(y,z)->((-1+(1+alpha)*z)/(alpha+2*(1-alpha)*y+(2+alpha-alpha^2)*z),(1-2*y+(1-alpha)*
z)/(alpha+2*(1-alpha)*y+(2+alpha-alpha^2)*z)):
> B:=(y,z)->((1-y-(1+alpha)*z)/(1+alpha+(alpha-3)*y+(1-alpha^2)*z),(1+y+(1-alpha)*z)/(1+alpha+
(alpha-3)*y+(1-alpha^2)*z)):
```

We define $m = \mathtt{AfAfA}$ and $n = \mathtt{BgBgB}$.

```
> m1:=(y,z)->op(1,[A(f(A(f(A(y,z))))))]):
> m2:=(y,z)->op(2,[A(f(A(f(A(y,z))))))]):
> m:=(y,z)->(m1(y,z), m2(y,z)):
> n1:=(y,z)->op(1,[B(g(B(g(B(y,z))))))]):
> n2:=(y,z)->op(2,[B(g(B(g(B(y,z))))))]):
> n:=(y,z)->(n1(y,z), n2(y,z)):
```

## 1.2 The procedures to compute direct images of vector fields

### 1.2.1 The procedures `champ`

The procedure `champ1` associates to a vector field $Z$ the coordinates of the Taylor expansion of the vector field $A_* Z$ in the following order: $\partial_y,\, y\,\partial_y,\, z\,\partial_y,\, \partial_z,\, y\,\partial_z,\, z\,\partial_z,\, y^2\,\partial_z$. The variable of the procedure is not the vector field itself but a curve $t \mapsto M_t$ whose derivate at 0 is $Z$. For example, if $Z = u(y,z)\partial_y + v(y,z)\partial_z$, we can take $M_t(y, z) = (u(y, z) + ty, v(y, z) + tz)$.

```
> champ1:=proc(M)
> local prod,res,res1,res2,u,v:
> global resu1,resu2,w:
> prod:=unapply([(A@(M@B))(y,z)],t):
> res:=simplify(subs(t=0,diff(prod(t),t))):
> res1:=unapply(op(1,res),(y,z)):
> res2:=unapply(op(2,res),(y,z)):
> resu1:=[subs({y=0,z=0},res1(y,z)),subs({y=0,z=0},diff(res1(y,z),y)),subs({y=0,z=0},diff(
res1(y,z),z))]:
> resu2:=[subs({y=0,z=0},res2(y,z)),subs({y=0,z=0},diff(res2(y,z),y)),subs({y=0,z=0},diff(
res2(y,z),z)),subs({y=0,z=0},diff(diff(res2(y,z),y),y)/2)]:
> u:=op(resu1):v:=op(resu2):
> w:=[u,v]:
> RETURN(convert(w,Vector))
> end:
```

The procedure `champ2` associates to a vector field $Z$ the coordinates of the Taylor expansion of the vector field $(\phi A)_*^{-1} Z$ in the following order: $\partial_y,\, y\,\partial_y,\, z\,\partial_y,\, \partial_z,\, y\,\partial_z,\, z\,\partial_z,\, y^2\,\partial_z$.

```
> champ2:=proc(M)
> local prod,res,res1,res2,u,v:
> global resu1,resu2,w:
> prod:=unapply([(B@(g@(M@(f@A))))(y,z)],t):
> res:=simplify(subs(t=0,diff(prod(t),t))):
> res1:=unapply(op(1,res),(y,z)):
> res2:=unapply(op(2,res),(y,z)):
```

```
>  resu1:=[subs({y=0,z=0},res1(y,z)),subs({y=0,z=0},diff(res1(y,z),y)),subs({y=0,z=0},diff(
res1(y,z),z))]:
>  resu2:=[subs({y=0,z=0},res2(y,z)),subs({y=0,z=0},diff(res2(y,z),y)),subs({y=0,z=0},diff(
res2(y,z),z)),subs({y=0,z=0},diff(diff(res2(y,z),y),y)/2)]:
>  u:=op(resu1):v:=op(resu2):
>  w:=[u,v]:
>  RETURN(convert(w,Vector))
>  end:
```

The procedure `champ3` associates to a vector field $Z$ the coordinates of the Taylor expansion of the vector field $(A\phi A)_*^{-1}(Z)$ in the following order: $\partial_y$, $y\,\partial_y$, $z\,\partial_y$, $\partial_z$, $y\,\partial_z$, $z\,\partial_z$, $y^2\,\partial_z$.

```
>  champ3:=proc(M)
>  local prod,res,res1,res2:
>  global resu1,resu2:
>  prod:=unapply([(B@(g@(B@(M@(A@(f@A))))))(y,z)],t):
>  res:=simplify(subs(t=0,diff(prod(t),t))):
>  res1:=unapply(op(1,res),(y,z)):
>  res2:=unapply(op(2,res),(y,z)):
>  resu1:=[subs({y=0,z=0},res1(y,z)),subs({y=0,z=0},diff(res1(y,z),y)),subs({y=0,z=0},diff(
res1(y,z),z))]:
>  resu2:=[subs({y=0,z=0},res2(y,z)),subs({y=0,z=0},diff(res2(y,z),y)),subs({y=0,z=0},diff(
res2(y,z),z)),subs({y=0,z=0},diff(diff(res2(y,z),y),y)/2)]:
>  RETURN(Transpose(convert([seq(factor(op(k,resu1)),k=1..nops(resu1)),seq(factor(op(k,resu2)
),k=1..nops(resu2))],Matrix))):
>  end:
```

The procedure `champ4` associates to a vector field $Z$ the coordinates of the Taylor expansion of the vector field $A_*^{-1}(Z)$ in the following order: $\partial_y$, $y\partial_y$, $z\partial_y$, $\partial_z$, $y\partial_z$, $z\partial_z$, $y^2\partial_z$.

```
>  champ4:=proc(M)
>  local prod,res,res1,res2,u,v:
>  global resu1,resu2,w:
>  prod:=unapply([(B@(M@A))(y,z)],t):
>  res:=simplify(subs(t=0,diff(prod(t),t))):
>  res1:=unapply(op(1,res),(y,z)):
>  res2:=unapply(op(2,res),(y,z)):
>  resu1:=[subs({y=0,z=0},res1(y,z)),subs({y=0,z=0},diff(res1(y,z),y)),subs({y=0,z=0},diff(
res1(y,z),z))]:
>  resu2:=[subs({y=0,z=0},res2(y,z)),subs({y=0,z=0},diff(res2(y,z),y)),subs({y=0,z=0},diff(
res2(y,z),z)),subs({y=0,z=0},diff(diff(res2(y,z),y),y)/2)]:
>  u:=op(resu1):v:=op(resu2):
>  w:=[u,v]:
>  RETURN(convert(w,Vector))
>  end:
```

The procedure `champ5` associates to a vector field $Z$ the coordinates of the Taylor expansion of the vector field $(A\phi A\phi A)_*^{-1}(Z)$ in the following order: $\partial_y$, $y\,\partial_y$, $z\,\partial_y$, $\partial_z$, $y\,\partial_z$, $z\,\partial_z$, $y^2\,\partial_z$.

```
>  champ5:=proc(M)
>  local prod,res,res1,res2,u,v:
>  global resu1,resu2,w:
>  prod:=unapply([(m@(M@n))(y,z)],t):
>  res:=(subs(t=0,diff(prod(t),t))):
>  res1:=unapply(op(1,res),(y,z)):
>  res2:=unapply(op(2,res),(y,z)):
>  resu1:=[subs({y=0,z=0},res1(y,z)),subs({y=0,z=0},diff(res1(y,z),y)),subs({y=0,z=0},diff(
res1(y,z),z))]:
>  resu2:=[subs({y=0,z=0},res2(y,z)),subs({y=0,z=0},diff(res2(y,z),y)),subs({y=0,z=0},diff(
res2(y,z),z)),subs({y=0,z=0},diff(diff(res2(y,z),y),y)/2)]:
>  u:=op(resu1):v:=op(resu2):
>  w:=[u,v]:
>  RETURN(convert(w,Vector))
>  end:
```

This procedure is too heavy for Maple, we will not use it. To replace it, we will do successive Taylor expansions at each step of blow-ups. We deal with it in the next section.

### 1.2.2 The procedures `devlim` and `extract`

We compute the points where we take the Taylor expansions.

```
>   A(0,0);
```

$-\alpha^{-1}, \alpha^{-1}$

```
>   simplify(op(1,[(A@f@A)(0,0)])), simplify(op(2,[(A@f@A)(0,0)]));
```

$\alpha^{-1}, \alpha^{-1}$

```
>   simplify(op(1,[(A@f@A@f@A)(0,0)])), simplify(op(2,[(A@f@A@f@A)(0,0)]));
```

$0, 0$

The procedure `devlim1` associates to a vector field $Z$ in coordinates $(y, z)$ (still given by the derivative at 0 of a curve $M_t$) the Taylor expansion of $A_*Z$ of order 2 at a point $(p, q)$. The result is viewed as a polynomial vector field in a neighborhood of $(p, q)$ and the procedure returns a curve $N_t$ associated to this vector field.

```
>   devlim1:=proc(M,p,q)
>   local prod,res,res1,res2,resu1,resu2,N,MMM:
>   MMM:=(y,z)->(op(M(y,z))):
>   prod:=unapply([(A@(MMM@B))(y,z)],t):
>   res:=simplify(subs(t=0,diff(prod(t),t))):
>   res1:=unapply(op(1,res),(y,z)):
>   res2:=unapply(op(2,res),(y,z)):
>   resu1:=simplify(mtaylor(res1(y,z),[y=p,z=q],3)):
>   resu2:=simplify(mtaylor(res2(y,z),[y=p,z=q],3)):
>   N:=unapply([y+t*resu1,z+t*resu2],(y,z)):
>   RETURN(N):
>   end:
```

The procedure `devlim2` is similar to devlim1 but computes the Taylor expansion of $(A\phi)_*Z$ instead of the Taylor expansion of $A_*Z$.

```
>   devlim2:=proc(M,p,q)
>   local prod,res,res1,res2,resu1,resu2,N,MMM:
>   MMM:=(y,z)->(op(M(y,z))):
>   prod:=unapply([(A@f@(MMM@g@B))(y,z)],t):
>   res:=simplify(subs(t=0,diff(prod(t),t))):
>   res1:=unapply(op(1,res),(y,z)):
>   res2:=unapply(op(2,res),(y,z)):
>   resu1:=simplify(mtaylor(res1(y,z),[y=p,z=q],3)):
>   resu2:=simplify(mtaylor(res2(y,z),[y=p,z=q],3)):
>   N:=unapply([y+t*resu1,z+t*resu2],(y,z)):
>   RETURN(N):
>   end:
```

The procedure `extract` associates to a vector field given by a curve $N_t(y, z)$ the components in $\partial_y$, $y\partial_y$, $z\partial_y$, $\partial_z$, $y\partial_z$, $z\partial_z$ and $y^2\partial_z$ in the Taylor expansion of the associated vector field at $(0,0)$.

```
>   extract:=proc(M)
>   local prod,res,res1,res2,w:
>   prod:=unapply(M(y,z),t):
>   res:=simplify(subs(t=0,diff(prod(t),t))):
>   res1:=unapply(op(1,res),(y,z)):
>   res2:=unapply(op(2,res),(y,z)):
>   w[1]:=simplify(res1(0,0)):
>   w[2]:=simplify(subs({y=0,z=0},diff(res1(y,z),y))):
>   w[3]:=simplify(subs({y=0,z=0},diff(res1(y,z),z))):
>   w[4]:=simplify(res2(0,0)):
>   w[5]:=simplify(subs({y=0,z=0},diff(res2(y,z),y))):
>   w[6]:=simplify(subs({y=0,z=0},diff(res2(y,z),z))):
>   w[7]:=simplify(subs({y=0,z=0},diff(res2(y,z),y$2)/2)):
>   w:=[w[1],w[2],w[3],w[4],w[5],w[6],w[7]]:
>   RETURN(convert(w,Vector)):
>   end:
```

The function $\text{psi}_1$ is the projection of the blow-up of $(0,0)_{y,z}$, and $\text{eta}_1$ is its inverse.

```
>   psi1:=(u,v)->(u,u*v):
>   eta1:=(y,z)->(y,z/y):
```

The function $\text{psi}_2$ is the projection of the blow-up of $(-1/\alpha, 1/\alpha)_{y,z}$ and $\text{eta}_2$ is its inverse.

```
> psi2:=(u,v)->(u-1/alpha,u*v+1/alpha):
> eta2:=(y,z)->(y+1/alpha,(z-1/alpha)/(y+1/alpha)):
```

The function `psi3` is the projection of the blow-up of $(1/\alpha, 1/\alpha)_{y,z}$ and `eta3` is its inverse.

```
> psi3:=(u,v)->(u+1/alpha, u*v+1/alpha):
> eta3:=(y,z)->(y-1/alpha,(z-1/alpha)/(y-1/alpha)):
```

The procedure `u1devlim1` associates to a vector field $Z$ on $\mathbb{P}^2$ blown up at $(0,0)_{y,z}$ written in the coordinates $(u_1, v_1)$ the Taylor expansion of $A_* Z$ of order 3 at a point $(p, q)$. The result is considered as a polynomial vector field in a neighborhood of $(p,q)$ in the coordinates $(u_1, v_1)$ which are this time the coordinates on $\mathbb{P}^2$ blown up at $\left(-\frac{1}{\alpha}, \frac{1}{\alpha}\right)_{y,z}$, and the procedure gives a curve $N_t$ associated to this vector field.

```
> u1devlim1:=proc(M,p,q)
> local prod,res,res1,res2,resu1,resu2,N,MMM:
> MMM:=(u,v)->(op(M(u,v))):
> prod:=unapply([(eta2@(A@(psi1@MMM@eta1@(B@psi2))))(u,v)],t):
> res:=simplify(subs(t=0,diff(prod(t),t))):
> res1:=unapply(simplify(op(1,res)),(u,v)):
> res2:=unapply(simplify(op(2,res)),(u,v)):
> resu1:=simplify(mtaylor(res1(u,v),[u=p,v=q],4)):
> resu2:=simplify(mtaylor(res2(u,v),[u=p,v=q],4)):
> N:=unapply([u+t*resu1,v+t*resu2],(u,v)):
> RETURN(N):
> end:
```

We compute the points where we take the Taylor expansions.

```
> limit(op(1,[(eta2@(A@psi1))(u,v)]),{u=0,v=0}), limit(op(2,[(eta2@(A@psi1))(u,v)]),{u=0,
v=0});
```

$0, -(1-\alpha)^{-1}$

```
> limit(op(1,[(eta3@(A@f@A))@psi1)(u,v)]),{u=0,v=0}), limit(op(2,[(eta3@(A@f@A)@psi1)(u,v)
]),{u=0,v=0});
```

$0, -(-1-\alpha)^{-1}$

```
> limit(op(1,[(eta1@(A@f@A@f@A)@psi1)(u,v)]),{u=0,v=0}), limit(op(2,[(eta1@(A@f@A@f@A)@psi1)
(u,v)]),{u=0,v=0});
```

$0, 0$

The procedure `u1devlim2` associates to a vector field $Z$ on the blow up of $\mathbb{P}^2$ at $\left(-\frac{1}{\alpha}, \frac{1}{\alpha}\right)_{y,z}$ expressed in the coordinates $(u_1, v_1)$ the Taylor expansion of $(A\phi)_* Z$ of order 3 at a point $(p, q)$. The result is viewed as a polynomial vector field in a neighborhood of $(p, q)$ in the coordinates $(u_1, v_1)$ which are this time some coordinates on $\mathbb{P}^2$ blown up at $\left(\frac{1}{\alpha}, \frac{1}{\alpha}\right)_{y,z}$, and the procedure returns a curve $N_t$ associated to this vector field.

```
> u1devlim2:=proc(M,p,q)
> local prod,res,res1,res2,resu1,resu2,N,MMM:
> MMM:=(u,v)->(op(M(u,v))):
> prod:=unapply([(eta3@(A@f@(psi2@MMM@eta2@g@B@psi3)))(u,v)],t):
> res:=simplify(subs(t=0,diff(prod(t),t))):
> res1:=unapply(simplify(op(1,res)),(u,v)):
> res2:=unapply(simplify(op(2,res)),(u,v)):
> resu1:=simplify(mtaylor(res1(u,v),[u=p,v=q],4)):
> resu2:=simplify(mtaylor(res2(u,v),[u=p,v=q],4)):
> N:=unapply([u+t*resu1,v+t*resu2],(u,v)):
> RETURN(N):
> end:
```

The procedure `u1devlim3` associates to a vector field $Z$ on the blow up of $\mathbb{P}^2$ at $\left(\frac{1}{\alpha}, \frac{1}{\alpha}\right)_{y,z}$ expressed in the coordinates $(u_1, v_1)$ the Taylor expansion of $(A\phi)_* Z$ of order 3 at a point $(p, q)$. The result is viewed as a polynomial vector field in a neighborhood of $(p, q)$ in the coordinates $(u_1, v_1)$ which are this time coordinates on $\mathbb{P}^2$ blown up at $(0, 0)_{y,z}$, and the procedure gives a curve $N_t$ associated to this vector field.

```
> u1devlim3:=proc(M,p,q)
> local prod,res,res1,res2,resu1,resu2,N,MMM:
> MMM:=(u,v)->(op(M(u,v))):
> prod:=unapply([(eta1@(A@f@(psi3@MMM@eta3@g@B@psi1)))(u,v)],t):
> res:=simplify(subs(t=0,diff(prod(t),t))):
> res1:=unapply(simplify(op(1,res)),(u,v)):
> res2:=unapply(simplify(op(2,res)),(u,v)):
> resu1:=simplify(mtaylor(res1(u,v),[u=p,v=q],4)):
```

```
>    resu2:=simplify(mtaylor(res2(u,v),[u=p,v=q],4)):
>    N:=unapply([u+t*resu1,v+t*resu2],(u,v)):
>    RETURN(N):
>    end:
```

The procedure `u1extract` associates to a vector field given by a curve $N_t(u_1, v_1)$ the components in $\partial_{u_1}$, $u_1\,\partial_{u_1}$, $v_1\,\partial_{u_1}$, $u_1 v_1\,\partial_{u_1}$, $v_1^2\,\partial_{u_1}$, $v_1^3\,\partial_{u_1}$, $\partial_{v_1}$, $u_1\,\partial_{v_1}$, $v_1\,\partial_{v_1}$ and $v_1^2\,\partial_{v_1}$ of the Taylor expansion of the associated vector field at $(0,0)$.

```
>    u1extract:=proc(M)
>    local prod,res,res1,res2,w:
>    prod:=unapply(M(u,v),t):
>    res:=simplify(subs(t=0,diff(prod(t),t))):
>    res1:=unapply(op(1,res),(u,v)):
>    res2:=unapply(op(2,res),(u,v)):
>    w[1]:=simplify(res1(0,0)):
>    w[2]:=simplify(subs({u=0,v=0},diff(res1(u,v),u))):
>    w[3]:=simplify(subs({u=0,v=0},diff(res1(u,v),v))):
>    w[4]:=simplify(subs({u=0,v=0},diff(res1(u,v),u,v))):
>    w[5]:=simplify(subs({u=0,v=0},diff(res1(u,v),v$2))/2):
>    w[6]:=simplify(subs({u=0,v=0},diff(res1(u,v),v$3))/3):
>    w[7]:=simplify(res2(0,0)):
>    w[8]:=simplify(subs({u=0,v=0},diff(res2(u,v),u))):
>    w[9]:=simplify(subs({u=0,v=0},diff(res2(u,v),v))):
>    w[10]:=simplify(subs({u=0,v=0},diff(res2(u,v),v$2))/2):
>    w:=[w[1],w[2],w[3],w[4],w[5],w[6],w[7],w[8],w[9],w[10]]:
>    RETURN(convert(w,Vector)):
>    end:
```

The function $\mathtt{ppsi}_1$ is the projection of the blow up of $(0,0)_{y,z}$ and $(0,0)_{u_1,v_1}$, and $\mathtt{eeta}_1$ is its inverse.

```
>    ppsi1:=(r,s)->(r*s,r*s^2):
>    eeta1:=(y,z)->(y^2/z,z/y):
```

The function $\mathtt{ppsi}_2$ is the projection of the blow up of $\left(-\frac{1}{\alpha}, \frac{1}{\alpha}\right)_{y,z}$ and $\left(0, \frac{1}{\alpha-1}\right)_{u_1,v_1}$, and $\mathtt{eeta}_2$ is its inverse.

```
>    ppsi2:=(r,s)->(r*s-1/alpha, r*s*(s+1/(alpha-1))+1/alpha):
>    eeta2:=(y,z)->((y+1/alpha)/((z-1/alpha)/(y+1/alpha)-1/(alpha-1)),(z-1/alpha)/(y+1/alpha)
-1/(alpha-1)):
```

The function $\mathtt{ppsi}_3$ is the projection of the blow up of $\left(\frac{1}{\alpha}, \frac{1}{\alpha}\right)_{y,z}$ and $\left(0, \frac{1}{\alpha+1}\right)_{u_1,v_1}$, and $\mathtt{eeta}_3$ is its inverse.

```
>    ppsi3:=(r,s)->(r*s+1/alpha, r*s*(s+1/(alpha+1))+1/alpha):
>    eeta3:=(y,z)->((y-1/alpha)/((z-1/alpha)/(y-1/alpha)-1/(alpha+1)),(z-1/alpha)/(y-1/alpha)
-1/(alpha+1)):
```

We compute the points where we take the Taylor expansions.

```
>    limit(op(1,[(eeta2@(A@ppsi1))(r,s)]),{r=0,s=0}), limit(op(2,[(eeta2@(A@ppsi1))(r,s)]),{r=0,
s=0});
```

$0, 0$

```
>    limit(op(1,[(eeta3@(A@f@A)@ppsi1)(r,s)]),{r=0,s=0}), limit(op(2,[(eeta3@(A@f@A)@ppsi1)(
r,s)]),{r=0,s=0});
```

$0, 0$

```
>    limit(op(1,[(eeta1@(A@f@A@f@A)@ppsi1)(r,s)]),{r=0,s=0}), limit(op(2,[(eeta1@(A@f@A@f@A)
@ppsi1)(r,s)]),{r=0,s=0});
```

$0, 0$

The procedure `r2devlim1` associates to a vector field $Z$ on the blow up of $\mathbb{P}^2$ at $(0,0)_{y,z}$ and $(0,0)_{u_1,v_1}$ expressed in coordinates $(r_2, s_2)$ the Taylor expansion of $A_*Z$ of order 2 at a point $(p,q)$. The result is viewed as a polynomial vector field in a neighborhood of $(p,q)$ in the coordinates $(r_2, s_2)$ which are this time coordinates on $\mathbb{P}^2$ blown up at $\left(-\frac{1}{\alpha}, \frac{1}{\alpha}\right)_{y,z}$ and $\left(0, \frac{1}{\alpha-1}\right)_{u_1,v_1}$, and the procedure returns a curve $N_t$ associated to this vector field.

```
>    r2devlim1:=proc(M,p,q)
>    local prod,res,res1,res2,resu1,resu2,N,MMM:
>    MMM:=(r,s)->(op(M(r,s))):
>    prod:=unapply([(eeta2@(A@(ppsi1@MMM@eeta1@(B@ppsi2))))(r,s)],t):
>    res:=simplify(subs(t=0,diff(prod(t),t))):
```

```
>    res1:=unapply(simplify(op(1,res)),(r,s)):
>    res2:=unapply(simplify(op(2,res)),(r,s)):
>    resu1:=simplify(mtaylor(res1(r,s),[r=p,s=q],3)):
>    resu2:=simplify(mtaylor(res2(r,s),[r=p,s=q],3)):
>    N:=unapply([r+t*resu1,s+t*resu2],(r,s)):
>    RETURN(N):
>    end:
```

The procedure `r2devlim2` associates to a vector field $Z$ on the blow up of $\mathbb{P}^2$ at $\left(-\frac{1}{\alpha}, \frac{1}{\alpha}\right)_{y,z}$ and $\left(0, \frac{1}{\alpha-1}\right)_{u_1,v_1}$ expressed in coordinates $(r_2, s_2)$ the Taylor expansion of $(A\phi)_* Z$ of order 2 at a point $(p, q)$. The result is viewed as a polynomial vector field in a neighborhood of $(p, q)$ in the coordinates $(r_2, s_2)$ which are this time the coordinates of $\mathbb{P}^2$ blown up at $\left(\frac{1}{\alpha}, \frac{1}{\alpha}\right)_{y,z}$ and $\left(0, \frac{1}{\alpha+1}\right)_{u_1,v_1}$, and the procedure returns a curve $N_t$ associated to this vector field.

```
>    r2devlim2:=proc(M,p,q)
>    local prod,res,res1,res2,resu1,resu2,N,MMM:
>    MMM:=(r,s)->(op(M(r,s))):
>    prod:=unapply([(eeta3@(A@f@(ppsi2@MMM@eeta2@g@B@ppsi3)))(r,s)],t):
>    res:=simplify(subs(t=0,diff(prod(t),t))):
>    res1:=unapply(simplify(op(1,res)),(r,s)):
>    res2:=unapply(simplify(op(2,res)),(r,s)):
>    resu1:=simplify(mtaylor(res1(r,s),[r=p,s=q],3)):
>    resu2:=simplify(mtaylor(res2(r,s),[r=p,s=q],3)):
>    N:=unapply([r+t*resu1,s+t*resu2],(r,s)):
>    RETURN(N):
>    end:
```

The procedure `r2devlim3` associates to a vector field $Z$ on the blow up of $\mathbb{P}^2$ at $\left(\frac{1}{\alpha}, \frac{1}{\alpha}\right)_{y,z}$ and $\left(0, \frac{1}{\alpha+1}\right)_{u_1,v_1}$ expressed in coordinates $(r_2, s_2)$ the Taylor expansion of $(A\phi)_* Z$ of order 2 at a point $(p, q)$. The result is viewed as a polynomial vector field in a neighborhood of $(p, q)$ in the coordinates $(r_2, s_2)$ which are this time the coordinates of $\mathbb{P}^2$ blown up at $(0, 0)_{y,z}$ and $(0, 0)_{u_1,v_1}$, and the procedure returns a curve $N_t$ associated to this vector field.

```
>    r2devlim3:=proc(M,p,q)
>    local prod,res,res1,res2,resu1,resu2,N,MMM:
>    MMM:=(r,s)->(op(M(r,s))):
>    prod:=unapply([(eeta1@(A@f@(ppsi3@MMM@eeta3@g@B@ppsi1)))(r,s)],t):
>    res:=simplify(subs(t=0,diff(prod(t),t))):
>    res1:=unapply(simplify(op(1,res)),(r,s)):
>    res2:=unapply(simplify(op(2,res)),(r,s)):
>    resu1:=simplify(mtaylor(res1(r,s),[r=p,s=q],3)):
>    resu2:=simplify(mtaylor(res2(r,s),[r=p,s=q],3)):
>    N:=unapply([r+t*resu1,s+t*resu2],(r,s)):
>    RETURN(N):
>    end:
```

The procedure `r2extract` associates to a vector field given by a curve $N_t(r_2, s_2)$ the components in $\partial_{r_2}$, $r_2 \partial_{r_2}$, $s_2 \partial_{r_2}$, $r_2 s_2 \partial_{r_2}$, $r_2^2 \partial_{r_2}$, $s_2^2 \partial_{r_2}$, $\partial_{s_2}$, $r_2 \partial_{s_2}$, $s_2 \partial_{s_2}$, $r_2 s_2 \partial_{s_2}$, $r_2^2 \partial_{s_2}$ and $s_2^2 \partial_{s_2}$ in the Taylor expansion of the associated vector field at $(0,0)$.

```
>    r2extract:=proc(M)
>    local prod,res,res1,res2,w:
>    prod:=unapply(M(r,s),t):
>    res:=simplify(subs(t=0,diff(prod(t),t))):
>    res1:=unapply(op(1,res),(r,s)):
>    res2:=unapply(op(2,res),(r,s)):
>    w[1]:=simplify(res1(0,0)):
>    w[2]:=simplify(subs({r=0,s=0},diff(res1(r,s),r))):
>    w[3]:=simplify(subs({r=0,s=0},diff(res1(r,s),s))):
>    w[4]:=simplify(subs({r=0,s=0},diff(res1(r,s),r,s))):
>    w[5]:=simplify(subs({r=0,s=0},diff(res1(r,s),r$2))/2):
>    w[6]:=simplify(subs({r=0,s=0},diff(res1(r,s),s$2))/2):
>    w[7]:=simplify(res2(0,0)):
>    w[8]:=simplify(subs({r=0,s=0},diff(res2(r,s),r))):
>    w[9]:=simplify(subs({r=0,s=0},diff(res2(r,s),s))):
>    w[10]:=simplify(subs({r=0,s=0},diff(res2(r,s),r,s))):
>    w[11]:=simplify(subs({r=0,s=0},diff(res2(r,s),r$2))/2):
>    w[12]:=simplify(subs({r=0,s=0},diff(res2(r,s),s$2))/2):
>    w:=[w[1],w[2],w[3],w[4],w[5],w[6],w[7],w[8],w[9],w[10],w[11],w[12]]:
>    RETURN(convert(w,Vector)):
>    end:
```

The function pppsi$_1$ is the projection of the blow up of $(0,0)_{y,z}$, $(0,0)_{u_1,v_1}$ and $(0,0)_{r_2,s_2}$, and eeeta$_1$ is its inverse.

```
> pppsi1:=(r,s)->(r*s*s,r*s*s^2):
> eeeta1:=(y,z)->((y^2/z)/(z/y),(z/y)):
```

The function pppsi$_2$ is the projection of the blow up of $\left(-\frac{1}{\alpha},\frac{1}{\alpha}\right)_{y,z}$, $\left(0,\frac{1}{\alpha-1}\right)_{u_1,v_1}$ and $(0,0)_{r_2,s_2}$, and eeeta$_2$ is its inverse.

```
> pppsi2:=(r,s)->(r*s*s-1/alpha, r*s*s*(s+1/(alpha-1))+1/alpha):
> eeeta2:=(y,z)->(((y+1/alpha)/((z-1/alpha)/(y+1/alpha)-1/(alpha-1)))/((z-1/alpha)/(y+1/alpha)
-1/(alpha-1)),((z-1/alpha)/(y+1/alpha)-1/(alpha-1))):
```

The function pppsi$_3$ is the projection of the blow up of $\left(\frac{1}{\alpha},\frac{1}{\alpha}\right)_{y,z}$, $\left(0,\frac{1}{\alpha+1}\right)_{u_1,v_1}$ and $(0,0)_{r_2,s_2}$, and eeeta$_3$ is its inverse.

```
> pppsi3:=(r,s)->(r*s*s+1/alpha, r*s*s*(s+1/(alpha+1))+1/alpha):
> eeeta3:=(y,z)->(((y-1/alpha)/((z-1/alpha)/(y-1/alpha)-1/(alpha+1)))/((z-1/alpha)/(y-1/alpha)
-1/(alpha+1)),((z-1/alpha)/(y-1/alpha)-1/(alpha+1))):
```

We compute the points where we take the Taylor expansions.

```
> limit(op(1,[(eeeta2@(A@pppsi1))(r,s)]),{r=1,s=0}), limit(op(2,[(eeeta2@(A@pppsi1))(r,s)
]),{r=1,s=0});
```

$$\frac{-(\alpha-1)^5}{2\alpha^4}, 0$$

```
> limit(op(1,[(eeeta3@(A@f@A)@pppsi1)(r,s)]),{r=1,s=0}), limit(op(2,[(eeeta3@(A@f@A)@pppsi1)
(r,s)]),{r=1,s=0});
```

$$\frac{(1+\alpha)^5}{2\alpha^4}, 0$$

```
> limit(op(1,[(eeeta1@(A@f@A@f@A)@pppsi1)(r,s)]),{r=1,s=0}), limit(op(2,[(eeeta1@(A@f@A@f@A)
@pppsi1)(r,s)]),{r=1,s=0});
```

$-1, 0$

The procedure r3devlim1 associates to a vector field $Z$ on the blow up of $\mathbb{P}^2$ at $(0,0)_{y,z}$, $(0,0)_{u_1,v_1}$ and $(0,0)_{r_2,s_2}$ expressed in coordinates $(r_3,s_3)$ the Taylor expansion of $A_*Z$ of order 1 at a point $(p,q)$. The result is viewed as a polynomial vector field in a neighborhood of $(p,q)$ in the coordinates $(r_3,s_3)$ which are this time the coordinates on $\mathbb{P}^2$ blown up at $\left(-\frac{1}{\alpha},\frac{1}{\alpha}\right)_{y,z}$, $\left(0,\frac{1}{\alpha-1}\right)_{u_1,v_1}$ and $(0,0)_{r_2,s_2}$, and the procedure gives a curve $N_t$ associated to this vector field.

```
> r3devlim1:=proc(M,p,q)
> local prod,res,res1,res2,resu1,resu2,N,MMM:
> MMM:=(r,s)->(op(M(r,s))):
> prod:=unapply([(eeeta2@(A@(pppsi1@MMM@eeeta1@(B@pppsi2))))(r,s)],t):
> res:=simplify(subs(t=0,diff(prod(t),t))):
> res1:=unapply(simplify(op(1,res)),(r,s)):
> res2:=unapply(simplify(op(2,res)),(r,s)):
> resu1:=simplify(mtaylor(res1(r,s),[r=p,s=q],2)):
> resu2:=simplify(mtaylor(res2(r,s),[r=p,s=q],2)):
> N:=unapply([r+t*resu1,s+t*resu2],(r,s)):
> RETURN(N):
> end:
```

The procedure r3devlim2 associates to a vector field $Z$ on the blow up of $\mathbb{P}^2$ at $\left(-\frac{1}{\alpha},\frac{1}{\alpha}\right)_{y,z}$, $\left(0,\frac{1}{\alpha-1}\right)_{u_1,v_1}$ and $(0,0)_{r_2,s_2}$ expressed in the coordinates $(r_3,s_3)$ the Taylor expansion of $(A\phi)_*Z$ of order 1 at a point $(p,q)$. The result is viewed as a polynomial vector field in a neighborhood of $(p,q)$ in the coordinates $(r_3,s_3)$ which are this time the coordinates of $\mathbb{P}^2$ blown up at $\left(\frac{1}{\alpha},\frac{1}{\alpha}\right)_{y,z}$, $\left(0,\frac{1}{\alpha+1}\right)_{u_1,v_1}$ and $(0,0)_{r_2,s_2}$, and the procedure returns a curve $N_t$ associated to this vector field.

```
> r3devlim2:=proc(M,p,q)
> local prod,res,res1,res2,resu1,resu2,N,MMM:
> MMM:=(r,s)->(op(M(r,s))):
> prod:=unapply([(eeeta3@(A@f@(pppsi2@MMM@eeeta2@g@B@pppsi3)))(r,s)],t):
> res:=simplify(subs(t=0,diff(prod(t),t))):
> res1:=unapply(simplify(op(1,res)),(r,s)):
> res2:=unapply(simplify(op(2,res)),(r,s)):
> resu1:=simplify(mtaylor(res1(r,s),[r=p,s=q],2)):
> resu2:=simplify(mtaylor(res2(r,s),[r=p,s=q],2)):
```

```
>   N:=unapply([r+t*resu1,s+t*resu2],(r,s)):
>   RETURN(N):
>   end:
```

The procedure `r3devlim3` associates to a vector field $Z$ on the blow up of $\mathbb{P}^2$ at $\left(\frac{1}{\alpha}, \frac{1}{\alpha}\right)_{y,z}$, $\left(0, \frac{1}{\alpha+1}\right)_{u_1,v_1}$ and $(0,0)_{r_2,s_2}$ expressed in coordinates $(r_3, s_3)$ the Taylor expansion of $(A\phi)_* Z$ of order 1 at a point $(p, q)$. The result is viewed as a polynomial vector field in a neighborhood of $(p, q)$ in the coordinates $(r_3, s_3)$ which are this time the coordinates on $\mathbb{P}^2$ blown up at $(0,0)_{y,z}$, $(0,0)_{u_1,v_1}$ and $(0,0)_{r_2,s_2}$, and the procedure returns a curve $N_t$ associated to this vector field.

```
>   r3devlim3:=proc(M,p,q)
>   local prod,res,res1,res2,resu1,resu2,N,MMM:
>   MMM:=(r,s)->(op(M(r,s))):
>   prod:=unapply([(eeeta1@(A@f@(pppsi3@MMM@eeeta3@g@B@pppsi1)))(r,s)],t):
>   res:=simplify(subs(t=0,diff(prod(t),t))):
>   res1:=unapply(simplify(op(1,res)),(r,s)):
>   res2:=unapply(simplify(op(2,res)),(r,s)):
>   resu1:=simplify(mtaylor(res1(r,s),[r=p,s=q],2)):
>   resu2:=simplify(mtaylor(res2(r,s),[r=p,s=q],2)):
>   N:=unapply([r+t*resu1,s+t*resu2],(r,s)):
>   RETURN(N):
>   end:
```

The procedure `r3extract` associates to a vector field given by a curve $N_t(r_3, s_3)$ the components of $\partial_{r_3}$, $s_3 \partial_{r_3}$ and $\partial_{s_3}$ in the Taylor expansion of the associated vector field at $(0, 0)$.

```
>   r3extract:=proc(M)
>   local prod,res,res1,res2,w:
>   prod:=unapply(M(r,s),t):
>   res:=simplify(subs(t=0,diff(prod(t),t))):
>   res1:=unapply(op(1,res),(r,s)):
>   res2:=unapply(op(2,res),(r,s)):
>   w[1]:=simplify(res1(-1,0)):
>   w[2]:=simplify(subs({r=-1,s=0},diff(res1(r,s),s))):
>   w[3]:=simplify(res2(-1,0)):
>   w:=[w[1],w[2],w[3]]:
>   RETURN(convert(w,Vector)):
>   end:
```

The function $\mathtt{ppppsi}_1$ is the projection of the blow up of $(0,0)_{y,z}$, $(0,0)_{u_1,v_1}$ $(0,0)_{r_2,s_2}$ and $(1,0)_{r_3,s_3}$, and $\mathtt{eeeeta}_1$ is its inverse.

```
>   ppppsi1:=(r,s)->((r*s+1)*s*s,(r*s+1)*s*s^2):
>   eeeeta1:=(y,z)->(((y^2/z)/(z/y)-1)/(z/y),(z/y)):
```

The function $\mathtt{ppppsi}_2$ is the projection of the blow up at $\left(-\frac{1}{\alpha}, \frac{1}{\alpha}\right)_{y,z}$, $\left(0, \frac{1}{\alpha-1}\right)_{u_1,v_1}$, $(0,0)_{r_2,s_2}$ and $\left(-\frac{(\alpha-1)^5}{2\alpha^4}, 0\right)_{r_3,s_3}$, and $\mathtt{eeeeta}_2$ is its inverse.

```
>   ppppsi2:=(r,s)->((r*s-(1/2)*(-1+alpha)^5/alpha^4)*s*s-1/alpha, (r*s-(1/2)*(-1+alpha)^5/
alpha^4)*s*s*(s+1/(alpha-1))+1/alpha):
>   eeeeta2:=(y,z)->((((y+1/alpha)/((z-1/alpha)/(y+1/alpha)-1/(alpha-1)))/((z-1/alpha)/(y+1/
alpha)-1/(alpha-1))+(1/2)*(-1+alpha)^5/alpha^4)/((z-1/alpha)/(y+1/alpha)-1/(alpha-1)),((z-1/
alpha)/(y+1/alpha)-1/(alpha-1))):
```

The function $\mathtt{ppppsi}_3$ is the projection of the blow up at $\left(\frac{1}{\alpha}, \frac{1}{\alpha}\right)_{y,z}$, $\left(0, \frac{1}{\alpha+1}\right)_{u_1,v_1}$, $(0,0)_{r_2,s_2}$ and $\left(\frac{(\alpha+1)^5}{2\alpha^4}, 0\right)$, and $\mathtt{eeeeta}_3$ is its inverse.

```
>   ppppsi3:=(r,s)->((r*s+(1/2)*(1+alpha)^5/alpha^4)*s*s+1/alpha, (r*s+(1/2)*(1+alpha)^5/alpha^4)
*s*s*(s+1/(alpha+1))+1/alpha):
>   eeeeta3:=(y,z)->((((y-1/alpha)/((z-1/alpha)/(y-1/alpha)-1/(alpha+1)))/((z-1/alpha)/(y-1/
alpha)-1/(alpha+1))-(1/2)*(1+alpha)^5/alpha^4)/((z-1/alpha)/(y-1/alpha)-1/(alpha+1)),((z-1/
alpha)/(y-1/alpha)-1/(alpha+1))):
```

The function $\mathtt{ppppsi}_4$ is the projection of the blow up at $(0,0)_{y,z}$, $(0,0)_{u_1,v_1}$ $(0,0)_{r_2,s_2}$ and $(-1,0)_{r_3,s_3}$, and $\mathtt{eeeeta}_1$ is its inverse.

```
>   ppppsi4:=(r,s)->((r*s-1)*s*s,(r*s-1)*s*s^2):
>   eeeeta4:=(y,z)->(((y^2/z)/(z/y)+1)/(z/y),(z/y)):
```

We compute the points where we take the Taylor expansions.

```
>   limit(op(1,[(eeeeta2@(A@ppppsi1))(r,s)]),{r=0,s=0}), limit(op(2,[(eeeeta2@(A@ppppsi1))(
r,s)]),{r=0,s=0});
```

$$\frac{3(\alpha-1)^4(1-\alpha^2+\alpha^3-\alpha)}{4\alpha^5},\ 0$$

```
> limit(op(1,[(eeeeta3@(A@f@A)@ppppsi1)(r,s)]),{r=0,s=0}), limit(op(2,[(eeeeta3@(A@f@A)@ppppsi1)
(r,s)]),{r=0,s=0});
```

$$\frac{-3(1+\alpha)^4(-\alpha-1+\alpha^2+\alpha^3)}{4\alpha^5},\ 0$$

```
> limit(op(1,[(eeeeta4@(A@f@A@f@A)@ppppsi1)(r,s)]),{r=0,s=0}), limit(op(2,[(eeeeta4@(A@f@A@f@A)
@ppppsi1)(r,s)]),{r=0,s=0});
```

$0,\ 0$

The procedure **r4devlim1** associates to a vector field $Z$ on the blow up of $\mathbb{P}^2$ at $(0,0)_{y,z}$, $(0,0)_{u_1,v_1}$, $(0,0)_{r_2,s_2}$ and $(1,0)_{r_3,s_3}$ expressed in the coordinates $(r_4,s_4)$ the Taylor expansion of $A_*Z$ of order 0 at a point $(p,q)$. The result is viewed as a polynomial vector field in a neighborhood of $(p,q)$ in the coordinates $(r_4,s_4)$ which are this time the coordinates on $\mathbb{P}^2$ blown up at $\left(-\frac{1}{\alpha},\frac{1}{\alpha}\right)_{y,z}$, $\left(0,\frac{1}{\alpha-1}\right)_{u_1,v_1}$, $(0,0)_{r_2,s_2}$ and $\left(-\frac{(\alpha-1)^5}{2\alpha^4},0\right)_{r_3,s_3}$, and the procedure returns a curve $N_t$ associated to this vector field.

```
> r4devlim1:=proc(M,p,q)
> local prod,res,res1,res2,resu1,resu2,N,MMM:
> MMM:=(r,s)->(op(M(r,s))):
> prod:=unapply([(eeeeta2@(A@(ppppsi1@MMM@eeeeta1@(B@ppppsi2))))(r,s)],t):
> res:=simplify(subs(t=0,diff(prod(t),t))):
> res1:=unapply(simplify(op(1,res)),(r,s)):
> res2:=unapply(simplify(op(2,res)),(r,s)):
> resu1:=simplify(mtaylor(res1(r,s),[r=p,s=q],1)):
> resu2:=simplify(mtaylor(res2(r,s),[r=p,s=q],1)):
> N:=unapply([r+t*resu1,s+t*resu2],(r,s)):
> RETURN(N):
> end:
```

The procedure **r4devlim2** associates to a vector field $Z$ on the blow up of $\mathbb{P}^2$ at $\left(-\frac{1}{\alpha},\frac{1}{\alpha}\right)_{y,z}$, $\left(0,\frac{1}{\alpha-1}\right)_{u_1,v_1}$, $(0,0)_{r_2,s_2}$ and $\left(-\frac{(\alpha-1)^5}{2\alpha^4},0\right)_{r_3,s_3}$ expressed in the coordinates $(r_4,s_4)$ the Taylor expansion of $(A\phi)_*Z$ of order 0 at a point $(p,q)$. The result is viewed as a polynomial vector field in a neighborhood of $(p,q)$ in the coordinates $(r_4,s_4)$ which are this time the coordinates on $\mathbb{P}^2$ blown up at $\left(\frac{1}{\alpha},\frac{1}{\alpha}\right)_{y,z}$, $\left(0,\frac{1}{\alpha+1}\right)_{u_1,v_1}$, $(0,0)_{r_2,s_2}$ and $\left(\frac{(\alpha+1)^5}{2\alpha^4},0\right)$, and the procedure returns a curve $N_t$ associated to this vector field.

```
> r4devlim2:=proc(M,p,q)
> local prod,res,res1,res2,resu1,resu2,N,MMM:
> MMM:=(r,s)->(op(M(r,s))):
> prod:=unapply([(eeeeta3@(A@f@(ppppsi2@MMM@eeeeta2@g@B@ppppsi3)))(r,s)],t):
> res:=simplify(subs(t=0,diff(prod(t),t))):
> res1:=unapply(simplify(op(1,res)),(r,s)):
> res2:=unapply(simplify(op(2,res)),(r,s)):
> resu1:=simplify(mtaylor(res1(r,s),[r=p,s=q],1)):
> resu2:=simplify(mtaylor(res2(r,s),[r=p,s=q],1)):
> N:=unapply([r+t*resu1,s+t*resu2],(r,s)):
> RETURN(N):
> end:
```

The procedure **r4devlim3** associates to a vector field $Z$ on the blow up of $\mathbb{P}^2$ at $\left(\frac{1}{\alpha},\frac{1}{\alpha}\right)_{y,z}$, $\left(0,\frac{1}{\alpha+1}\right)_{u_1,v_1}$, $(0,0)_{r_2,s_2}$ and $\left(\frac{(\alpha+1)^5}{2\alpha^4},0\right)$, and expressed in coordinates $(r_4,s_4)$ the Taylor expansion of $(A\phi)_*Z$ of order 0 at a point $(p,q)$. The result is viewed as a polynomial vector field in a neighborhood of $(p,q)$ in the coordinates $(r_4,s_4)$ which are this time the coordinates on $\mathbb{P}^2$ blown up at $(0,0)_{y,z}$, $(0,0)_{u_1,v_1}$, $(0,0)_{r_2,s_2}$ and $(-1,0)_{r_3,s_3}$, and the procedure returns a curve $N_t$ associated to this vector field.

```
> r4devlim3:=proc(M,p,q)
> local prod,res,res1,res2,resu1,resu2,N,MMM:
> MMM:=(r,s)->(op(M(r,s))):
> prod:=unapply([(eeeeta4@(A@f@(ppppsi3@MMM@eeeeta3@g@B@ppppsi4)))(r,s)],t):
> res:=simplify(subs(t=0,diff(prod(t),t))):
> res1:=unapply(simplify(op(1,res)),(r,s)):
> res2:=unapply(simplify(op(2,res)),(r,s)):
> resu1:=simplify(mtaylor(res1(r,s),[r=p,s=q],1)):
> resu2:=simplify(mtaylor(res2(r,s),[r=p,s=q],1)):
> N:=unapply([r+t*resu1,s+t*resu2],(r,s)):
```

```
>    RETURN(N):
>  end:
```
The procedure `r4extract` associates to a vector field given by a curve $N_t(r_4, s_4)$ the components of $\partial_{r_4}$ and $\partial_{s_4}$ in the Taylor expansion of the vector field associated at $(0,0)$.

```
>  r4extract:=proc(M)
>  local prod,res,res1,res2,w:
>  prod:=unapply(M(r,s),t):
>  res:=simplify(subs(t=0,diff(prod(t),t))):
>  res1:=unapply(op(1,res),(r,s)):
>  res2:=unapply(op(2,res),(r,s)):
>  w[1]:=simplify(res1(0,0)):
>  w[2]:=simplify(res2(0,0)):
>  w:=[w[1],w[2]]:
>  RETURN(convert(w,Vector)):
>  end:
```

## 1.3 Definition of the residues matrices

$R_1$, denoted by $\mathscr{Z}_1$ in the article, is the $(10 \times 7)$-matrix of $\partial_y, y\,\partial_y, z\,\partial_y, \partial_z, y\,\partial_z, z\,\partial_z, y^2\,\partial_z$ in the basis $(\mathfrak{l}_i)_{1 \leq i \leq 10}$.

```
>  R_1:=Transpose(Matrix([[-lambda[1],-1,2,0,0,0,-lambda[4]^2,-2*lambda[4],0,0],[0,0,0,0,0,
0,-3,0,0,0],[0,0,0,0,0,0,0,0,-3,0],[1,0,0,0,0,0,lambda[4]^3,3*lambda[4]^2,2*lambda[5]^2,4*
lambda[5]],[0,0,-lambda[2],-1,-2*lambda[3],-2,-3*lambda[4],-3,-4*lambda[5],-4],[0,0,0,0,0,
0,2,0,0,0],[0,0,0,0,0,0,0,0,-2,0]]));
```

$$R\_1 := \begin{bmatrix} -\lambda_1 & 0 & 0 & 1 & 0 & 0 & 0 \\ -1 & 0 & 0 & 0 & 0 & 0 & 0 \\ 2 & 0 & 0 & 0 & -\lambda_2 & 0 & 0 \\ 0 & 0 & 0 & 0 & -1 & 0 & 0 \\ 0 & 0 & 0 & 0 & -2\lambda_3 & 0 & 0 \\ 0 & 0 & 0 & 0 & -2 & 0 & 0 \\ -\lambda_4{}^2 & -3 & 0 & \lambda_4{}^3 & -3\lambda_4 & 2 & 0 \\ -2\lambda_4 & 0 & 0 & 3\lambda_4{}^2 & -3 & 0 & 0 \\ 0 & 0 & -3 & 2\lambda_5{}^2 & -4\lambda_5 & 0 & -2 \\ 0 & 0 & 0 & 4\lambda_5 & -4 & 0 & 0 \end{bmatrix}$$

$R_2$, denoted by $\mathscr{Z}_2$ in the article, is the $(10 \times 7)$-matrix of $\partial_y, y\,\partial_y, z\,\partial_y, \partial_z, y\,\partial_z, z\,\partial_z, y^2\,\partial_z$ in the basis $(\mathfrak{m}_i)_{1 \leq i \leq 10}$.

```
>  R_2:=Transpose(Matrix([[-mu[1],-1,2,0,0,0,-mu[4]^2,-2*mu[4],0,0],[0,0,0,0,0,0,3,0,0,0],
[0,0,0,0,0,0,0,0,3,0],[1,0,0,0,0,0,-mu[4]^3,-3*mu[4]^2,2*mu[5]^2,4*mu[5]],[0,0,-mu[2],-1,-2*
mu[3],-2,-3*mu[4],-3,-4*mu[5],-4],[0,0,0,0,0,0,-2,0,0,0],[0,0,0,0,0,0,0,0,-2,0]]));
```

$$R\_2 := \begin{bmatrix} -\mu_1 & 0 & 0 & 1 & 0 & 0 & 0 \\ -1 & 0 & 0 & 0 & 0 & 0 & 0 \\ 2 & 0 & 0 & 0 & -\mu_2 & 0 & 0 \\ 0 & 0 & 0 & 0 & -1 & 0 & 0 \\ 0 & 0 & 0 & 0 & -2\mu_3 & 0 & 0 \\ 0 & 0 & 0 & 0 & -2 & 0 & 0 \\ -\mu_4{}^2 & 3 & 0 & -\mu_4{}^3 & -3\mu_4 & -2 & 0 \\ -2\mu_4 & 0 & 0 & -3\mu_4{}^2 & -3 & 0 & 0 \\ 0 & 0 & 3 & 2\mu_5{}^2 & -4\mu_5 & 0 & -2 \\ 0 & 0 & 0 & 4\mu_5 & -4 & 0 & 0 \end{bmatrix}$$

$RR_1$ is the $(10 \times 10)$-matrix that gives the residues of $\partial_{u_1}, u_1\,\partial_{u_1}, v_1\,\partial_{u_1}, u_1 v_1\,\partial_{u_1}, v_1^2\,\partial_{u_1}, v_1^3\,\partial_{u_1}, \partial_{v_1}, u_1\,\partial_{v_1}, v_1\,\partial_{v_1}, v_1^2\,\partial_{v_1}$ in the basis $(\mathfrak{l}_i)_{1 \leq i \leq 10}$.

```
> RR_1:=Matrix([[0,0,0,0,0,0,0,0,0,0],[0,0,0,0,0,0,0,0,0,0],[1,0,0,0,0,0,-lambda[2],0,0,0],
[0,0,0,0,0,0,-1,0,0,0],[0,0,1,0,0,0,-2*lambda[3],0,0,0],[0,0,0,0,0,0,-2,0,0,0],[0,-1,0,0,1,
0,-3*lambda[4],0,2,0],[0,0,0,0,0,0,-3,0,0,0],[0,0,0,-1,0,1,-4*lambda[5],-2,0,2],[0,0,0,0,0,
0,-4,0,0,0]]);
```

$$RR\_1 := \begin{bmatrix} 0 & 0 & 0 & 0 & 0 & 0 & 0 & 0 & 0 & 0 \\ 0 & 0 & 0 & 0 & 0 & 0 & 0 & 0 & 0 & 0 \\ 1 & 0 & 0 & 0 & 0 & 0 & -\lambda_2 & 0 & 0 & 0 \\ 0 & 0 & 0 & 0 & 0 & 0 & -1 & 0 & 0 & 0 \\ 0 & 0 & 1 & 0 & 0 & 0 & -2\lambda_3 & 0 & 0 & 0 \\ 0 & 0 & 0 & 0 & 0 & 0 & -2 & 0 & 0 & 0 \\ 0 & -1 & 0 & 0 & 1 & 0 & -3\lambda_4 & 0 & 2 & 0 \\ 0 & 0 & 0 & 0 & 0 & 0 & -3 & 0 & 0 & 0 \\ 0 & 0 & 0 & -1 & 0 & 1 & -4\lambda_5 & -2 & 0 & 2 \\ 0 & 0 & 0 & 0 & 0 & 0 & -4 & 0 & 0 & 0 \end{bmatrix}$$

$RRR_1$ is the $(10 \times 12)$-matrix that gives the residues of $\partial_{r_2}$, $r_2\,\partial_{r_2}$, $s_2\,\partial_{r_2}$, $r_2 s_2\,\partial_{r_2}$, $r_2^2\partial_{r_2}$, $s_2^2\partial_{r_2}$, $\partial_{s_2}$, $r_2\,\partial_{s_2}$, $s_2\,\partial_{s_2}$, $r_2 s_2\,\partial_{s_2}$, $r_2^2\,\partial_{s_2}$, $s_2^2\,\partial_{s_2}$ in the basis $(\mathfrak{l}_i)_{1\leq i\leq 10}$.

```
> RRR_1:=Matrix([[0,0,0,0,0,0,0,0,0,0,0,0],[0,0,0,0,0,0,0,0,0,0,0,0],[0,0,0,0,0,0,0,0,0,0,
0,0],[0,0,0,0,0,0,0,0,0,0,0,0],[1,0,0,0,0,0,-lambda[3],0,0,0,0,0],[0,0,0,0,0,0,-1,0,0,0,0,
0],[0,-1,1,0,0,0,-2*lambda[4],-1,1,0,0,0],[0,0,0,0,0,0,-2,0,0,0,0,0],[0,0,0,-1,1,1,-3*lambda[5],
0,0,-1,1,1],[0,0,0,0,0,0,-3,0,0,0,0,0]]);
```

$$RRR\_1 := \begin{bmatrix} \text{` 10 x 12 `} (Matrix) \\ \text{`Data Type: `} anything \\ \text{`Storage: `} rectangular \\ \text{`Order: `} Fortran\_order \end{bmatrix}$$

```
> evalm(RRR_1);
```

$$\begin{bmatrix} 0 & 0 & 0 & 0 & 0 & 0 & 0 & 0 & 0 & 0 & 0 & 0 \\ 0 & 0 & 0 & 0 & 0 & 0 & 0 & 0 & 0 & 0 & 0 & 0 \\ 0 & 0 & 0 & 0 & 0 & 0 & 0 & 0 & 0 & 0 & 0 & 0 \\ 0 & 0 & 0 & 0 & 0 & 0 & 0 & 0 & 0 & 0 & 0 & 0 \\ 1 & 0 & 0 & 0 & 0 & 0 & -\lambda_3 & 0 & 0 & 0 & 0 & 0 \\ 0 & 0 & 0 & 0 & 0 & 0 & -1 & 0 & 0 & 0 & 0 & 0 \\ 0 & -1 & 1 & 0 & 0 & 0 & -2\lambda_4 & -1 & 1 & 0 & 0 & 0 \\ 0 & 0 & 0 & 0 & 0 & 0 & -2 & 0 & 0 & 0 & 0 & 0 \\ 0 & 0 & 0 & -1 & 1 & 1 & -3\lambda_5 & 0 & 0 & -1 & 1 & 1 \\ 0 & 0 & 0 & 0 & 0 & 0 & -3 & 0 & 0 & 0 & 0 & 0 \end{bmatrix}$$

$RRRR_1$ is the $(10 \times 3)$-matrix that gives the residues of $\partial_{r_3}$, $s_3\partial_{r_3}$, $\partial_{s_3}$ in the basis $(\mathfrak{l}_i)_{1\leq i\leq 10}$.

```
> RRRR_1:=Matrix([[0,0,0],[0,0,0],[0,0,0],[0,0,0],[0,0,0],[0,0,0],[1,0,-lambda[4]],[0,0,-1],
[0,1,-2*lambda[5]],[0,0,-2]]);
```

$$RRRR\_1 := \begin{bmatrix} 0 & 0 & 0 \\ 0 & 0 & 0 \\ 0 & 0 & 0 \\ 0 & 0 & 0 \\ 0 & 0 & 0 \\ 0 & 0 & 0 \\ 1 & 0 & -\lambda_4 \\ 0 & 0 & -1 \\ 0 & 1 & -2\lambda_5 \\ 0 & 0 & -2 \end{bmatrix}$$

RRRRR$_1$ is the $(10 \times 2)$-matrix that gives the residues of $\partial_{r_4}$, $\partial_{s_4}$ in the basis $(\mathfrak{l}_i)_{1 \leq i \leq 10}$.

```
> RRRRR_1:=Matrix([[0,0],[0,0],[0,0],[0,0],[0,0],[0,0],[0,0],[0,0],[1,-lambda[5]],[0,-1]])
;
```

$$RRRRR\_1 := \begin{bmatrix} 0 & 0 \\ 0 & 0 \\ 0 & 0 \\ 0 & 0 \\ 0 & 0 \\ 0 & 0 \\ 0 & 0 \\ 0 & 0 \\ 1 & -\lambda_5 \\ 0 & -1 \end{bmatrix}$$

P is the transition matrix from the algebraic basis $(\mathfrak{m}_i)_{1 \leq 10}$ to the geometric basis for $\widehat{P}_2$.

```
> P:=Matrix([[-mu[1],1,0,0,0,0,0,0,0,0],[-1,0,0,0,0,0,0,0,0,0],[2,0,1,-mu[2],0,0,0,0,0,0],
[0,0,0,-1,0,0,0,0,0,0],[0,0,0,-2*mu[3],1,-mu[3],0,0,0,0],[0,0,0,-2,0,-1,0,0,0,0],[-mu[4]^2,
-mu[4]^3,0,-3*mu[4],0,-2*mu[4],1,-mu[4],0,0],[-2*mu[4],-3*mu[4]^2,0,-3,0,-2,0,-1,0,0],[0,2*
mu[5]^2,0,-4*mu[5],0,-3*mu[5],0,-2*mu[5],1,-mu[5]],[0,4*mu[5],0,-4,0,-3,0,-2,0,-1]]);
```

$$P := \begin{bmatrix} -\mu_1 & 1 & 0 & 0 & 0 & 0 & 0 & 0 & 0 & 0 \\ -1 & 0 & 0 & 0 & 0 & 0 & 0 & 0 & 0 & 0 \\ 2 & 0 & 1 & -\mu_2 & 0 & 0 & 0 & 0 & 0 & 0 \\ 0 & 0 & 0 & -1 & 0 & 0 & 0 & 0 & 0 & 0 \\ 0 & 0 & 0 & -2\mu_3 & 1 & -\mu_3 & 0 & 0 & 0 & 0 \\ 0 & 0 & 0 & -2 & 0 & -1 & 0 & 0 & 0 & 0 \\ -\mu_4{}^2 & -\mu_4{}^3 & 0 & -3\mu_4 & 0 & -2\mu_4 & 1 & -\mu_4 & 0 & 0 \\ -2\mu_4 & -3\mu_4{}^2 & 0 & -3 & 0 & -2 & 0 & -1 & 0 & 0 \\ 0 & 2\mu_5{}^2 & 0 & -4\mu_5 & 0 & -3\mu_5 & 0 & -2\mu_5 & 1 & -\mu_5 \\ 0 & 4\mu_5 & 0 & -4 & 0 & -3 & 0 & -2 & 0 & -1 \end{bmatrix}$$

PP, denoted by $\mathscr{K}$ in the article, is the transition matrix from the algebraic basis $(\mathfrak{l}_i)_{1 \leq i \leq 10}$ to the geometric basis $(\mathfrak{e}_i)_{1 \leq i \leq 10}$ for $\widehat{P}_1$.

```
> PP:=Matrix([[-lambda[1],1,0,0,0,0,0,0,0,0],[-1,0,0,0,0,0,0,0,0,0],[2,0,1,-lambda[2],0,0,
0,0,0,0],[0,0,0,-1,0,0,0,0,0,0],[0,0,0,-2*lambda[3],1,-lambda[3],0,0,0,0],[0,0,0,-2,0,-1,0,
0,0,0],[-lambda[4]^2, lambda[4]^3,0,-3*lambda[4],0,-2*lambda[4],1,-lambda[4],0,0],[-2*lambda[4],
3*lambda[4]^2,0,-3,0,-2,0,-1,0,0],[0,2*lambda[5]^2,0,-4*lambda[5],0,-3*lambda[5],0,-2*lambda[5],
1,-lambda[5]],[0,4*lambda[5],0,-4,0,-3,0,-2,0,-1]]);
```

$$PP := \begin{bmatrix} -\lambda_1 & 1 & 0 & 0 & 0 & 0 & 0 & 0 & 0 & 0 \\ -1 & 0 & 0 & 0 & 0 & 0 & 0 & 0 & 0 & 0 \\ 2 & 0 & 1 & -\lambda_2 & 0 & 0 & 0 & 0 & 0 & 0 \\ 0 & 0 & 0 & -1 & 0 & 0 & 0 & 0 & 0 & 0 \\ 0 & 0 & 0 & -2\lambda_3 & 1 & -\lambda_3 & 0 & 0 & 0 & 0 \\ 0 & 0 & 0 & -2 & 0 & -1 & 0 & 0 & 0 & 0 \\ -\lambda_4{}^2 & \lambda_4{}^3 & 0 & -3\lambda_4 & 0 & -2\lambda_4 & 1 & -\lambda_4 & 0 & 0 \\ -2\lambda_4 & 3\lambda_4{}^2 & 0 & -3 & 0 & -2 & 0 & -1 & 0 & 0 \\ 0 & 2\lambda_5{}^2 & 0 & -4\lambda_5 & 0 & -3\lambda_5 & 0 & -2\lambda_5 & 1 & -\lambda_5 \\ 0 & 4\lambda_5 & 0 & -4 & 0 & -3 & 0 & -2 & 0 & -1 \end{bmatrix}$$

PPP is the transition matrix from the algebraic basis $(\mathfrak{l}_i)_{1 \leq i \leq 10}$ to the geometric basis $(\mathfrak{e}_i)_{1 \leq i \leq 10}$ for $\widehat{P}_1$.

```
> PPP:=MatrixInverse(PP):
```

## 1.4 Construction of the matrix $\mathscr{Y}$

### 1.4.1 Construction of a block matrix

We input the coefficients of the matrix $\mathscr{L}$.

```
> col[1]:=[-mu[1], -1, 2, 0, 0, 0, 2*mu[4]^2, 4*mu[4], 0, 0, 0, 0, 0, 0, 0, 0, -4, 0, 0, 0,
0, 0, 0, 0, 0, 0, 0, 0, 0, 0, 0, 0, 0, -3, 0, 0, 0, 0, 0, 0, 0, 0, 0, 0, 0, 0,
0, 0, 0, 0, 0, 0, 0, 0, 0, 0, 0, 0, 3, 0];
```

$col_1 := [-\mu_1, -1, 2, 0, 0, 0, 2\mu_4^2, 4\mu_4, 0, 0, 0, 0, 0, 0, 0, 0, -4, 0, 0, 0, 0, 0, 0, 0, 0, 0, 0, 0, 0, 0, 0, 0, 0, -3, 0, 0, 0, 0, 0, 0, 0, 0, 0, 0, 0, 0, 0, 0, 0, 0, 0, 0, 0, 0, 0, 0, 0, 0, 3, 0]$

```
> col[2]:=
> [1, 0, 0, 0, 0, 0, -mu[4]^3, -3*mu[4]^2, 0, 0, 0, 0, 0, 0, 0, 0, 0, 0, 0, 0, 0, 2, 0, 0,
0, 0, 0, 0, 0, 0, 0, 0, 0, 0, 0, 0, 0, 1, 0, 0, 0, 0, 0, 0, 0, 0, 0, 0, 0, 0,
0, 0, 0, 0, 0, 0, 0, 0, 0, -2];
```

$col_2 := [1, 0, 0, 0, 0, 0, -\mu_4^3, -3\mu_4^2, 0, 0, 0, 0, 0, 0, 0, 0, 0, 0, 0, 0, 0, 2, 0, 0, 0, 0, 0, 0, 0, 0, 0, 0, 0, 0, 0, 0, 0, 1, 0, 0, 0, 0, 0, 0, 0, 0, 0, 0, 0, 0, 0, 0, 0, 0, 0, 0, 0, 0, 0, -2]$

```
> col[3]:=
> [0, 0, 1, 0, 0, 0, mu[4]^2, 2*mu[4], 0, 0, 0, 0, 0, 0, 0, 0, -2, 0, 0, 0, 0, 0, 0, 0, 0,
0, 0, 0, 0, 0, 0, 0, 0, 0, 0, -2, 0, 0, 0, 0, 0, 0, 0, 0, 0, 0, 0, 0, 0, 0,
0, 0, 0, 0, 0, 0, 0, 0, 1, 0];
```

$col_3 := [0, 0, 1, 0, 0, 0, \mu_4^2, 2\mu_4, 0, 0, 0, 0, 0, 0, 0, 0, -2, 0, 0, 0, 0, 0, 0, 0, 0, 0, 0, 0, 0, 0, 0, 0, 0, 0, 0, -2, 0, 0, 0, 0, 0, 0, 0, 0, 0, 0, 0, 0, 0, 0, 0, 0, 0, 0, 0, 0, 0, 0, 1, 0]$

```
> col[4]:=
> [0, 0, -mu[2], -1, 2*mu[3]^2-2*mu[3], 4*mu[3]-2, mu[4], 1, 0, 0, 0, 0, 0, 0, 0, 0, 0, 0, 0,
0, 0, 0, 0, 0, 0, 0, 0, 0, 0, 0, 0, 0, 0, 0, 0, 0, 0, 0, 0, 0, 0, 0, 0, 0, 0, 0,
0, 0, 0, 0, 0, 0, 0, 0, 0, 0, 0, 0, 2, 0, 0];
```

$col_4 := [0, 0, -\mu_2, -1, 2\mu_3^2 - 2\mu_3, 4\mu_3 - 2, \mu_4, 1, 0, 0, 0, 0, 0, 0, 0, 0, 0, 0, 0, 0, 0, 0, 0, 0, 0, 0, 0, 0, 0, 0, 0, 0, 0, 0, 0, 0, 0, 0, 0, 0, 0, 0, 0, 0, 0, 0, 0, 0, 0, 0, 0, 0, 0, 0, 0, 0, 0, 0, 2, 0, 0]$

```
> col[5]:=
> [0, 0, -2*mu[2], -2, (mu[3]-1)^2, 2*(mu[3]-1), 0, 0, 0, 0, 0, 0, -2, 0, 0, 0, 0, 0, 0, 0,
0, 0, 0, 0, 0, 0, 0, 0, 0, 0, 0, 0, 0, 0, 0, 0, 0, 0, 0, 0, 0, 0, 0, 0, 0, 0, 0,
0, 0, 0, 0, 0, 0, 0, 0, 0, 0, 1, 0, 0];
```

$col_5 := [0, 0, -2\mu_2, -2, (\mu_3-1)^2, 2\mu_3 - 2, 0, 0, 0, 0, 0, 0, -2, 0, 0, 0, 0, 0, 0, 0, 0, 0, 0, 0, 0, 0, 0, 0, 0, 0, 0, 0, 0, 0, 0, 0, 0, 0, 0, 0, 0, 0, 0, 0, 0, 0, 0, 0, 0, 0, 0, 0, 0, 0, 0, 0, 0, 1, 0, 0]$

```
> col[6]:=
> [0, 0, 0, 0, mu[3]^2-mu[3], 2*mu[3]-1, 0, 0, 0, 0, 0, 0, 1, 0, 0, 0, 0, 0, 0, 0, 0, 0,
0, 0, 0, 0, 0, 0, 0, 0, 0, 0, 0, 0, 0, 0, 0, 0, 0, 0, 0, 0, 0, 0, 0, 0, 0, 0, 0,
0, 0, 0, 0, 0, 0, 0, 1, 0, 0];
```

$col_6 := [0, 0, 0, 0, \mu_3^2 - \mu_3, 2\mu_3 - 1, 0, 0, 0, 0, 0, 0, 1, 0, 0, 0, 0, 0, 0, 0, 0, 0, 0, 0, 0, 0, 0, 0, 0, 0, 0, 0, 0, 0, 0, 0, 0, 0, 0, 0, 0, 0, 0, 0, 0, 0, 0, 0, 0, 0, 0, 0, 0, 0, 0, 0, 0, 1, 0, 0]$

```
> col[7]:=
```

```
>    [0, 0, mu[2]^2, 2*mu[2], 0, 0, 0, 0, 0, 0, 0, 0, 0, 0, 0, 0, 0, 0, 0, 0, 0, 0, 0, 0, 0,
     0, 0, 0, 0, 0, 0, 0, 0, 0, 0, 0, 0, 0, 0, 0, 0, 0, 0, 0, 0, 0, 0, 0, 0, 0, 0, 0, 0, 0,
     0, 0, 0, 0, 0, 1, 0, 0, 0];
```

$col_7 := [0,0,\mu_2^2, 2\mu_2, 0,0,0,0,0,0,0,0,0,0,0,0,0,0,0,0,0,0,0,0,0,0,0,0,0,0,0,0,0,0,0,0,0,0,0,0,$
$0,0,0,0,0,0,0,0,0,0,0,0,0,0,0,0,0,0,0,1,0,0,0]$

```
>    col[8]:=
>    [0, 0, mu[2], 1, 0, 0, 0, 0, 0, 0, 0, 0, 2, 0, 0, 0, 0, 0, 0, 0, 0, 0, 0, 0, 0, 0,
     0, 0, 0, 0, 0, 0, 0, 0, 0, 0, 0, 0, 0, 0, 0, 0, 0, 0, 0, 0, 0, 0, 0, 0, 0, 0, 0, 0, 0,
     0, 0, 0, 0, 0, 0];
```

$col_8 := [0,0,\mu_2,1,0,0,0,0,0,0,0,0,2,0,0,0,0,0,0,0,0,0,0,0,0,0,0,0,0,0,0,0,0,0,0,0,0,0,0,0$
$,0,0,0,0,0,0,0,0,0,0,0,0,0,0,0,0,0,0,0,0,0,0,0]$

```
>    col[9]:=[0, 0, 0, 0, 0, 0, 0, 0, 0, 0, 0, 0, 0, 0, 0, 0, 0, 0, 0, 0, 0, 0, 0, 0, 0,
     0, 0, 0, 0, 0, 0, 0, 0, 0, 0, 0, 0, 0, 0, 0, 0, 0, 0, 0, 0, 0, 0, 0, 0, 0, 0, 0, 0, 0,
     0, 0, 1, 0, 0, 0, 0];
```

$col_9 := [0,0,0,0,0,0,0,0,0,0,0,0,0,0,0,0,0,0,0,0,0,0,0,0,0,0,0,0,0,0,0,0,0,0,0,0,0,0,0,0,0,$
$0,0,0,0,0,0,0,0,0,0,0,0,0,0,0,1,0,0,0,0]$

```
>    col[10]:=
>    [0, 0, 0, 0, 0, 0, 0, 0, 0, 0, 0, 0, 1, 0, 0, 0, 0, 0, 0, 0, 0, 0, 0, 0, 0, 0,
     0, 0, 0, 0, 0, 0, 0, 0, 0, 0, 0, 0, 0, 0, 0, 0, 0, 0, 0, 0, 0, 0, 0, 0, 0, 0, 0, 0, 0,
     0, 0, 0, 0, 0];
```

$col_{10} := [0,0,0,0,0,0,0,0,0,0,0,0,1,0,0,0,0,0,0,0,0,0,0,0,0,0,0,0,0,0,0,0,0,0,0,0,0,0,0,0,0,0,$
$0,0,0,0,0,0,0,0,0,0,0,0,0,0,0,0,0,0]$

The matrix `PART` is by definition $\mathscr{L}\mathscr{K}^{-1}$.

```
>    PART:=MatrixMatrixMultiply(Transpose(Matrix([col[1],col[2],col[3],col[4],col[5],col[6],
     col[7],col[8],col[9],col[10]])),PPP);
```

$$PART := \begin{bmatrix} \text{` 65 x 10 `}(Matrix) \\ \text{`Data Type: `}anything \\ \text{`Storage: `}rectangular \\ \text{`Order: `}Fortran\_order \end{bmatrix}$$

```
>    for k from 1 to 85 do
>    for j from 1 to 30 do
>    L[k][j]:=0
>    od
>    od:
```

We construct the matrix $\mathscr{Y}$, whose upper right block $\mathscr{Q}$ is indeterminate.

```
>    for k from 1 to 10 do
>    for l from 1 to 10 do
>    L[k][l+20]:=Q[k,l]
>    od
>    od:
>    for k from 76 to 85 do
>    L[k][k-65]:=1
>    od:
>    for k from 11 to 75 do
>    for j from 1 to 10 do
>    L[k][j]:=PART[k-10,j]
>    od
>    od:
>    for k from 1 to 85 do
>    L[k]:=convert(L[k],list)
>    od:
>    MM:=Matrix([seq(L[k],k=1..85)]):
```

### 1.4.2 Computation of $\mathscr{Q}$

To compute the coefficients of $\mathscr{Q}$, we consider the matrix `R` obtained by multiplying $\mathscr{Q}$ at right by the transition matrix from the algebraic basis to the geometric basis of $\widehat{P}_2$. The matrix `R` is then the matrix

$$(A\phi A\phi A)_* : W(\widehat{P}_2, \mathscr{B}_{\texttt{geom}}) \longrightarrow W(\widehat{P}_1, \mathscr{B}_{\texttt{alg}}).$$

Computation of $(A\phi A\phi A)_* \partial_y$ in the basis $(\mathfrak{l}_i)_{1 \leq i \leq 10}$.

```
> y:='y':  z:='z':
> M:=(y,z)->[y+t,z]:
> U:=MatrixMatrixMultiply(R_1,extract(devlim2(devlim2((devlim1(M,-1/alpha,1/alpha)), 1/alpha,
1/alpha),0,0))):
> for k from 1 to 10 do
> R[k][1]:=U[k];
> od:
```

Computation of $(A\phi A\phi A)_* \partial_z$ in the basis $(\mathfrak{l}_i)_{1 \leq i \leq 10}$.

```
> y:='y':  z:='z':
> M:=(y,z)->[y,z+t]:
> U:=MatrixMatrixMultiply(R_1,extract(devlim2(devlim2((devlim1(M,-1/alpha,1/alpha)), 1/alpha,
1/alpha),0,0))):
> for k from 1 to 10 do
> R[k][2]:=U[k];
> od:
```

Computation of $(A\phi A\phi A)_* \partial_{u_1}$ in the basis $(\mathfrak{l}_i)_{1 \leq i \leq 10}$.

```
> u:='u':  v:='v':
> M:=(u,v)->[u+t,v]:
> U:=MatrixMatrixMultiply(RR_1,u1extract(u1devlim3(u1devlim2(u1devlim1(M,0,1/(alpha-1)), 0,
1/(1+alpha)),0,0))):
> for k from 1 to 10 do
> R[k][3]:=U[k];
> od:
```

Computation of $(A\phi A\phi A)_* \partial_{v_1}$ in the basis $(\mathfrak{l}_i)_{1 \leq i \leq 10}$.

```
> u:='u':  v:='v':
> M:=(u,v)->[u,v+t]:
> U:=MatrixMatrixMultiply(RR_1,u1extract(u1devlim3(u1devlim2(u1devlim1(M,0,1/(alpha-1)), 0,
1/(1+alpha)),0,0))):
> for k from 1 to 10 do
> R[k][4]:=U[k];
> od:
```

Computation of $(A\phi A\phi A)_* \partial_{r_2}$ in the basis $(\mathfrak{l}_i)_{1 \leq i \leq 10}$.

```
> r:='r':  s:='s':
> M:=(r,s)->[r+t,s]:
> U:=MatrixMatrixMultiply(RRR_1,r2extract(r2devlim3(r2devlim2(r2devlim1(M,0,0),0,0),0,0))
):
> for k from 1 to 10 do
> R[k][5]:=U[k];
> od:
```

Computation of $(A\phi A\phi A)_* \partial_{s_2}$ in the basis $(\mathfrak{l}_i)_{1 \leq i \leq 10}$.

```
> r:='r':  s:='s':
> M:=(r,s)->[r,s+t]:
> U:=MatrixMatrixMultiply(RRR_1,r2extract(r2devlim3(r2devlim2(r2devlim1(M,0,0),0,0),0,0))
):
> for k from 1 to 10 do
> R[k][6]:=U[k];
> od:
```

Computation of $(A\phi A\phi A)_* \partial_{r_3}$ in the basis $(\mathfrak{l}_i)_{1 \leq i \leq 10}$.

```
> r:='r':  s:='s':
> M:=(r,s)->[r+t,s]:
> U:=MatrixMatrixMultiply(RRRR_1,r3extract(r3devlim3(r3devlim2(r3devlim1(M,-(1/2)*(-1+alpha)
^5/alpha^4,0),(1/2)*(1+alpha)^5/alpha^4, 0),-1,0))):
> for k from 1 to 10 do
> R[k][7]:=U[k];
> od:
```

Computation of $(A\phi A\phi A)_* \partial_{s_3}$ in the basis $(\mathfrak{l}_i)_{1 \leq i \leq 10}$.

```
> r:='r':  s:='s':
> M:=(r,s)->[r,s+t]:
> U:=MatrixMatrixMultiply(RRRR_1,r3extract(r3devlim3(r3devlim2(r3devlim1(M,-(1/2)*(-1+alpha)
^5/alpha^4,0),(1/2)*(1+alpha)^5/alpha^4, 0),-1,0))):
> for k from 1 to 10 do
> R[k][8]:=U[k];
> od:
```

Computation of $(A\phi A\phi A)_* \partial_{c_4}$ in the basis $(\mathfrak{l}_i)_{1 \leq i \leq 10}$.

```
> r:='r':  s:='s':
> M:=(r,s)->[r+t,s]:
```

```
> U:=MatrixMatrixMultiply(RRRRR_1,r4extract(r4devlim3(r4devlim2(r4devlim1(M,(3/4)*(-1+alpha)
  ^4*(alpha^3-alpha-alpha^2+1)/alpha^5,0),-(3/4)*(1+alpha)^4*(-alpha-1+alpha^3+alpha^2)/alpha^5,
  0),0,0))):
> for k from 1 to 10 do
> R[k][9]:=U[k];
> od:
```

Computation of $(A\phi A\phi A)_* \partial_{d_4}$ in the basis $(\mathfrak{l}_i)_{1\leq i\leq 10}$.

```
> r:='r':  s:='s':
> M:=(r,s)->[r,s+t]:
> U:=MatrixMatrixMultiply(RRRRR_1,r4extract(r4devlim3(r4devlim2(r4devlim1(M,(3/4)*(-1+alpha)
  ^4*(alpha^3-alpha-alpha^2+1)/alpha^5,0),-(3/4)*(1+alpha)^4*(-alpha-1+alpha^3+alpha^2)/alpha^5,
  0),0,0))):
> for k from 1 to 10 do
> R[k][10]:=U[k];
> od:
```

We construct the matrix R, then we compute $\mathcal{Q}$.

```
> for k from 1 to 10 do
> R[k]:=convert(R[k],list)
> od:
> TEMP:=Matrix([seq(R[k],k=1..10)]):
> Q:=MatrixMatrixMultiply(TEMP, MatrixInverse(P)):
```

We can display the coefficients of $\mathcal{Q}$.

```
> simplify(Q[9,4]);
```

$4(23 + \alpha - 6\lambda_5)\mu_2$

## 1.5 The $8$ vectors of $V(\mathfrak{D}_1)$ that form the columns of the matrix $\mathscr{V}_1$

### 1.5.1 Columns of $\mathscr{M}_a$ (eight $(10 \times 1)$-vectors)

$a_1$

```
> pro1:=[-lambda[1],-1,2,0,0,0,-lambda[4]^2,-2*lambda[4],0,0];
> a[1]:=convert(pro1,Vector):
```

$pro1 := [-\lambda_1, -1, 2, 0, 0, 0, -\lambda_4{}^2, -2\lambda_4, 0, 0]$

$a_2$

```
> pro2:=[0,0,0,0,0,0,-3,0,0,0];
> a[2]:=convert(pro2,Vector):
```

$pro2 := [0, 0, 0, 0, 0, 0, -3, 0, 0, 0]$

$a_3$

```
> pro3:=[0,0,0,0,0,0,0,0,-3,0];
> a[3]:=convert(pro3,Vector):
```

$pro3 := [0, 0, 0, 0, 0, 0, 0, 0, -3, 0]$

$a_4$

```
> pro4:=[1,0,0,0,0,0,lambda[4]^3,3*lambda[4]^2,2*lambda[5]^2,4*lambda[5]];
> a[4]:=convert(pro4,Vector):
```

$pro4 := [1, 0, 0, 0, 0, 0, \lambda_4{}^3, 3\lambda_4{}^2, 2\lambda_5{}^2, 4\lambda_5]$

$a_5$

```
> pro5:=[0,0,-lambda[2],-1,-2*lambda[3],-2,-3*lambda[4],-3,-4*lambda[5],-4];
> a[5]:=convert(pro5,Vector):
```

$pro5 := [0, 0, -\lambda_2, -1, -2\lambda_3, -2, -3\lambda_4, -3, -4\lambda_5, -4]$

$a_6$

```
> pro6:=[0,0,0,0,0,0,2,0,0,0];
> a[6]:=convert(pro6,Vector):
```

$pro6 := [0, 0, 0, 0, 0, 0, 2, 0, 0, 0]$

$a_7$

```
> pro7:=[0,0,0,0,0,0,0,0,0,0];
> a[7]:=convert(pro7,Vector):
```

$pro7 := [0, 0, 0, 0, 0, 0, 0, 0, 0, 0]$

$a_8$

```
> pro8:=[0,0,0,0,0,0,0,0,0,0];
> a[8]:=convert(pro8,Vector):
```

$pro8 := [0, 0, 0, 0, 0, 0, 0, 0, 0, 0]$

### 1.5.2  Columns of $\mathscr{M}_a'$ (eight $(10 \times 1)$-vectors)

$a_1$

```
> M:=(y,z)->(t+y,z):
> aa[1]:=MatrixMatrixMultiply(R_2,champ4(M)):
```

$a_2$

```
> M:=(y,z)->((1+t)*y,z):
> aa[2]:=MatrixMatrixMultiply(R_2,champ4(M)):
```

$a_3$

```
> M:=(y,z)->(y+t*z,z):
> aa[3]:=MatrixMatrixMultiply(R_2,champ4(M)):
```

$a_4$

```
> M:=(y,z)->(y,t+z):
> aa[4]:=MatrixMatrixMultiply(R_2,champ4(M)):
```

$a_5$

```
> M:=(y,z)->(y,t*y+z):
> aa[5]:=MatrixMatrixMultiply(R_2,champ4(M)):
```

$a_6$

```
> M:=(y,z)->(y,(1+t)*z):
> aa[6]:=MatrixMatrixMultiply(R_2,champ4(M)):
```

$a_7$

```
> M:=(y,z)->(y/(1-t*y),z/(1-t*y)):
> aa[7]:=MatrixMatrixMultiply(R_2,champ4(M)):
```

$a_8$

```
> M:=(y,z)->(y/(1-t*z),z/(1-t*z)):
> aa[8]:=MatrixMatrixMultiply(R_2,champ4(M)):
```

### 1.5.3  Columns of $\mathscr{M}_a''$ (eight $(10 \times 1)$-vectors)

$a_1$

```
> M:=(y,z)->(t+y,z):
> aaa[1]:=MatrixMatrixMultiply(R_2,champ3(M)):
```

$a_2$

```
> M:=(y,z)->((1+t)*y,z):
> aaa[2]:=MatrixMatrixMultiply(R_2,champ3(M)):
```

$a_3$

```
> M:=(y,z)->(y+t*z,z):
> aaa[3]:=MatrixMatrixMultiply(R_2,champ3(M)):
```

$a_4$

```
> M:=(y,z)->(y,t+z):
> aaa[4]:=MatrixMatrixMultiply(R_2,champ3(M)):
```

$a_5$

```
> M:=(y,z)->(y,t*y+z):
> aaa[5]:=MatrixMatrixMultiply(R_2,champ3(M)):
```

$a_6$

```
> M:=(y,z)->(y,(1+t)*z):
> aaa[6]:=MatrixMatrixMultiply(R_2,champ3(M)):
```

$a_7$

```
> M:=(y,z)->(y/(1-t*y),z/(1-t*y)):
> aaa[7]:=MatrixMatrixMultiply(R_2,champ3(M)):
```

$a_8$

```
> M:=(y,z)->(y/(1-t*z),z/(1-t*z)):
> aaa[8]:=MatrixMatrixMultiply(R_2,champ3(M)):
```

### 1.5.4 Construction of the matrix $\mathcal{V}_1$

```
> mise:=proc(k::integer)
> local u,v,w,h:
> u:=convert(Transpose(a[k]),list):
> v:=convert(Transpose(aa[k]),list):
> w:=convert(Transpose(aaa[k]),list):
> h:=[op(u),op(v),op(w)]:
> RETURN(convert(h,Vector))
> end:
> for k from 1 to 8 do
> VectDef[k]:=mise(k)
> od:
```

## 1.6 The 63 vectors of $V(\mathfrak{D}_2)$ that form the columns of the matrix $\mathcal{V}_2$

### 1.6.1 Vectors of type a (we compute the columns of the matrices $\mathcal{N}'_a$, $\mathcal{N}_a$ and $\mathcal{N}''_a$)

*Columns of $\mathcal{N}'_a$ (eight $(10 \times 1)$-vectors)*

$Aa_1$

```
> M:=(y,z)->(t+y,z):
> aA[1]:=MatrixMatrixMultiply(R_1,champ1(M)):
```

$Aa_2$

```
> M:=(y,z)->((1+t)*y,z):
> aA[2]:=MatrixMatrixMultiply(R_1,champ1(M)):
```

$Aa_3$

```
> M:=(y,z)->(y+t*z,z):
> aA[3]:=MatrixMatrixMultiply(R_1,champ1(M)):
```

$Aa_4$

```
> M:=(y,z)->(y,t+z):
> aA[4]:=MatrixMatrixMultiply(R_1,champ1(M)):
```

$Aa_5$

```
> M:=(y,z)->(y,t*y+z):
> aA[5]:=MatrixMatrixMultiply(R_1,champ1(M)):
```

$Aa_6$

```
> M:=(y,z)->(y,(1+t)*z):
> aA[6]:=MatrixMatrixMultiply(R_1,champ1(M)):
```

$Aa_7$

```
> M:=(y,z)->(y/(1-t*y),z/(1-t*y)):
```

```
>  aA[7]:=MatrixMatrixMultiply(R_1,champ1(M)):
```

$Aa_8$

```
>  M:=(y,z)->(y/(1-t*z),z/(1-t*z)):
>  aA[8]:=MatrixMatrixMultiply(R_1,champ1(M)):
```

Columns of $\mathcal{N}_a$ to which we add 55 zeros

$Aa_1$

```
>  pro1:=[-mu[1],-1,2,0,0,0,-mu[4]^2,-2*mu[4],0,0,0,0,0,0,0,0,0,0,0,0,0,0,0,0,0,0,0,0,0,
   0,0,0,0,0,0,0,0,0,0,0,0,0,0,0,0,0,0,0,0,0,0,0,0,0,0,0,0,0,0,0,0,0,0,0,0];
>  nops(pro1);
>  aaA[1]:=convert(pro1,Vector):
```

$pro1 := [-\mu_1, -1, 2, 0, 0, 0, -\mu_4^2, -2\mu_4, 0, 0, 0, 0, 0, 0, 0, 0, 0, 0, 0, 0, 0, 0, 0, 0, 0, 0, 0, 0, 0, 0, 0, 0, 0, 0, 0, 0, 0, 0, 0, 0, 0, 0, 0, 0, 0, 0,$
$0, 0, 0, 0, 0, 0, 0, 0, 0, 0, 0, 0, 0, 0, 0, 0, 0, 0, 0]$

65

$Aa_2$

```
>  pro2:=[0,0,0,0,0,0,3,0,0,0,0,0,0,0,0,0,0,0,0,0,0,0,0,0,0,0,0,0,0,0,0,0,0,0,0,0,0,0,
   0,0,0,0,0,0,0,0,0,0,0,0,0,0,0,0,0,0,0,0,0,0,0,0,0,0,0];
>  nops(pro2);
>  aaA[2]:=convert(pro2,Vector):
```

$pro2 := [0, 0, 0, 0, 0, 0, 3, 0, 0, 0, 0, 0, 0, 0, 0, 0, 0, 0, 0, 0, 0, 0, 0, 0, 0, 0, 0, 0, 0, 0, 0, 0, 0, 0, 0, 0, 0, 0, 0, 0, 0, 0, 0, 0, 0, 0, 0, 0$
$, 0, 0, 0, 0, 0, 0, 0, 0, 0, 0, 0, 0, 0, 0, 0, 0, 0]$

65

$Aa_3$

```
>  pro3:=[0,0,0,0,0,0,0,0,3,0,0,0,0,0,0,0,0,0,0,0,0,0,0,0,0,0,0,0,0,0,0,0,0,0,0,0,0,0,
   0,0,0,0,0,0,0,0,0,0,0,0,0,0,0,0,0,0,0,0,0,0,0,0,0,0,0];
>  nops(pro3);
>  aaA[3]:=convert(pro3,Vector):
```

$pro3 := [0, 0, 0, 0, 0, 0, 0, 0, 3, 0, 0, 0, 0, 0, 0, 0, 0, 0, 0, 0, 0, 0, 0, 0, 0, 0, 0, 0, 0, 0, 0, 0, 0, 0, 0, 0, 0, 0, 0, 0, 0, 0, 0, 0, 0, 0, 0, 0$
$, 0, 0, 0, 0, 0, 0, 0, 0, 0, 0, 0, 0, 0, 0, 0, 0, 0]$

65

$Aa_4$

```
>  pro4:=[1,0,0,0,0,0,-mu[4]^3,-3*mu[4]^2,2*mu[5]^2,4*mu[5],0,0,0,0,0,0,0,0,0,0,0,0,0,0,0,
   0,0,0,0,0,0,0,0,0,0,0,0,0,0,0,0,0,0,0,0,0,0,0,0,0,0,0,0,0,0,0,0,0,0,0,0,0,0,0,0];
>  nops(pro4);
>  aaA[4]:=convert(pro4,Vector):
```

$pro4 := [1, 0, 0, 0, 0, 0, -\mu_4^3, -3\mu_4^2, 2\mu_5^2, 4\mu_5, 0, 0, 0, 0, 0, 0, 0, 0, 0, 0, 0, 0, 0, 0, 0, 0, 0, 0, 0, 0, 0, 0, 0, 0, 0, 0, 0, 0, 0, 0, 0,$
$0, 0, 0, 0, 0, 0, 0, 0, 0, 0, 0, 0, 0, 0, 0, 0, 0, 0, 0, 0, 0, 0, 0, 0]$

65

$Aa_5$

```
>  pro5:=[0,0,-mu[2],-1,-2*mu[3],-2,-3*mu[4],-3,-4*mu[5],-4,0,0,0,0,0,0,0,0,0,0,0,0,0,0,
   0,0,0,0,0,0,0,0,0,0,0,0,0,0,0,0,0,0,0,0,0,0,0,0,0,0,0,0,0,0,0,0,0,0,0,0,0,0,0,0,0];
>  nops(pro5);
>  aaA[5]:=convert(pro5,Vector):
```

$pro5 := [0, 0, -\mu_2, -1, -2\mu_3, -2, -3\mu_4, -3, -4\mu_5, -4, 0, 0, 0, 0, 0, 0, 0, 0, 0, 0, 0, 0, 0, 0, 0, 0, 0, 0, 0, 0, 0, 0, 0,$
$0, 0, 0, 0, 0, 0, 0, 0, 0, 0, 0, 0, 0, 0, 0, 0, 0, 0, 0, 0, 0, 0, 0, 0, 0, 0, 0, 0, 0, 0, 0, 0]$

65

$Aa_6$

```
>  pro6:=[0,0,0,0,0,0,-2,0,0,0,0,0,0,0,0,0,0,0,0,0,0,0,0,0,0,0,0,0,0,0,0,0,0,0,0,0,0,0,
   0,0,0,0,0,0,0,0,0,0,0,0,0,0,0,0,0,0,0,0,0,0,0,0,0,0,0];
>  nops(pro6);
>  aaA[6]:=convert(pro6,Vector):
```

$pro6 := [0, 0, 0, 0, 0, 0, -2, 0, 0, 0, 0, 0, 0, 0, 0, 0, 0, 0, 0, 0, 0, 0, 0, 0, 0, 0, 0, 0, 0, 0, 0, 0, 0, 0, 0, 0, 0, 0, 0, 0, 0, 0, 0, 0, 0, 0, 0, 0$
$, 0, 0, 0, 0, 0, 0, 0, 0, 0, 0, 0, 0, 0, 0, 0, 0, 0]$

$Aa_7$

```
>  pro7:=[0,0,0,0,0,0,0,0,0,0,0,0,0,0,0,0,0,0,0,0,0,0,0,0,0,0,0,0,0,0,0,0,0,0,0,0,0,
   0,0,0,0,0,0,0,0,0,0,0,0,0,0,0,0,0,0,0,0,0,0,0,0,0,0,0];
>  nops(pro7);
>  aaA[7]:=convert(pro7,Vector):
```

$pro7 := [0,0,0,0,0,0,0,0,0,0,0,0,0,0,0,0,0,0,0,0,0,0,0,0,0,0,0,0,0,0,0,0,0,0,0,0,0,0,0,0,0,0,$
$0,0,0,0,0,0,0,0,0,0,0,0,0,0,0,0,0,0,0,0,0,0,0]$

65

$Aa_8$

```
>  pro8:=[0,0,0,0,0,0,0,0,0,0,0,0,0,0,0,0,0,0,0,0,0,0,0,0,0,0,0,0,0,0,0,0,0,0,0,0,0,
   0,0,0,0,0,0,0,0,0,0,0,0,0,0,0,0,0,0,0,0,0,0,0,0,0,0,0,0];
>  nops(pro8);
>  aaA[8]:=convert(pro7,Vector):
```

$pro8 := [0,0,0,0,0,0,0,0,0,0,0,0,0,0,0,0,0,0,0,0,0,0,0,0,0,0,0,0,0,0,0,0,0,0,0,0,0,0,0,0,0,0,$
$0,0,0,0,0,0,0,0,0,0,0,0,0,0,0,0,0,0,0,0,0,0,0]$

65

Columns of $\mathcal{N}_a''$ (eight $(10 \times 1)$-vectors)

$Aa_1$

```
>  M:=(y,z)->(t+y,z):
>  aaaA[1]:=simplify(MatrixMatrixMultiply(R_2,champ2(M))):
```

$Aa_2$

```
>  M:=(y,z)->((1+t)*y,z):
>  aaaA[2]:=simplify(MatrixMatrixMultiply(R_2,champ2(M))):
```

$Aa_3$

```
>  M:=(y,z)->(y+t*z,z):
>  aaaA[3]:=simplify(MatrixMatrixMultiply(R_2,champ2(M))):
```

$Aa_4$

```
>  M:=(y,z)->(y,t+z):
>  aaaA[4]:=simplify(MatrixMatrixMultiply(R_2,champ2(M))):
```

$Aa_5$

```
>  M:=(y,z)->(y,t*y+z):
>  aaaA[5]:=simplify(MatrixMatrixMultiply(R_2,champ2(M))):
```

$Aa_6$

```
>  M:=(y,z)->(y,(1+t)*z):
>  aaaA[6]:=simplify(MatrixMatrixMultiply(R_2,champ2(M))):
```

$Aa_7$

```
>  M:=(y,z)->(y/(1-t*y),z/(1-t*y)):
>  aaaA[7]:=simplify(MatrixMatrixMultiply(R_2,champ2(M))):
```

$Aa_8$

```
>  M:=(y,z)->(y/(1-t*z),z/(1-t*z)):
>  aaaA[8]:=simplify(MatrixMatrixMultiply(R_2,champ2(M))):
```

Glueing

We form the first eight vectors of type a of $V(\mathfrak{D}_2)$.

```
>  for k from 1 to 8 do
>  for j from 1 to 10 do
>  VectDef_a[k][j]:=aA[k][j]
>  od
>  od:
>  for k from 1 to 8 do
>  for j from 11 to 75 do
>  VectDef_a[k][j]:=aaA[k][j-10]
```

```
>    od
>    od:
>    for k from 1 to 8 do
>    for j from 76 to 85 do
>    VectDef_a[k][j]:=aaaA[k][j-75]
>    od
>    od:
>    for k from 1 to 8 do
>    VectDef_a[k]:=convert(VectDef_a[k],list)
>    od:
>    for k from 1 to 8 do
>    VectDef_a[k]:=convert(VectDef_a[k],Vector)
>    od:
```

### 1.6.2 Vectors of type b (we compute the columns of the matrices $\mathcal{N}_b'$, $\mathcal{N}_b$ and $\mathcal{N}_b''$)

*Columns of $\mathcal{N}_b'$ (twenty five $(10 \times 1)$-vectors)*

```
>    Mg:=(y,z,p,q)->(y+t*y^p/z^q,z);
```

$$Mg := (y, z, p, q) \mapsto y + \frac{ty^p}{z^q}, z$$

```
>    M:=(y,z)->Mg(y,z,p,q);
```

$$M := (y, z) \mapsto Mg(y, z, p, q)$$

$Ab_{p,q}$

```
>    bA:=unapply(MatrixMatrixMultiply(R_1,champ1(M)),(p,q)):
```

*Colums of $\mathcal{N}_b$ to which we add a $(25 \times 25)$-identity block then 30 zeros*

$Ab_{0,1}$

```
>    pro1:=[0,0,0,0,0,0,2*mu[4]^5,10*mu[4]^4,3*mu[5]^3,9*mu[5]^2, 1,0,0,0,0,0,0,0,0,0,0,0,0,0,0,
0,0,0,0,0,0,0,0,0,0,0,0,0,0,0,0,0,0,0,0,0,0,0,0,0,0,0,0,0,0,0,0,0,0];
>    nops(pro1);
>    bbA[0,1]:=convert(pro1,Vector):
>    bbA[0,1][11];
```

$pro1 := [0, 0, 0, 0, 0, 0, 2\,\mu_4^5, 10\,\mu_4^4, 3\,\mu_5^3, 9\,\mu_5^2, 1, 0, 0, 0, 0, 0, 0, 0, 0, 0, 0, 0, 0, 0, 0, 0, 0, 0, 0, 0, 0, 0, 0, 0, 0, 0, 0, 0, 0, 0, 0, 0, 0, 0, 0, 0, 0, 0, 0, 0, 0, 0, 0, 0, 0, 0, 0, 0, 0, 0]$

65

1

$Ab_{1,1}$

```
>    pro2:=[-1,0,0,0,0,0,mu[4]^3,3*mu[4]^2,-2*mu[5]^2,-4*mu[5],0,1,0,0,0,0,0,0,0,0,0,0,0,0,0,0,
0,0,0,0,0,0,0,0,0,0,0,0,0,0,0,0,0,0,0,0,0,0,0,0,0,0,0,0,0,0,0,0,0,0];
>    nops(pro2);
>    bbA[1,1]:=convert(pro2,Vector):
>    bbA[1,1][12];
```

$pro2 := [-1, 0, 0, 0, 0, 0, \mu_4^3, 3\,\mu_4^2, -2\,\mu_5^2, -4\,\mu_5, 0, 1, 0, 0, 0, 0, 0, 0, 0, 0, 0, 0, 0, 0, 0, 0, 0, 0, 0, 0, 0, 0, 0, 0, 0, 0, 0, 0, 0, 0, 0, 0, 0, 0, 0, 0, 0, 0, 0, 0, 0, 0, 0, 0, 0, 0, 0, 0, 0, 0]$

65

1

$Ab_{2,1}$

```
>    pro3:=[0,0,2*mu[2],2,3*mu[3],3,4*mu[4],4,5*mu[5],5, 0,0,1,0,0,0,0,0,0,0,0,0,0,0,0,0,0,0,0,0,
0,0,0,0,0,0,0,0,0,0,0,0,0,0,0,0,0,0,0,0,0,0,0,0,0,0,0,0,0,0];
>    nops(pro3);
>    bbA[2,1]:=convert(pro3,Vector):
>    bbA[2,1][13];
```

$pro3 := [0, 0, 2\,\mu_2, 2, 3\,\mu_3, 3, 4\,\mu_4, 4, 5\,\mu_5, 5, 0, 0, 1, 0, 0, 0, 0, 0, 0, 0, 0, 0, 0, 0, 0, 0, 0, 0, 0, 0, 0, 0, 0, 0, 0, 0, 0, 0, 0, 0, 0, 0, 0, 0, 0, 0, 0, 0, 0, 0, 0, 0, 0, 0, 0, 0, 0, 0, 0, 0]$

65

1

$Ab_{0,2}$

```
>    pro4:=[0,0,0,0,0,0,-9*mu[4]^8,-72*mu[4]^7,0,0,0,0,0,1,0,0,0,0,0,0,0,0,0,0,0,0,0,0,0,0,0,
0,0,0,0,0,0,0,0,0,0,0,0,0,0,0,0,0,0,0,0,0,0,0,0,0,0,0,0,0,0,0,0,0,0,0,0];
>    nops(pro4);
>    bbA[0,2]:=convert(pro4,Vector):
>    bbA[0,2][14];
```

$pro4 := [0,0,0,0,0,0,-9\,\mu_4^8,-72\,\mu_4^7,0,0,0,0,0,1,0,0,0,0,0,0,0,0,0,0,0,0,0,0,0,0,0,0,0,0,0,0,$
$0,0,0,0,0,0,0,0,0,0,0,0,0,0,0,0,0,0,0,0,0,0,0,0,0,0,0,0,0]$

65

1

$Ab_{1,2}$

```
>    pro5:=[0,0,0,0,0,0,-3*mu[4]^6,-18*mu[4]^5,0,0,0,0,0,0,1,0,0,0,0,0,0,0,0,0,0,0,0,0,0,0,0,0,
0,0,0,0,0,0,0,0,0,0,0,0,0,0,0,0,0,0,0,0,0,0,0,0,0,0,0,0,0,0,0,0,0];
>    nops(pro5);
>    bbA[1,2]:=convert(pro5,Vector):
>    bbA[1,2][15];
```

$pro5 := [0,0,0,0,0,0,-3\,\mu_4^6,-18\,\mu_4^5,0,0,0,0,0,0,1,0,0,0,0,0,0,0,0,0,0,0,0,0,0,0,0,0,0,0,0,0,$
$0,0,0,0,0,0,0,0,0,0,0,0,0,0,0,0,0,0,0,0,0,0,0,0,0,0,0]$

65

1

$Ab_{2,2}$

```
>    pro6:=[-1/mu[1],1/mu[1]^2,0,0,0,0,-mu[4]^4,-4*mu[4]^3,0,0,0,0,0,0,0,1,0,0,0,0,0,0,0,0,0,0,
0,0,0,0,0,0,0,0,0,0,0,0,0,0,0,0,0,0,0,0,0,0,0,0,0,0,0,0,0,0,0,0,0,0,0,0,0,0,0];
>    nops(pro6);
>    bbA[2,2]:=convert(pro6,Vector):
>    bbA[2,2][16];
```

$pro6 := [-\mu_1^{-1},\mu_1^{-2},0,0,0,0,-\mu_4^4,-4\,\mu_4^3,0,0,0,0,0,0,0,1,0,0,0,0,0,0,0,0,0,0,0,0,0,0,0,0,0,0,0,0,$
$0,0,0,0,0,0,0,0,0,0,0,0,0,0,0,0,0,0,0,0,0,0,0,0,0,0,0]$

65

1

$Ab_{3,2}$

```
>    pro7:=[0,0,0,0,0,0,0,0,0,0,0,0,0,0,0,0,0,1,0,0,0,0,0,0,0,0,0,0,0,0,0,0,0,0,0,0,0,0,0,
0,0,0,0,0,0,0,0,0,0,0,0,0,0,0,0,0,0,0,0,0,0,0,0,0,0];
>    nops(pro7);
>    bbA[3,2]:=convert(pro7,Vector):
>    bbA[3,2][17];
```

$pro7 := [0,0,0,0,0,0,0,0,0,0,0,0,0,0,0,0,0,1,0,0,0,0,0,0,0,0,0,0,0,0,0,0,0,0,0,0,0,0,0,0,0,0,0,0,0,0,0,$
$0,0,0,0,0,0,0,0,0,0,0,0,0,0,0,0,0,0]$

65

1

$Ab_{0,3}$

```
>    pro8:=[0,0,0,0,0,0,52*mu[4]^11,572*mu[4]^10,-28*mu[5]^6,-168*mu[5]^5,0,0,0,0,0,0,0,0,1,0,
0,0,0,0,0,0,0,0,0,0,0,0,0,0,0,0,0,0,0,0,0,0,0,0,0,0,0,0,0,0,0,0,0,0,0,0,0,0,0,0,0,0,0,0,
0];
>    nops(pro8);
>    bbA[0,3]:=convert(pro8,Vector):
>    bbA[0,3][18];
```

$pro8 := [0,0,0,0,0,0,52\,\mu_4^{11},572\,\mu_4^{10},-28\,\mu_5^6,-168\,\mu_5^5,0,0,0,0,0,0,0,0,1,0,0,0,0,0,0,0,0,0,0,0,0,0,0,0,$
$0,0,0,0,0,0,0,0,0,0,0,0,0,0,0,0,0,0,0,0,0,0,0,0,0,0,0,0,0,0,0]$

65

$Ab_{1,3}$

```
>  pro9:=[0,0,0,0,0,0,15*mu[4]^9,135*mu[4]^8,12*mu[5]^5,60*mu[5]^4,0,0,0,0,0,0,0,0,1,0,0,0,
0,0,0,0,0,0,0,0,0,0,0,0,0,0,0,0,0,0,0,0,0,0,0,0,0,0,0,0,0,0,0,0,0,0,0,0,0,0,0,0,0];
>  nops(pro1);
>  bbA[1,3]:=convert(pro9,Vector):
>  bbA[1,3][19];
```

$pro9 := [0,0,0,0,0,0,15\mu_4^9,135\mu_4^8,12\mu_5^5,60\mu_5^4,0,0,0,0,0,0,0,0,1,0,0,0,0,0,0,0,0,0,0,0,0,0,0,0,0,0,0,0,0,0,0,0,0,0,0,0,0,0,0,0,0,0,0,0,0,0,0,0,0,0,0,0,0,0,0]$

65

1

$Ab_{2,3}$

```
>  pro10:=[0,0,0,0,0,0,4*mu[4]^7,28*mu[4]^6,-5*mu[5]^4,-20*mu[5]^3,0,0,0,0,0,0,0,0,0,1,0,0,
0,0,0,0,0,0,0,0,0,0,0,0,0,0,0,0,0,0,0,0,0,0,0,0,0,0,0,0,0,0,0,0,0,0,0,0,0,0,0,0,0,0,0];
>  nops(pro10);
>  bbA[2,3]:=convert(pro10,Vector):
>  bbA[2,3][20];
```

$pro10 := [0,0,0,0,0,0,4\mu_4^7,28\mu_4^6,-5\mu_5^4,-20\mu_5^3,0,0,0,0,0,0,0,0,0,1,0,0,0,0,0,0,0,0,0,0,0,0,0,0,0,0,0,0,0,0,0,0,0,0,0,0,0,0,0,0,0,0,0,0,0,0,0,0,0,0,0,0,0,0,0]$

65

1

$Ab_{3,3}$

```
>  pro11:=[-1/mu[1]^2,2/mu[1]^3,0,0,0,0,mu[4]^5,5*mu[4]^4,2*mu[5]^3,6*mu[5]^2,0,0,0,0,0,0,0,
0,0,0,1,0,0,0,0,0,0,0,0,0,0,0,0,0,0,0,0,0,0,0,0,0,0,0,0,
0,0,0,0];
>  nops(pro11);
>  bbA[3,3]:=convert(pro11,Vector):
>  bbA[3,3][21];
```

$pro11 := [-\mu_1^{-2},2\mu_1^{-3},0,0,0,0,\mu_4^5,5\mu_4^4,2\mu_5^3,6\mu_5^2,0,0,0,0,0,0,0,0,0,0,1,0,0,0,0,0,0,0,0,0,0,0,0,0,0,0,0,0,0,0,0,0,0,0,0,0,0,0,0,0,0,0,0,0,0,0,0,0,0,0,0,0,0,0,0]$

65

1

$Ab_{4,3}$

```
>  pro12:=[0,0,0,0,0,0,0,0,0,0,0,0,0,0,0,0,0,0,0,0,0,1,0,0,0,0,0,0,0,0,0,0,0,0,0,0,0,0,0,0,
0,0,0,0,0,0,0,0,0,0,0,0,0,0,0,0,0,0,0,0,0,0,0,0,0];
>  nops(pro12);
>  bbA[4,3]:=convert(pro12,Vector):
>  bbA[4,3][22];
```

$pro12 := [0,0,0,0,0,0,0,0,0,0,0,0,0,0,0,0,0,0,0,0,0,1,0,0,0,0,0,0,0,0,0,0,0,0,0,0,0,0,0,0,0,0,0,0,0,0,0,0,0,0,0,0,0,0,0,0,0,0,0,0,0,0,0,0,0]$

65

1

$Ab_{0,4}$

```
>  pro13:=[0,0,0,0,0,0,-340*mu[4]^14,-4760*mu[4]^13,0,0,0,0,0,0,0,0,0,0,0,0,0,0,1,0,0,0,0,
0,0,0,0,0,0,0,0,0,0,0,0,0,0,0,0,0,0,0,0,0,0,0,0,0,0,0,0,0,0,0,0,0,0,0,0,0,0];
>  nops(pro13);
>  bbA[0,4]:=convert(pro13,Vector):
>  bbA[0,4][23];
```

$pro13 := [0,0,0,0,0,0,-340\mu_4^{14},-4760\mu_4^{13},0,0,0,0,0,0,0,0,0,0,0,0,0,0,1,0,0,0,0,0,0,0,0,0,0,0,0,0,0,0,0,0,0,0,0,0,0,0,0,0,0,0,0,0,0,0,0,0,0,0,0,0,0,0,0,0,0]$

65

1

$Ab_{1,4}$

```
>  pro14:=[0,0,0,0,0,0,-91*mu[4]^12,-1092*mu[4]^11,0,0,0,0,0,0,0,0,0,0,0,0,0,0,0,1,0,0,0,0,
0,0,0,0,0,0,0,0,0,0,0,0,0,0,0,0,0,0,0,0,0,0,0,0,0,0,0,0,0,0,0,0,0,0,0,0,0,0];
>  nops(pro14);
>  bbA[1,4]:=convert(pro14,Vector):
>  bbA[1,4][24];
```

$pro14 := [0,0,0,0,0,0,-91\,\mu_4^{12},-1092\,\mu_4^{11},0,0,0,0,0,0,0,0,0,0,0,0,0,0,0,1,0,0,0,0,0,0,0,0,0,0,0,0,0,0,0,0,0,0,0,0,0,0,0,0,0,0,0,0,0,0,0,0,0,0,0,0,0,0,0,0,0]$

65

1

$Ab_{2,4}$

```
>  pro15:=[0,0,0,0,0,0,-22*mu[4]^10,-220*mu[4]^9,0,0,0,0,0,0,0,0,0,0,0,0,0,0,0,0,1,0,0,0,0,
0,0,0,0,0,0,0,0,0,0,0,0,0,0,0,0,0,0,0,0,0,0,0,0,0,0,0,0,0,0,0,0,0,0,0,0];
>  nops(pro15);
>  bbA[2,4]:=convert(pro15,Vector):
>  bbA[2,4][25];
```

$pro15 := [0,0,0,0,0,0,-22\,\mu_4^{10},-220\,\mu_4^{9},0,0,0,0,0,0,0,0,0,0,0,0,0,0,0,0,1,0,0,0,0,0,0,0,0,0,0,0,0,0,0,0,0,0,0,0,0,0,0,0,0,0,0,0,0,0,0,0,0,0,0,0,0,0,0,0,0]$

65

1

$Ab_{3,4}$

```
>  pro16:=[0,0,0,0,0,0,-5*mu[4]^8,-40*mu[4]^7,0,0,0,0,0,0,0,0,0,0,0,0,0,0,0,0,0,1,0,0,0,0,
0,0,0,0,0,0,0,0,0,0,0,0,0,0,0,0,0,0,0,0,0,0,0,0,0,0,0,0,0,0,0,0,0,0,0];
>  nops(pro16);
>  bbA[3,4]:=convert(pro16,Vector):
>  bbA[3,4][26];
```

$pro16 := [0,0,0,0,0,0,-5\,\mu_4^{8},-40\,\mu_4^{7},0,0,0,0,0,0,0,0,0,0,0,0,0,0,0,0,0,1,0,0,0,0,0,0,0,0,0,0,0,0,0,0,0,0,0,0,0,0,0,0,0,0,0,0,0,0,0,0,0,0,0,0,0,0,0,0,0]$

65

1

$Ab_{4,4}$

```
>  pro17:=[-1/mu[1]^3,3/mu[1]^4,0,0,0,0,-mu[4]^6,-6*mu[4]^5,0,0,0,0,0,0,0,0,0,0,0,0,0,0,0,0,
0,0,0,1,0,0,0,0,0,0,0,0,0,0,0,0,0,0,0,0,0,0,0,0,0,0,0,0,0,0,0,0,0,0,0,0,0,0,0,0,0];
>  nops(pro17);
>  bbA[4,4]:=convert(pro17,Vector):
>  bbA[4,4][27];
```

$pro17 := [-\mu_1^{-3},3\,\mu_1^{-4},0,0,0,0,-\mu_4^{6},-6\,\mu_4^{5},0,0,0,0,0,0,0,0,0,0,0,0,0,0,0,0,0,0,0,1,0,0,0,0,0,0,0,0,0,0,0,0,0,0,0,0,0,0,0,0,0,0,0,0,0,0,0,0,0,0,0,0,0,0,0,0,0]$

65

1

$Ab_{5,4}$

```
>  pro18:=[0,0,0,0,0,0,0,0,0,0,0,0,0,0,0,0,0,0,0,0,0,0,0,0,0,0,0,1,0,0,0,0,0,0,0,0,0,0,0,
0,0,0,0,0,0,0,0,0,0,0,0,0,0,0,0,0,0,0,0,0,0,0,0,0,0];
>  nops(pro18);
>  bbA[5,4]:=convert(pro18,Vector):
>  bbA[5,4][28];
```

$pro18 := [0,0,0,0,0,0,0,0,0,0,0,0,0,0,0,0,0,0,0,0,0,0,0,0,0,0,0,1,0,0,0,0,0,0,0,0,0,0,0,0,0,0,0,0,0,0,0,0,0,0,0,0,0,0,0,0,0,0,0,0,0,0,0,0,0]$

65

1

$Ab_{0,5}$

```
>  pro19:=[0,0,0,0,0,0,2394*mu[4]^17,40698*mu[4]^16,429*mu[5]^9,3861*mu[5]^8,0,0,0,0,0,0,0,
0,0,0,0,0,0,0,0,0,0,0,0,1,0,0,0,0,0,0,0,0,0,0,0,0,0,0,0,0,0,0,0,0,0,0,0,0,0,0,0,0,0,0,0,
0,0,0];
>  nops(pro19);
>  bbA[0,5]:=convert(pro19,Vector):
>  bbA[0,5][29];
```

$pro19 := [0,0,0,0,0,0,2394\,\mu_4^{17},40698\,\mu_4^{16},429\,\mu_5^9,3861\,\mu_5^8,0,0,0,0,0,0,0,0,0,0,0,0,0,0,0,0,0,0,1,0,0,0,0,$
$0,0,0,0,0,0,0,0,0,0,0,0,0,0,0,0,0,0,0,0,0,0,0,0,0,0,0,0,0,0,0]$

65

1

$Ab_{1,5}$

```
>  pro20:=[0,0,0,0,0,0,612*mu[4]^15,9180*mu[4]^14,-165*mu[5]^8,-1320*mu[5]^7,0,0,0,0,0,0,0,
0,0,0,0,0,0,0,0,0,0,0,0,1,0,0,0,0,0,0,0,0,0,0,0,0,0,0,0,0,0,0,0,0,0,0,0,0,0,0,0,0,0,0,0,
0,0,0];
>  nops(pro20);
>  bbA[1,5]:=convert(pro20,Vector):
>  bbA[1,5][30];
```

$pro20 := [0,0,0,0,0,0,612\,\mu_4^{15},9180\,\mu_4^{14},-165\,\mu_5^8,-1320\,\mu_5^7,0,0,0,0,0,0,0,0,0,0,0,0,0,0,0,0,0,0,0,1,0,0,0,$
$0,0,0,0,0,0,0,0,0,0,0,0,0,0,0,0,0,0,0,0,0,0,0,0,0,0,0,0,0,0,0]$

65

1

$Ab_{2,5}$

```
>  pro21:=[0,0,0,0,0,0,140*mu[4]^13,1820*mu[4]^12,60*mu[5]^7,420*mu[5]^6,0,0,0,0,0,0,0,0,0,
0,0,0,0,0,0,0,0,0,0,0,1,0,0,0,0,0,0,0,0,0,0,0,0,0,0,0,0,0,0,0,0,0,0,0,0,0,0,0,0,0,0,0,0,
0];
>  nops(pro21);
>  bbA[2,5]:=convert(pro21,Vector):
>  bbA[2,5][31];
```

$pro21 := [0,0,0,0,0,0,140\,\mu_4^{13},1820\,\mu_4^{12},60\,\mu_5^7,420\,\mu_5^6,0,0,0,0,0,0,0,0,0,0,0,0,0,0,0,0,0,0,0,0,1,0,0,0,0,0,$
$0,0,0,0,0,0,0,0,0,0,0,0,0,0,0,0,0,0,0,0,0,0,0,0,0,0,0,0,0]$

65

1

$Ab_{3,5}$

```
>  pro22:=[0,0,0,0,0,0,30*mu[4]^11,330*mu[4]^10,-21*mu[5]^6,-126*mu[5]^5,0,0,0,0,0,0,0,0,0,0,
0,0,0,0,0,0,0,0,0,0,0,1,0,0,0,0,0,0,0,0,0,0,0,0,0,0,0,0,0,0,0,0,0,0,0,0,0,0,0,0,0,0,0,0,
0];
>  nops(pro22);
>  bbA[3,5]:=convert(pro22,Vector):
>  bbA[3,5][32];
```

$pro22 := [0,0,0,0,0,0,30\,\mu_4^{11},330\,\mu_4^{10},-21\,\mu_5^6,-126\,\mu_5^5,0,0,0,0,0,0,0,0,0,0,0,0,0,0,0,0,0,0,0,0,0,1,0,0,0,$
$0,0,0,0,0,0,0,0,0,0,0,0,0,0,0,0,0,0,0,0,0,0,0,0,0,0,0,0,0,0]$

65

1

$Ab_{4,5}$

```
>  pro23:=[0,0,0,0,0,0,6*mu[4]^9,54*mu[4]^8,7*mu[5]^5,35*mu[5]^4,0,0,0,0,0,0,0,0,0,0,0,0,0,0,
0,0,0,0,0,0,0,0,1,0,0,0,0,0,0,0,0,0,0,0,0,0,0,0,0,0,0,0,0,0,0,0,0,0,0,0,0,0,0,0,0];
>  nops(pro23);
>  bbA[4,5]:=convert(pro23,Vector):
>  bbA[4,5][33];
```

$pro23 := [0,0,0,0,0,0,6\,\mu_4^9,54\,\mu_4^8,7\,\mu_5^5,35\,\mu_5^4,0,0,0,0,0,0,0,0,0,0,0,0,0,0,0,0,0,0,0,0,0,0,1,0,0,0,0,0,0,0,$
$0,0,0,0,0,0,0,0,0,0,0,0,0,0,0,0,0,0,0,0,0,0,0,0,0]$

65

1

$Ab_{5,5}$

```
>  pro24:=[-1/mu[1]^4,4/mu[1]^5,0,0,0,0,mu[4]^7,7*mu[4]^6,-2*mu[5]^4,-8*mu[5]^3,0,0,0,0,0,
   0,0,0,0,0,0,0,0,0,0,0,0,0,0,0,0,0,1,0,0,0,0,0,0,0,0,0,0,0,0,0,0,0,0,0,0,0,0,0,0,0,0,0,
   0,0,0,0,0];
>  nops(pro24);
>  bbA[5,5]:=convert(pro24,Vector):
>  bbA[5,5][34];
```

$pro24 := [-\mu_1^{-4}, 4\mu_1^{-5}, 0, 0, 0, 0, \mu_4^7, 7\mu_4^6, -2\mu_5^4, -8\mu_5^3, 0, 0, 0, 0, 0, 0, 0, 0, 0, 0, 0, 0, 0, 0, 0, 0, 0, 0, 0, 0, 0, 0, 0, 1, 0, 0, 0, 0, 0, 0, 0, 0, 0, 0, 0, 0, 0, 0, 0, 0, 0, 0, 0, 0, 0, 0, 0, 0, 0, 0, 0, 0, 0, 0, 0]$

65

1

$Ab_{6,5}$

```
>  pro25:=[0,0,0,0,0,0,0,0,0,0,0,0,0,0,0,0,0,0,0,0,0,0,0,0,0,0,0,0,0,0,0,0,0,0,1,0,0,0,0,0,
   0,0,0,0,0,0,0,0,0,0,0,0,0,0,0,0,0,0,0,0,0,0,0,0,0];
>  nops(pro25);
>  bbA[6,5]:=convert(pro25,Vector):
>  bbA[6,5][35];
```

$pro25 := [0, 0, 0, 0, 0, 0, 0, 0, 0, 0, 0, 0, 0, 0, 0, 0, 0, 0, 0, 0, 0, 0, 0, 0, 0, 0, 0, 0, 0, 0, 0, 0, 0, 0, 1, 0, 0, 0, 0, 0, 0, 0, 0, 0, 0, 0, 0, 0, 0, 0, 0, 0, 0, 0, 0, 0, 0, 0, 0, 0, 0, 0, 0, 0, 0]$

65

1

Columns of $\mathscr{N}_b''$ (twenty five $(10 \times 1)$-vectors)

```
>  Mg:=(y,z,p,q)->(y+t*y^p/z^q,z);
```

$Mg := (y, z, p, q) \mapsto y + \dfrac{ty^p}{z^q}, z$

```
>  M:=(y,z)->Mg(y,z,p,q);
```

$M := (y, z) \mapsto Mg(y, z, p, q)$

$Ab_{p,q}$

```
>  bbbA:=unapply(MatrixMatrixMultiply(R_2,champ2(M)),(p,q)):
```

Glueing

We form the 25 vectors of type b of $V(\mathfrak{D}_2)$.

```
>  for q from 1 to 5 do
>  for p from 0 to q+1 do
>  for j from 1 to 10 do
>  u[p,q][j]:=bA(p,q)[j]
>  od
>  od
>  od:
>  for q from 1 to 5 do
>  for p from 0 to q+1 do
>  for j from 11 to 75 do
>  u[p,q][j]:=bbA[p,q][j-10]
>  od
>  od
>  od:
>  for q from 1 to 5 do
>  for p from 0 to q+1 do
>  for j from 76 to 85 do
>  u[p,q][j]:=bbbA(p,q)[j-75]
>  od
>  od
>  od:
>  for q from 1 to 5 do
>  for p from 0 to q+1 do
>  u[p,q]:=convert(u[p,q],list)
>  od
```

```
>   od:
>   q:=1:
>   for p from 0 to q+1 do
>   VectDef_b[p+1]:=u[p,q]
>   od:
>   q:=2:
>   for p from 0 to q+1 do
>   VectDef_b[p+4]:=u[p,q]
>   od:
>   q:=3:
>   for p from 0 to q+1 do
>   VectDef_b[p+8]:=u[p,q]
>   od:
>   q:=4:
>   for p from 0 to q+1 do
>   VectDef_b[p+13]:=u[p,q]
>   od:
>   q:=5:
>   for p from 0 to q+1 do
>   VectDef_b[p+19]:=u[p,q]
>   od:
>   p:='p':
>   q:='q':
```

### 1.6.3  Vectors of type c (we compute the columns of the matrices $\mathscr{N}'_c$, $\mathscr{N}_c$ and $\mathscr{N}''_c$)

Columns of $\mathscr{N}'_c$ (twenty five $(10 \times 1)$-vectors)

```
>   Mg:=(y,z,p,q)->(y,z+t*y^p/z^q);
```

$$Mg := (y,z,p,q) \mapsto y, z + \frac{ty^p}{z^q}$$

```
>   M:=(y,z)->Mg(y,z,p,q);
```

$M := (y,z) \mapsto Mg(y,z,p,q)$

$Ac_{p,q}$

```
>   cA:=unapply(MatrixMatrixMultiply(R_1,champ1(M)),(p,q)):
```

Columns of $\mathscr{N}_c$ to which we add 25 zeros, a $(25 \times 25)$-identity block then 5 zeros

$Ac_{0,1}$

```
>   pro1:=[0,0,0,0,0,0,4*mu[4]^6,24*mu[4]^5,0,0,0,0,0,0,0,0,0,0,0,0,0,0,0,0,0,0,0,0,0,0,0,0,0,0,0,1,0,0,0,0,0,0,0,0,0,0,0,0,0,0,0,0,0,0,0,0,0,0,0,0,0,0,0,0,0];
>   nops(pro1);
>   ccA[0,1]:=convert(pro1,Vector):
>   ccA[0,1][36];
```

$pro1 := [0,0,0,0,0,0,4\mu_4^6, 24\mu_4^5, 0,0,0,0,0,0,0,0,0,0,0,0,0,0,0,0,0,0,0,0,0,0,0,0,0,0,0,1,0,0,0,0,0,0,0,$
$0,0,0,0,0,0,0,0,0,0,0,0,0,0,0,0,0,0,0,0,0,0]$

65

1

$Ac_{1,1}$

```
>   pro2:=[1/mu[1],-1/mu[1]^2,0,0,0,0,mu[4]^4,4*mu[4]^3,0,0,0,0,0,0,0,0,0,0,0,0,0,0,0,0,0,0,0,0,0,0,0,0,0,0,0,0,1,0,0,0,0,0,0,0,0,0,0,0,0,0,0,0,0,0,0,0,0,0,0,0,0,0,0,0,0];
>   nops(pro2);
>   ccA[1,1]:=convert(pro2,Vector):
>   ccA[1,1][37];
```

$pro2 := [\mu_1^{-1}, -\mu_1^{-2}, 0,0,0,0, \mu_4^4, 4\mu_4^3, 0,0,0,0,0,0,0,0,0,0,0,0,0,0,0,0,0,0,0,0,0,0,0,0,0,0,0,0,1,0,0,0,0,$
$0,0,0,0,0,0,0,0,0,0,0,0,0,0,0,0,0,0,0,0,0,0,0,0]$

65

1

$Ac_{2,1}$

```
> pro3:=[0,0,0,0,0,0,0,0,0,0,0,0,0,0,0,0,0,0,0,0,0,0,0,0,0,0,0,0,0,0,0,0,0,0,0,0,0,1,0,0,
  0,0,0,0,0,0,0,0,0,0,0,0,0,0,0,0,0,0,0,0,0,0,0,0,0];
> nops(pro3);
> ccA[2,1]:=convert(pro3,Vector):
> ccA[2,1][38];
```

$pro3 := [0,0,0,0,0,0,0,0,0,0,0,0,0,0,0,0,0,0,0,0,0,0,0,0,0,0,0,0,0,0,0,0,0,0,0,0,0,1,0,0,0,0,0,0,0,0,0,0,0,0,0,0,0,0,0,0,0,0,0,0,0,0,0,0,0]$

65

1

$Ac_{0,2}$

```
> pro4:=[0,0,0,0,0,0,-25*mu[4]^9,-225*mu[4]^8,-18*mu[5]^5,-90*mu[5]^4,0,0,0,0,0,0,0,0,0,0,
  0,0,0,0,0,0,0,0,0,0,0,0,0,0,0,1,0,0,0,0,0,0,0,0,0,0,0,0,0,0,0,0,0,0,0,0,0,0,0,0,0,0,0,0,0];
> nops(pro4);
> ccA[0,2]:=convert(pro4,Vector):
> ccA[0,2][39];
```

$pro4 := [0,0,0,0,0,0,-25\mu_4^9,-225\mu_4^8,-18\mu_5^5,-90\mu_5^4,0,0,0,0,0,0,0,0,0,0,0,0,0,0,0,0,0,0,0,0,0,0,0,0,0,1,0,0,0,0,0,0,0,0,0,0,0,0,0,0,0,0,0,0,0,0,0,0,0,0,0,0,0,0,0]$

65

1

$Ac_{1,2}$

```
> pro5:=[0,0,0,0,0,0,-5*mu[4]^7,-35*mu[4]^6,6*mu[5]^4,24*mu[5]^3,0,0,0,0,0,0,0,0,0,0,
  0,0,0,0,0,0,0,0,0,0,0,0,0,0,0,0,1,0,0,0,0,0,0,0,0,0,0,0,0,0,0,0,0,0,0,0,0,0,0,0,0,0,0,0,0];
> nops(pro5);
> ccA[1,2]:=convert(pro5,Vector):
> ccA[1,2][40];
```

$pro5 := [0,0,0,0,0,0,-5\mu_4^7,-35\mu_4^6,6\mu_5^4,24\mu_5^3,0,0,0,0,0,0,0,0,0,0,0,0,0,0,0,0,0,0,0,0,0,0,0,0,0,0,1,0,0,0,0,0,0,0,0,0,0,0,0,0,0,0,0,0,0,0,0,0,0,0,0,0,0,0,0]$

65

1

$Ac_{2,2}$

```
> pro6:=[1/mu[1]^2,-2/mu[1]^3,0,0,0,0,-mu[4]^5,-5*mu[4]^4,-2*mu[5]^3,-6*mu[5]^2,0,0,0,0,
  0,0,0,0,0,0,0,0,0,0,0,0,0,0,0,0,0,0,0,0,0,0,0,1,0,0,0,0,0,0,0,0,0,0,0,0,0,0,0,0,0,0,0,0,0,
  0,0,0,0,0];
> nops(pro6);
> ccA[2,2]:=convert(pro6,Vector):
> ccA[2,2][41];
```

$pro6 := [\mu_1^{-2},-2\mu_1^{-3},0,0,0,0,-\mu_4^5,-5\mu_4^4,-2\mu_5^3,-6\mu_5^2,0,0,0,0,0,0,0,0,0,0,0,0,0,0,0,0,0,0,0,0,0,0,0,0,0,0,0,0,1,0,0,0,0,0,0,0,0,0,0,0,0,0,0,0,0,0,0,0,0,0,0,0,0,0,0]$

65

1

$Ac_{3,2}$

```
> pro7:=[0,0,0,0,0,0,0,0,0,0,0,0,0,0,0,0,0,0,0,0,0,0,0,0,0,0,0,0,0,0,0,0,0,0,0,0,0,0,0,0,
  0,1,0,0,0,0,0,0,0,0,0,0,0,0,0,0,0,0,0,0,0,0,0,0,0];
> nops(pro7);
> ccA[3,2]:=convert(pro7,Vector):
> ccA[3,2][42];
```

$pro7 := [0,0,0,0,0,0,0,0,0,0,0,0,0,0,0,0,0,0,0,0,0,0,0,0,0,0,0,0,0,0,0,0,0,0,0,0,0,0,0,0,0,1,0,0,0,0,0,0,0,0,0,0,0,0,0,0,0,0,0,0,0,0,0,0,0]$

65

1

$Ac_{0,3}$

```
>  pro8:=[0,0,0,0,0,0,182*mu[4]^12,2184*mu[4]^11,0,0,0,0,0,0,0,0,0,0,0,0,0,0,0,0,0,0,0,0,
   0,0,0,0,0,0,0,0,0,0,0,0,0,0,1,0,0,0,0,0,0,0,0,0,0,0,0,0,0,0,0,0,0,0,0,0,0,0];
>  nops(pro8);
>  ccA[0,3]:=convert(pro8,Vector):
>  ccA[0,3][43];
```

$pro8 := [0,0,0,0,0,0,182\mu_4^{12},2184\mu_4^{11},0,0,0,0,0,0,0,0,0,0,0,0,0,0,0,0,0,0,0,0,0,0,0,0,0,0,0,0,0,0,0,0,0,0,$
$0,0,1,0,0,0,0,0,0,0,0,0,0,0,0,0,0,0,0,0,0,0,0,0]$

65

1

$Ac_{1,3}$

```
>  pro9:=[0,0,0,0,0,0,33*mu[4]^10,330*mu[4]^9,0,0,0,0,0,0,0,0,0,0,0,0,0,0,0,0,0,0,0,0,0,0,
   0,0,0,0,0,0,0,0,0,0,0,0,0,1,0,0,0,0,0,0,0,0,0,0,0,0,0,0,0,0,0,0,0,0,0];
>  nops(pro1);
>  ccA[1,3]:=convert(pro9,Vector):
>  ccA[1,3][44];
```

$pro9 := [0,0,0,0,0,0,33\mu_4^{10},330\mu_4^9,0,0,0,0,0,0,0,0,0,0,0,0,0,0,0,0,0,0,0,0,0,0,0,0,0,0,0,0,0,0,0,0,0,0,0,$
$0,0,1,0,0,0,0,0,0,0,0,0,0,0,0,0,0,0,0,0,0,0,0]$

65

1

$Ac_{2,3}$

```
>  pro10:=[0,0,0,0,0,0,6*mu[4]^8,48*mu[4]^7,0,0,0,0,0,0,0,0,0,0,0,0,0,0,0,0,0,0,0,0,0,0,
   0,0,0,0,0,0,0,0,0,0,0,0,0,0,1,0,0,0,0,0,0,0,0,0,0,0,0,0,0,0,0,0,0,0,0];
>  nops(pro10);
>  ccA[2,3]:=convert(pro10,Vector):
>  ccA[2,3][45];
```

$pro10 := [0,0,0,0,0,0,6\mu_4^8,48\mu_4^7,0,0,0,0,0,0,0,0,0,0,0,0,0,0,0,0,0,0,0,0,0,0,0,0,0,0,0,0,0,0,0,0,0,0,0,$
$0,1,0,0,0,0,0,0,0,0,0,0,0,0,0,0,0,0,0,0,0,0,0]$

65

1

$Ac_{3,3}$

```
>  pro11:=[1/mu[1]^3,-3/mu[1]^4,0,0,0,0,mu[4]^6,6*mu[4]^5,0,0,0,0,0,0,0,0,0,0,0,0,0,0,0,0,0,
   0,0,0,0,0,0,0,0,0,0,0,0,0,0,0,0,0,0,0,0,1,0,0,0,0,0,0,0,0,0,0,0,0,0,0,0,0,0];
>  nops(pro11);
>  ccA[3,3]:=convert(pro11,Vector):
>  ccA[3,3][46];
```

$pro11 := [\mu_1^{-3},-3\mu_1^{-4},0,0,0,0,\mu_4^6,6\mu_4^5,0,0,0,0,0,0,0,0,0,0,0,0,0,0,0,0,0,0,0,0,0,0,0,0,0,0,0,0,0,0,0,0,0,0,0,$
$0,0,0,0,1,0,0,0,0,0,0,0,0,0,0,0,0,0,0,0,0,0]$

65

1

$Ac_{4,3}$

```
>  pro12:=[0,0,0,0,0,0,0,0,0,0,0,0,0,0,0,0,0,0,0,0,0,0,0,0,0,0,0,0,0,0,0,0,0,0,0,0,0,0,0,0,
   0,0,0,0,0,0,1,0,0,0,0,0,0,0,0,0,0,0,0,0,0,0,0,0];
>  nops(pro12);
>  ccA[4,3]:=convert(pro12,Vector):
>  ccA[4,3][47];
```

$pro12 := [0,0,0,0,0,0,0,0,0,0,0,0,0,0,0,0,0,0,0,0,0,0,0,0,0,0,0,0,0,0,0,0,0,0,0,0,0,0,0,0,0,0,0,0,0,0,$
$1,0,0,0,0,0,0,0,0,0,0,0,0,0,0,0,0,0]$

65

1

$Ac_{0,4}$

```
>  pro13:=[0,0,0,0,0,0,-1428*mu[4]^15,-21420*mu[4]^14,330*mu[5]^8,2640*mu[5]^7,
```

```
>   0,0,0,0,0,0,0,0,0,0,0,0,0,0,0,0,0,0,0,0,0,0,0,0,0,0,0,0,0,0,0,0,0,0,0,0,0,0,0,1,0,0,0,0,0,0,
    0,0,0,0,0,0,0,0,0,0,0,0];
>   nops(pro13);
>   ccA[0,4]:=convert(pro13,Vector):
>   ccA[0,4][48];
```

$pro13 := [0,0,0,0,0,0,-1428\,\mu_4^{15},-21420\,\mu_4^{14},330\,\mu_5^8,2640\,\mu_5^7,0,0,0,0,0,0,0,0,0,0,0,0,0,0,0,0,0,0,0,0,0,0,0,0,0,0,0,0,0,0,0,0,0,0,0,0,1,0,0,0,0,0,0,0,0,0,0,0,0,0,0,0,0,0,0]$

65

1

$Ac_{1,4}$

```
>   pro14:=[0,0,0,0,0,0,-245*mu[4]^13,-3185*mu[4]^12,-96*mu[5]^7,-672*mu[5]^6,
>   0,0,0,0,0,0,0,0,0,0,0,0,0,0,0,0,0,0,0,0,0,0,0,0,0,0,0,0,0,0,0,0,0,0,0,0,1,0,0,0,0,0,
    0,0,0,0,0,0,0,0,0,0,0];
>   nops(pro14);
>   ccA[1,4]:=convert(pro14,Vector):
>   ccA[1,4][49];
```

$pro14 := [0,0,0,0,0,0,-245\,\mu_4^{13},-3185\,\mu_4^{12},-96\,\mu_5^7,-672\,\mu_5^6,0,0,0,0,0,0,0,0,0,0,0,0,0,0,0,0,0,0,0,0,0,0,0,0,0,0,0,0,0,0,0,0,0,0,0,0,1,0,0,0,0,0,0,0,0,0,0,0,0,0,0,0,0,0,0]$

65

1

$Ac_{2,4}$

```
>   pro15:=[0,0,0,0,0,0,-42*mu[4]^11,-462*mu[4]^10,28*mu[5]^6,168*mu[5]^5,0,0,0,0,0,0,0,0,0,
    0,0,0,0,0,0,0,0,0,0,0,0,0,0,0,0,0,0,0,0,0,0,0,0,0,0,0,0,1,0,0,0,0,0,0,0,0,0,0,0,0,0,0,
    0];
>   nops(pro15);
>   ccA[2,4]:=convert(pro15,Vector):
>   ccA[2,4][50];
```

$pro15 := [0,0,0,0,0,0,-42\,\mu_4^{11},-462\,\mu_4^{10},28\,\mu_5^6,168\,\mu_5^5,0,0,0,0,0,0,0,0,0,0,0,0,0,0,0,0,0,0,0,0,0,0,0,0,0,0,0,0,0,0,0,0,0,0,0,0,1,0,0,0,0,0,0,0,0,0,0,0,0,0,0,0,0,0,0]$

65

1

$Ac_{3,4}$

```
>   pro16:=[0,0,0,0,0,0,-7*mu[4]^9,-63*mu[4]^8,-8*mu[5]^5,-40*mu[5]^4,0,0,0,0,0,0,0,0,0,0,0,
    0,0,0,0,0,0,0,0,0,0,0,0,0,0,0,0,0,0,0,0,0,0,0,0,0,0,0,1,0,0,0,0,0,0,0,0,0,0,0,0,0,0,0,0];
>   nops(pro16);
>   ccA[3,4]:=convert(pro16,Vector):
>   ccA[3,4][51];
```

$pro16 := [0,0,0,0,0,0,-7\,\mu_4^9,-63\,\mu_4^8,-8\,\mu_5^5,-40\,\mu_5^4,0,0,0,0,0,0,0,0,0,0,0,0,0,0,0,0,0,0,0,0,0,0,0,0,0,0,0,0,0,0,0,0,0,0,0,0,0,0,1,0,0,0,0,0,0,0,0,0,0,0,0,0,0,0,0]$

65

1

$Ac_{4,4}$

```
>   pro17:=[1/mu[1]^4,-4/mu[1]^5,0,0,0,0,-mu[4]^7,-7*mu[4]^6,2*mu[5]^4,8*mu[5]^3,0,0,0,0,0,
    0,0,0,0,0,0,0,0,0,0,0,0,0,0,0,0,0,0,0,0,0,0,0,0,0,0,0,0,0,0,0,0,0,1,0,0,0,0,0,0,0,0,
    0,0,0,0,0];
>   nops(pro17);
>   ccA[4,4]:=convert(pro17,Vector):
>   ccA[4,4][52];
```

$pro17 := [\mu_1^{-4},-4\,\mu_1^{-5},0,0,0,0,-\mu_4^7,-7\,\mu_4^6,2\,\mu_5^4,8\,\mu_5^3,0,0,0,0,0,0,0,0,0,0,0,0,0,0,0,0,0,0,0,0,0,0,0,0,0,0,0,0,0,0,0,0,0,0,0,0,0,0,1,0,0,0,0,0,0,0,0,0,0,0,0,0,0,0,0]$

65

1

$Ac_{5,4}$

```
>  pro18:=[0,0,0,0,0,0,0,0,0,0,0,0,0,0,0,0,0,0,0,0,0,0,0,0,0,0,0,0,0,0,0,0,0,0,0,0,
   0,0,0,0,0,0,0,0,0,0,0,0,0,0,0,1,0,0,0,0,0,0,0,0,0,0,0,0];
>  nops(pro18);
>  ccA[5,4]:=convert(pro18,Vector):
>  ccA[5,4][53];
```

$pro18 := [0,0,0,0,0,0,0,0,0,0,0,0,0,0,0,0,0,0,0,0,0,0,0,0,0,0,0,0,0,0,0,0,0,0,0,0,0,0,0,0,0,0,0,0,0,0,0,0,0,0,0,0,1,0,0,0,0,0,0,0,0,0,0,0,0]$

65

1

$Ac_{0,5}$

```
>  pro19:=[0,0,0,0,0,0,11704*mu[4]^18,210672*mu[4]^17,0,0,0,0,0,0,0,0,0,0,0,0,0,0,0,0,
   0,0,0,0,0,0,0,0,0,0,0,0,0,0,0,0,0,0,0,0,0,0,0,0,0,0,0,1,0,0,0,0,0,0,0,0,0,0,0,0];
>  nops(pro19);
>  ccA[0,5]:=convert(pro19,Vector):
>  ccA[0,5][54];
```

$pro19 := [0,0,0,0,0,0,11704\,\mu_4^{18},210672\,\mu_4^{17},0,0,0,0,0,0,0,0,0,0,0,0,0,0,0,0,0,0,0,0,0,0,0,0,0,0,0,0,0,0,0,0,0,0,0,0,0,0,0,0,0,0,0,0,0,1,0,0,0,0,0,0,0,0,0,0,0]$

65

1

$Ac_{1,5}$

```
>  pro20:=[0,0,0,0,0,0,1938*mu[4]^16,31008*mu[4]^15,0,0,0,0,0,0,0,0,0,0,0,0,0,0,0,0,0,
   0,0,0,0,0,0,0,0,0,0,0,0,0,0,0,0,0,0,0,0,0,0,0,0,0,0,0,0,1,0,0,0,0,0,0,0,0,0,0];
>  nops(pro20);
>  ccA[1,5]:=convert(pro20,Vector):
>  ccA[1,5][55];
```

$pro20 := [0,0,0,0,0,0,1938\,\mu_4^{16},31008\,\mu_4^{15},0,0,0,0,0,0,0,0,0,0,0,0,0,0,0,0,0,0,0,0,0,0,0,0,0,0,0,0,0,0,0,0,0,0,0,0,0,0,0,0,0,0,0,0,0,0,1,0,0,0,0,0,0,0,0,0,0]$

65

1

$Ac_{2,5}$

```
>  pro21:=[0,0,0,0,0,0,320*mu[4]^14,4480*mu[4]^13,0,0,0,0,0,0,0,0,0,0,0,0,0,0,0,0,0,
   0,0,0,0,0,0,0,0,0,0,0,0,0,0,0,0,0,0,0,0,0,0,0,0,0,0,0,0,0,1,0,0,0,0,0,0,0,0,0,0];
>  nops(pro21);
>  ccA[2,5]:=convert(pro21,Vector):
>  ccA[2,5][56];
```

$pro21 := [0,0,0,0,0,0,320\,\mu_4^{14},4480\,\mu_4^{13},0,0,0,0,0,0,0,0,0,0,0,0,0,0,0,0,0,0,0,0,0,0,0,0,0,0,0,0,0,0,0,0,0,0,0,0,0,0,0,0,0,0,0,0,0,0,0,1,0,0,0,0,0,0,0,0,0]$

65

1

$Ac_{3,5}$

```
>  pro22:=[0,0,0,0,0,0,52*mu[4]^12,624*mu[4]^11,0,0,0,0,0,0,0,0,0,0,0,0,0,0,0,0,0,
   0,0,0,0,0,0,0,0,0,0,0,0,0,0,0,0,0,0,0,0,0,0,0,0,0,0,0,0,0,0,1,0,0,0,0,0,0,0,0];
>  nops(pro22);
>  ccA[3,5]:=convert(pro22,Vector):
>  ccA[3,5][57];
```

$pro22 := [0,0,0,0,0,0,52\,\mu_4^{12},624\,\mu_4^{11},0,0,0,0,0,0,0,0,0,0,0,0,0,0,0,0,0,0,0,0,0,0,0,0,0,0,0,0,0,0,0,0,0,0,0,0,0,0,0,0,0,0,0,0,0,0,0,0,1,0,0,0,0,0,0,0,0]$

65

1

$Ac_{4,5}$

```
>   pro23:=[0,0,0,0,0,0,8*mu[4]^10,80*mu[4]^9,0,0,0,0,0,0,0,0,0,0,0,0,0,0,0,0,0,0,0,0,0,0,
    0,0,0,0,0,0,0,0,0,0,0,0,0,0,0,0,0,0,0,0,0,0,0,0,0,0,1,0,0,0,0,0,0,0];
>   nops(pro23);
>   ccA[4,5]:=convert(pro23,Vector):
>   ccA[4,5][58];
```

$pro23 := [0,0,0,0,0,0,8\mu_4^{10},80\mu_4^9,0,0,0,0,0,0,0,0,0,0,0,0,0,0,0,0,0,0,0,0,0,0,0,0,0,0,0,0,0,$
$0,0,0,0,0,0,0,0,0,0,0,0,0,0,0,0,0,0,1,0,0,0,0,0,0,0]$

65

1

$Ac_{5,5}$

```
>   pro24:=[1/mu[1]^5,-5/mu[1]^6,0,0,0,0,mu[4]^8,8*mu[4]^7,0,0,0,0,0,0,0,0,0,0,0,0,0,0,0,0,
    0,0,0,0,0,0,0,0,0,0,0,0,0,0,0,0,0,0,0,0,0,0,0,0,0,0,0,0,0,0,0,0,0,0,1,0,0,0,0,0,0];
>   nops(pro24);
>   ccA[5,5]:=convert(pro24,Vector):
>   ccA[5,5][59];
```

$pro24 := [\mu_1^{-5},-5\mu_1^{-6},0,0,0,0,\mu_4^8,8\mu_4^7,0,0,0,0,0,0,0,0,0,0,0,0,0,0,0,0,0,0,0,0,0,0,0,0,0,0,0,0,$
$0,0,0,0,0,0,0,0,0,0,0,0,0,0,0,0,0,0,0,0,1,0,0,0,0,0,0,0]$

65

1

$Ac_{6,5}$

```
>   pro25:=[0,0,0,0,0,0,0,0,0,0,0,0,0,0,0,0,0,0,0,0,0,0,0,0,0,0,0,0,0,0,0,0,0,0,0,0,0,0,
    0,0,0,0,0,0,0,0,0,0,0,0,0,0,0,0,0,0,0,0,1,0,0,0,0,0,0];
>   nops(pro25);
>   ccA[6,5]:=convert(pro25,Vector):
>   ccA[6,5][60];
```

$pro25 := [0,0,0,0,0,0,0,0,0,0,0,0,0,0,0,0,0,0,0,0,0,0,0,0,0,0,0,0,0,0,0,0,0,0,0,0,0,0,0,0,0,0,0,0,0,0,$
$0,0,0,0,0,0,0,0,0,0,0,0,1,0,0,0,0,0,0]$

65

1

Columns of $\mathcal{N}_c''$ (twenty five $(10 \times 1)$-vectors)

```
>   Mg:=(y,z,p,q)->(y,z+t*y^p/z^q);
```

$Mg := (y,z,p,q) \mapsto y, z + \dfrac{ty^p}{z^q}$

```
>   M:=(y,z)->Mg(y,z,p,q);
```

$M := (y,z) \mapsto Mg(y,z,p,q)$

$Ac_{p,q}$

```
>   cccA:=unapply(MatrixMatrixMultiply(R_2,champ2(M)),(p,q)):
```

Glueing

We form the 25 vectors of type c of $V(\mathfrak{D}_2)$.

```
>   for q from 1 to 5 do
>   for p from 0 to q+1 do
>   for j from 1 to 10 do
>   uu[p,q][j]:=cA(p,q)[j]
>   od
>   od
>   od:
>   for q from 1 to 5 do
>   for p from 0 to q+1 do
>   for j from 11 to 75 do
>   uu[p,q][j]:=ccA[p,q][j-10]
>   od
>   od
>   od:
>   for q from 1 to 5 do
>   for p from 0 to q+1 do
```

```
>   for j from 76 to 85 do
>     uu[p,q][j]:=cccA(p,q)[j-75]
>   od
>  od
>  od:
>  for q from 1 to 5 do
>    for p from 0 to q+1 do
>      uu[p,q]:=convert(uu[p,q],list)
>    od
>  od:
>  q:=1:
>  for p from 0 to q+1 do
>    VectDef_c[p+1]:=uu[p,q]
>  od:
>  q:=2:
>  for p from 0 to q+1 do
>    VectDef_c[p+4]:=uu[p,q]
>  od:
>  q:=3:
>  for p from 0 to q+1 do
>    VectDef_c[p+8]:=uu[p,q]
>  od:
>  q:=4:
>  for p from 0 to q+1 do
>    VectDef_c[p+13]:=uu[p,q]
>  od:
>  q:=5:
>  for p from 0 to q+1 do
>    VectDef_c[p+19]:=uu[p,q]
>  od:
>  p:='p':
>  q:='q':
```

### 1.6.4 Vectors of type d (we compute the columns of the matrices $\mathcal{N}_d'$, $\mathcal{N}_d$ and $\mathcal{N}_d''$)

*Columns of $\mathcal{N}_d'$ (five $(10 \times 1)$-vectors)*

```
>   Mg:=(y,z,p)->(y+t*y^(p+2)/z^p,z+t*y^(p+1)/z^(p-1));
```

$$Mg := (y,z,p) \mapsto y + \frac{ty^{p+2}}{z^p}, z + \frac{ty^{p+1}}{z^{p-1}}$$

```
>   M:=(y,z)->Mg(y,z,p);
```

$M := (y,z) \mapsto Mg(y,z,p)$

$Ad_p$

```
>   dA:=unapply(MatrixMatrixMultiply(R_1,champ1(M)),p):
```

*Columns of $\mathcal{N}_d$ to which we add 50 zeros then a $(5 \times 5)$-identity block*

$Ad_1$

```
>  pro1:=[0,0,0,0,0,0,0,0,1,0,0,0,0,0,0,0,0,0,0,0,0,0,0,0,0,0,0,0,0,0,0,0,0,0,0,0,0,
0,0,0,0,0,0,0,0,0,0,0,0,0,0,0,0,0,0,0,0,0,0,0,1,0,0,0,0];
>  nops(pro1);
>  ddA[1]:=convert(pro1,Vector):
>  ddA[1][61];
```

*pro1* := [0,0,0,0,0,0,0,0,1,0,0,0,0,0,0,0,0,0,0,0,0,0,0,0,0,0,0,0,0,0,0,0,0,0,0,0,0, 0,0,0,0,0,0,0,0,0,0,0,0,0,0,0,0,0,0,0,0,0,0,0,1,0,0,0,0]

65

1

$Ad_2$

```
>  pro2:=[0,0,mu[2]^2,2*mu[2],0,0,1,0,0,0,0,0,0,0,0,0,0,0,0,0,0,0,0,0,0,0,0,0,0,0,0,0,0,0,
0,0,0,0,0,0,0,0,0,0,0,0,0,0,0,0,0,0,0,0,0,0,0,0,0,0,0,1,0,0,0];
>  nops(pro2);
>  ddA[2]:=convert(pro2,Vector):
>  ddA[2][62];
```

$pro2 := [0, 0, \mu_2^2, 2\mu_2, 0, 0, 1, 0, 0, 0, 0, 0, 0, 0, 0, 0, 0, 0, 0, 0, 0, 0, 0, 0, 0, 0, 0, 0, 0, 0, 0, 0, 0, 0, 0, 0, 0, 0, 0, 0, 0, 0, 0, 0, 0, 0, 0, 0, 0, 0, 0, 0, 0, 0, 0, 0, 0, 0, 0, 0, 1, 0, 0, 0]$

65

1

$Ad_3$

```
>   pro3:=[0,0,0,0,mu[3]^2,2*mu[3],2*mu[4],2,2*mu[5],2,0,0,0,0,0,0,0,0,0,0,0,0,0,0,0,
0,0,0,0,0,0,0,0,0,0,0,0,0,0,0,0,0,0,0,0,0,0,0,0,0,0,0,0,0,0,0,0,0,0,0,1,0,0];
>   nops(pro3);
>   ddA[3]:=convert(pro3,Vector):
>   ddA[3][63];
```

$pro3 := [0, 0, 0, 0, \mu_3^2, 2\mu_3, 2\mu_4, 2, 2\mu_5, 2, 0, 0, 0, 0, 0, 0, 0, 0, 0, 0, 0, 0, 0, 0, 0, 0, 0, 0, 0, 0, 0, 0, 0, 0, 0, 0, 0, 0, 0, 0, 0, 0, 0, 0, 0, 0, 0, 0, 0, 0, 0, 0, 0, 0, 0, 0, 0, 0, 0, 0, 0, 1, 0, 0]$

65

1

$Ad_4$

```
>   pro4:=[0,0,0,0,0,0,mu[4]^2,2*mu[4],0,0,0,0,0,0,0,0,0,0,0,0,0,0,0,0,0,0,0,0,0,0,0,
0,0,0,0,0,0,0,0,0,0,0,0,0,0,0,0,0,0,0,0,0,0,0,0,0,0,0,0,0,0,0,0,1,0];
>   nops(pro4);
>   ddA[4]:=convert(pro4,Vector):
>   ddA[4][64];
```

$pro4 := [0, 0, 0, 0, 0, 0, \mu_4^2, 2\mu_4, 0, 0, 0, 0, 0, 0, 0, 0, 0, 0, 0, 0, 0, 0, 0, 0, 0, 0, 0, 0, 0, 0, 0, 0, 0, 0, 0, 0, 0, 0, 0, 0, 0, 0, 0, 0, 0, 0, 0, 0, 0, 0, 0, 0, 0, 0, 0, 0, 0, 0, 0, 0, 0, 0, 0, 1, 0]$

65

1

$Ad_5$

```
>   pro5:=[0,0,0,0,0,0,0,0,mu[5]^2,2*mu[5],0,0,0,0,0,0,0,0,0,0,0,0,0,0,0,0,0,0,0,0,
0,0,0,0,0,0,0,0,0,0,0,0,0,0,0,0,0,0,0,0,0,0,0,0,0,0,0,0,0,0,0,0,0,0,1];
>   nops(pro5);
>   ddA[5]:=convert(pro5,Vector):
>   ddA[5][65];
```

$pro5 := [0, 0, 0, 0, 0, 0, 0, 0, \mu_5^2, 2\mu_5, 0, 0, 0, 0, 0, 0, 0, 0, 0, 0, 0, 0, 0, 0, 0, 0, 0, 0, 0, 0, 0, 0, 0, 0, 0, 0, 0, 0, 0, 0, 0, 0, 0, 0, 0, 0, 0, 0, 0, 0, 0, 0, 0, 0, 0, 0, 0, 0, 0, 0, 0, 0, 0, 0, 1]$

65

1

Columns of $\mathscr{N}_d''$ (five $(10 \times 1)$-vectors)

```
>   Mg:=(y,z,p)->(y+t*y^(p+2)/z^p,z+t*y^(p+1)/z^(p-1));
```

$Mg := (y, z, p) \mapsto y + \dfrac{ty^{p+2}}{z^p}, z + \dfrac{ty^{p+1}}{z^{p-1}}$

```
>   M:=(y,z)->Mg(y,z,p);
```

$M := (y, z) \mapsto Mg(y, z, p)$

$Ad_p$

```
>   dddA:=unapply(MatrixMatrixMultiply(R_2,champ2(M)),p):
```

Glueing

We form the 5 vectors of type d of $V(\mathfrak{D}_2)$.

```
>   for k from 1 to 5 do
>   for j from 1 to 10 do
>   VectDef_d[k][j]:=dA(k)[j]
>   od
>   od:
>   for k from 1 to 5 do
>   for j from 11 to 75 do
```

```
>    VectDef_d[k][j]:=ddA[k][j-10]
>    od
>    od:
>    for k from 1 to 5 do
>    for j from 76 to 85 do
>    VectDef_d[k][j]:=dddA(k)[j-75]
>    od
>    od:
>    for k from 1 to 5 do
>    VectDef_d[k]:=convert(VectDef_d[k],list)
>    od:
>    for k from 1 to 5 do
>    VectDef_d[k]:=convert(VectDef_d[k],Vector)
>    od:
```

### 1.6.5 Construction of the matrix $\mathscr{V}_2$

We put together the 8 vectors of type a, the 25 vectors of type b, the 25 vectors of type c and the 5 vectors of type d.

```
>    for k from 1 to 8 do
>    V[k]:=VectDef_a[k]
>    od:
>    for k from 9 to 33 do
>    V[k]:=VectDef_b[k-8]
>    od:
>    for k from 34 to 58 do
>    V[k]:=VectDef_c[k-33]
>    od:
>    for k from 59 to 63 do
>    V[k]:=VectDef_d[k-58]
>    od:
>    for k from 1 to 63 do
>    WW[k]:=convert(V[k],list)
>    od:
>    MatVectDef:=Transpose(Matrix([seq(WW[k],k=1..63)])):
```

We define the images of the 8 vectors of $V(\mathfrak{D}_1)$ under the action of $\mathscr{Y}$.

```
>    for k from 1 to 8 do
>    ImageVectDef[k]:=MatrixMatrixMultiply(MM,VectDef[k])
>    od:
```

We know how to compute these 8 vectors in the basis with the 65 vectors of $V(\mathfrak{D}_2)$ thanks to the relations (5.3) given in the article. We check that these relations are satisfied (these 8 instructions give 85 zeros).

```
>    for k from 1 to 85 do
>    simplify(ImageVectDef[1][k]-WW[1][k]-3*WW[62][k]+4*WW[15][k]+3*WW[36][k]);
>    od:
>    for k from 1 to 85 do
>    simplify(ImageVectDef[2][k]-WW[2][k]+3*WW[60][k]);
>    od:
>    for k from 1 to 85 do
>    simplify(ImageVectDef[3][k]-WW[3][k]+3*WW[59][k]);
>    od:
>    for k from 1 to 85 do
>    simplify(ImageVectDef[4][k]-WW[4][k]-2*WW[20][k]-WW[40][k]+2*WW[63][k]);
>    od:
>    for k from 1 to 85 do
>    simplify(ImageVectDef[5][k]-WW[5][k]-2*WW[61][k]);
>    od:
>    for k from 1 to 85 do
>    simplify(ImageVectDef[6][k]-WW[6][k]-2*WW[60][k]);
>    od:
>    for k from 1 to 85 do
>    simplify(ImageVectDef[7][k]-WW[7][k]);
>    od:
```

## 1.7 Computation of the action of $\psi_*$ on $\mathrm{H}^1(X, \mathrm{T}X)$

Construction of the supplementary $\mathscr{E}_1$ of $V(\mathfrak{D}_1)$ of dimension 22 in $W(\mathfrak{D}_1)$.

We form a $(30 \times 30)$-matrix $\mathtt{NN}_1$ whose the first 22 columns are now zero and the last 8 are the basis vectors of $V(\mathfrak{D}_1)$ computed previously.

```
>   for k from 1 to 30 do
>   for j from 23 to 30 do
>   N1[k][j]:=VectDef[j-22][k]
>   od
>   od:
>   for k from 1 to 30 do
>   for j from 1 to 22 do
>   N1[k][j]:=0
>   od
>   od:
>   for k from 1 to 30 do
>   N1[k]:=convert(N1[k],list)
>   od:
>   NN1:=Matrix([seq(N1[k],k=1..30)]):
```

We fill the first 22 columns of $\mathtt{NN}_1$ by some vectors chosen among the algebraic basis of $W(\mathfrak{D}_1)$; they will define a supplementary $\mathscr{E}_1$ of $V(\mathfrak{D}_1)$.

```
>   NN1[1,1]:=1:
>   NN1[2,2]:=1:
>   NN1[3,3]:=1:
>   NN1[4,4]:=1:
>   NN1[5,5]:=1:
>   NN1[6,6]:=1:
>   NN1[7,7]:=1:
>   NN1[8,8]:=1:
>   NN1[9,9]:=1:
>   NN1[10,10]:=1:
>   NN1[11,11]:=1:
>   NN1[12,12]:=1:
>   NN1[13,13]:=1:
>   NN1[14,14]:=1:
>   NN1[15,15]:=1:
>   NN1[16,16]:=1:
>   NN1[17,17]:=1:
>   NN1[21,18]:=1:
>   NN1[22,19]:=1:
>   NN1[23,20]:=1:
>   NN1[24,21]:=1:
>   NN1[25,22]:=1:
```

We compute the determinant of $\mathtt{NN}_1$, that is nonzero for generic parameters $\mu_i$; this shows that $\mathscr{E}_1$ is a supplementary to $V(\mathfrak{D}_1)$.

```
>   Determinant(NN1);
```

$12\,\mu_4\mu_5(-6\,\mu_4\mu_5 - 4\,\mu_4\mu_5^2 + 9\,\mu_5\alpha + 9\,\mu_4^2\alpha + 45\,\mu_5 + 45\,\mu_4^2 - 36\,\mu_4)$

We extract the first 22 columns of $\mathtt{NN}_1$ in order to form the matrix $\mathtt{left}$, that is the injection matrix of $\mathscr{E}_1$ into $W(\mathfrak{D}_1)$.

```
>   for k from 1 to 30 do
>   for j from 1 to 22 do
>   coeff1[k][j]:=NN1[k,j]:
>   od:
>   od:
>   for k from 1 to 30 do
>   coeff1[k]:=convert(coeff1[k],list)
>   od:
>   left:=Matrix([seq(coeff1[k],k=1..30)]):
```

Definition of the supplementary $\mathscr{E}_2$ of $V(\mathfrak{D}_2)$ in $W(\mathfrak{D}_2)$ of dimension 22 which is the image of $\mathscr{E}_1$ by the inclusion of $W(\mathfrak{D}_1)$ into $W(\mathfrak{D}_2)$.

We form a $(85 \times 85)$-matrix $\mathtt{NN}_2$ whose the first 22 columns are now zero and the last 63 are some basis vectors of $V(\mathfrak{D}_2)$ computed previously.

```
>  for k from 1 to 85 do
>  for j from 23 to 85 do
>  N2[k][j]:=WW[j-22][k]
>  od
>  od:
>  for k from 1 to 85 do
>  for j from 1 to 22 do
>  N2[k][j]:=0
>  od
>  od:
>  for k from 1 to 85 do
>  N2[k]:=convert(N2[k],list)
>  od:
>  NN2:=Matrix([seq(N2[k],k=1..85)]):
```

We fill the first 22 columns of $\mathtt{NN}_2$ by the basis vectors of $\mathscr{E}_2$.

```
>  NN2[1,1]:=1:
>  NN2[2,2]:=1:
>  NN2[3,3]:=1:
>  NN2[4,4]:=1:
>  NN2[5,5]:=1:
>  NN2[6,6]:=1:
>  NN2[7,7]:=1:
>  NN2[8,8]:=1:
>  NN2[9,9]:=1:
>  NN2[10,10]:=1:
>  NN2[11,11]:=1:
>  NN2[12,12]:=1:
>  NN2[13,13]:=1:
>  NN2[14,14]:=1:
>  NN2[15,15]:=1:
>  NN2[16,16]:=1:
>  NN2[17,17]:=1:
>  NN2[76,18]:=1:
>  NN2[77,19]:=1:
>  NN2[78,20]:=1:
>  NN2[79,21]:=1:
>  NN2[80,22]:=1:
```

The matrix `right` is the projection matrix of $W(\mathfrak{D}_2)$ onto $\mathscr{E}_2$.

```
>  NNN2:=MatrixInverse(NN2):
>  for k from 1 to 22 do
>  for j from 1 to 85 do
>  coeff2[k][j]:=NNN2[k,j]:
>  od:
>  od:
>  for k from 1 to 22 do
>  coeff2[k]:=convert(coeff2[k],list)
>  od:
>  right:=Matrix([seq(coeff2[k],k=1..22)]):
```

The matrix `OMG` is the matrix of $f_*$ acting on $\mathrm{H}^1(X, \mathrm{T}X)$.

```
>  OMG:=MatrixMatrixMultiply(right,MatrixMatrixMultiply(MM,left)):
>  OMG;
```

$$\begin{bmatrix} \text{` 22 x 22 `} \, (Matrix) \\ \text{`Data Type: `} \, anything \\ \text{`Storage: `} \, rectangular \\ \text{`Order: `} \, Fortran\_order \end{bmatrix}$$

We assign the parameters $\lambda_i$ et $\mu_i$ otherwise the computation is too heavy.

```
>  lambda[1]:=10:
>  lambda[2]:=-6:
>  lambda[3]:=5:
>  lambda[4]:=2:
>  lambda[5]:=3:
>  mu[1]:=3:
>  mu[2]:=-2:
```

```
>  mu[3]:=4:
>  mu[4]:=3:
>  mu[5]:=29:
```

Computation of the eigenvalues of `OMG`.

```
>  factor(CharacteristicPolynomial(OMG,x));
```

$(x^2 + 3x + 1)(x^2 + 18x + 1)(x^2 - 7x + 1)(x^2 + x + 1)(x - 1)^2(x + 1)^4(x^2 - x + 1)^4$

We look at the size of the Jordan blocks for the multiple eigenvalues.

```
>  Id:=Matrix(22,22,shape=identity):
>  Rank(OMG+Id);
>  Rank((OMG+Id)^2);
>  Rank((OMG+Id)^3);
```

19

18

18

```
>  Rank(OMG-Id);
>  Rank((OMG-Id)^2);
```

20

20

```
>  Rank(OMG-(1+I*sqrt(3))/2*Id);
>  Rank((OMG-(1+I*sqrt(3))/2*Id^2));
```

19

19

# 2 Quadratic maps fixing a cuspidal cubic curve

```
>  restart;
```

## 2.1 Kähler differentials on a cuspidal cubic curve

### 2.1.1 Automorphisms in affine coordinates

The cuspidal cubic $C$ is defined in affine coordinates $(x, y)$ by the equation $y^2 = x^3$. It can be parameterized away from the cusp by a coordinate $t$, where $x = t^{-2}$ and $y = t^{-3}$. Let $(x, y) \to (f(x, y), g(x, y))$ be a the following birational transformation of $\mathbb{A}^2$, which is regular at $(0,0)$.

```
>  f:=(x, y)->(b^2*x^2-2*a*b*y+a^2*x)/(b^2*x-a^2)^2;
>  g:=(x, y)->(b^3*y^2-3*b^2*a*x*y+3*b*a^2*x^2-a^3*y)/(b^2*x-a^2)^3;
```

$$f := (x, y) \mapsto \frac{b^2 x^2 + a^2 x - 2aby}{(b^2 x - a^2)^2}$$

$$g := (x, y) \mapsto \frac{3 ba^2 x^2 - 3 b^2 axy + b^3 y^2 - a^3 y}{(b^2 x - a^2)^3}$$

We check that $p$ induces the automorphism $t \to at + b$ of $C$

```
>  simplify (f(t^(-2), t^(-3))/g(t^(-2), t^(-3)));
```

$at + b$

### 2.1.2 Action on the Zariski cotangent space of the cusp

We compute the action of $p$ on the Zariski cotangent space of the cusp.

```
>  subs(x=0, y=0, simplify(diff(f(x,y), x)));
>  subs(x=0, y=0, simplify(diff(f(x,y), y)));
>  subs(x=0, y=0, simplify(diff(g(x,y), x)));
>  subs(x=0, y=0, simplify(diff(g(x,y), y)));
```

$a^{-2}$

$$-2\frac{b}{a^3}$$

$$0$$

$$a^{-3}$$

### 2.1.3 Action on torsion differentials

If $\tau$ is the torsion differential $2xdy - 3ydx$, we write $p^*\tau = udx + vdy$ and we compute $u$ and $v$.

```
> u:=(x, y)->2*f(x, y)*diff(g(x, y), x)-3*g(x, y)*diff(f(x, y), x):
> v:=(x, y)->2*f(x, y)*diff(g(x, y), y)-3*g(x, y)*diff(f(x, y), y):
> factor(u(x,y));
> factor(v(x,y));
```

$$-3\frac{-2a^2b^3x^3 + 4ab^4x^2y - 2b^5xy^2 + a^4bx^2 - 5a^3b^2xy + 3a^2b^3y^2 + a^5y}{(-b^2x + a^2)^5}$$

$$2\frac{3ab^4x^3 - 2b^5x^2y - 5a^3b^2x^2 + a^2b^3xy + ab^4y^2 + a^5x + a^4by}{(-b^2x + a^2)^5}$$

Using that $x^2 = y^3$, we see that $p^*\tau = w(x,y) \times \tau$ where $w$ is defined as follows:

```
> w:=(x, y)->(2*b^5*x*y-4*b^4*a*x^2-b^4*a*y^2-b^3*a^2*y+5*b^2*a^3*x-a^5)/(b^2*x-a^2)^5;
```

$$w := (x,y) \mapsto \frac{-4b^4ax^2 - b^4ay^2 + 2b^5xy + 5b^2a^3x - b^3a^2y - a^5}{(b^2x - a^2)^5}$$

We have $p^*\tau = w(0,0) \times \tau + \frac{\partial w}{\partial x}(0,0) \times x\tau$ and $p^*(x\tau) = \frac{\partial (fw)}{\partial x}(0,0) \times x\tau$. We compute the three coefficients $w(0,0), \frac{\partial w}{\partial x}(0,0)$ and $\frac{\partial (fw)}{\partial x}(0,0)$.

```
> w(0,0);
> subs(x=0, y=0, simplify(diff(w(x, y), x)));
> subs(x=0, y=0, simplify(diff(f(x,y)*w(x, y), x)));
```

$$a^{-5}$$

$$0$$

$$a^{-7}$$

```
> subs(x=0, y=0, simplify(diff(w(x, y), x)));
```

$$0$$

```
> subs(x=0, y=0, simplify(diff(f(x,y)*w(x, y), x)));
```

$$a^{-7}$$

```
> subs(x=0, y=0, simplify(diff(f(x,y), x)));
> subs(x=0, y=0, simplify(diff(f(x,y), y)));
> subs(x=0, y=0, simplify(diff(g(x,y), x)));
> subs(x=0, y=0, simplify(diff(g(x,y), y)));
```

$$a^{-2}$$

$$-2\frac{b}{a^3}$$

$$0$$

$$a^{-3}$$

## 2.2 First definitions

```
> restart;
```

### 2.2.1 The quadratic transform

Definition of the linear transform mapping $[1,1,1], [a,1,a^3]$ and $[b,1,b^3]$ on $[1,0,0], [0,1,0]$ and $[0,0,1]$.

```
> A:=(b^3-a^3)*x+(b*a^3-a*b^3)*y+(a-b)*z;
> B:=(1-b^3)*x+(b^3-b)*y+(b-1)*z;
> C:=(a^3-1)*x+(a-a^3)*y+(1-a)*z;
```

$$A := (b^3 - a^3)x + (ba^3 - ab^3)y + (a-b)z$$

$$B := (1 - b^3)x + (b^3 - b)y + (b - 1)z$$

$$C := (a^3 - 1)x + (a - a^3)y + (1 - a)z$$

Composition with the Cremona involution. The transformation $[x, y, z] \to [X, Y, Z]$ is a quadratic birational transformation whose indeterminacy points are $[1, 1, 1]$, $[a, 1, a^3]$ and $[b, 1, b^3]$.

```
>   X:=B*C:
>   Y:=A*C:
>   Z:=A*B:
```

Definition of a 3x3 matrix.

```
>   m[1]:=3*(a-b)^4*(b+1+4*a)^2*(2*b-1+2*a)^3*(2*a-b+2)^2*(b-1)*(b^2+2*b-b*a+2*a+4+a^2):
>   m[2]:=-(3*(a-2-2*b))*(a-b)^2*(b+1+4*a)^2*(2*b-1+2*a)^2*(2*a-b+2)^2*(b-1)^3*(b^2+2*b-b*a+
2*a+4+a^2):
>   m[3]:=3*(-1+a)^2*(b^2+2*b-b*a+2*a+4+a^2)*(b-1)*(2*a-b+2)^3*(2*b-1+2*a)^2*(a-b)^2*(b+1+4*
a)^2:
>   m[4]:=-(2*a-b+2)^3*(2*b-1+2*a)^3*(b+1+4*a)^3*(a-b)^5:
>   m[5]:=-(2*a-b+2)^3*(b+1+4*a)^3*(a-b)^3*(2*b-1+2*a)^3*(b-1)^2:
>   m[6]:=-(2*b-1+2*a)^3*(a-b)^3*(b+1+4*a)^3*(2*a-b+2)^3*(-1+a)^2:
>   m[7]:=27*(b^2+2*b-b*a+2*a+4+a^2)^3*(b-1)^3*(2*b-1+2*a)^3*(a-b)^2:
>   m[8]:=-27*(b-1)^5*(b^2+2*b-b*a+2*a+4+a^2)^3*(a-2-2*b)^3:
>   m[9]:=27*(b-1)^3*(2*a-b+2)^3*(-1+a)^2*(b^2+2*b-b*a+2*a+4+a^2)^3:
```

Introduction of an auxiliary coefficient `c`.

```
>   c:=(1/9)*(2*a-b+2)*(2*b-1+2*a)*(b+1+4*a)*(a-b)/((b-1)*(b^2+2*b-b*a+2*a+4+a^2));
```

$$c := 1/9 \, \frac{(2a - b + 2)(2b - 1 + 2a)(b + 1 + 4a)(a - b)}{(b - 1)(b^2 + 2b - ba + 2a + 4 + a^2)}$$

Composition of $[x, y, z] \to [X, Y, Z]$ with $m$ to abtain a birational transform fixing the cubic $y^2 z = x^3$. Next we multiply by the diagonal element diag $\{c\mu, 1, (c\mu)^3\}$. The constant $c$ has been chosen in order that $\mu$ is the multiplier of the transformation acting on the curve.

```
>   U:=(c*mu)*(m[1]*X+m[2]*Y+m[3]*Z):
>   V:=m[4]*X+m[5]*Y+m[6]*Z:
>   W:=(c*mu)^3*(m[7]*X+m[8]*Y+m[9]*Z):
```

We express the previous transform in the affine coordinates $(x, z)$.

```
>   f:=unapply(U/V, x, z):
>   y:=1:
>   g:=unapply(W/V, x, z):
```

We check that the cubic $y^2 z = x^3$ is fixed.

```
>   simplify(g(t, t^3)-f(t, t^3)^3);
```

0

Computation of the multiplier.

```
>   factor(f(t, t^3));
```

$$1/3 \, (1 + 3t + b + a)\mu$$

Definition of the translation factor.

```
>   epsilon:=(a+b+1)*mu/3;
```

$$\epsilon := 1/3 \, (a + b + 1)\mu$$

Definition of the points $p_i^-$.

```
>   p1:=(1-2*a-2*b)*mu/3;
>   p2:=(-2+a-2*b)*mu/3;
>   p3:=(-2-2*a+b)*mu/3;
```

$$p1 := 1/3 \, (-2b + 1 - 2a)\mu$$

$$p2 := 1/3 \, (a - 2 - 2b)\mu$$

$$p3 := 1/3 \, (b - 2 - 2a)\mu$$

### 2.2.2 The fixed point

Computation of the smooth fixed point on the cubic.

```
> solve(f(t, t^3)-t, t);
```

$$-1/3\,{\frac { \left( a+b+1 \right) \mu}{\mu-1}}$$

This point is $(p, q)$ in the coordinates $(x, z)$.

```
> p:=-(1/3)*mu*(a+b+1)/(-1+mu);
> q:=p^3;
```

$$p := -1/3\,{\frac { \left( a+b+1 \right) \mu}{\mu-1}}$$

$$q := -1/27\,{\frac { \left( a+b+1 \right)^{3}{\mu}^{3}}{ \left( \mu-1 \right) ^{3}}}$$

We check that $\mu$ is an eigenvalue of the differential at a fixed point.

```
> factor((coeftayl(f(x,z), [x, z]=[p, q], [1, 0])-mu)*(coeftayl(g(x,z), [x, z]=[p, q], [0,
1])-mu)-coeftayl(g(x,z), [x, z]=[p, q], [1, 0])*coeftayl(f(x,z), [x, z]=[p, q], [0, 1]));
```

0

Calculation of the other eigenvalue: we divide the determinant by $\mu$.

```
> zeta:=simplify((((coeftayl(f(x,z), [x, z]=[p, q], [1, 0]))*(coeftayl(g(x,z), [x, z]=[p, q],
[0, 1]))-coeftayl(g(x,z), [x, z]=[p, q], [1, 0])*coeftayl(f(x,z), [x, z]=[p, q], [0, 1])))
/mu:
```

## 2.3 Computations for the permutation (1)(2)(3)

```
> a:='a':  b:='b':
```

### 2.3.1 Computations of a and b

```
> solve({mu^(i-1)*(p1-p)+p=1, mu^(j-1)*(p2-p)+p=a}, {a, b});
```

$$\{a = {\frac{2\,{\mu}^{j-1}{\mu}^{2}{\mu}^{i-1} - 3\,{\mu}^{j-1}\mu\,{\mu}^{i-1} - 2\,{\mu}^{j-1}\mu + {\mu}^{i-1}\mu - 1 + 3\,{\mu}^{j-1}}{\left(2\,{\mu}^{i-1}\mu - 3\,{\mu}^{i-1} + 1\right)\left({\mu}^{j-1}\mu - 1\right)}}, b = -{\frac{{\mu}^{j-1}{\mu}^{3}{\mu}^{i-1} - 3\,{\mu}^{j-1}{\mu}^{2}{\mu}^{i-1} + 2\,{\mu}^{j-1}{\mu}^{2} + 2\,{\mu}^{i-1}{\mu}^{2}}{\left(2\,{\mu}^{i-1}\mu - 3\,{\mu}^{i-1} + 1\right)\mu\left({\mu}^{j-1}\mu - 1\right)}}$$

```
> a:=simplify((2*mu^(j-1)*mu^2*mu^(i-1)-3*mu^(j-1)*mu*mu^(i-1)-2*mu^(j-1)*mu+mu^(i-1)*mu-1+
3*mu^(j-1))/((2*mu^(i-1)*mu-3*mu^(i-1)+1)*(mu^(j-1)*mu-1)));
```

$$a := {\frac{2\,{\mu}^{j+1+i} - 3\,{\mu}^{j+i} - 2\,{\mu}^{j+1} + {\mu}^{i+1} - \mu + 3\,{\mu}^{j}}{\left({\mu}^{j} - 1\right)\left(2\,{\mu}^{i+1} - 3\,{\mu}^{i} + \mu\right)}}$$

```
> b:=simplify(-(mu^(j-1)*mu^3*mu^(i-1)-3*mu^(j-1)*mu^2*mu^(i-1)+2*mu^(j-1)*mu^2+2*mu^(i-1)
*mu^2-5*mu+3)/((2*mu^(i-1)*mu-3*mu^(i-1)+1)*mu*(mu^(j-1)*mu-1)));
```

$$b := {\frac{-{\mu}^{j+2+i} + 3\,{\mu}^{j+1+i} - 2\,{\mu}^{j+2} - 2\,{\mu}^{2+i} + 5\,{\mu}^{2} - 3\,\mu}{\mu\left({\mu}^{j} - 1\right)\left(2\,{\mu}^{i+1} - 3\,{\mu}^{i} + \mu\right)}}$$

Verification: $\mu^{k-1} \times (p3 - p) + p = b$ when $\mu$ is a root of $P_\tau$.

```
> factor(numer(expand((mu^(k-1)*(p3-p)+p-b))));
```

$$-3\,(\mu - 1)(2\,\mu - {\mu}^{i}\mu - {\mu}^{j}\mu - {\mu}^{k}\mu + {\mu}^{k}\mu\,{\mu}^{j}{\mu}^{i} + {\mu}^{k}{\mu}^{i} - 1 + {\mu}^{j}{\mu}^{i} - 2\,{\mu}^{k}{\mu}^{j}{\mu}^{i} + {\mu}^{k}{\mu}^{j})$$

### 2.3.2 Computation of $\zeta$

```
> factor(expand(zeta*mu^(i+j+k-3)-1));
```

$$-\frac{-\mu^i\mu-\mu^j\mu+2\mu-\mu^k\mu+\mu^k\mu^j\mu^i+\mu^k\mu^i+\mu^k\mu^j+\mu^j\mu^i-2\mu^k\mu^j\mu^i-1}{-\mu^i\mu+2\mu-\mu^j\mu-1+\mu^j\mu^i}$$

## 2.4 Computations for the permutation (123)

```
> a:='a':  b:='b':
```

### 2.4.1 Computations of `a` and `b`

```
> solve({mu^(i-1)*(p1-p)+p=a, mu^(j-1)*(p2-p)+p=b}, {a, b});
```

$$\{a = \frac{(2\mu^{j-1}\mu^2\mu^{i-1}+2\mu^{i-1}\mu-3\mu^{j-1}\mu\mu^{i-1}-1)\mu}{2\mu^{j-1}\mu^3\mu^{i-1}-3\mu^{j-1}\mu^2\mu^{i-1}+2\mu^{i-1}\mu^2+3\mu^{j-1}\mu^2-3\mu^{i-1}\mu-3\mu^{j-1}\mu+5\mu-3}, b =$$

$$-\frac{(\mu^{j-1}\mu^2\mu^{i-1}-3\mu^{j-1}\mu\mu^{i-1}+\mu^{i-1}\mu-3\mu^{j-1}+1+3\mu^{j-1}\mu)\mu}{2\mu^{j-1}\mu^3\mu^{i-1}-3\mu^{j-1}\mu^2\mu^{i-1}+2\mu^{i-1}\mu^2+3\mu^{j-1}\mu^2-3\mu^{i-1}\mu-3\mu^{j-1}\mu+5\mu-3}\}$$

```
> a:=(2*mu^(j-1)*mu^2*mu^(i-1)+2*mu^(i-1)*mu-3*mu^(j-1)*mu*mu^(i-1)-1)*mu/(2*mu^(j-1)*mu^3*
mu^(i-1)-3*mu^(j-1)*mu^2*mu^(i-1)+2*mu^(i-1)*mu^2+3*mu^(j-1)*mu^2-3*mu^(i-1)*mu-3*mu^(j-1)
*mu+5*mu-3);
```

$$a := \frac{(2\mu^{j-1}\mu^2\mu^{i-1}+2\mu^{i-1}\mu-3\mu^{j-1}\mu\mu^{i-1}-1)\mu}{2\mu^{j-1}\mu^3\mu^{i-1}-3\mu^{j-1}\mu^2\mu^{i-1}+2\mu^{i-1}\mu^2+3\mu^{j-1}\mu^2-3\mu^{i-1}\mu-3\mu^{j-1}\mu+5\mu-3}$$

```
> b:=-(mu^(j-1)*mu^2*mu^(i-1)-3*mu^(j-1)*mu*mu^(i-1)+mu^(i-1)*mu-3*mu^(j-1)+1+3*mu^(j-1)*
mu)*mu/(2*mu^(j-1)*mu^3*mu^(i-1)-3*mu^(j-1)*mu^2*mu^(i-1)+2*mu^(i-1)*mu^2+3*mu^(j-1)*mu^2-3*
mu^(i-1)*mu-3*mu^(j-1)*mu+5*mu-3);
```

$$b := -\frac{(\mu^{j-1}\mu^2\mu^{i-1}-3\mu^{j-1}\mu\mu^{i-1}+\mu^{i-1}\mu-3\mu^{j-1}+1+3\mu^{j-1}\mu)\mu}{2\mu^{j-1}\mu^3\mu^{i-1}-3\mu^{j-1}\mu^2\mu^{i-1}+2\mu^{i-1}\mu^2+3\mu^{j-1}\mu^2-3\mu^{i-1}\mu-3\mu^{j-1}\mu+5\mu-3}$$

Verification: $\mu^{k-1} \times (p3 - p) + p = 1$ when $\mu$ is a root of $P_\tau$.

```
> factor(numer(expand((mu^(k-1)*(p3-p)+p-1))));
```

$$-3(\mu-1)(\mu^k\mu\mu^i+\mu^k\mu\mu^j+\mu^j\mu\mu^i+\mu^i\mu+2\mu+\mu^j\mu+\mu^k\mu+\mu^k\mu\mu^j\mu^i-\mu^k\mu^i-\mu^i-\mu^k\mu^j-\mu^k-\mu^j\mu^i-1-\mu^j-2\mu^k\mu^j\mu^i)$$

### 2.4.2 Computations of $\zeta$

```
> factor(expand(zeta*mu^(i+j+k-3)-1));
```

$$-\frac{\mu^i\mu+2\mu+\mu^k\mu+\mu^j\mu\mu^i+\mu^j\mu+\mu^k\mu\mu^j\mu^i+\mu^k\mu\mu^j+\mu^k\mu\mu^i-\mu^k\mu^j-\mu^i-1-\mu^j-\mu^k-\mu^k\mu^i-2\mu^k\mu^j\mu^i-\mu^j\mu^i}{\mu^i\mu+2\mu+\mu^j\mu\mu^i+\mu^j\mu-\mu^i-1-\mu^j-\mu^j\mu^i}$$

## 2.5 Computations for the permutation (1)(23)

```
> a:='a':  b:='b':
```

### 2.5.1 Computations of `a` and `b`

```
> solve({mu^(i-1)*(p1-p)+p=1, mu^(j-1)*(p2-p)+p=b}, {a, b});
```

$$\{a = \frac{2\mu^{j-1}\mu^3\mu^{i-1}-2\mu^{j-1}\mu^2+2\mu^{i-1}\mu^2-3\mu^{j-1}\mu^2\mu^{i-1}-5\mu+3\mu^{j-1}\mu+3}{(2\mu^{i-1}\mu-3\mu^{i-1}+1)(\mu^{j-1}\mu+1)\mu}, b =$$

$$-\frac{\mu^{j-1}\mu^2\mu^{i-1}+2\mu^{j-1}\mu+\mu^{i-1}\mu-3\mu^{j-1}\mu\mu^{i-1}-1}{(2\mu^{i-1}\mu-3\mu^{i-1}+1)(\mu^{j-1}\mu+1)}\}$$

```
> a:=(2*mu^(j-1)*mu^3*mu^(i-1)-3*mu^(j-1)*mu^2*mu^(i-1)+2*mu^(i-1)*mu^2-2*mu^(j-1)*mu^2-5*
mu+3*mu^(j-1)*mu+3)/((-3*mu^(i-1)+2*mu^(i-1)*mu+1)*(mu^(j-1)*mu+1)*mu);
```

$$a := \frac{2\mu^{j-1}\mu^3\mu^{i-1}-2\mu^{j-1}\mu^2+2\mu^{i-1}\mu^2-3\mu^{j-1}\mu^2\mu^{i-1}-5\mu+3\mu^{j-1}\mu+3}{(2\mu^{i-1}\mu-3\mu^{i-1}+1)\quad(\mu^{j-1}\mu+1)\mu}$$

```
> b:=-(mu^(j-1)*mu^2*mu^(i-1)-3*mu^(j-1)*mu*mu^(i-1)+mu^(i-1)*mu+2*mu^(j-1)*mu-1)/((-3*mu^(
i-1)+2*mu^(i-1)*mu+1)*(mu^(j-1)*mu+1));
```

$$b := -\frac{\mu^{j-1}\mu^2\mu^{i-1} + 2\mu^{j-1}\mu + \mu^{i-1}\mu - 3\mu^{j-1}\mu\mu^{i-1} - 1}{(2\mu^{i-1}\mu - 3\mu^{i-1} + 1)\quad(\mu^{j-1}\mu + 1)}$$

Verification: $\mu^{k-1} \times (p_3 - p) + p = a$ when $\mu$ is a root of $P_\tau$.

> `factor(numer(expand((mu^(k-1)*(p3-p)+p-a))));`

$$-3\mu(\mu-1)(-2\mu + \mu^j\mu\mu^i + \mu^i\mu - \mu^j\mu - \mu^k\mu + \mu^k\mu\mu^i + \mu^k\mu\mu^j\mu^i + \mu^j + 1 - \mu^j\mu^i + \mu^k\mu^j - \mu^k\mu^i - 2\mu^k\mu^j\mu^i + \mu^k)$$

### 2.5.2 Computations of $\zeta$

> `factor(expand(zeta*mu^(i+j+k-3)-1));`

$$-\frac{\mu^i\mu - 2\mu - \mu^k\mu + \mu^j\mu\mu^i - \mu^j\mu + \mu^k\mu\mu^j\mu^i + \mu^k\mu\mu^i + \mu^j + \mu^k\mu^j + 1 + \mu^k - \mu^k\mu^i - 2\mu^k\mu^j\mu^i - \mu^j\mu^i}{\mu^j\mu\mu^i + \mu^i\mu - \mu^j\mu^i - 2\mu - \mu^j\mu + \mu^j + 1}$$

CNRS, I2M (Marseille) & IHÉS
*E-mail address*: `jgrivaux@math.cnrs.fr`



\end{document}